%% file: loubert-berta.tex
\documentclass[reqno,12pt]{amsart}
\pdfoutput=1
\usepackage{hbpaper,verbatim}
\usepackage{enumitem,genyoungtabtikz,pdflscape,rotating,leftindex,relsize,xcolor,etoolbox,tikz-cd}
\usepackage{amscd}
\usepackage[all]{xy}
\usetikzlibrary{shapes.geometric, arrows}
\usetikzlibrary{calc}
\usetikzlibrary{cd}

\definecolor{answercolor}{RGB}{0, 112, 48}
\excludecomment{answer}      

\allowdisplaybreaks

\def\Ht{{\operatorname{ht}}}

\DeclareMathOperator\ldom {\triangleleft}
\DeclareMathOperator\ledom{\trianglelefteq}

\def\bi{\text{\boldmath$i$}}
\def\bj{\text{\boldmath$j$}}

\def\b1{\text{\boldmath$1$}}
\def\pmod#1{\text{ }(\text{\rm mod } #1)\,}
\def\Belt{\mathbf{B}}
\def\caly{\mathcal{Y}}
\def\calx{\mathcal{X}}

\newcommand{\tij}{\mathcal{T}_{i,j}}
\newcommand\Par[1][]{{\mathcal P}}
\newcommand{\cont}{\operatorname{cont}}
\newcommand{\std}{\operatorname{Std}}

\newcommand{\cmA}{\mathsf{A}}

\newcommand{\ct}{c_{\mathtt{T}}}

\newcommand{\vtt}{v_{\mathtt{T}}}
\newcommand{\vtd}[1]{v_{{\mathtt{T}_{(#1)}}}}

\newcommand{\tj}[2]{\tau_{(#1,#2)}}
\newcommand{\tu}[2]{\tau^{(#1,#2)}}
\newcommand{\vt}[1]{v_{\tau({#1})}}

\newcommand{\vs}{v_{\mathtt{S}}}

\newcommand{\vtp}{v_{\mathtt{T'}}}
\newcommand{\vtb}{v_{\bar{\mathtt{T}}}}
\newcommand{\vst}{v_{{\tts_{(2)}}}}
\newcommand{\vsk}{v_{{\tts_{(3)}}}}
\newcommand{\vsf}{v_{{\tts_{(4)}}}}

\newcommand{\Arm}{\operatorname{Arm}}
\newcommand{\Leg}{\operatorname{Leg}}
\newcommand{\Psichaind}[2]{\Psi\hspace{-4pt}\underset{\scriptscriptstyle{#2}}{\overset{\scriptscriptstyle{#1}}{\downarrow}}}
\newcommand{\Psichainu}[2]{\Psi\hspace{-4pt}\underset{\scriptscriptstyle{#1}}{\overset{\scriptscriptstyle{#2}}{\uparrow}}}
\newcommand{\tp}{\mathtt{p}}

\newcommand{\last}{\mathfrak{l}}
\newcommand{\first}{\mathfrak{f}}
\renewcommand{\P}{\sfp}
\newcommand{\id}{\operatorname{id}}

\renewcommand\hom{\operatorname{Hom}_{R_n^{\La}}}
\newcommand{\Boxedi}[2][]{\boxed{\quad\cali_{g_{#2}#1,h_{#2}}\quad}}
\newcommand{\iBoxed}[2][]{\boxed{\quad\cali^{G_{#2}#1,H_{#2}}\quad}}
\newcommand{\lzero}{\underline{0}}

\newcommand\md{\operatorname{mod}}

\numberwithin{equation}{section}

\newcommand\SetBox[2][35mm]{\Big\{\vcenter{\hsize#1\centering#2}\Big\}}

\newcommand\YoungDiagram[2][\relax]{
  \begin{tikzpicture}[scale=0.5,draw/.append style={thick,black}]
    \ifx\relax#1\relax%
    \else 
    \foreach\box in {#1} {
      \filldraw[blue!30]\box rectangle ++(1,1);
    }
    \fi
    \newcount\row
    \row=0
    \foreach \col in {#2} {
       \draw(1,\the\row)grid ++(\col,1);
       \global\advance\row by -1
    }
  \end{tikzpicture}
}

\newcommand\Tableau[2][\relax]{
  \begin{tikzpicture}[scale=0.5,draw/.append style={thick,black}]
    \ifx\relax#1\relax%
    \else 
      \foreach\box in {#1} { \filldraw[blue!30]\box+(-.5,-.5)rectangle++(.5,.5); }
    \fi
    \newcount\row\newcount\col
    \row=0
    \foreach \Row in {#2} {
       \col=1
       \foreach\k in \Row {
          \draw(\the\col,\the\row)+(-.5,-.5)rectangle++(.5,.5);
          \draw(\the\col,\the\row)node{\k};
          \global\advance\col by 1
       }
       \global\advance\row by -1
    }
  \end{tikzpicture}
}

\title[Homomorphisms into Specht modules labelled by hooks]{Homomorphisms into Specht modules labelled by hooks in quantum characteristic two}
\author[Berta Hudak]{Berta Hudak}
\address{Okinawa Institute of Science and Technology, Okinawa 904-0495, Japan}
\email{berta.hudak@oist.jp}
\date{}

\begin{document}


\renewcommand\auth{Berta Hudak}

\let\thefootnote\relax

\begin{abstract}
Let $R_n$ denote the KLR algebra of type $A^{(1)}_{e-1}$.
Using the presentation of Specht modules given by Kleschev--Mathas--Ram, Loubert completely determined $\Hom_{R_n}(S^\mu,S^\la)$ where $\mu$ is an arbitrary partition, $\la$ is a hook and $e\neq2$.
In this paper, we investigate the same problem when $e=2$.
First we give a complete description of the action of the generators on the basis elements of $S^\la$.
We use this result to identify a large family of partitions $\mu$ such that there exists at least one non-zero homomorphism from $S^\mu$ to $S^\la$, explicitly describe these maps and give their grading.
Finally, we generalise James's result for the trivial module.
\end{abstract}

\maketitle

\tableofcontents

\section{Introduction}\label{intro}

\input{parts/intro}

\section{Background}\label{backgr}

\input{parts/prelim}

\input{parts/klr-alg}

\section{Specht modules indexed by hook partitions}\label{hookaction}

\input{parts/hook-action}

\section{Homomorphisms into \texorpdfstring{$S^\la$}{S-lambda}}\label{hookhoms}

\input{parts/homs-1}

\input{parts/homs-2}

\bibliographystyle{amsalpha}
\bibliography{master}

\end{document}

%% file: parts/intro.tex
The {\em Iwahori--Hecke algebra} $\calh_n$ of type $A$ arises in various mathematical settings, and its representation theory is very similar to the modular representation theory of the symmetric group $\fkS_ n$.
Let $q$ denote the deformation parameter of $\calh_n$ and set $e$ to be the smallest integer such that $1+q+q^2+\dots+q^{e-1}=0$ (if no such integer exists, we set $e=0$).
For every partition $\la$ of $n$ we can associate a corresponding $\calh_n$-module, called a {\em Specht module} $S^\la$.

Brundan and Kleshchev showed in \cite{bkisom}
that over an arbitrary field $\bbf$, the algebra $\calh_n$ is isomorphic to a certain algebra $R^\La_n$, known as the \emph{cyclotomic Khovanov--Lauda--Rouquier} (KLR) algebra of type $A$.
Moreover, $R^{\La}_n$ is $\bbz$-graded and this grading can be transferred to $\calh_n$ via the mentioned isomorphism. 
In \cite{bkw11}, Specht modules were described as graded modules over $R^{\La}_n$ and this result was extended in \cite{kmr} where the authors gave a construction of Specht modules in terms of generators and relations.

In this paper, we use this presentation to consider $\hom(S^\mu,S^\la)$ where $\mu$ is an arbitrary partition, $\la$ is a hook and $e=2$.
The significantly simpler situation when the quantum characteristic is at least $3$ was solved earlier by Loubert \cite{loub17}.
The defining relations of $R^\La_n$ are different when $e=2$, so it requires a completely separate treatment and our work turns out to be a lot more complicated than the $e\ge3$ case.
In particular, when $e\ge3$ $\hom(S^\mu,S^\la)$ is at most one-dimensional with any nonzero homomorphism sending the cyclic generator of $S^\mu$ to a single standard basis vector in $S^\la$.
As we will later see, here this is different: dimensions of the homomorphism space quickly grow very large and each map involves an integer linear combination of standard basis vectors.
Nevertheless, in the case when $\mu$ is a \emph{two-column hook partition}, a partition of the form $(c,2^s,1^d)$, we completely determine when $\hom(S^\mu,S^\la)$ is nontrivial if $c$ is even or $s=0$ and construct at least one homomorphism explicitly.

\vspace{10pt}
\noindent
{\bf Main Theorem}
\emph{
Let $\la=(a,1^b)$ and $\mu=(c,2^s,1^d)$ where if $s>0$ then $c$ is even. Then $\hom(S^\mu,S^\la)$ is at least one dimensional if for $c$ even, we have that $a\ge s+2$; and
\begin{itemize}
    \item $n$ is odd, $n+a\equiv_2 s+1$ and $s+1\le a\le n-(s+1)$, or
    \item $n$ is even and $s+2\le a\le c+s$ or $a=c+s+1+2y$ for $0\le y\le d/2-2$,
\end{itemize}
and zero otherwise.
}
\vspace{10pt}

(The above theorem appears as \cref{thm:main} later.)

In the special cases when $\la$ is a single row or column, our result allows us to determine whether $S^\mu$ has the trivial module as a quotient.
Furthermore, by the following isomorphism, it also enables us to immediately see whether there exists a non-trivial homomorphism in the other direction (see \cite{loub17} for more details).
Let $\la'$ denote the {\em conjugate} of $\la$, which is the partition obtained from $\la$ by swapping its rows and columns.
Then we have that
\[
\hom(S^\mu,S^\la) \cong \hom(S^{\la'}, S^{\mu'}).
\]
By the above isomorphism, it is easy to apply \cref{thm:main} to determine whether $S^\mu$ has a one-dimensional submodule.\\

\textbf{Structure of the paper} In \cref{backgr}, we set up notation and collect all necessary background material.
In particular, we introduce cyclotomic KLR algebras of type $A$ and the universal Specht modules, following \cite{kmr}.
In \cref{hookaction}, we completely determine the action of the generator elements on the basis elements $\vtt$ of $S^\la$.
Our findings (see \cref{action}) generalise the results in \cite{ls14} and \cite[Chapter~3.]{thesis} for all standard tableaux of hook shapes and all parities of $n$.
In \cref{hookhoms}, we first discover a novel way to label hook tableaux by partitions and a rule to assign a coefficient to each one of them.
This leads us to completely determine the homomorphisms from $S^\mu$ to $S^\la$ where $\mu$ is a two-column hook partition, that is a partition of the form $(c,2^s,1^d)$ if $c$ is even or $s=0$ (see \cref{thm:main}).
Finally, we generalise James's well-known result for the trivial module.
Using a different technique inspired by \cite{dg15}, in the special case when $\mu$ is a two-part partition and $n$ is odd, we are able to prove when the homomorphism space $\hom(S^\mu,S^\la)$ is non-trivial without explicitly describing the maps.\\

\textbf{ArXiv version}
We decided to omit some of the more lengthy but less informative calculations from this pdf as they are similar to others included.
However for those interested, we suggest downloading the \LaTeX~source and find the toggle near the beginning of the file which allows one to compile the paper with these details included.
\\

\textbf{Acknowledgements}
The author thanks Dr Liron Speyer for his helpful comments and inspiring correspondence.
She would also like to thank Dr Haralampos Geranios for pointing out the argument given in \cref{lem:tworows} after most of the work in this paper has been completed.

%% file: parts/prelim.tex

For $m,m'\in\bbz$, we write $m\equiv_2 m'$ if $m-m'$ is even and $m\not\equiv_2 m'$ otherwise.

\subsection{Lie theoretical notation}

Take $e=2$ and $I:=\bbz/e\bbz=\{0,1\}$.
Let $\Gamma$ be the quiver of type $A_{e-1}^{(1)}=A_{1}^{(1)}$ with vertex set $I$ as depicted below.

\[
\begin{tikzcd}
\qquad
0 \arrow[r, shift left]
& 1 \arrow[l, shift left]
\end{tikzcd}
\]

The corresponding {\em Cartan matrix} 
$\cmA =(a_{i,j})_{i, j \in I}$ is given by 
\[
\begin{pmatrix}
    2&-2\\
    -2&2
\end{pmatrix}.
\]

Let $(\cmA,P,\Pi,\Pi^\vee)$ be a realization of the Cartan matrix $(a_{ij})_{i,j\in I}$, so we have the {\em simple roots} 
$\Pi=\{\al_0,\al_1\},$ and  
the {\em fundamental dominant weights} 
$\{\La_0,\La_1\}$ in the \emph{weight lattice} $P$.
Let $Q^+:= \bigoplus_{i\in I} \bbz_{\ge0} \alpha_i$ be the \emph{positive cone of the root lattice}.

For $\al = \sum_{i\in I} a_i \alpha_i \in Q^+$, let $\Ht(\al)$ denote the {\em height of} $\al$, that is $\Ht(\al) = \sum_{i\in I} a_i$.

Let $\fkS_n$ denote the {\em symmetric group} on $n$ letters and let $s_r=(r,r+1)$, for $1\le r<n$, be the simple transpositions of~$\fkS_n$. Then $\fkS_n$
acts from the left on the set $ I^n$ by place permutations. If $\bi=(i_1,\dots,i_n) \in I^n$
then its {\em weight} is $|\bi|:=\alpha_{i_1}+\cdots+\alpha_{i_n}\in Q^+$.  Then the
$\fkS_n$-orbits on $I^n$ are the sets 
\begin{equation*} I^\alpha := \{\bi \in I^n \mid \al=|\bi|\} 
\end{equation*} 
parameterised by all $\alpha \in Q^+$ of height $n$.

\subsection{Partitions and tableaux}\label{Par}
Let $\Par_n$ be the set of all partitions of $n$ and fix $\Par:=\bigsqcup_{n\ge 0}\Par_n$. 
The {\em Young diagram} of the partition $\mu\in \Par$ is defined as
$$
[\mu]:=\{(r,c)\in\bbz_{>0}\times\bbz_{>0}\mid 1\le c\le \mu_r\}.
$$
We call the elements of this set the {\em nodes of $\mu$}.
To each node $A=(r,c)\in[\mu]$ 
we associate its {\em residue} given by:
$$
\res A=(c-r)\; \pmod{2}.
$$
An {\em $i$-node} is a node of residue $i$.
Define the {\em residue content of $\mu$} to be
$$
\cont(\mu):=\sum_{A\in\mu}\al_{\res A} \in Q^+.
$$

A node $A\in[\mu]$ is a {\em removable node} (of~$\mu$), if $[\mu]\setminus \{A\}$ is (the diagram of) a partition. A node $B\not\in[\mu]$ is an {\em addable node} (of $\mu$), if $[\mu]\cup \{B\}$ is a partition.
We use the notation
$$
\mu_A:=[\mu]\setminus \{A\}.
$$

If $\mu,\nu\in \Par_n$, we say that $\mu$ {\em dominates} $\nu$, and write $\mu\unrhd\nu$, if 
$$
\sum_{j=1}^i \mu_j \ge \sum_{j=1}^i \nu_j 
$$
for all $i \ge 1$.
In other words, $\mu$ dominates $\nu$ if $[\mu]$ can be obtained from $[\nu]$ by moving nodes up in the diagram.

If $\mu=(\mu_1,\dots,\mu_k)\in\Par_n$, then the {\em conjugate} of $\mu$ is the partition
$\mu'\in \Par_n$ where $\mu'$ is obtained by swapping the rows and columns of $\mu$. 

Let $\mu\in\Par_n$. 
A {\em $\mu$-tableau} $\ttt$ is obtained from $[\mu]$ by 
inserting the integers $1,\dots,n$ into the nodes without any repeats. 
If the node $A=(r,c)\in[\mu]$ is occupied by the integer $m$ in $\ttt$, we write $m=\ttt(r,c)$ and set
$\res_\ttt(m)=\res A$. The {\em residue sequence} of~$\ttt$ is
$
\bi^\ttt=(i_1,\dots,i_n)\in I^n,
$
where $i_m=\res_\ttt(m)$ is the residue of the node occupied by 
$m$ in $\ttt$ ($1\le m\le n$).

A $\mu$-tableau $\ttt$ is {\em row-standard} (resp.\ {\em column-standard}) if its entries increase from left to right (resp.\ from top to bottom) along the rows (resp.\ columns) of each component of $\ttt$. 
A $\mu$-tableau $\ttt$ is {\em standard} if it is both row- and column-standard. 
 Let $\std(\mu)$ be the set of standard $\mu$-tableaux.

For $\mu\in\Par$, $i\in I$, and a removable $i$-node $A \in [\mu]$, we define

\[
d_A(\mu)=\#\SetBox{addable $i$-nodes of\\[-4pt] $\mu$ strictly below $A$}
                -\#\SetBox[38mm]{removable $i$-nodes of\\[-4pt]$\mu$ strictly below $A$}.
\]
\vspace{1mm}

Given $\mu \in \Par_n$ and $\ttt \in \std(\mu)$, the {\em degree} of $\ttt$ is defined inductively. If $n=0$, then $\ttt=\emptyset$, and set $\deg(\ttt):=0$.
Otherwise, let $A$ be the node occupied by $n$ in $\ttt$. Let $\ttt_{<n}\in\std(\mu_A)$
be the tableau obtained from $\ttt$ by removing this node and set
\[
\deg(\ttt):=d_A(\mu)+\deg(\ttt_{<n}).
\]

The symmetric group $\fkS_n$ acts on the set of $\mu$-tableaux from the left by acting on the entries of the tableaux.
Let $\ttt^\mu$ be the unique $\mu$-tableau in which the numbers $1,2,\dots,n$ appear in order from left to right along the successive rows,
working from top to bottom.
We call $\ttt^\mu$ the \emph{(row) initial tableau}.

Set 
$$
\bi^\mu:=\bi^{\ttt^\mu}.
$$
For each $\mu$-tableau $\ttt$ define permutations $w^\ttt \in \fkS_n$ by the equation
$$
w^\ttt  \ttt^\mu=\ttt.
$$

\subsection{Bruhat order}
Let $\ell$ be the length function on $\fkS_n$ with respect to the Coxeter generators
$s_1,s_2,\dots,s_{n-1}$. Let $\le$ be the {\em Bruhat order} on $\fkS_n$ (so that $1\le w$ for all
$w\in\fkS_n)$.
Define a related
partial order on $\std(\mu)$ as follows: if $\tts,\ttt\in\std(\mu)$ then
$$
\ttt \ledom \tts \quad \text{if and only if} \quad w^{\tts} \le w^\ttt.
$$
If $\tts\ledom\ttt$ and $\tts\ne\ttt$, we write $\tts\ldom\ttt$.
Observe that if  $\ttt\in \std(\mu)$ then $\ttt\ledom\ttt^\mu$.

%% file: parts/klr-alg.tex

\subsection{KLR algebras}\label{relations}
We mostly follow the notations and conventions of \cite{kmr}.
Let $\bbf$ be a field and fix $\al\in Q^+$ such that $\Ht(\al)=n$.
We define the algebra $R_\al$ to be the unital associative $\bbf$-algebra generated by elements
\[
\{e(\bi)\:|\: \bi\in I^\al\}\cup\{y_1,\dots,y_{n}\}\cup\{\psi_1, \dots,\psi_{n-1}\},
\]
subject only to the following relations:

\begin{align*}
e(\bi) e(\bj) &= \de_{\bi,\bj} e(\bi);\\
\sum_{\bi \in I^\al} e(\bi) &= 1;\\
y_r e(\bi) &= e(\bi) y_r;\\
\psi_r e(\bi) &= e(s_r{ }\bi) \psi_r;\\
y_r y_s &= y_s y_r;\\
\psi_r y_s  &= y_s \psi_r\hspace{5cm}\text{if $s \neq r,r+1$};\\
\psi_r \psi_s &= \psi_s \psi_r\hspace{4.9cm}\text{if }r\neq s\pm1;\\
\psi_r y_{r+1} e(\bi) &=  (y_r\psi_r+\de_{i_r,i_{r+1}})e(\bi)\\
y_{r+1} \psi_re(\bi) &= (\psi_ry_r+\de_{i_r,i_{r+1}})e(\bi)\\
\psi_r^2e(\bi) &= 
\begin{cases}
0&\quad\text{if }i_r = i_{r+1},\\
(y_{r+1} - y_{r})(y_{r}-y_{r+1}) e(\bi)&\quad\text{if }i_r \neq i_{r+1};
\end{cases}
\\
(\psi_{r}\psi_{r+1} \psi_{r}-\psi_{r+1} \psi_{r} \psi_{r+1} ) e(\bi)
&=
\begin{cases}
    (y_r-2y_{r+1}+y_{r+2}\big)e(\bi)&\qquad\;\text{if }i_{r+2}=i_r \neq i_{r+1},\\
    e(\bi)&\qquad\;\text{otherwise}.
\end{cases}
\end{align*}
\vspace{.5cm}
for all admissible $\bi,\bj,r,s$.

We define the {\em affine Khovanov--Lauda--Rouquier (KLR) algebra} $R_n$ of type $A^{(1)}_1$ to be the direct sum
\[
R_n:=\bigoplus_{\substack{\al\in Q^+\\ \Ht(\al)=n}}R_\al
\]

Fix $\La=\La_0\in P^+$.
The corresponding {\em cyclotomic KLR algebra (of type $A$)} $R_n^{\La}$ is defined to be the quotient of $R_n$ subject to the additional {\em cyclotomic relations} (of level one)
\begin{equation}
y_1=0\qquad\text{and}\qquad e(\bi)=0\text{ if } i_1\neq0.
\end{equation}

The algebras $R_n$ and $R_n^{\La}$ are non-trivially $\bbz$-graded, by setting 
$e(\bi)$ to be of degree 0,
$y_r$ of degree $2$, and
$\psi_r e(\bi)$ of degree $-a_{i_r,i_{r+1}}$
for all $r$ and $\bi \in I^\al$.

For the rest of this paper, we fix a {\em preferred reduced expression} $w=s_{r_1}\dots s_{r_m}$ for each $w \in \fkS_n$ ($1 \le r_1, \dots, r_m < n$) and we define the elements $\psi_w := \psi_{r_1}\dots \psi_{r_m} \in R_n$.

Furthermore, for $\mu\in\Par_n$ and some tableau $\ttt$ of shape $\mu$, we set $ \psi^\ttt:=\psi_{w^\ttt}$.

\subsection{Universal Specht modules}\label{SecSpecht}
Fix a partition $\mu \in \Par_n$.
A node $A=(r,c) \in \mu$ is called a \emph{Garnir node} of $\mu$ if we also have $(r+1,c) \in \mu$.
We define the \emph{Garnir belt} of $A$ to be the set $\Belt^A$ of nodes of $\mu$ containing $A$ and all nodes directly to the right of $A$, along with the node directly below $A$ and all nodes directly to the left of this node:
\[
\Belt^A = \{(r,s) \in \mu \mid c \le s \le \mu_r\} \cup \{(r+1, s) \in \mu \mid 1 \le s \le c \}.
\]

For example, if $A=(2,3)$ and $\mu=(6,5,3,2)$ then $\Belt^A$ consists of the highlighted nodes:

$$\begin{array}{l}
\begin{tikzpicture}[scale=0.5,draw/.append style={thick,black}]
  \fill[orange!50](2.5,-1.5)rectangle(5.5,-0.5);
  \fill[orange!50](0.5,-2.5)rectangle(3.5,-1.5);
  \newcount\col
  \foreach\Row/\row in {{,,,,,}/0,{,,A,,}/-1,{,,}/-2,{,}/-3} {
     \col=1
    \foreach\k in \Row {
        \draw(\the\col,\row)+(-.5,-.5)rectangle++(.5,.5);
        \draw(\the\col,\row)node{\k};
    \global\advance\col by 1
     }
  }
   \draw[magenta,line width=2pt](2.5,-1.5)--(2.5,-0.5)--(5.5,-0.5)--(5.5,-1.5)--(3.5,-1.5)
                     --(3.5,-2.5)--(0.5,-2.5)--(0.5,-1.5)--(2.5,-1.5);
\end{tikzpicture}
\end{array}$$

We define an \emph{$A$-brick} to be a set of $e$ successive nodes in the same row
\[
\{(t,s), (t,s+1),\dots,(t,s+e-1)\} \subseteq \Belt^A
\]
such that $\res(t,s) = \res(A)$.
Let $k \ge 0$ be the number of bricks in $\Belt^A$.
We label the bricks
$$B^A_1, B^A_2, \dots, B^A_k$$
going from left to right along row $r$, and then from left to right along row $r+1$.
Define $C^A$ to be the set of nodes in row $r$ of $\Belt^A$ not contained in any $A$-brick, and $D^A$ to be the set of nodes in row $r+1$ of $\Belt^A$ not contained in any $A$-brick.
We see that as $e=2$, each brick consists of two nodes and $C^A$ and $D^A$ each contain at most one node.

Continuing with the previous example, taking $\mu=(6,5,3,2)$ and $A=(2,3)$:

$$\begin{array}{l}
\begin{tikzpicture}[scale=0.5,draw/.append style={thick,black}]
 \fill[lime!70](2.5,-1.5)rectangle(4.5,-0.5);
  \fill[lime!70](1.5,-2.5)rectangle(3.5,-1.5);
  \fill[orange!50](4.5,-1.5)rectangle(5.5,-0.5);
  \fill[orange!50](0.5,-2.5)rectangle(1.5,-1.5);
  \newcount\col
  \foreach\Row/\row in {{1,2,3,4,5,6}/0,{7,8,9,10,11}/-1,{12,13,14}/-2,{15,16}/-3} {
     \col=1
     \foreach\k in \Row {
        \draw(\the\col,\row)+(-.5,-.5)rectangle++(.5,.5);
        \draw(\the\col,\row)node{\k};
        \global\advance\col by 1
      }
   }
   \draw[magenta,line width=2pt](2.5,-1.5)--(2.5,-0.5)--(5.5,-0.5)--(5.5,-1.5)--(3.5,-1.5)
                     --(3.5,-2.5)--(0.5,-2.5)--(0.5,-1.5)--(2.5,-1.5);
   \draw[magenta,line width=2pt](4.5,-0.5)--(4.5,-1.5);
   \draw[magenta,line width=2pt](1.5,-1.5)--(1.5,-2.5);
   \draw[magenta](5.5,-0.5)--(5.5,-1.5);
   \draw[magenta,line width=2pt](1.5,-1.5)--(3.5,-1.5);
   \draw[black,->](4.2,1)node[gray,above=0mm]{$B_1^A$}--(3.6,-0.8);
   \draw[black,->](4.05,-3.5)node[gray,right=-1mm]{$B_2^A$}--(2.6,-2.2);
   \draw[black,->](6.05,-2.5)node[gray,right=-1mm]{$C^A$}--(5,-1.2);
      \draw[black,->](-.8,-2.5)node[gray,left=0mm]{$D^A$}--(0.75,-2.2);
\end{tikzpicture}
\end{array}$$

Let $u, u+1, \dots, v$ be the values in $\Belt^A$ in the initial tableau $\ttt^\mu$.
We obtain a new tableau $\ttt^A$ by placing the numbers $u, u+1 \dots, v$ from left to right in the following order: $D^A$, $B^A_1$, $B^A_2$, \dots, $B^A_k$, $C^A$.
The rest of the numbers are placed in the same positions as in $T^\mu$.
We also define $\bi^A := \bi^{\ttt^A}$.

If $\mu$ and $A$ are the same as before, then we have that

$$\ttt^A=\begin{array}{l}
\begin{tikzpicture}[scale=0.5,draw/.append style={thick,black}]
\fill[lime!70](2.5,-1.5)rectangle(4.5,-0.5);
  \fill[lime!70](1.5,-2.5)rectangle(3.5,-1.5);
  \fill[orange!50](4.5,-1.5)rectangle(5.5,-0.5);
  \fill[orange!50](0.5,-2.5)rectangle(1.5,-1.5);
  \newcount\col
  \foreach\Row/\row in {{1,2,3,4,5,6}/0,{7,8,10,11,14}/-1,{9,12,13}/-2,{15,16}/-3} {
     \col=1
     \foreach\k in \Row {
        \draw(\the\col,\row)+(-.5,-.5)rectangle++(.5,.5);
        \draw(\the\col,\row)node{\k};
        \global\advance\col by 1
      }
   }
   \draw[magenta,line width=2pt](2.5,-1.5)--(2.5,-0.5)--(5.5,-0.5)--(5.5,-1.5)--(3.5,-1.5)
                     --(3.5,-2.5)--(0.5,-2.5)--(0.5,-1.5)--(2.5,-1.5);
   \draw[magenta,line width=2pt](4.5,-0.5)--(4.5,-1.5);
   \draw[magenta,line width=2pt](1.5,-1.5)--(1.5,-2.5);
   \draw[magenta,line width=2pt](5.5,-0.5)--(5.5,-1.5);
   \draw[magenta,line width=2pt](1.5,-1.5)--(3.5,-1.5);
   \draw[black,->](3.85,1)node[gray,above=0mm]{$B_1^A$}--(3.6,-0.8);
   \draw[black,->](4.05,-3.5)node[gray,right=-1mm]{$B_2^A$}--(2.6,-2.2);
   \draw[black,->](6.05,-2.5)node[gray,right=-1mm]{$C^A$}--(5,-1.2);
   \draw[black,->](-0.8,-2.5)node[gray,left=0mm]{$D^A$}--(0.75,-2.2);
\end{tikzpicture}
\end{array}$$

Assume that $k > 0$, and let $l = \ttt^A(A)$ and for $1\le r < k$ define
\[
w^A_r := \prod_{s=l+re-e}^{l+re-1} (s, s+e) \in \fkS_n.
\]
Informally, $w^A_r$ swaps the bricks $B^A_r$ and $B^A_{r+1}$. The elements $w^A_1, w^A_2, \dots, w^A_{k-1}$ are Coxeter generators of the group
\[
\fkS^A := \langle w^A_1, w^A_2, \dots, w^A_{k-1} \rangle \cong \fkS_k.
\]
By convention, if $k=0$ we set $\fkS^A$ to be the trivial group.
Let $f$ be the number of $A$-bricks in row $r$ of $\Belt^A$. 

We denote by $\scrd^A$ the set of minimal length left coset representatives of $\fkS_f \times \fkS_{k-f}$ in $\fkS^A \cong \fkS_k$.
The corresponding elements of $R^{\La}_n$ are given by
\[
\si^A_r := \psi_{w^A_r} e(\bi^A) \quad\text{and}\quad \tau^A_r := (\si^A_r + 1)e(\bi^A).
\]
Let $u \in \scrd^A$ with reduced expression $u = w^A_{r_1} \dots w^A_{r_a}$. 
Since every element of $\scrd^A$ is fully commutative, the element $\tau^A_u := \tau^A_{r_1} \dots \tau^A_{r_a}$ does not depend upon this reduced expression.
Finally, we define the \emph{Garnir element} to be
\[
g^A := \sum_{u \in \scrd^A} \tau^A_u \psi^{\ttt^A} \in R^{\La}_n.
\]

\begin{defn}
Let $\al\in Q_+$, $n=\Ht(\al)$ and $\mu\in\Par_n$.
The \emph{universal graded (row) Specht module} $S^\mu$ is defined to be the graded $R^{\La}_n$-module generated by the vector $z^\mu$ of degree $\deg(\ttt^\mu)$ subject only to the following relations:
\begin{enumerate}
\item[{\rm (i)}] $e(\bj)z^\mu= \de_{\bj,\bi^\mu}z^\mu$ for all $\bj\in I^n$;
\item[{\rm (ii)}] $y_rz^\mu=0$ for all $1\le r\le n$;
\item[{\rm (iii)}] $\psi_rz^\mu=0$ for all $1\le r< n$ such that $r$ and $r+1$ appear in the same row of $\ttt^\mu$;
\item[{\rm (iv)}] \emph{(homogeneous Garnir relations) }$g^A z^\mu=0$ for all Garnir nodes $A\in[\mu]$. 
\end{enumerate}
\end{defn}

For each $\ttt \in \std(\mu)$, we define $\vtt := \psi_{w^\ttt} z^\mu \in S^\mu$ with residue sequence $\bi^\ttt := w^\ttt \bi^\mu$.

\begin{thmc}{kmr}{Corollary 6.24}
Let $\mu \in \Par_n$. Then the universal Specht module $S^\mu$ for $R^{\La}_n$ has homogeneous $\bbf$-basis
\[
\{\vtt \mid \ttt\in\std(\mu)\}.
\]
\end{thmc}

%% file: parts/hook-action.tex
In this section, we give a complete description of the action of the generators on the basis elements of the Specht module $S^\la$.

For the rest of this paper, we fix $\la := (a,1^b)$ such that $a>0,b\ge0$ and $n = a +b$.

\begin{defn}\label{4-1}
  Given $\ttt \in \std(\la)$, we call the nodes in positions $(1,y)$ of $\la$ for $y= 2, 3, \dots, a$ {\em arm nodes} and we define $\Arm(\ttt)=\{w^\ttt (2), \dots, w^\ttt(a)\}$.
  Similarly, we call the nodes in positions $(x,1)$ of $\la$ for $x=2, 3, \dots, b$ {\em leg nodes} and $\Leg(\ttt)=\{w^\ttt(a+1), \dots, w^\ttt(b)\}$.
  \end{defn}

 \begin{defn}\label{pairs}
 For any odd $3 \le u \le n$, we call $u-1$ and $u$ a {\em pair} of $\ttt$, denoted $\tp_u$, whenever $u-1$ appears directly in front of $u$ in $\ttt$, i.e.\ if $u-1\in\Leg(\ttt)$ (resp., $u-1\in\Arm(\ttt)$) we also have that $u\in\Leg(\ttt)$ (resp., $u\in\Arm(\ttt)$).
 We denote by $\bi_{\tp_u}:=(i_{u-1}, i_u)$ the residue of the pair.
 \end{defn}

\begin{rem}
The previous definition and many of the following lemmas are very similar to those in \cite[\S5]{ls14} and \cite[Chapter~3]{thesis}.
In particular, if both $n$ and $a$ are odd, our results agree.
However, in our case the residues of the pairs might be different.
\end{rem}

\begin{lemc}{bkw11}{Lemma~4.4}
    Then for any $\ttt\in\std(\la)$, we have that $e(\bi)\vtt=\de_{\bi,\bi^\ttt}\vtt$.
 \end{lemc}

\begin{lem}\label{standard-y}
Let $\ttt \in \std(\la)$ and write $\bi^\ttt = (i_1, \dots, i_n)$. For $1 \le r \le n$, we have that
          \[ y_r \vtt = \begin{cases}
            -v_{ s_r \ttt }& \text{ if $i_r = i_{r+1}$, $r \in \Leg(\ttt)$, and $r+1 \in \Arm(\ttt)$},\\
              v_{ s_{r-1} \ttt}& \text{ if $i_{r-1} = i_r$, $r-1 \in \Leg(\ttt)$, and $r \in \Arm(\ttt)$},\\
             0,& \text{ otherwise.}
              \end{cases} \]
\end{lem}

\begin{proof}
We set $\scrd^\la(\bi) = \{ w^\ttt |\  \ttt \in \std(\la) \text{ and } \bi^\ttt = \bi \}$ define the subgroup of $\fkS_n$ $H(\bi):=\langle s_m \mid s_m\bi=\bi\rangle$ for $(m \ge 2) $.
Fix $x \in \fkS_n$ and let $\tts$ be the most dominant standard $\la$-tableau with residue sequence $\bi^\ttt$ such that $\tts=x \ttt^\la$.
By \cite[Proposition~2.13]{loub17}, there exists a unique minimal length permutation $a \in \scrd^\la(\bi^\ttt)$ such that $w^\ttt = u a$ for some $u \in H(\bi^\ttt)$.
Given a reduced expression $u = s_{r_1} \dots s_{r_a}$, we see that $\vtt = \psi_{r_1} \dots \psi_{r_a} \vs$.
Furthermore, $\deg(\vs)-\deg(\vtt) = 2a$, and so $\vs$ is the unique largest degree vector with residue sequence $\bi^\ttt$. In particular, $y_s v_\tts = 0$ for $1 \le s \le n$.
As $H(\bi^\ttt)$ is commutative, at most one of $r-1$ and $r$ is an element of $\{r_1, \dots, r_a\}$.

First, assume that $r \in \{r_1, \dots, r_a\}$, that is we have $r \in \Leg(\ttt)$ and $r+1 \in \Arm(\ttt)$ with $i_r = i_{r+1}$. As each $s_{r_i}$ involved in $w$ commutes with each other, we can assume that $r = r_1$. Then we have that
    \begin{align*}
         y_r \vtt &= y_r \psi_r \psi_{r_2} \dots \psi_{r_a} \vs \\
               &= (\psi_r y_{r+1} - 1) \psi_{r_2} \dots \psi_{r_a} \vs\\
               &= \psi_r \psi_{r_2} \dots \psi_{r_a} y_{r+1} \vs - \psi_{r_2} \dots \psi_{r_a} \vs \\
               &= -\psi_{r_2} \dots \psi_{r_a} \vs= -v_{ s_r \ttt }.
    \end{align*}

   Similarly, taking $r-1 \in \{r_1, \dots, r_a\}$ is the same as saying that $r-1 \in \Leg(\ttt)$, and $r \in \Arm(\ttt)$ with $i_{r-1} = i_r$. Suppose $r-1 = r_1$, then we can compute:
           \begin{align*}
         y_{r} \vtt &= y_{r} \psi_{r-1} \psi_{r_2} \dots \psi_{r_a} \vs \\
               &= (\psi_{r-1} y_{r-1} +1) \psi_{r_2} \dots \psi_{r_a} \vs\\
               &= \psi_{r-1} \psi_{r_2} \dots \psi_{r_a} y_{r-1} \vs + \psi_{r_2} \dots \psi_{r_a} \vs \\
               &= \psi_{r_2} \dots \psi_{r_a} \vs= v_{ s_{r-1} \ttt }.
    \end{align*}

    If $r,r-1\notin\{r_1, \dots, r_a\}$, then $y_r$ commutes with each $\psi_{r_i}$, and hence $y_r \psi_{r_1} \dots \psi_{r_a} \vs= \psi_{r_1} \dots \psi_{r_a} y_r \vs= 0$.
\end{proof}

\begin{rem}
Observe that due to the imposed residue conditions, for $r$ even $y_r\vtt \neq 0$ if and only if $r \in \Leg(\ttt)$ and $r+1\in \Arm(\ttt)$.
If $r\in\Arm(\ttt)$ and $r-1\in\Leg(\ttt)$, we would have $r-3$ entries to place in front of them in $\ttt$, where the number of entries in front of $r$ and $r-1$ needs to have the same parity in order to satisfy the condition $i_{r-1}=i_{r}$.
But as $r-3$ is odd, we see that this is impossible. 
\end{rem}

\begin{lem}\label{psi-1}
Let $\ttt \in \std(\la)$ and $\bi^\ttt = (i_1, \dots, i_n)$.
We have that $\psi_1 \vtt = 0$ and if $r \in \Arm(\ttt)$ and $r+1 \in \Leg(\ttt)$, then $\psi_r \vtt = v_{ s_{r} \ttt }$ for $2 \le r \le n-1$.
\end{lem}

\begin{proof}
We see that $\bi^\ttt$ begins with  $(0,1,\dots)$.
By the cyclotomic relations (see \cref{relations}), $\psi_1\vtt=e(s_1\bi^\ttt)\psi_1\vtt$.
Hence, by looking at the residues we must have $\psi_1 \vtt = 0$.

The second assertion is obvious.
\end{proof}

\begin{lem}\label{psi-2}
    Let $\ttt \in \std(\la)$ and $\bi^\ttt = (i_1, \dots, i_n)$. If $r-1,r\in \Leg(\ttt)$, and $r+1,r+2 \in \Arm(\ttt)$, then 
\[ \psi_r \vtt =
\begin{cases}
             -2v_{ s_{r-1} s_{r+1} s_r \ttt } & \text{ if $i_{r-1}=i_{r+1}$ and $i_r=i_{r+2}$},\\
         0& \text{ otherwise.}
      \end{cases}
      \]
\end{lem}

\begin{proof}
Assume that $r-1 \in \Leg(\ttt)$, and $r+1\in \Arm(\ttt)$, so we have that $\vtt = \psi_r v_{ s_{r} \ttt }$.
Thus $\psi_r \vtt = \psi_r^2 v_{ s_{r} \ttt }$.
If $i_r=i_{r+1}$, then $\psi_r^2 v_{ s_{r} \ttt } = 0$.

    If $i_r \neq i_{r+1}$, then $\psi_r^2 v_{ s_{r} \ttt }= (y_{r+1} - y_r)(y_r - y_{r+1})v_{ s_{r} \ttt }=(2y_{r+1} y_r - y_r^2 - y_{r+1}^2)v_{ s_{r} \ttt }$.
    We consider each of the three terms separately.
    
    As $r \in \Arm(s_r \ttt)$, \cref{standard-y} shows that if $y_r v_{ s_{r} \ttt } \neq 0$ then $r-1 \in \Leg(s_r \ttt)$ and $i_{r-1} = i_{r+1}$.
    In this case, $y_r v_{ s_{r} \ttt }= v_{ s_{r-1} s_{r} \ttt }$.
    Applying $y_r$ once again, we see that $y_r v_{ s_{r-1} s_{r} \ttt } \neq 0$ if and only if $r+1 \in \Arm(s_{r-1} s_r \ttt)$.
    But if $r+1$ is in $\Arm(s_{r-1} s_r \ttt)$, we must have that $r+1 \in \Leg(\ttt)$, a contradiction.
    Hence $y_r^2 v_{ s_{r} \ttt } = 0$. 
    
    We deal with the term $y_{r+1}^2 v_{ s_{r} \ttt }$ in a similar manner.
    As $r+1 \in \Leg(s_r \ttt)$, if $y_{r+1} v_{ s_{r} \ttt } \neq 0$ then $r+2 \in \Arm(s_r \ttt)$ and $i_{r+1} \neq i_{r+2}$.
    Then $y_{r+1} v_{ s_{r} \ttt }= -v_{ s_{r+1} s_{r} \ttt }$.
    Applying $y_{r+1}$ once again, we see that $y_{r+1} v_{ s_{r+1} s_{r} \ttt } \neq 0$ if and only if $r \in \Leg(s_{r+1} s_r \ttt)$. But if $r\in\Leg(s_{r+1} s_r \ttt)$, then we must have that $r \in \Arm(\ttt)$, which contradicts our original assumption.
    Thus $y_{r+1}^2 v_{ s_{r} \ttt } = 0$.

    Finally, we look at the cross term.
    If $y_{r+1} y_r v_{ s_{r} \ttt } \neq 0$, then we have that $y_r v_{ s_{r} \ttt } \neq 0$ if and only if $r-1 \in \Leg( s_r \ttt)$ and $i_{r-1} = i_{r+1}$. In this case, $y_r v_{ s_{r} \ttt } = v_{ s_{r-1}  s_{r} \ttt }$.
    We also see that $y_{r+1} v_{ s_{r-1}  s_{r} \ttt } \neq 0$ if and only if $r+2 \in \Arm(s_{r-1}  s_r \ttt)$ and $i_r = i_{r+ 2}$. But $r+2 \in \Arm(s_{r-1}  s_r \ttt)$ if and only if $r+2 \in \Arm(\ttt)$ and we conclude that $y_{r+1}y_r v_{ s_{r} \ttt } = -v_{ s_{r-1} s_{r+1} s_r \ttt }$.
\end{proof}

\begin{rem}
Note that if an even $r$ is as in the previous statement, then $\psi_r\vtt=0$.
To see this, observe that due to the residue conditions $i_{r-1}=i_{r+1}, i_r=i_{r+2}$, the number of entries in front of $r-1$ and $r+1$ in $\ttt$ must have the same parity.
In total, we have $r-3$ entries to divide between the arm and the leg of $\ttt$.
If $r$ is even, then $r-3$ is odd, hence it would be impossible to have both odd or even number of entries in front of $r-1$ and $r+1$.
\end{rem}

\begin{defn}[\cite{ls14}, Definition~5.1]\label{doms}
We define $\Psi_r := \psi_r \psi_{r+1} \psi_{r-1} \psi_r$.
For two odd integeres $3\le x\le y\le n-2$, we set:

\[\Psichaind{y}{x}:=\Psi_y\Psi_{y-2}\dots\Psi_x\quad\text{and}\quad\Psichainu{x}{y}:=\Psi_x\Psi_{x+2}\dots\Psi_{y}.\]

If $y<x$, we consider both of the above terms to be the identity.
\end{defn}

We will now describe how to write $\vtt$ as a product of $\Psichaind{y}{x}$'s.
Let $\ttt\in\std(\lambda)$ be a tableau in which all entries appear in pairs as in \cref{pairs}.
Let $d$ be the number of pairs in the arm of $\ttt$ which differ from the corresponding pairs in $\ttt^\la$, we will label these by $y_1',\dots,y_d'$ where $i_{j_i'}=1$ and define
\[
j_i:=\begin{cases}
j_i'-1&\text{if $j'_i$ is even}\\
j_i'-2 &\text{if $j'_i$ is odd.}
\end{cases}
\]
This way we will always put in the highest possible pair into a given position.
Since $\ttt\in\std(\la)$, we see that these are precisely the final $d$ pairs in $\Arm(\ttt)$.

\begin{eg}
    For example, if $\la=(7,1^6)$ and $\ttt\in\std(\la)$ is as below
\[
\ttt^\lambda=
\begin{array}{l}
\Tableau{{1,2,3,4,5,6,7},{8},{9},{\ten},{\eleven},{\twelve},{\thirteen}}
\end{array}
\qquad
\text{ and }
\qquad
\ttt=
\begin{array}{l}
\Tableau{{1,2,3,8,9,\twelve,\thirteen},{4},{5},{6},{7},{\ten},{\eleven}}
\end{array},
\]
then $d=2$ and we see that $j_1=7$ and $j_2=11$. 
\end{eg}

Observe that $\vtt$ has the reduced expression
\[
\vtt=
\Psichaind{j_1}{a-2d+2}\Psichaind{j_2}{a-2d+4}\dots\Psichaind{j_{d-1}}{a-2}\Psichaind{j_d}{a}z^\la.
\]
We will refer to this as the \textbf{standard form} for $\vtt$.
Notice that $j_{i+1}>j_i$ for all $i=1,\dots,d-1$.
Moreover, if $\vtt \in S^\la$ is in standard form, then any expression obtained from it by deleting $\Psi$ terms from the left is also in $S^\la$.

\begin{lemc}{ls14}{Lemma 5.3} \label{commuting}
Let $\la$ and $\ttt$ be as above. Then
\begin{enumerate}
\item $e(\bi^\ttt)\Psi_r=\Psi_re(\bi^\ttt)$ for all $r$;\\
\item $y_k\Psi_r=\Psi_ry_k$ for all $k\ge r+3$ and for all $k\le r-2$;\\
\item $\psi_k\Psi_r=\Psi_r\psi_k$ for all $k\ge r+3$ and for all $k\le r-3$.
\end{enumerate}
\end{lemc}

If both $n$ and $a$ are odd and all entries appear in pairs in $\ttt\in\std(\lambda)$, then \cite[Proposition 3.27]{thesis} gives a complete description of the action of $\psi_r$ for all $\vtt$ and $2\le r \le n-1$.
Now, we will extend these results for all $\ttt\in\std(\la)$ on a case by case approach.
Set $\ell:=\ttt(1,a)$.

First, we will illustrate how to obtain $\vtt$ for all $\ttt\in\std(\la)$.
\begin{enumerate}
\item Move $\ell$ into position $\ttt(1,a)$ using $\psi_{\ell-1} \psi_{\ell-2} \dots \psi_{a+1} \psi_a$ where $a<\ell \le n-1$ (this part only appears if $a$ is even).
\item Move the pairs into position using some $\Psichaind{x}{y}$ terms: $\tp_y\mapsto\tp_{x+2}$ where $x$ is maximal such that $x+1\le\ttt(1,y-1)$ and $x+2\le\ttt(1,y)$.
\item Using some individual $\psi$'s, we either
\begin{itemize}
\item ``break up'' the pairs (i.e.\  an even entry $j-1$ will not have $j$ next to it),  or
\item shift the residues of a pair (i.e.\ for some $\tp_j$ we will have that $\bi_{\tp_j}=(0,1)$ instead of $(1,0)$).
\end{itemize}
\end{enumerate}

\begin{eg}
Suppose that $\la=(10,1^7)$ and let $\ttt$ be as below.
\[
\ttt^\lambda=
\begin{array}{l}
\Tableau{{1,2,3,4,5,6,7,8,9,10},{\eleven},{\twelve},{\thirteen},{\fourteen},{\fifteen},{\sixteen},{\seventeen}}
\end{array}
\qquad
\qquad
\ttt=
\begin{array}{l}
\Tableau{{1,2,3,6,12,\thirteen,14,\fifteen,16,17},{4},{5},{7},{8},{9},{\ten},{\eleven}}
\end{array}
\]

As $a$ is even, we first need to move $17$ to the end of $\Arm(\ttt)$:
\[
\ttt_1=
\begin{array}{l}
\Tableau{{1,2,3,4,5,6,7,8,9,17},{10},{\eleven},{\twelve},{\thirteen},{\fourteen},{\fifteen},{\sixteen}}
\end{array}
v_{\ttt_1}=\psi_{16}\psi_{15}\dots\psi_{10}
z^\la
\]

Next, we identify $j_i'$ as explained after \cref{doms}:
\[
\ttt=
\begin{array}{l}
\qquad
\begin{tikzpicture}[scale=0.5,draw/.append style={thick,black}]
  \fill[magenta!40](7.5,-0.5)rectangle(9.5,0.5);
    \fill[yellow!50](5.5,-0.5)rectangle(7.5,0.5);
     \fill[blue!30](3.5,-0.5)rectangle(5.5,0.5);
  \newcount\col
  \foreach\Row/\row in {{1,2,3,6,12,13,14,15,16,17}/0,{4}/-1,{5}/-2,{7}/-3,{8}/-4,{9}/-5,{10}/-6,{11}/-7}{
     \col=1
    \foreach\k in \Row {
        \draw(\the\col,\row)+(-.5,-.5)rectangle++(.5,.5);
        \draw(\the\col,\row)node{\k};
    \global\advance\col by 1
     }
  }
    \draw (4,1) node[gray] {\scriptsize$j'_1=6$};
    \draw (6,-1) node[gray] {\scriptsize$j'_2=13$};
    \draw (8,1) node[gray] {\scriptsize$j'_3=15$};
\end{tikzpicture}
\end{array}
\]
so $j_1=5,j_2=11,j_3=13$, hence
\[
\ttt_2=
\begin{array}{l}
\Tableau{{1,2,3,6,7,12,13,14,15,17},{4},{5},{8},{9},{\ten},{\eleven},{16}}
\end{array}
v_{\ttt_2}=\Psi_5
\Psichaind{11}{7}
\Psichaind{13}{9}
\psi_{16}\psi_{15}\dots\psi_{10}
z^\la.
\]

Lastly, we see that we need to increase the entries $12,13,14,15$ by one and place $12$ in the position of $7$:
\[
\ttt_3=\ttt=
\begin{array}{l}
\Tableau{{1,2,3,6,12,13,14,15,16,17},{4},{5},{7},{8},{9},{\ten},{\eleven}}
\end{array},
\]
\[
v_{\ttt}=
\underbrace{\psi_{\eleven}\psi_{\ten}\psi_9\psi_8\psi_7}_{\scriptsize\text{``breaks" up $\tp_7$}}
\underbrace{\psi_{12}\psi_{13}\psi_{14}\psi_{15}}_{\scriptsize\text{puts $13,14,15,16$ in front of $17$}}
\Psi_5
\Psichaind{11}{7}
\Psichaind{13}{9}
\psi_{16}\psi_{15}\dots\psi_{10}
z^\la.
\]

We have $\tp_3,\tp_{13},\tp_{15},\tp_{17}\in\Arm(\ttt)$ and $\tp_5,\tp_9,\tp_{11}\in\Leg(\ttt)$, but $\tp_{13},\tp_{15}$ and $\tp_{17}$ have shifted residues, thus these numbers appear as indices in the front $\psi$-chain.
We will use two different expressions of $\vtt$ throughout our proofs: in one the $\Psi$ terms are enclosed by two individual $\psi$-chains (\cref{form1}) and in the other one we push the $\psi$ terms appearing on the left as much to the right as possible (\cref{form2}):
\begin{align}
\vtt&=
\psi_{\eleven}\psi_{\ten}\psi_9\psi_8\psi_7
\psi_{12}\psi_{13}
\psi_{14}\psi_{15}
\Psi_5
\Psichaind{11}{7}
\Psichaind{13}{9}
\psi_{16}\psi_{15}\dots\psi_{10}
z^\la \label{form1} \\
&=
\psi_{\eleven}\psi_{\ten}\psi_9\psi_8\psi_7
\Psi_5
\psi_{12}\psi_{13}
\Psichaind{11}{7}
\psi_{14}\psi_{15}
\Psichaind{13}{9}
\psi_{16}\psi_{15}\dots\psi_{10}
z^\la. \label{form2}
\end{align}
\end{eg}

In general, we can write \cref{form1} as:
\begin{equation}\label{short1}
    \vtt=\psi_w\Psichaind{y_1}{m_1}\Psichaind{y_2}{m_2}\dots\Psichaind{y_d}{a(-1)}(\psi_{\ell-1}\dots\psi_a)z^\la,
\end{equation}
for some minimal length $w\in \fkS_n$, where the bracketed terms only appear if $a$ is even.
We will be using this bracket notation throughout our proofs. 
(We will introduce a similar expression for \cref{form2} before \cref{3pairs}.)
When acting on $\vtt$ by a generator $\psi_k$, we will consider $\vtt$ in normal form as in \cref{form1}.
It will be useful to have compact notation for the front terms of $\vtt$ that do not play a role in our computations.

Let $\psi^*$ denote the leading terms (in the normal form for $\psi_\ttt$) that commute with $\psi_k$ ($\psi^*$ may be empty), so that $\vtt = 
\psi^* \vtp$ for some $\ttt'\ge \ttt$, and $\psi_k \vtt = \psi^* \psi_k \vtp$.
For example:
\begin{align*}
\psi_{12}\vtt&=
\psi_{12}\psi_{\ten}\psi_9\psi_8\psi_7
\Psi_5
\Psichaind{11}{7}
\psi_{14}\psi_{15}
\Psichaind{13}{9}
\psi_{16}\psi_{15}\dots\psi_{10}
z^\la\\
&=
\underbrace{
\psi_{\ten}\psi_9\psi_8\psi_7
\Psi_5}_{=\psi^*}
\psi_{12}
\underbrace{
\Psichaind{11}{7}
\psi_{14}\psi_{15}
\Psichaind{13}{9}
\psi_{16}\psi_{15}\dots\psi_{10}
z^\la}_{=\vtp}\\
&=
\psi^*
\psi_{12}
\Psi_{11}
\underbrace{
\Psichaind{9}{7}
\psi_{14}\psi_{15}
\Psichaind{13}{9}
\psi_{16}\psi_{15}\dots\psi_{10}
z^\la}_{=\vs \text{ for some } \tts>\ttt'}\\
&=
\psi^*
\psi_{12}
\Psi_{11}
\vs.
\end{align*}

\subsection{Action of $\psi_k$ with $k$ even}

The following statement describes the action of $y_r$ for any $r$ and $\psi_k$ for $k$ even (cf.\ \cite[Proposition 5.4]{ls14}).

\begin{prop}\label{pre-even}
Assume that $\ttt\in\std(\la)$ with corresponding $\vtt$.
Then
\begin{enumerate}
\item $y_r\vtt=-y_{r-1}\vtt=v_{s_{r-1}\ttt}$ if $r-1\in\Leg(\ttt),r\in\Arm(\ttt)$ and $r$ is odd;
\item $y_k\vtt=0$ otherwise;
\item $\psi_{r}\vtt=v_{s_{r}\ttt}$ if $r+1\in\Leg(\ttt),r\in\Arm(\ttt)$ and $r$ is even;
\item $\psi_{k}\vtt=0$ for all other even $k$, unless $\bi_{\tp_{k+1}}=(0,1)$, $\tp_{k+1}\in\Leg(\ttt)$.
\end{enumerate}
\end{prop}

\begin{proof}
Cases $1$ and $2$ follow immediately from \cref{standard-y}.

For the third statement, if $r\in\Leg(\ttt)$ and $r-1\in\Arm(\ttt)$, then $\psi_{r-1}\vtt$ is already in reduced form as $w^\ttt < w^{s_{r-1}\ttt}$.

For case $4$, if $k$ is in $\Leg(\ttt)$ and $k+1$ is in $\Arm(\ttt)$, we have that $\psi_k\vtt=\psi^*\psi_k\vtp=\psi^*\psi_k^2v_{s_k\ttt'}=0$ as $i_k=i_{k+1}$ in the tableau $s_k\ttt'$.

Finally, assume that $\tp_{k+1}\in\Arm(\ttt)$ or $\tp_{k+1}\in\Leg(\ttt)$, that is $k$ and $k+1$ are adjacent in $\ttt$.
To prove case 4 $\psi_k$ on $\vtt$, we have five cases to consider:
\begin{itemize}
\item[$(A)$] $\psi_k$ commutes with $\psi_\ttt$;
\item[$(B_r)$] $\bi_{\tp_{k+1}}=(1,0)$, $\tp_{k+1}\in\Arm(\ttt)$ (i.e.\ the first term that $\psi_k$ does not commute with is $\Psi_{k-1}$ and $\psi_\ttt$ involves $r$ $\Psi$ terms  unless $r=0$);
\item[$(C_r)$] $\bi_{\tp_{k+1}}=(1,0)$, $\tp_{k+1}\in\Leg(\ttt)$ (i.e.\ the first term that $\psi_k$ does not commute with is $\Psi_{k+1}$ and $\psi_\ttt$ involves $r$ $\Psi$ terms unless $r=0$);
\item[$(D_s)$] $\bi_{\tp_{k+1}}=(0,1)$, $\tp_{k+1}\in\Arm(\ttt)$ (i.e.\ the first term that $\psi_k$ does not commute with is $\psi_{k-1}$ and $\ell(w^{\ttt'})=s$ );
\item[$(E)$] $\bi_{\tp_{k+1}}=(0,1)$, $\tp_{k+1}\in\Leg(\ttt)$ (i.e.\ the first term that $\psi_k$ does not commute with is $\psi_{k+1}$).
\underline{Note}: this is a special case which will be proved later in \cref{even}.
\end{itemize}

\underline{Case $A$}
Follows immediately from \cref{4-1}.

\underline{Cases $B$ and $C$}\\
The proof is essentially the same as \cite[Proposition 5.4]{ls14}.
We use induction on $r$, the number of $\Psi$ terms appearing in $\vtp$.
Note that $B_0$ can only occur if $\ttt=\ttt^\la$ and $C_0$ can only occur if $\ttt=\ttt^\la$ where $a$ is odd or $\vtt=\psi_{\ell-1}\dots\psi_az^\la$ where $a$ is even and $a\le \ell\le n$ (recall that $\ell=\ttt(1,a)$).
In the latter case if $k<\ell$ we have that:
\begin{align*}
    \psi_k\vtt&=\psi_k\psi_{\ell-1}\dots\psi_az^\la\\
    &=\psi_{\ell-1}\dots\psi_{k+2}(\psi_k\psi_{k+1}\psi_k)\psi_{k-1}\dots\psi_az^\la\\
    &=\psi_{\ell-1}\dots\psi_{k+2}(\psi_{k+1}\psi_k\psi_{k+1})\psi_{k-1}\dots\psi_az^\la\\
    &=0.
\end{align*}

If $k>\ell$, then we are in case $A$.

For $B_r$ we have that
\begin{align*}
\psi_k\vtt&=\psi_k\psi^*\Psi_{k-1}\vs\\
&=\psi^*(\psi_k\psi_{k-1}\psi_k)\psi_{k-2}\psi_{k-1}\vs\\
&=\psi^*\psi_{k-1}\psi_{k}(\psi_{k-1}\psi_{k-2}\psi_{k-1})\vs\\
&=\psi^*\psi_{k-1}\psi_{k}\psi_{k-2}\psi_{k-1}\psi_{k-2}\vs=0\text{ by $B_{r-1}$ or $A$}\\
&\begin{rcases}
-\psi^*\psi_{k-1}\psi_{k}y_{k-2}\vs\\
+2\psi^*\psi_{k-1}\psi_{k}y_{k-1}\vs\\
-\psi^*\psi_{k-1}\psi_{k}y_k\vs\\
\end{rcases}
=0\text{ by \cref{standard-y}}
\\
&=0.
\end{align*}
Similarly, for $C_r$ we have that
\begin{align*}
\psi_k\vtt&=\psi_k\psi^*\Psi_{k+1}\vs\\
&=\psi^*(\psi_k\psi_{k+1}\psi_k)\psi_{k+2}\psi_{k+1}\vs\\
&=\psi^*\psi_{k+1}\psi_{k}(\psi_{k+1}\psi_{k+2}\psi_{k+1})\vs\\
&=\psi^*\psi_{k+1}\psi_{k}\psi_{k+2}\psi_{k+1}\psi_{k+2}\vs=0\text{ by $C_{r-1}$ or $A$}\\
&\begin{rcases}
+\psi^*\psi_{k-1}\psi_{k}y_{k+1}\vs\\
-2\psi^*\psi_{k-1}\psi_{k}y_{k+2}\vs\\
+\psi^*\psi_{k-1}\psi_{k}y_{k+3}\vs
\end{rcases}
=0\text{ by \cref{standard-y}}\\ \nopagebreak
&=0.
\end{align*}

\underline{Case $D$}
We prove the statement by induction on $s$, the length of $\psi_{\ttt'}$.
Details can be included by switching the toggle in the {\tt arXiv} version of this paper.
\begin{answer}
For the base case, if $a$ is even, we have that $\vtt=\psi_{k-1}\psi_kz^\la$, so
\begin{align*}
\psi_k\vtt&=(\psi_k\psi_{k-1}\psi_{k})z^\la\\
&=(\psi_{k-1}\psi_k\psi_{k-1})z^\la\\
&=0.
\end{align*}
If $a$ is odd, then $\vtt=\psi_{k-1}\psi_k\psi_{\ell-1}\dots\psi_az^\la$, hence:
\begin{align*}
\psi_k\vtt&=(\psi_k\psi_{k-1}\psi_k)\psi_{\ell-1}\dots\psi_az^\la\\
&=(\psi_{k-1}\psi_k\psi_{k-1})\psi_{\ell-1}\dots\psi_az^\la\\
&=0,
\end{align*}
as $k-1=a$, otherwise we would have $\Psi$ terms appearing in the expression.

Now, we have four different subcases to consider.
\begin{itemize}

 \item[i)] The entry to the left of $k$ is either $k-2$ or $k-1$ and $\tp_{k+1}$ is at the end of the arm:
    
        \begin{align*}
\psi_k\vtt&=\psi_k\psi^*\psi_{k-1}\Psichaind{k-3}{a-1}\psi_k\psi_{k-1}\dots\psi_az^\la\\
&=\psi^*(\psi_k\psi_{k-1}\psi_k)\Psichaind{k-3}{a-1}\psi_{k-1}\dots\psi_az^\la\\
&=\psi^*(\psi_{k-1}\psi_k\psi_{k-1})\Psichaind{k-3}{a-1}\psi_{k-1}\dots\psi_az^\la\\
&=\psi^*\psi_{k-1}\psi_k\psi_{k-3}\psi_{k-4}(\psi_{k-1}\psi_{k-2}\psi_{k-1})\psi_{k-3}\Psichaind{k-5}{a-1}\psi_{k-2}\dots\psi_az^\la\\
&=\psi^*\psi_{k-1}\psi_k\psi_{k-3}\psi_{k-4}(\psi_{k-2}\psi_{k-1}\underbrace{\psi_{k-2})\psi_{k-3}\Psichaind{k-5}{a-1}\psi_{k-2}\dots\psi_az^\la}_{=0\text{\;by $D_s$ for some $s<r$}}\\
&\begin{rcases}
-\psi^*\psi_{k-1}\psi_k\psi_{k-3}\psi_{k-4}(y_{k-2})\psi_{k-3}\Psichaind{k-5}{a-1}\psi_{k-2}\dots\psi_az^\la\\
+2\psi^*\psi_{k-1}\psi_k\psi_{k-3}\psi_{k-4}(y_{k-1})\psi_{k-3}\Psichaind{k-5}{a-1}\psi_{k-2}\dots\psi_az^\la\\
-\psi^*\psi_{k-1}\psi_k\psi_{k-3}\psi_{k-4}(y_k)\psi_{k-3}\Psichaind{k-5}{a-1}\psi_{k-2}\dots\psi_az^\la
\end{rcases}
=0\text{ by \cref{standard-y}}\\ 
&=0.
\end{align*}

    \item[ii)] The entry to the left of $k$ is $x-1\neq k-1,k-2$ and $\tp_{k+1}$ is at the end of the arm:
    
    \begin{align*}
\psi_k\vtt&=\psi_k\psi^*\psi_{k-1}\psi_{k-2}\dots\psi_x\Psichaind{x-2}{a-1}\psi_k\psi_{k-1}\dots\psi_az^\la\\
&=\psi^*(\psi_k\psi_{k-1}\psi_k)\psi_{k-2}\dots\psi_x\Psichaind{x-2}{a-1}\psi_{k-1}\dots\psi_az^\la\\
&=\psi^*(\psi_{k-1}\psi_k\psi_{k-1})\psi_{k-2}\dots\psi_x\Psichaind{x-2}{a-1}\psi_{k-1}\dots\psi_az^\la\\
&=\psi^*\psi_{k-1}\psi_k(\psi_{k-1}\psi_{k-2}\psi_{k-1})\psi_{k-3}\dots\psi_x\Psichaind{x-2}{a-1}\psi_{k-2}\dots\psi_az^\la\\
&=\psi^*\psi_{k-1}\psi_k(\psi_{k-2}\psi_{k-1}\underbrace{\psi_{k-2})\psi_{k-3}\dots\psi_x\Psichaind{x-2}{a-1}\psi_{k-2}\dots\psi_az^\la}_{=0\text{\;by $D_s$ for some $s<r$}}\\
&\begin{rcases}
-\psi^*\psi_{k-1}\psi_ky_{k-2}\psi_{k-3}\dots\psi_x\Psichaind{x-2}{a-1}\psi_{k-2}\dots\psi_az^\la\\
+2\psi^*\psi_{k-1}\psi_ky_{k-1}\psi_{k-3}\dots\psi_x\Psichaind{x-2}{a-1}\psi_{k-2}\dots\psi_az^\la\\
-\psi^*\psi_{k-1}\psi_ky_{k}\psi_{k-3}\dots\psi_x\Psichaind{x-2}{a-1}\psi_{k-2}\dots\psi_az^\la
\end{rcases}
=0\text{ by \cref{standard-y}}\\ 
&=0.
\end{align*}

 \item[iii)] The entry to the left of $k$ is either $k-2$ or $k-1$ and $\tp_{k+1}$ is not at the end of the arm:
 
  \begin{align*}
\psi_k\vtt&=\psi_k\psi^*\psi_{k-1}\Psichaind{k-3}{p}\psi_k\psi_{k+1}\Psichaind{k-1}{p+2}\vs\\
&=\psi^*(\psi_k\psi_{k-1}\psi_k)\Psichaind{k-3}{p}\psi_{k+1}\Psichaind{k-1}{p+2}\vs\\
&=\psi^*(\psi_{k-1}\psi_k\psi_{k-1})\Psichaind{k-3}{p}\psi_{k+1}\Psichaind{k-1}{p+2}\vs\\
&=\psi^*\psi_{k-1}\psi_k\psi_{k-3}\psi_{k-4}(\psi_{k-1}\psi_{k-2}\psi_{k-1})\psi_{k-3}\Psichaind{k-5}{p}\psi_{k+1}\psi_k\psi_{k-2}\psi_{k-1}\Psichaind{k-3}{p+2}\vs\\
&=\psi^*\psi_{k-1}\psi_k\psi_{k-3}\psi_{k-4}(\psi_{k-2}\psi_{k-1}\underbrace{\psi_{k-2})\psi_{k-3}\Psichaind{k-5}{p}\psi_{k+1}\psi_k\psi_{k-2}\psi_{k-1}\Psichaind{k-3}{p+2}\vs}_{=0\text{\;by $D_s$ for some $s<r$}}\\
&\begin{rcases}
-\psi^*\psi_{k-1}\psi_k\psi_{k-3}\psi_{k-4}(y_{k-2})\psi_{k-3}\Psichaind{k-5}{p}\psi_{k+1}\psi_k\psi_{k-2}\psi_{k-1}\Psichaind{k-3}{p+2}\vs\\
+2\psi^*\psi_{k-1}\psi_k\psi_{k-3}\psi_{k-4}(y_{k-1})\psi_{k-3}\Psichaind{k-5}{p}\psi_{k+1}\psi_k\psi_{k-2}\psi_{k-1}\Psichaind{k-3}{p+2}\vs\\
-\psi^*\psi_{k-1}\psi_k\psi_{k-3}\psi_{k-4}(y_k)\psi_{k-3}\Psichaind{k-5}{p}\psi_{k+1}\psi_k\psi_{k-2}\psi_{k-1}\Psichaind{k-3}{p+2}\vs
\end{rcases}
=0\text{ by \cref{standard-y}}\\ 
&=0.
\end{align*}

    \item[iv)] The entry to the left of $k$ is $x-1\neq k-1,k-2$, $\tp_{k+1}$ is not at the end of the arm and the entry to the right of $k+1$ is $y\ge k+2$ (note that the chain $(\psi_y\dots\psi_{k+2})$ might be empty):

    \begin{align*}
\psi_k\vtt&=\psi_k\psi^*\psi_{k-1}\psi_{k-2}\dots\psi_x\Psichaind{x-2}{p}\psi_k(\psi_y\dots\psi_{k+2})\psi_{k+1}\Psichaind{k-1}{p+2}\vs\\
&=\psi^*(\psi_k\psi_{k-1}\psi_k)\psi_{k-2}\dots\psi_x\Psichaind{x-2}{p}(\psi_y\dots\psi_{k+2})\psi_{k+1}\Psichaind{k-1}{p+2}\vs\\
&=\psi^*(\psi_{k-1}\psi_k\psi_{k-1})\psi_{k-2}\dots\psi_x\Psichaind{x-2}{p}(\psi_y\dots\psi_{k+2})\psi_{k+1}\Psichaind{k-1}{p+2}\vs\\
&=\psi^*\psi_{k-1}\psi_k(\psi_{k-1}\psi_{k-2}\psi_{k-1})\psi_{k-3}\dots\psi_x\Psichaind{x-2}{p}(\psi_y\dots\psi_{k+2})\psi_{k+1}\psi_{k}\psi_{k-2}\psi_{k-1}\Psichaind{k-3}{p+2}\vs\\
&=\psi^*\psi_{k-1}\psi_k(\psi_{k-2}\psi_{k-1}\underbrace{\psi_{k-2})\psi_{k-3}\dots\psi_x\Psichaind{x-2}{p}(\psi_y\dots\psi_{k+2})\psi_{k+1}\psi_{k}\psi_{k-2}\psi_{k-1}\Psichaind{k-3}{p+2}\vs}_{=0\text{\;by $D_s$ for some $s<r$}}\\
&\begin{rcases}
-\psi^*\psi_{k-1}\psi_k(y_{k-2})\psi_{k-3}\dots\psi_x\Psichaind{x-2}{p}\\
 \hspace{4cm}\cdot\,(\psi_y\dots\psi_{k+2})\psi_{k+1}\psi_{k}\psi_{k-2}\psi_{k-1}\Psichaind{k-3}{p+2}\vs\\
+2\psi^*\psi_{k-1}\psi_k(y_{k-1})\psi_{k-3}\dots\psi_x\Psichaind{x-2}{p}\\
 \hspace{4cm}\cdot\,(\psi_y\dots\psi_{k+2})\psi_{k+1}\psi_{k}\psi_{k-2}\psi_{k-1}\Psichaind{k-3}{p+2}\vs\\
-\psi^*\psi_{k-1}\psi_k(y_{k})\psi_{k-3}\dots\psi_x\Psichaind{x-2}{p}\\
 \hspace{4cm}\cdot\,(\psi_y\dots\psi_{k+2})\psi_{k+1}\psi_{k}\psi_{k-2}\psi_{k-1}\Psichaind{k-3}{p+2}\vs
\end{rcases}
=0\text{ by \cref{standard-y}}\\ 
&=0.
\qedhere
\end{align*}
\end{itemize}
\end{answer}
\end{proof}

The next lemma is very similar to \cite[Lemma 5.5]{ls14}.
\begin{lem}\label{handy}
Suppose $k$ is odd, $\vtt$ is as in \cref{form2} and let $w\in\fkS_n$ be some permutation with corresponding reduced expression $\psi^*$ with which $\psi_k$ commutes.
(Note that in all cases below we have that $\ttt' \ledom \tts$ and that $\ell(w^\ttt)=\ell(w^\ttu)+4$ or $\ell(w^\ttt)=\ell(w^\tts)+8$.)
\begin{enumerate}
\item If $\vtt=\psi^*\vtp$ where $\vtp=\Psi_k v_\ttu$, then $\psi_k\vtt=-2\psi_k\psi^*v_\ttu$ (compare with \cref{standard-y});
\item  If $\vtt=\psi^*\vtp$ where $\vtp=\Psi_{k+2}\Psi_k\vs$, then $\psi_k\vtt=\psi_k\psi^*\vs$,
\item If $\vtt=\psi^*\vtp$ where $\vtp=\Psi_{k-2}\Psi_k \vs$, then $\psi_k\vtt=\psi_k\psi^*\vs$,
\end{enumerate}
\end{lem}

\begin{proof}

\begin{enumerate}

\item[1)] If $\vtp$ has $\Psi_k$ on the left, we see that this means that $k-1,k\in\Leg(\ttt')$ and $k+1,k+2\in\Arm(\ttt')$, so the result follows by \cref{psi-2}.

\item[2)] If $\vtp$ has $\Psi_{k+2}\Psi_k$ on the left, we see that this means that $k-1,k,k+1,k+2\in\Leg(\ttt')$ and $k+3,k+4\in\Arm(\ttt')$, so we have that:
\begin{align*}
\psi_k\vtt&=
\psi_k\psi^*\Psi_{k+2}\Psi_k \vs\\
&=\psi^*\psi_k\psi_{k+2}\psi_{k+3}\psi_{k+1}\psi_{k+2}\psi_k\psi_{k-1}\psi_{k+1}\psi_k \vs\\
&=\psi^*\psi_{k+2}\psi_{k+3}(\psi_k\psi_{k+1}\psi_k)\psi_{k+2}\psi_{k-1}\psi_{k+1}\psi_k \vs\\
&=\psi^*\psi_{k+2}\psi_{k+3}(\psi_{k+1}\psi_{k}\psi_{k+1})\psi_{k+2}\psi_{k-1}\psi_{k+1}\psi_k \vs\\
&+\psi^*\psi_{k+2}\psi_{k+3}(y_k)\psi_{k+2}\psi_{k-1}\psi_{k+1}\psi_k \vs\\
&\begin{rcases}
-2\psi^*\psi_{k+2}\psi_{k+3}(y_{k+1})\psi_{k+2}\psi_{k-1}\psi_{k+1}\psi_k \vs\\
+\psi^*\psi_{k+2}\psi_{k+3}(y_{k+2})\psi_{k+2}\psi_{k-1}\psi_{k+1}\psi_k \vs
\end{rcases}
=0\text{ by \cref{standard-y}}\\
&=\psi^*\psi_{k+2}\psi_{k+3}\psi_{k+1}\psi_{k}(\psi_{k+1}\psi_{k+2}\psi_{k+1})\psi_{k-1}\psi_k \vs\\
&+\psi^*\psi_{k+2}\psi_{k+3}\psi_{k+2}\psi_{k+1}\psi_k \vs\\
&=\psi^*\psi_{k+2}\psi_{k+3}\psi_{k+1}\psi_{k}(\psi_{k+2}\psi_{k+1}\psi_{k+2})\psi_{k-1}\psi_k \vs\\
&+\psi^*(\psi_{k+2}\psi_{k+3}\psi_{k+2})\psi_{k+1}\psi_k \vs\\
&=\psi^*\Psi_{k+2}\Psi_k\psi_{k+2} \vs\\
&+\psi^*(\psi_{k+3}\psi_{k+2}\psi_{k+3}+y_{k+2}-2y_{k+3}+y_{k+4})\psi_{k+1}\psi_k \vs\\
&=\psi^*\Psi_{k+2}\Psi_k\psi_{k+2} \vs\\
&+\psi^*\psi_{k+3}\psi_{k+2}\underbrace{\psi_{k+3}\psi_{k+1}\psi_k \vs}_{=0\text{ by \cref{even} case $4$}}\\
&+\psi^* (y_{k+2})\psi_{k+1}\psi_k \vs\\
&\begin{rcases}
-2\psi^* (y_{k+3})\psi_{k+1}\psi_k \vs\\
+\psi^*(y_{k+4})\psi_{k+1}\psi_k \vs
\end{rcases}
=0\text{ by \cref{standard-y}}\\
&=\psi_k\psi^* \vs+\psi^*\Psi_{k+2}\Psi_k\psi_{k+2} \vs\\
\end{align*}

\item[3)] Similarly, if $\vtt$ has $\Psi_{k-2}\Psi_k$ on the left, we see that this means that $k-3,k-2\in\Leg(\ttt')$ and $k-1,k,k+1,k+2\in\Arm(\ttt')$.
Then:
\begin{align*}
\psi_k\vtt&=
\psi_k\psi^*\Psi_{k-2}\Psi_k \vs\\
&=\psi^*\psi_k\psi_{k-2}\psi_{k-3}\psi_{k-1}\psi_{k-2}\psi_k\psi_{k-1}\psi_{k+1}\psi_k \vs\\
&=\psi^*\psi_{k-2}\psi_{k-3}(\psi_k\psi_{k-1}\psi_k)\psi_{k-2}\psi_{k-1}\psi_{k+1}\psi_k \vs\\
&=\psi^*\psi_{k-2}\psi_{k-3}(\psi_{k-1}\psi_{k}\psi_{k-1})\psi_{k-2}\psi_{k-1}\psi_{k+1}\psi_k \vs\\
&\begin{rcases}
-\psi^*\psi_{k-2}\psi_{k-3}(y_{k-1})\psi_{k-2}\psi_{k-1}\psi_{k+1}\psi_k \vs\\
+2\psi^*\psi_{k-2}\psi_{k-3}(y_{k})\psi_{k-2}\psi_{k-1}\psi_{k+1}\psi_k \vs
\end{rcases}
=0\text{ by \cref{standard-y}}\\
&-\psi^*\psi_{k-2}\psi_{k-3}(y_{k+1})\psi_{k-2}\psi_{k-1}\psi_{k+1}\psi_k \vs\\
&=\psi^*\psi_{k-2}\psi_{k-3}\psi_{k-1}\psi_{k}(\psi_{k-1}\psi_{k-2}\psi_{k-1})\psi_{k+1}\psi_k \vs\\
&+\psi^*\psi_{k-2}\psi_{k-3}\psi_{k-2}\psi_{k-1}\psi_k \vs\\
&=\psi^*\psi_{k-2}\psi_{k-3}\psi_{k-1}\psi_{k}(\psi_{k-2}\psi_{k-1}\psi_{k-2})\psi_{k+1}\psi_k \vs\\
&+\psi^*(\psi_{k-2}\psi_{k-3}\psi_{k-2})\psi_{k-1}\psi_k \vs\\
&=\psi^*\Psi_{k-2}\Psi_{k}\psi_{k-2} \vs\\
&+\psi^*(\psi_{k-3}\psi_{k-2}\psi_{k-3}-y_{k-3}+2y_{k-2}-y_{k-1})\psi_{k-1}\psi_k \vs\\
&=\psi^*\Psi_{k-2}\Psi_{k}\psi_{k-2} \vs\\
&+\psi^*\underbrace{(0-0+0-y_{k-1})}_{\text{by \cref{even,standard-y}}}\psi_{k-1}\psi_k \vs\\
&=\psi^*\Psi_{k-2}\Psi_{k}\psi_{k-2} \vs
+\psi^*\psi_k \vs
\end{align*}
\end{enumerate}

In order to prove the results in the second and third assertions, we will show that 
\begin{align*}
\Psi_k\psi_{k+2} \vs=0.&\hspace{3cm} (B)\\
\Psi_k\psi_{k-2} \vs=0,&\hspace{3cm} (C)
\end{align*}
where $\vtt=\psi^*\Psi_{k+2}\Psi_k\vs$ (case 2) and $\vtt=\psi^*\Psi_{k-2}\Psi_k\vs$ (case 3) for $\ttt\in\std(\la)$.
We will do this by simultaneous induction on $r$, the number of $\Psi$ terms appearing in $\vs$.

Suppose we can write $\vs$ above as $\vs=\Psichaind{y_1}{m}\dots\Psichaind{y_d}{a(-1)}(\psi_{\ell-1} \psi_{\ell-2} \dots \psi_{a+1} \psi_a)z^\la$.
We also denote $\vst=\Psichaind{y_2}{m+2}\dots\Psichaind{y_d}{a(-1)}(\psi_{\ell-1} \psi_{\ell-2} \dots \psi_{a+1} \psi_a)z^\la$. (Recall the bracket notation from \cref{short1}.)

Let $B_r$ and $C_r$ denote that statements $B$ and $C$ hold if $\vs$ contains $r$ $\Psi$ terms.

First, we prove that $(B_r)$ follows if both $(B_s)$ and $(C_s)$ hold for all $s<r$.
For $B_0$,
\begin{itemize}
\item[i)] if $a$ is odd, then
\[
\Psi_k\psi_{k+2}z^\la=\psi_k\psi_{k+1}\psi_{k-1}\psi_{k+2}\psi_k z^\la=0
\]
as at least one of $\psi_k$ and $\psi_{k+2}$ kills $z^\la$; and

\item[ii)] if $a$ is even, then
\begin{align*}
\Psi_k\psi_{k+2}\psi_{\ell-1}\dots\psi_az^\la&=\Psi_k\psi_{\ell-1}\dots\psi_{k+4}(\psi_{k+2}\psi_{k+3}\psi_{k+2})\psi_{k+1}\dots\psi_az^\la\\
&=\Psi_k\psi_{\ell-1}\dots\psi_{k+4}(\psi_{k+3}\psi_{k+2}\psi_{k+3})\psi_{k+1}\dots\psi_az^\la\\
&+\Psi_k\psi_{\ell-1}\dots \psi_{k+4}y_{k+2}\psi_{k+1}\dots\psi_az^\la\\
&-2\Psi_k\psi_{\ell-1}\dots\psi_{k+4} y_{k+3}\psi_{k+1}\dots\psi_az^\la\\
&+\Psi_k\psi_{\ell-1}\dots\psi_{k+4} y_{k+4}\psi_{k+1}\dots\psi_az^\la\\
&=0.
\end{align*}
\end{itemize}

Next assume $r>0$ and that $B_s$ holds for all $s<r$.

\begin{itemize}

\item[a)] Let $y_1\ge k+6$. As $\ttt$ is standard, this is only possible if $\Psi_{k+2}\Psi_k$ formed the chain on the right of $\Psi_{y_1}$, hence we must have that $m_1=k+2$, then
\[
\Psi_k\psi_{k+2}\vs=\Psi_k\psi_{k+2}\Psichaind{y_1}{k+2}\vst=\Psi_{y_1}\Psi_k\psi_{k+2}\Psichaind{y_1-2}{k+2}\vst=0\text{ by $B_{r-1}$.}
\]

\item[b)] Let $y_1=k+4$.
Once again $m_1=k+2$, hence
\begin{align*}
\Psi_k(\psi_{k+2}\Psi_{k+4}\Psi_{k+2})\Psichaind{k+4}{k+2}\vst&=\Psi_k(\psi_{k+2}+\Psi_{k+4}\Psi_{k+2}\psi_{k+4})\vst\text{ by our previous result}\\
&=0 \text{ by $B_{r-2}$.}
\end{align*}

\item[c)] Finally, let $y_1=k-2$.
We immediately see that in this case $y_2\ge k+4$ (as $\vtp=\Psi_{k+2}\Psi_k\vs$, that is $\Psichaind{k+2}{k}$ sat on top of the first $\Psi$ chain).
Let $\vsk=\Psichaind{y_3}{m_3}\dots\Psichaind{y_d}{a(-1)}(\psi_{\ell-1}\dots\psi_a)z^\la$.
Then
\begin{align*}
\Psi_k\psi_{k+2}\vs&=\Psi_k\psi_{k+2}\Psichaind{k-2}{m_1}\Psichaind{y_2}{k+6}\Psi_{k+4}\Psi_{k+2}\Psichaind{k}{m_2}\vsk\\
&=\Psi_k\Psichaind{y_2}{k+6}\Psichaind{k-2}{m_1}(\psi_{k+2}\Psi_{k+4}\Psi_{k+2})\Psichaind{k}{m_2}\vsk\\
&=\Psi_k\Psichaind{y_2}{k+6}\Psichaind{k-2}{m_1}(\psi_{k+2}+\Psi_{k+4}\underbrace{\Psi_{k+2}\psi_{k+4})\Psichaind{k}{m_2}\vsk}_{=0\text{ by $B_{s}$ for some $s<r$}}\\
&=\Psichaind{y_2}{k+6}\Psi_k\psi_{k+2}\Psichaind{k-2}{m_1}\Psichaind{k}{m_2}\vsk\\
&=\Psichaind{y_2}{k+6}\psi_k\psi_{k-1}\psi_{k+1}\psi_{k+2}(\psi_k\Psi_{k-2}\Psi_k)\Psichaind{k-4}{m_1}\Psichaind{k-2}{m_2}\vsk\\
&=\Psichaind{y_2}{k+6}\psi_k\psi_{k-1}\psi_{k+1}\psi_{k+2}\underbrace{(\psi_k+\Psi_{k-2}\Psi_k\psi_{k-2})}_{\text{ by our previous result}}\Psichaind{k-4}{m_1}\Psichaind{k-2}{m_2}\vsk\\
&=\Psichaind{y_2}{k+6}\psi_k\psi_{k-1}\psi_{k+1}\psi_{k+2}\psi_k\Psichaind{k-4}{m_1}\Psichaind{k-2}{m_2}\vsk\\
&+\Psichaind{y_2}{k+6}\psi_k\psi_{k-1}\psi_{k+1}\psi_{k+2}\Psi_{k-2}\underbrace{\Psi_k\psi_{k-2}\Psichaind{k-4}{m_1}\Psichaind{k-2}{m_2}\vsk}_{=0\text{ by $C_s$ for some $s<r$}}\\
&=\Psichaind{y_2}{k+6}\Psichaind{k-4}{m_1}\Psi_k\psi_{k+2}\Psichaind{k-2}{m_2}\vsk\\
&=0\text{ by $B_s$ for some $s<r$.}
\end{align*}

\end{itemize}

Now we prove that $(C_r)$ holds if $(B_s)$ and $(C_s)$ hold for all $s<r$.
For $C_0$,
\begin{itemize}
    \item[i)] if $a$ is odd, then
\[
\Psi_k\psi_{k-2}z^\la=\psi_k\psi_{k+1}\psi_{k-1}\psi_{k-2}\psi_k z^\la=0
\]
since at least one of $\psi_k$ and $\psi_{k-2}$ kills $z^\la$; and

\item[ii)] if $a$ is even, then

\begin{align*}
\Psi_k\psi_{k-2}\psi_{\ell-1}\dots\psi_az^\la&=\Psi_k\psi_{\ell-1}\dots(\psi_{k-2}\psi_{k-1}\psi_{k-2})\psi_{k-3}\dots\psi_az^\la\\
&=\Psi_k\psi_{\ell-1}\dots(\psi_{k-1}\psi_{k-2}\psi_{k-1})\psi_{k-3}\dots\psi_az^\la\\
&-\Psi_k\psi_{\ell-1}\dots y_{k-2}\psi_{k-3}\dots\psi_az^\la\\
&+2\Psi_k\psi_{\ell-1}\dots y_{k-1}\psi_{k-3}\dots\psi_az^\la\\
&-\Psi_k\psi_{\ell-1}\dots y_k\psi_{k-3}\dots\psi_az^\la\\
&=0.
\end{align*}
\end{itemize}

Next assume $r>0$  and that $C_s$ holds for all $s<r$.
Let $\vs, \vst$, $\vsk$ be as above and write $\vsf=\Psichaind{y_4}{m_4}\dots\Psichaind{y_d}{a(-1)}(\psi_{\ell-1}\dots\psi_a)z^\la$.
As $\ttt\in\std(\la)$ and $\vtt=\Psi_{k-2}\Psi_k\vs$, we immediately see that $\vs$ starts with $\Psi_{k-4}$, $\Psi_{k-2}$ or some $\Psi_y$ such that $y\ge k+2$.

\begin{itemize}
\item[a)] If $y_1\ge k+2$, then $m_1=k+2$ by a similar reasoning as above, hence
\[
\Psi_k\psi_{k-2}\vs=0,
\]
as in this case $\psi_{k-2}$ commutes with everything to its left because the lowest indexed $\Psi$-term in $\vst$ is $\Psi_{k+4}$.

\item[b)] Let $y_1=k-4$.
We immediately see that in this case $y_2= k-2$, as $\vtp=\Psi_{k-2}\Psi_k\vs$.
Here we have two subcases to consider.
\begin{itemize}

    \item [i)] $y_3\ge k+2$:
    
    \begin{align*}
        \Psi_k\psi_{k-2}\vs&=\Psi_k\psi_{k-2}\Psi_{k-4}\Psi_{k-2}\Psichaind{k-6}{m_1}\Psichaind{k-4}{m_2}\Psichaind{y_3}{m_3}\vsf\\
    &=\Psi_k\underbrace{(\psi_{k-2}+\Psi_{k-4}\Psi_{k-2}\psi_{k-4})}_{\text{by our previous result}}\Psichaind{k-6}{m_1}\Psichaind{k-4}{m_2}\Psichaind{y_3}{m_3}\vsf\\
    &=\Psi_k\psi_{k-2}\Psichaind{k-6}{m_1}\Psichaind{k-4}{m_2}\Psichaind{y_3}{m_3}\vsf\\
    &+\Psi_k\Psi_{k-4}\Psi_{k-2}\psi_{k-4}\Psichaind{k-6}{m_1}\Psichaind{k-4}{m_2}\Psichaind{y_3}{m_3}\vsf=0\text{ by $C_{r-2}$}\\
    &=\Psichaind{y_3}{k+4}\Psichaind{k-6}{m_1}\psi_k\psi_{k-1}\psi_{k+1}\psi_{k-2}\Psichaind{k-4}{m_2}\psi_k\Psi_{k+2}\Psi_{k}\Psichaind{k-2}{m_3}\vsf\\
    &=\Psichaind{y_3}{k+4}\Psichaind{k-6}{m_1}\psi_k\psi_{k-1}\psi_{k+1}\psi_{k-2}\Psichaind{k-4}{m_2}\underbrace{(\psi_k+\Psi_{k+2}\Psi_{k}\psi_{k+2})}_{\text{by our previous result}}\Psichaind{k-2}{m_3}\vsf\\
    &=\Psichaind{y_3}{k+4}\Psichaind{k-6}{m_1}\psi_k\psi_{k-1}\psi_{k+1}\psi_{k-2}\Psichaind{k-4}{m_2}\psi_k\Psichaind{k-2}{m_3}\vsf\\
    &+\Psichaind{y_3}{k+4}\Psichaind{k-6}{m_1}\psi_k\psi_{k-1}\psi_{k+1}\psi_{k-2}\Psichaind{k-4}{m_2}\Psi_{k+2}\underbrace{\Psi_{k}\psi_{k+2}\Psichaind{k-2}{m_3}\vsf}_{=0\text{ $B_s$ for some $s<r$}}\\
    &=\Psichaind{y_3}{k+4}\Psichaind{k-6}{m_1}\Psi_k\psi_{k-2}\Psichaind{k-4}{m_2}\Psichaind{k-2}{m_3}\vsf\\
    &=0\text{ by $C_s$ for some $s<r$.}
    \end{align*}

\item [ii)] $y_3=k$:

\begin{align*}
        \Psi_k\psi_{k-2}\vs&=\Psi_k\psi_{k-2}\Psi_{k-4}\Psi_{k-2}\Psichaind{k-6}{m_1}\Psichaind{k-4}{m_2}\Psichaind{k}{m_3}\vsf\\
    &=\Psi_k\underbrace{(\psi_{k-2}+\Psi_{k-4}\Psi_{k-2}\psi_{k-4})}_{\text{by our previous result}}\Psichaind{k-6}{m_1}\Psichaind{k-4}{m_2}\Psichaind{k}{m_3}\vsf\\
    &=\Psi_k\psi_{k-2}\Psichaind{k-6}{m_1}\Psichaind{k-4}{m_2}\Psichaind{k}{m_3}\vsf\\
    &+\Psi_k\Psi_{k-4}\Psi_{k-2}\psi_{k-4}\Psichaind{k-6}{m_1}\Psichaind{k-4}{m_2}\Psichaind{k}{m_3}\vsf=0\text{ by $C_{r-2}$}\\
    &=\Psichaind{k-6}{m_1}\psi_k\psi_{k+1}\psi_{k-1}\psi_{k-2}\Psichaind{k-4}{m_2}\psi_k\Psichaind{k}{m_3}\vsf\\
    &=-2\Psichaind{k-6}{m_1}\psi_k\psi_{k+1}\psi_{k-1}\psi_{k-2}\Psichaind{k-4}{m_2}\psi_k\Psichaind{k-2}{m_3}\vsf\text{ by case $A$}\\
    &=-2\Psichaind{k-6}{m_1}\Psi_k\psi_{k-2}\Psichaind{k-4}{m_2}\Psichaind{k-2}{m_3}\vsf\\
    &=0\text{ by $C_s$ for some $s<r$.}
    \qedhere
    \end{align*}

\end{itemize}
\end{itemize}
\end{proof}

Now we are ready to prove the remaining case from \cref{pre-even}:

\begin{prop}\label{even}
If $k$ is even and $k$ and $k+1$ are adjacent in $\ttt\in\std(\la)$, then $\psi_k\vtt=0$.
\end{prop}

\begin{answer}
\begin{rmk}
Notice that if $k+1=n$, the result holds by \cref{even} case $4$ part $A$, so we will assume that $k+1\le n-1$.

Let $\ttt(1,m-1)=x+1$.
In the proofs below, it might seem that we require $x \le k-2$, however this is not an extra condition, but follows immediately from the reduced expression for $\vtt$: if $x=k-1$ we have an additional $\psi_{k-2}$ on the left.
As $\psi_k$ and $\psi_{k-2}$ commute, $\psi_{k-2}$ is in $\psi^*$ and it is irrelevant for our computations.

By a similar reasoning, we only need to consider the cases when $\ttt(1,m)=k+2$: if $\ttt'(1,m)=f$ for some $f>k+2$, then $\psi^*$ involves an additional chain $\psi_{f-1}\dots\psi_{k+2}$.
\end{rmk}
\end{answer}

\begin{proof}
By \cref{pre-even}, we only need to consider the case when $\tp_{k+1}\in\Leg(\ttt)$ and $\bi_{\tp_{k+1}}=(0,1)$.

Let $m$ be minimal such that $\ttt(1,m-1)\ge m+1$ and $\ttt(1,m)\ge m+2$ for $m$ odd and let $E_m$ denote that the statement holds if $\Psichaind{x}{m}$ is the first $\Psi$ chain that $\psi_k$ does not commute with.
We will proceed by reverse induction on $m$.
Details can be included by switching the toggle in the {\tt arXiv} version of this paper.
\begin{answer}
For the base case, assume $m>\ell$, so we have no $\Psi$ terms appearing in $\vtt$.
If $k>\ell$ ($\ell=\ttt(1,a)$), the result holds by \cref{pre-even} case $4$.

If $k<\ell$, we have two cases to consider.
\begin{itemize}
    \item[i)] If $a$ is odd, then 
\begin{align*}
\psi_k\vtt&=\psi^*\psi_k\psi_{\ell-1}\dots\psi_az^\la\\
&=\psi^*\psi_{\ell-1}\dots\psi_{k+2}(\psi_k\psi_{k+1}\psi_k)\psi_{k-1}\dots\psi_az^\la\\
&=\psi^*\psi_{\ell-1}\dots\psi_{k+2}(\psi_{k+1}\psi_k\psi_{k+1})\psi_{k-1}\dots\psi_az^\la\\
&=0.
\end{align*}

\item[ii)] If $a$ is even, we have that
\begin{align*}
\psi_k\vtt&=\psi^*\psi_k\psi_{k+1}\dots\psi_{a-1}\psi_{k+2}\psi_{k+1}\dots\psi_az^\la\\
&=\psi^*(\psi_k\psi_{k+1}\psi_k)\dots\psi_{a-1}\psi_{k+2}\psi_{k+1}\dots\psi_az^\la\\
&=\psi^*\psi_{k+1}\psi_k\dots\psi_{a-1}(\psi_{k+1}\psi_{k+2}\psi_{k+1})\psi_k\dots\psi_az^\la\\
&=\psi^*\psi_{k+1}\psi_k\dots\psi_{a-1}(\psi_{k+2}\psi_{k+1}\psi_{k+2})\psi_k\dots\psi_az^\la=0\\
&+\psi^*\psi_{k+1}\psi_k\dots\psi_{a-1}(y_{k+1})\psi_k\dots\psi_az^\la\\
&\begin{rcases}
-\psi^*2\psi_{k+1}\psi_k\dots\psi_{a-1}(y_{k+2})\psi_k\dots\psi_az^\la\\
+\psi^*\psi_{k+1}\psi_k\dots\psi_{a-1}(y_{k+3})\psi_k\dots\psi_az^\la\\
\end{rcases}
=0\text{ by \cref{standard-y}}\\
&=\psi^*\psi_{k+1}\psi_k\dots\psi_{a-1}\psi_{k-1}\dots\psi_az^\la\\
&=\psi^*\psi_{k+1}\psi_k(\psi_{k-1}\psi_{k-2}\psi_{k-1})\dots\psi_{a-1}\psi_{k-2}\dots\psi_az^\la\\
&=\psi^*\psi_{k+1}\psi_k(\psi_{k-2}\psi_{k-1}\underbrace{\psi_{k-2})\dots\psi_{a-1}\psi_{k-2}\dots\psi_az^\la}_{=0\text{ by \cref{pre-even} case $4$}}\\
&\begin{rcases}
-\psi^*\psi_{k+1}\psi_k(y_{k-2})\dots\psi_{a-1}\psi_{k-2}\dots\psi_az^\la\\
+\psi^*2\psi_{k+1}\psi_k(y_{k-1})\dots\psi_{a-1}\psi_{k-2}\dots\psi_az^\la\\
-\psi^*\psi_{k+1}\psi_k(y_k)\dots\psi_{a-1}\psi_{k-2}\dots\psi_az^\la\\
\end{rcases}
=0\text{ by \cref{standard-y}}\\
&=0.
\end{align*}

\end{itemize}

For the induction step, we have several subcases to look at.

\begin{itemize}
\item[i)] $k+2\in\Arm(\ttt)$ and $\ttt(1,a)=k+2$ (i.e.\ $m=a$):

Observe that here we must have that $a$ is odd and $x\le k-3$.
\begin{align*}
    \psi_k\vtt&=\psi^*(\psi_k\psi_{k+1}\psi_k)\psi_{k-1}\dots\psi_{x+2}\Psichaind{x}{a}z^\la\\
    &=\psi^*(\psi_{k+1}\psi_k\psi_{k+1})\psi_{k-1}\dots\psi_{x+2}\Psichaind{x}{a}z^\la\\
    &=\psi^*\psi_{k+1}\psi_k\psi_{k-1}\dots\psi_{x+2}\Psichaind{x}{a}\psi_{k+1}z^\la\\
    &=0
\end{align*}

\item[ii)] $k+2\in\Arm(\ttt)$ and the entry to the right $k+2$ is $\ttt(1,a)=\ell$ (i.e.\ $m=a-1$):

Note that here we must have that $a$ is even and $x\le k-3$.

\begin{align*}
    \psi_k\vtt&=\psi^*(\psi_k\psi_{k+1}\psi_k)\psi_{k-1}\dots\psi_{x+2}\Psichaind{x}{a-1}\psi_{\ell-1}\dots\psi_az^\la\\
    &=\psi^*(\psi_{k+1}\psi_k\psi_{k+1})\psi_{k-1}\dots\psi_{x+2}\Psichaind{x}{a-1}\psi_{\ell-1}\dots\psi_az^\la\\
    &=\psi^*\psi_{k+1}\psi_k\psi_{k-1}\dots\psi_{x+2}\Psichaind{x}{a-1}\psi_{\ell-1}\dots\psi_{k+3}(\psi_{k+1}\psi_{k+2}\psi_{k+1})\psi_k\dots\psi_az^\la\\
    &=\psi^*\psi_{k+1}\psi_k\psi_{k-1}\dots\psi_{x+2}\Psichaind{x}{a-1}\psi_{\ell-1}\dots\psi_{k+3}(\psi_{k+2}\psi_{k+1}\underbrace{\psi_{k+2})\psi_k\dots\psi_az^\la}_{=0}\\
    &+\psi^*\psi_{k+1}\psi_k\psi_{k-1}\dots\psi_{x+2}\Psichaind{x}{a-1}\psi_{\ell-1}\dots\psi_{k+3}(y_{k+1})\psi_k\dots\psi_az^\la\\
    &\begin{rcases}
    -2\psi^*\psi_{k+1}\psi_k\psi_{k-1}\dots\psi_{x+2}\Psichaind{x}{a-1}\\
     \hspace{4cm}\cdot\psi_{\ell-1}\dots\psi_{k+3}(y_{k+2})\psi_k\dots\psi_az^\la\\
    +\psi^*\psi_{k+1}\psi_k\psi_{k-1}\dots\psi_{x+2}\Psichaind{x}{a-1}\\
     \hspace{4cm}\cdot\psi_{\ell-1}\dots\psi_{k+3}(y_{k+3})\psi_k\dots\psi_az^\la
    \end{rcases}
    =0\text{ by \cref{standard-y}}\\
     &=\psi^*\psi_{k+1}\psi_k\psi_{k-1}\dots\psi_{x+2}\Psichaind{x}{a-1}\psi_{\ell-1}\dots\psi_{k+3}\psi_{k-1}\dots\psi_az^\la\\
    &=0
\end{align*}

\item[iii)] $k+2,k+3\in\Arm(\ttt)$, $\ttt'(1,a)\neq k+3$ and $a\ge m+2$:

We have two cases to consider.
\begin{itemize}
    \item[a)] $m\le x=k-3$:

\begin{align*}
    \psi_k\vtt&=\psi^*(\psi_k\psi_{k+1}\psi_k)\psi_{k-1}\Psichaind{k-3}{m}\psi_{k+2}\psi_{k+3}\Psichaind{k+1}{m+2}\\
    &\hspace{3.5cm}\cdot\underbrace{\Psichaind{y_2}{m+4}\Psichaind{y_3}{m+6}\dots\Psichaind{y_d}{a(-1)}(\psi_{\ell-1}\dots\psi_a)z^\la}_{=\vtp}\\
    &=\psi^*(\psi_{k+1}\psi_k\psi_{k+1})\psi_{k-1}\Psichaind{k-3}{m}\psi_{k+2}\psi_{k+3}\Psichaind{k+1}{m+2}\vtp\\
    &=\psi^*\psi_{k+1}\psi_k\psi_{k-1}\Psichaind{k-3}{m}(\psi_{k+1}\psi_{k+2}\psi_{k+1})\psi_{k+3}\psi_k\psi_{k+2}\psi_{k+1}\Psichaind{k-1}{m+2}\vtp\\
    &=\psi^*\psi_{k+1}\psi_k\psi_{k-1}\Psichaind{k-3}{m}(\psi_{k+2}\psi_{k+1}\underbrace{\psi_{k+2})\psi_{k+3}\psi_k\psi_{k+2}\psi_{k+1}\Psichaind{k-1}{m+2}\vtp}_{=0\text{\;by $E_{m+2}$}}\\
    &+\psi^*\psi_{k+1}\psi_k\psi_{k-1}\Psichaind{k-3}{m}(y_{k+1})\psi_{k+3}\psi_k\psi_{k+2}\psi_{k+1}\Psichaind{k-1}{m+2}\vtp\\
     &\begin{rcases}
    -2\psi^*\psi_{k+1}\psi_k\psi_{k-1}\Psichaind{k-3}{m}(y_{k+2})\\
     \hspace{4cm}\cdot\psi_{k+3}\psi_k\psi_{k+2}\psi_{k+1}\Psichaind{k-1}{m+2}\vtp\\
    +\psi^*\psi_{k+1}\psi_k\psi_{k-1}\Psichaind{k-3}{m}(y_{k+3})\cdot\\
     \hspace{4cm}\cdot\psi_{k+3}\psi_k\psi_{k+2}\psi_{k+1}\Psichaind{k-1}{m+2}\vtp
    \end{rcases}\
    =0\text{ by \cref{standard-y}}\\
    &=\psi^*\psi_{k+1}\psi_k\psi_{k-1}\Psichaind{k-3}{m}\psi_{k+3}\psi_{k+2}\psi_{k+1}\Psichaind{k-1}{m+2}\vtp\\
    &=\psi^*\psi_{k+1}\psi_k\psi_{k-3}\psi_{k-4}\psi_{k-2}\psi_{k-1}\cdot\\
    &\hspace{1cm}\cdot\underbrace{\psi_{k-2}\psi_{k-3}\Psichaind{k-5}{m}\psi_{k+3}\psi_{k+2}\psi_{k+1}\psi_{k-2}\psi_{k}\psi_{k-1}\Psichaind{k-3}{m+2}\vtp}_{=0\text{ by \cref{pre-even} case $4$}}\\
    &\begin{rcases}
    +\psi^*\psi_{k+1}\psi_k\psi_{k-3}\psi_{k-4}(y_{k-2})\psi_{k-3}\Psichaind{k-5}{m}\\
     \hspace{2cm}\cdot\psi_{k+3}\psi_{k+2}\psi_{k+1}\psi_{k-2}\psi_{k}\psi_{k-1}\Psichaind{k-3}{m+2}\vtp\\
    -2\psi^*\psi_{k+1}\psi_k\psi_{k-3}\psi_{k-4}(y_{k-1})\psi_{k-3}\Psichaind{k-5}{m}\\
     \hspace{2cm}\cdot\psi_{k+3}\psi_{k+2}\psi_{k+1}\psi_{k-2}\psi_{k}\psi_{k-1}\Psichaind{k-3}{m+2}\vtp\\
    +\psi^*\psi_{k+1}\psi_k\psi_{k-3}\psi_{k-4}(y_k)\psi_{k-3}\Psichaind{k-5}{m}\\
     \hspace{2cm}\cdot\psi_{k+3}\psi_{k+2}\psi_{k+1}\psi_{k-2}\psi_{k}\psi_{k-1}\Psichaind{k-3}{m+2}\vtp
    \end{rcases}
    =0\text{ by \cref{standard-y}}\\
    &=0,
\end{align*}
where $\Psichaind{k-5}{m}$ and $\Psichaind{k-3}{m+2}$might be empty.

    \item[b)] $m\le x < k-3$:

\begin{align*}
    \psi_k\vtt&=\psi^*(\psi_k\psi_{k+1}\psi_k)\psi_{k-1}\dots\psi_x\Psichaind{x}{m}\psi_{k+2}\psi_{k+3}\Psichaind{k+1}{m+2}\\
    &\hspace{3.5cm}\cdot\underbrace{\Psichaind{y_2}{m+4}\Psichaind{y_3}{m+6}\dots\Psichaind{y_d}{a(-1)}(\psi_{\ell-1}\dots\psi_a)z^\la}_{=\vtp}\\
    &=\psi^*(\psi_{k+1}\psi_k\psi_{k+1})\psi_{k-1}\dots\psi_x\Psichaind{x}{m}\psi_{k+2}\psi_{k+3}\Psichaind{k+1}{m+2}\vtp\\
    &=\psi^*\psi_{k+1}\psi_k\psi_{k-1}\dots\psi_x\Psichaind{x}{m}(\psi_{k+1}\psi_{k+2}\psi_{k+1})\psi_{k+3}\psi_k\psi_{k+2}\psi_{k+1}\Psichaind{k-1}{m+2}\vtp\\
     &=\psi^*\psi_{k+1}\psi_k\psi_{k-1}\dots\psi_x\Psichaind{x}{m}(\psi_{k+2}\psi_{k+1}\underbrace{\psi_{k+2})\psi_{k+3}\psi_k\psi_{k+2}\psi_{k+1}\Psichaind{k-1}{m+2}\vtp}_{=0\text{ by $E_{m+2}$}}\\
    &+\psi^*\psi_{k+1}\psi_k\psi_{k-1}\dots\psi_x\Psichaind{x}{m}(y_{k+1})\psi_{k+3}\psi_k\psi_{k+2}\psi_{k+1}\Psichaind{k-1}{m+2}\vtp\\
     &\begin{rcases}
    -2\psi^*\psi_{k+1}\psi_k\psi_{k-1}\dots\psi_x\Psichaind{x}{m}\\
     \hspace{2cm}\cdot(y_{k+2})\psi_{k+3}\psi_k\psi_{k+2}\psi_{k+1}\Psichaind{k-1}{m+2}\vtp\\
    +\psi^*\psi_{k+1}\psi_k\psi_{k-1}\dots\psi_x\Psichaind{x}{m}\\
     \hspace{2cm}\cdot(y_{k+3})\psi_{k+3}\psi_k\psi_{k+2}\psi_{k+1}\Psichaind{k-1}{m+2}\vtp
    \end{rcases}\
    =0\text{ by \cref{standard-y}}\\
    &=\psi^*\psi_{k+1}\psi_k\psi_{k-1}\dots\psi_x\Psichaind{x}{m}\psi_{k+3}\psi_{k+2}\psi_{k+1}\Psichaind{k-1}{m+2}\vtp\\
    &=\psi^*\psi_{k+1}\psi_k(\psi_{k-1}\psi_{k-2}\psi_{k-1})\psi_{k-3}\dots\psi_x\Psichaind{x}{m}\\
    &\hspace{4cm}\cdot\psi_{k+3}\psi_{k+2}\psi_{k+1}\psi_{k-2}\psi_{k}\psi_{k-1}\Psichaind{k-3}{m+2}\vtp\\
    &=\psi^*\psi_{k+1}\psi_k\psi_{k-2}\psi_{k-1}\\
    &\hspace{1cm}\underbrace{\psi_{k-2}\psi_{k-3}\dots\psi_x\Psichaind{x}{m}\psi_{k+3}\psi_{k+2}\psi_{k+1}\psi_{k-2}\psi_{k}\psi_{k-1}\Psichaind{k-3}{m+2}\vtp}_{=0\text{ by \cref{pre-even} case $4$}}\\
    &\begin{rcases}
    +\psi^*\psi_{k+1}\psi_k(y_{k-2})\psi_{k-3}\dots\psi_x\Psichaind{x}{m}\\
     \hspace{2cm}\cdot\psi_{k+3}\psi_{k+2}\psi_{k+1}\psi_{k-2}\psi_{k}\psi_{k-1}\Psichaind{k-3}{m+2}\vtp\\
    -2\psi^*\psi_{k+1}\psi_k(y_{k-1})\psi_{k-3}\dots\psi_x\Psichaind{x}{m}\\
     \hspace{2cm}\cdot\psi_{k+3}\psi_{k+2}\psi_{k+1}\psi_{k-2}\psi_{k}\psi_{k-1}\Psichaind{k-3}{m+2}\vtp\\
    +\psi^*\psi_{k+1}\psi_k(y_{k})\psi_{k-3}\dots\psi_x\Psichaind{x}{m}\\
     \hspace{2cm}\cdot\psi_{k+3}\psi_{k+2}\psi_{k+1}\psi_{k-2}\psi_{k}\psi_{k-1}\Psichaind{k-3}{m+2}\vtp
    \end{rcases}
    =0\text{ by \cref{standard-y}}\\
    &=0
\end{align*}
\end{itemize}

\item[iv)] $k+3\in\Leg(\ttt)$, $k+2\in\Arm(\ttt)$ and $a\ge m+2$:

\begin{itemize}
    \item[a)] $m=x=k-3$:
    
    \begin{align*}
    \psi_k\vtt&=\psi^*(\psi_k\psi_{k+1}\psi_k)\psi_{k-1}\Psi_{k-3}\Psichaind{y_1}{k-1}\Psichaind{y_2}{k+1}\underbrace{\Psichaind{y_3}{k+3}\dots\Psichaind{y_d}{a(-1)}(\psi_{\ell-1}\dots\psi_a)z^\la}_{=\vtp}\\
    &=\psi^*(\psi_{k+1}\psi_k\psi_{k+1})\psi_{k-1}\Psi_{k-3}\Psichaind{y_1}{k-1}\Psichaind{y_2}{k+1}\vtp\\
    &=\psi^*\psi_{k+1}\psi_k\Psichaind{y_1}{k+5}\psi_{k-1}\Psi_{k-3}\psi_{k+1}\Psichaind{k+3}{k+1}\Psi_{k-1}\Psichaind{y_2}{k+1}\vtp\\
    &=\psi^*\psi_{k+1}\psi_k\Psichaind{y_1}{k+5}\psi_{k-1}\Psi_{k-3}\psi_{k+1}\Psi_{k-1}\Psichaind{y_2}{k+1}\vtp\text{ by \cref{handy} case $B$}\\
    &=\psi^*\psi_{k+1}\psi_k\Psichaind{y_1}{k+5}\psi_{k+1}\psi_{k-1}\Psi_{k-3}\Psi_{k-1}\Psichaind{y_2}{k+1}\vtp\\
    &=\psi^*\psi_{k+1}\psi_k\Psichaind{y_1}{k+5}\psi_{k+1}\psi_{k-1}\Psichaind{y_2}{k+1}\vtp \text{ by \cref{handy} case $C$}\\
    &=\psi^*\psi_{k+1}\psi_k\Psichaind{y_1}{k+5}\Psichaind{y_2}{k+5}\psi_{k-1}\psi_{k+1}\Psichaind{k+3}{k+1}\vtp\\
    &=\psi^*\psi_{k+1}\psi_k\Psichaind{y_1}{k+5}\Psichaind{y_2}{k+5}\psi_{k-1}\psi_{k+1}\vtp \text{ by \cref{handy} case $B$}\\
    &=0,
    \end{align*}
    as $\psi_{k-1}$ commutes with every $\psi$ term in $\vtp$.
    
  \item[b)] $m\le x \le k-5$:
\begin{align*}
    \psi_k\vtt&=\psi^*(\psi_k\psi_{k+1}\psi_k)\psi_{k-1}\dots\psi_{x+2}\Psichaind{x}{m}\Psichaind{y_1}{m+2}\\
    &\hspace{3.5cm}\cdot\underbrace{\Psichaind{y_2}{m+4}\Psichaind{y_3}{m+6}\dots\Psichaind{y_d}{a(-1)}(\psi_{\ell-1}\dots\psi_a)z^\la}_{=\vtp}\\
    &=\psi^*(\psi_{k+1}\psi_k\psi_{k+1})\psi_{k-1}\dots\psi_{x+2}\Psichaind{x}{m}\Psichaind{y}{m+2}\vtp\\
    &=\psi^*\Psichaind{y}{k+5}\psi_{k+1}\psi_k\psi_{k-1}\dots\psi_{x+2}\Psichaind{x}{m}\psi_{k+1}\Psi_{k+3}\Psi_{k+1}\Psichaind{k-1}{m+2}\vtp\\
    &=\psi^*\Psichaind{y}{k+5}\psi_{k+1}\psi_k\psi_{k-1}\dots\psi_{x+2}\Psichaind{x}{m}\psi_{k+1}\Psichaind{k-1}{m+2}\vtp\text{ by \cref{handy} case $B$}\\
   &=\psi^*\Psichaind{y}{k+5}\psi_{k+1}\psi_k(\psi_{k-1}\psi_{k-2}\psi_{k-1})\psi_{k-3}\dots\psi_{x+2}\\
   &\hspace{4cm}\cdot\Psichaind{x}{m}\psi_{k+1}\psi_k\psi_{k-2}\psi_{k-1}\Psichaind{k-3}{m+2}\vtp\\
    &=\psi^*\Psichaind{y}{k+5}\psi_{k+1}\psi_k\psi_{k-2}\psi_{k-1}\\
    &\hspace{1cm}\underbrace{\psi_{k-2}\psi_{k-3}\dots\psi_{x+2}\Psichaind{x}{m}\psi_{k+1}\psi_k\psi_{k-2}\psi_{k-1}\Psichaind{k-3}{m+2}\vtp}_{=0\text{ by \cref{pre-even} case $4$}}\\
    &\begin{rcases}
    +\psi^*\Psichaind{y}{k+5}\psi_{k+1}\psi_k(y_{k-2})\psi_{k-3}\dots\psi_{x+2}\Psichaind{x}{m}\\
    \hspace{4cm}\cdot\psi_{k+1}\psi_k\psi_{k-2}\psi_{k-1}\Psichaind{k-3}{m+2}\vtp\\
    -2\psi^*\Psichaind{y}{k+5}\psi_{k+1}\psi_k(y_{k-1})\psi_{k-3}\dots\psi_{x+2}\Psichaind{x}{m}\\
    \hspace{4cm}\cdot\psi_{k+1}\psi_k\psi_{k-2}\psi_{k-1}\Psichaind{k-3}{m+2}\vtp\\
    +\psi^*\Psichaind{y}{k+5}\psi_{k+1}\psi_k(y_k)\psi_{k-1}\dots\psi_{x+2}\Psichaind{x}{m}\\
    \hspace{4cm}\cdot\psi_{k+1}\psi_k\psi_{k-2}\psi_{k-1}\Psichaind{k-3}{m+2}\vtp
    \end{rcases}
    =0\text{ by \cref{standard-y}}\\
    &=0.
    \qedhere
    \end{align*}
\end{itemize}
\end{itemize}
\end{answer}
\end{proof}

\subsection{Three consecutive elements in $\Leg(\ttt)$ or $\Arm(\ttt)$}

From now on, we will be working with \cref{form2}.
In order to fully describe the action of the $\psi$-generators on the basis elements $\vtt$, we still need to consider the cases when the first element in $\vtt$ that $\psi_k$ does not commute with is $\psi_{k-1}$ or $\psi_{k+1}$.
We will do this in the next two statements where we also give examples of the types of tableaux corresponding to such expressions.

\begin{lem}\label{halfdom-arm}
Let $k$ be odd and $\ttt,\ttt'\in\std(\la)$ with corresponding $\vtt,\vtp$.
\begin{itemize}
    \item If $\vtt=\psi^{*}\psi_{k-1}\psi_{k-2}(\psi_{k-3}\dots\psi_{x+2})\Psichaind{x}{m-2}\Psichaind{k}{m}\Psichaind{y_1}{m+2}\vtp$ where $x\le k-4$ and $y_1\ge k+2$, then $\psi_k\vtt=\psi^{*}\psi_{k-2}\psi_{k-1}\psi_k\Psichaind{x}{m-2}\Psichaind{k-2}{m}\vtp$.
    \item If $\vtt=\psi^*\psi_{k-1}\psi_{k}\psi_{k+1}\psi_{k+2}\Psichaind{k-2}{m-2}\Psichaind{k}{m}\Psichaind{y_1}{m+2}\vtp$ where $y_1\ge k+4$, then $\psi_k\vtt=\psi^*\psi_{k+2}\Psichaind{k-2}{m-2}\Psichaind{k}{m}\Psichaind{y_1}{m+2}\vtp$.
\end{itemize}
\end{lem}

\begin{lem}\label{halfdom-leg}
Let $k$ be odd $\ttt,\ttt'\in\std(\la)$ with corresponding $\vtt,\vtp$.
\begin{itemize}
    \item If $\vtt=\psi^*\psi_{k+1}\psi_k\psi_{k-1}\psi_{k-2}(\psi_{k-3}\dots\psi_{x+2})\Psichaind{x}{m}\Psichaind{y_1}{m+2}\vtp$ where $x\le k-4$ and $y_1\ge k+2$, then $\psi_k\vtt=\psi^{*}\psi_{k-2}(\psi_{k-3}\dots\psi_{x+2})\Psichaind{x}{m}\Psichaind{y_1}{m+2}\vtp$.
    \item If $\vtt=\psi^*\psi_{k+1}\psi_{k+2}\Psichaind{k}{m}\Psichaind{y_1}{m+2}\vtp$ where $y_1\ge k+4$, then $\psi_k\vtt=\psi^*\psi_{k+2}\psi_{k+1}\Psichaind{k-2}{m}\Psichaind{y_1}{m+2}\vtp$.
\end{itemize}
\end{lem}

\begin{rmk}
Observe that \cref{halfdom-arm} can be visualised the following way: if $k-1\in\Leg(\ttt)$, $k,k+1,k+2\in\Arm(\ttt)$ and $i_{k-1}=i_k$, then $\psi_k\vtt=\vs$ where $\tts$ is the tableau with $k-1,k,k+1\in\Arm(\tts),k+2\in\Leg(\tts)$ and outside of these entries it agrees with $\ttt$.

Similarly, \cref{halfdom-leg} corresponds to the case when $k-1,k,k+1\in\Leg(\ttt), k+2\in\Arm(\ttt)$ and $i_{k+1}=i_{k+2}$. Then $\psi_k\vtt=\vs$ where $\tts$ is the tableau with $k,k+1,k+2\in\Leg(\tts),k-1\in\Arm(\tts)$ and outside of these entries it agrees with $\ttt$.
\end{rmk}

\begin{eg}
Let $\la=(8,1^6)$ and take
\[
\ttt=
\begin{array}{l}
\qquad
\begin{tikzpicture}[scale=0.5,draw/.append style={thick,black}]
     \fill[magenta!40](3.5,-0.5)rectangle(6.5,0.5);
      \fill[yellow!50](0.5,-5.5)rectangle(1.5,-2.5);
  \newcount\col
  \foreach\Row/\row in {{1,2,3,9,10,11,12,14}/0,{4}/-1,{5}/-2,{6}/-3,{7}/-4,{8}/-5,{13}/-6}{
     \col=1
    \foreach\k in \Row {
        \draw(\the\col,\row)+(-.5,-.5)rectangle++(.5,.5);
        \draw(\the\col,\row)node{\k};
    \global\advance\col by 1
     }
  }
    \draw (2,1) node[red] {\scriptsize$1$};
     \draw (3,1) node[red] {\scriptsize$0$};
      \draw (4,1) node[red] {\scriptsize$1$};
    \draw (5,1) node[red] {\scriptsize$0$};
    \draw (6,1) node[red] {\scriptsize$1$};
     \draw (7,1) node[red] {\scriptsize$0$};
    \draw (8,1) node[red] {\scriptsize$1$};
     \draw (0,-1) node[red] {\scriptsize$1$};
    \draw (0,-2) node[red] {\scriptsize$0$};
  \draw (0,-3) node[red] {\scriptsize$1$};
        \draw (0,-4) node[red] {\scriptsize$0$};   
       \draw (0,-5) node[red] {\scriptsize$1$};
        \draw (0,-6) node[red] {\scriptsize$0$};              
\end{tikzpicture}
\end{array}
v_{\ttt}=
\psi_{8}\psi_{9}\psi_{10}\psi_{11}
\Psichaind{7}{5}
\Psichaind{9}{7}
\psi_{13}\psi_{12}\dots\psi_{8}
z^\la.
\]
(Above, we added the corresponding residue next to each node for the reader's convenience.)
Here we have that $i_8=i_9=1$. Then
\[
\tts=
\begin{array}{l}
\qquad
\begin{tikzpicture}[scale=0.5,draw/.append style={thick,black}]
     \fill[blue!30](3.5,-0.5)rectangle(6.5,0.5);
      \fill[blue!30](0.5,-5.5)rectangle(1.5,-4.5);
  \newcount\col
  \foreach\Row/\row in {{1,2,3,8,9,10,12,14}/0,{4}/-1,{5}/-2,{6}/-3,{7}/-4,{11}/-5,{13}/-6}{
     \col=1
    \foreach\k in \Row {
        \draw(\the\col,\row)+(-.5,-.5)rectangle++(.5,.5);
        \draw(\the\col,\row)node{\k};
    \global\advance\col by 1
     }
  }             
\end{tikzpicture}
\end{array}
\vs=\psi_9v_{\ttt}=
\psi_{11}
\Psichaind{7}{5}
\Psichaind{9}{7}
\psi_{13}\psi_{12}\dots\psi_{8}
z^\la.
\]

Similarly,
\[
\ttp=
\begin{array}{l}
\qquad
\begin{tikzpicture}[scale=0.5,draw/.append style={thick,black}]
     \fill[blue!30](3.5,-0.5)rectangle(4.5,0.5);
      \fill[blue!30](0.5,-5.5)rectangle(1.5,-2.5);
  \newcount\col
  \foreach\Row/\row in {{1,2,3,6,10,11,12,14}/0,{4}/-1,{5}/-2,{7}/-3,{8}/-4,{9}/-5,{13}/-6}{
     \col=1
    \foreach\k in \Row {
        \draw(\the\col,\row)+(-.5,-.5)rectangle++(.5,.5);
        \draw(\the\col,\row)node{\k};
    \global\advance\col by 1
     }
  }        
\end{tikzpicture}
\end{array}
v_\ttp=\psi_7v_{\ttt}=
\psi_{9}\psi_8\psi_7\psi_{10}\psi_{11}
\Psi_5
\Psichaind{9}{7}
\psi_{13}\psi_{12}\dots\psi_{8}
z^\la.
\]

\end{eg}

\begin{proof}[Proof of \cref{halfdom-arm}]
We have two options: $i_{k-1}=0=i_k$ or $i_{k-1}=1=i_k$.
Fix some odd $m<a$.

\underline{Case $1$: $i_{k-1}=0=i_k$} 
Assume that $\ttt(1,m)=k$. In this case $\vtt=\psi^{*}\psi_{k-1}\psi_{k-2}\linebreak(\psi_{k-3}\dots\psi_{x+2})\Psichaind{x}{m-2}\Psichaind{k}{m}\Psichaind{y_1}{m+2}\dots \Psichaind{y_d}{a(-1)}(\psi_{\ell}\psi_{\ell-1}\dots\psi_a ) z^\la$ (here $x\le k-4$ and $y_1\ge k+2$).
\begin{align*}
\psi_k\vtt&=\psi_k\psi^{*}\psi_{k-1}\psi_{k-2}(\psi_{k-3}\dots\psi_{x+2})\Psichaind{x}{m-2}\Psichaind{k}{m}\Psichaind{y_1}{m+2}\dots \Psichaind{y_d}{a(-1)}( \psi_{\ell}\psi_{\ell-1}\dots\psi_a) z^\la \\
&=\psi^{*}\psi_k \psi_{k-1}\psi_{k-2}\Psi_k\underbrace{(\psi_{k-3}\dots\psi_{x+2})\Psichaind{x}{m-2}\Psichaind{k-2}{m}\Psichaind{y_1}{m+2}\dots \Psichaind{y_d}{a(-1)} ( \psi_{\ell}\psi_{\ell-1}\dots\psi_a )z^\la }_{\scriptsize\vs}\\
&=\psi^{*}\psi_k\psi_{k-1}\psi_{k-2}\psi_k\psi_{k-1}\psi_{k+1}\psi_k\vs\\
&=\psi^{*}(\psi_k\psi_{k-1}\psi_k)\psi_{k-2}\psi_{k-1}\psi_{k+1}\psi_k\vs\\
&=\psi^{*}(\psi_{k-1}\psi_{k}\underbrace{\psi_{k-1})\psi_{k-2}\psi_{k-1}\psi_{k+1}\psi_k\vs}_{=0\text{ by \cref{even}}}\\
&\begin{rcases}
-\psi^{*}(y_{k-1})\psi_{k-2}\psi_{k-1}\psi_{k+1}\psi_k\vs\\
+2\psi^{*}(y_{k})\psi_{k-2}\psi_{k-1}\psi_{k+1}\psi_k\vs\\
\end{rcases}
=0\text{ by \cref{standard-y}}\\
&-\psi^{*}(y_{k+1})\psi_{k-2}\psi_{k-1}\psi_{k+1}\psi_k\vs\\
&=\psi^{*}\psi_{k-2}\psi_{k-1}\psi_k\vs
\end{align*}

\underline{Case $2$: $i_{k-1}=1=i_k$}
Assume that $\ttt(1,m-3)=k$, then $\vtt=\psi^*\psi_{k-1}\psi_{k}\psi_{k+1}\linebreak\psi_{k+2}\Psichaind{k-2}{m-2}\Psichaind{k}{m}\Psichaind{y_1}{m+2}\dots \Psichaind{y_d}{a(-1)} (\psi_{\ell-1}\dots\psi_a) z^\la$ (here $y_1\ge k+4$ and $x<y_1$).
\nopagebreak
\begin{align*}
\psi_k\vtt&
=\psi_k\psi^{*} \psi_{k-1}\psi_k\psi_{k+1}(\psi_x\dots\psi_{k+3})\psi_{k+2}\underbrace{\Psichaind{k-2}{m-2}\Psichaind{k}{m}\Psichaind{y_1}{m+2}\dots \Psichaind{y_d}{a(-1)}(\psi_{\ell-1}\dots\psi_a )z^\la }_{\vs}\\
&=\psi^{*}(\psi_k\psi_{k-1}\psi_{k})\psi_{k+1}(\psi_x\dots\psi_{k+3})\psi_{k+2}\vs\\
&=\psi^{*}(\psi_{k-1}\psi_{k}\underbrace{\psi_{k-1})\psi_{k+1}(\psi_x\dots\psi_{k+3})\psi_{k+2}\vs}_{=0\text{ by \cref{even}}}\\
&\begin{rcases}
-\psi^{*}(y_{k-1})\psi_{k+1}(\psi_x\dots\psi_{k+3})\psi_{k+2}\vs\\
+2\psi^{*}(y_{k})\psi_{k+1}(\psi_x\dots\psi_{k+3})\psi_{k+2}\vs
\end{rcases}
=0\text{ by \cref{standard-y}}\\
&-\psi^{*}(y_{k+1})\psi_{k+1}(\psi_x\dots\psi_{k+3})\psi_{k+2}\vs\\
&=\psi^{*}(\psi_x\dots\psi_{k+3})\psi_{k+2}\vs
\end{align*}
Note that if $y_1=k+2$, we have that $\vtt=\psi^*\psi_{k-1}\psi_{k}\psi_{k+1}\psi_{k+2}\Psichaind{k-2}{m-2}\Psichaind{k}{m}\psi_{k+3}\psi_{k+4}\Psichaind{k+2}{m+2}\Psichaind{y_2}{m+4}\dots \Psichaind{y_d}{a(-1)} (\psi_{\ell-1}\dots\psi_a) z^\la$, but the proof is identical.
\end{proof}

\begin{proof}[Proof of \cref{halfdom-leg}]
Similar to the proof of \cref{halfdom-arm}.
Details can be included by switching the toggle in the {\tt arXiv} version of this paper.
\begin{answer}
We have two cases to consider: $i_{k+1}=0=i_{k+2}$ or $i_{k+1}=1=i_{k+2}$.

\underline{Case $1$: $i_{k+1}=0=i_{k+2}$} 
Suppose that $\ttt(1,m)=k+2$. In this case $\vtt=\psi^*\psi_{k+1}\psi_k\psi_{k-1}\psi_{k-2}\linebreak(\psi_{k-3}\dots\psi_{x+2})\Psichaind{x}{m}\Psichaind{y_1}{m+2}\dots  \Psichaind{y_d}{a(-1)}(\psi_{\ell-1}\dots\psi_a )z^\la$ (here $x\le k-4$ and $y_1\ge k+2$).
\begin{align*}
\psi_k\vtt&=\psi_k\psi^*\psi_{k+1}\psi_k\psi_{k-1}\psi_{k-2}
\underbrace{(\psi_{k-3}\dots\psi_{x+2})\Psichaind{x}{m}\Psichaind{y_1}{m+2}\dots  \Psichaind{y_d}{a(-1)}(\psi_{\ell-1}\dots\psi_a )z^\la }_{\scriptsize\vs}\\
&=\psi^*(\psi_k\psi_{k+1}\psi_k)\psi_{k-1}\psi_{k-2}\vs\\
&=\psi^*(\psi_{k+1}\psi_{k}\underbrace{\psi_{k+1})\psi_{k-1}\psi_{k-2}\vs}_{=0\text{ by \cref{even}}}\\
&+\psi^*(y_{k})\psi_{k-1}\psi_{k-2}\vs\\
&\begin{rcases}
-2\psi^*(y_{k+1})\psi_{k-1}\psi_{k-2}\vs\\
+\psi^*(y_{k+2})\psi_{k-1}\psi_{k-2}\vs
\end{rcases}
=0\text{ by \cref{standard-y}}\\
&=\psi^*\psi_{k-2}\vs
\end{align*}

\underline{Case $2$: $i_{k+1}=1=i_{k+2}$} 
Suppose that $\ttt(1,m-1)=k+2$, so $\vtt=\psi^*\psi_{k+1}\psi_{k+2}\Psichaind{k}{m}\Psichaind{y_1}{m+2}\dots \Psichaind{y_d}{a(-1)}(\psi_{\ell-1}\dots\psi_a) z^\la$ (here $y_1\ge k+4$).
\begin{align*}
\psi_k\vtt&=\psi_k\psi^{*}\psi_{k+1}\psi_{k+2}\Psichaind{k}{m}\Psichaind{y_1}{m+2}\dots \Psichaind{j_d}{a(-1)}(\psi_{\ell-1}\dots\psi_a )z^\la\\
&=\psi^{*}\psi_k \psi_{k+1}\psi_{k+2}\Psi_k\underbrace{\Psichaind{k-2}{m}\Psichaind{y_1}{m+2}\dots \Psichaind{j_d}{a(-1)}(\psi_{\ell-1}\dots\psi_a )z^\la}_{\scriptsize\vs}\\
&=\psi^{*}(\psi_k\psi_{k+1}\psi_k)\psi_{k+2}\psi_{k-1}\psi_{k+1}\psi_k\vs\\
&=\psi^{*}(\psi_{k+1}\psi_{k}\underbrace{\psi_{k+1}\psi_{k+2}\psi_{k-1}\psi_{k+1}\psi_k\vs}_{=0\text{ by \cref{even}}}\\
&+\psi^{*}(y_{k})\psi_{k+2}\psi_{k-1}\psi_{k+1}\psi_k\vs\\
&\begin{rcases}
-2\psi^{*}(y_{k+1})\psi_{k+2}\psi_{k-1}\psi_{k+1}\psi_k\vs\\
+\psi^{*}(y_{k+2})\psi_{k+2}\psi_{k-1}\psi_{k+1}\psi_k\vs
\end{rcases}
=0\text{ by \cref{standard-y}}\\
&=\psi^{*}\psi_{k+2}\psi_{k+1}\psi_k\vs
\end{align*}
Note that if $y_1=k+2$, we have $\vtt=\psi^*\psi_{k+1}\psi_{k+2}\Psichaind{k}{m}\psi_{k+3}\psi_{k+4}\Psichaind{k+2}{m+2}\Psichaind{y_2}{m+4}\dots \Psichaind{y_d}{a(-1)}(\psi_{\ell-1}\dots\psi_a) z^\la$, but the proof is identical.
\end{answer}
\end{proof}

\subsection{Two consecutive pairs in $\Leg(\ttt)$ or $\Arm(\ttt)$}

Now, we have all the tools necessary to prove the analogous statements of Corollaries $3.21$ and $3.25$ in \cite{thesis} for all $\ttt\in\std(\la)$.

For the rest of proofs in this section, we will be working with \cref{form2}.
To further simplify notation, we let $\chi_i$ denote the terms sandwiched between the $\Psichaind{y_i}{m_i}$-terms, that is $\chi_i$ corresponds to the individual $\psi$ terms between $\psi_{m_i}$ and $\psi_{y_i}$.
If $i=1$, $\chi_1$ denotes the first $\psi$-chain in $\vtt$ on the left.
For example (compare with \cref{form2}):
\begin{align}\label{short2}
\nonumber\vtt&=\psi_{y_1+1}\psi_{y_1+3}\psi_{y_1+2}\Psichaind{y_1}{m_1}\psi_{y_2+2}\Psichaind{y_2}{m_2}\Psichaind{y_3}{m_3}\psi_{y_4+1}\psi_{y_4+2}\Psichaind{y_4}{m_4}\vtp\\
&=\chi_1\Psichaind{y_1}{m_1}\chi_2\Psichaind{y_2}{m_2}\chi_3\Psichaind{y_3}{m_3}\chi_4\Psichaind{y_4}{m_4}\vtp,
\end{align}
where $\vtp$ corresponds to the rest of $\vtt$ on the right (it is in reduced form and shorter than $\vtt$).
Here we have
\begin{itemize}
    \item $\chi_1=\psi_{y_1+1}\psi_{y_1+3}\psi_{y_1+2}$;
    \item $\chi_2=\psi_{y_2+2}$;
    \item $\chi_3=\id$; and
    \item $\chi_4=\psi_{y_4+1}\psi_{y_4+2}$.
\end{itemize}

Unless $\chi_i=\id$, we define $\chi_i(\last)$ to be the leftmost $\psi$-term in $\chi_i$.
Note that $\chi_i(\last)=\psi_{y_i+2}$.
We also set $\chi_i(\first)=\psi_{y_i+1}$ ($\chi_i(\first)$ might not appear).
Also, consistently with this new notation, if $a$ is even, we write $\chi_\ell=\psi_{\ell-1}\dots\psi_a$.

The next proposition is the analogous version of \cite[Corollary 3.17]{thesis}.

\begin{prop}\label{3pairs}
Suppose $\ttt\in\std(\la)$ and $1\le k\le n-1$ is odd.
\begin{enumerate}
\item If $\tp_k,\tp_{k+2}\in\Leg(\ttt)$ and $\tp_{c}\in\Arm(\ttt)$ is minimal such that $c>k+2$ and the entries $k+3,\dots,c-1$ are all unpaired, then $\psi_k\vtt=\psi_k\vs$ where $\tts\in\std(\la)$ agrees with $\ttt$ outside of these three pairs, but has $\tp_k\in\Arm(\tts)$ and $\tp_{k+2},\tp_c\in\Leg(\tts)$.
\item If $\tp_{k},\tp_{k+2}\in\Arm(\ttt)$ and $\tp_c\in\Leg(\ttt)$ is maximal such that $c<k$ and the entries $c+1,\dots,k-2$ are all unpaired, then $\psi_k\vtt=\psi_k\vs$ where $\tts\in\std(\la)$ agrees with $\ttt$ outside of these three pairs, but has $\tp_{k+2}\in\Leg(\tts)$ and $\tp_c,\tp_k\in\Arm(\tts)$.
\end{enumerate}
\end{prop}

\begin{rem}
    In fact, the converse of \cref{3pairs} is also true: if there exists no $\tp_c$ satisfying cases 1 or 2, then $\psi_k\vtt=0$.
    We will explain the proof in the  remarks after \cref{4killing-1,4killing-2}.
\end{rem}

\begin{proof}
Throughout the proof we omit $\psi^*$ on the left (as usual $\psi^*$ denotes everything that $\psi_k$ commutes through on the left).
We also write
\[
\vtp=\chi_{h+1}\Psichaind{y_{h+1}}{m_{h+1}}\chi_{h+2}\Psichaind{y_{h+2}}{m_{h+2}}\dots\chi_{d}\Psichaind{y_{d}}{m_{d}}\chi_\ell z^\la.
\]

\underline{Case 1}

\begin{itemize}

    \item[a)] $\bi_{\tp_k}=(1,0)$, $\bi_{\tp_c}=(0,1)$ and $\tp_c$ appears at the end of the arm:
    
    We see that if $\ttt$ satisfies the conditions of the statement, then it has corresponding reduced expression
    \[
    \vtt=\chi_1\Psichaind{k+2}{m_1}(\chi_2\Psichaind{k+6}{m_2}\dots\chi_{h-1}\Psichaind{y_{h-1}}{m_{h-1}})\chi_{h}\Psichaind{y_h}{m_h}\chi_\ell z^\la,\\
    \]
   
where the indices $k+2, k+6, \dots, y_{h-1},y_h$ must be four apart by the minimality of $c$.
Hence we have that $y_{h-1}=k+2+4(h-2)$ and $y_h=k+2+4(h-1)= c-4$.
It is also clear that the first entry that $\psi_k$ does not commute with is $\Psi_{k+2}$ as $\tp_k,\tp_{k+2}$ are in the leg and $\bi_k=0$.
We also know that $\chi_i$ and $\chi_{i+1}$ commute if the corresponding $\Psi_{y_i}$ commute, that is $y_i\le y_{i+1}-4$.
Further, $\chi_1$ is always nonempty.
Then by \cref{handy}:
    
    \begin{align*}
\psi_k\vtt&=\psi_k\chi_1\Psichaind{k+2}{m_1}(\chi_2\Psichaind{k+6}{m_2}\dots\chi_{h-1}\Psichaind{k+2+4(h-2)}{m_{h-1}})\chi_{h}\Psichaind{c-4}{m_h}\chi_\ell z^\la\\
    &=\chi_1\psi_k\Psi_{k+2}\Psi_k\Psichaind{k-2}{m_1}(\chi_2\Psichaind{k+6}{m_2}\dots\chi_{h-1}\Psichaind{k+2+4(h-2)}{m_{h-1}})\chi_{h}\Psichaind{c-4}{m_h}\chi_\ell z^\la\\
    &=\chi_1\psi_k\Psichaind{k-2}{m_1}(\chi_2\Psichaind{k+6}{m_2}\dots\chi_{h-1}\Psichaind{k+2+4(h-2)}{m_{h-1}})\chi_{h}\Psichaind{c-4}{m_h}\chi_\ell z^\la\\
    &=\psi_k\Psichaind{k-2}{m_1}(\chi_2\chi_1\Psichaind{k+6}{m_2}\dots\chi_{h-1}\Psichaind{k+2+4(h-2)}{m_{h-1}})\chi_{h}\Psichaind{c-4}{m_h}\chi_\ell z^\la\\
    &=\psi_k\Psichaind{k-2}{m_1}(\underbrace{\chi_2\chi_1\Psi_{k+6}\Psi_{k+4}}_{\text{as $\chi_1$ and $\chi_2$ commute and $\chi_1(\last)=\psi_{k+4}$}}\Psichaind{k+2}{m_2}\dots\chi_{h-1}\Psichaind{k+2+4(h-2)}{m_{h-1}}\chi_{h}\Psichaind{c-4}{m_h}\chi_\ell z^\la\\
    &=\psi_k\Psichaind{k-2}{m_1}(\chi_2\chi_1\Psichaind{k+2}{m_2}\dots\chi_{h-1}\Psichaind{k+2+4(h-2)}{m_{h-1}})\chi_{h}\Psichaind{c-4}{m_h}\chi_\ell z^\la\\
    &\;\,\vdots\\
    &=\psi_k\Psichaind{k-2}{m_1}(\chi_1\Psichaind{k+2}{m_2}\chi_2\Psichaind{k+6}{m_3}\dots\chi_{h-2}\Psichaind{k+2+4(h-3)}{m_{h-1}})\chi_{h-1}\chi_h\Psichaind{c-4}{m_h+1}\chi_\ell z^\la\\
&=\psi_k\Psichaind{k-2}{m_1}\chi_1(\Psichaind{k+2}{m_2}\chi_2\dots\chi_{h-2}\Psichaind{k+2+4(h-3)}{m_{h-1}})\chi_{h-1}\Psichaind{c-8}{m_h}\chi_h\chi_\ell z^\la.
\end{align*}
    
We must have $\chi_h(\last)=\psi_{c-2}$, otherwise we cannot have $\tp_c\in\Arm(\ttt)$.
We will write $\chi_h=\widehat{\chi}_h\chi_h(\last)$.
Then
\begin{align*}
\chi_h(\last)\chi_\ell&=\psi_{c-2}\psi_{c-1}\psi_{c-2}\psi_{c-3}\dots\psi_az^\la  \\
&=\psi_{c-2}\psi_{c-1}\psi_{c-2}\psi_{c-3}\dots\psi_az^\la  \\
&=(\psi_{c-1}\psi_{c-2}\psi_{c-1})\psi_{c-3}\dots\psi_az^\la=0  \\
&+(y_{c-2})\psi_{c-3}\psi_{c-4}\dots\psi_az^\la  \\
&-2(y_{c-1})\psi_{c-3}\psi_{c-4}\dots\psi_az^\la =0 \\
&+(y_{c})\psi_{c-3}\psi_{c-4}\dots\psi_az^\la =0 \\
&=\psi_{c-4}\psi_{c-5}\dots\psi_az^\la,
\end{align*}

hence
\[
\psi_k\vtt=\psi_k\Psichaind{k-2}{m_1}\chi_1(\Psichaind{k+2}{m_2}\chi_2\dots\chi_{h-2}\Psichaind{k+2+4(h-3)}{m_{h-1}})\chi_{h-1}\Psichaind{c-8}{m_h}\widehat{\chi}_h\psi_{c-4}\psi_{c-5}\dots\psi_az^\la
\]
as required.

   \item[b)] $\bi_{\tp_k}=(1,0),\bi_{\tp_c}=(0,1)$ and $\tp_c$ is not at the end of the arm:
   
   We have that
   \[
   \vtt=\psi_k\chi_1\Psichaind{k+2}{m_1}(\chi_2\Psichaind{k+6}{m_2}\dots\chi_{h-2}\Psichaind{y_{h-3}}{m_{h-2}}\chi_{h-1}\Psichaind{y_{h-1}}{m_{h-1}})\chi_{h}\Psichaind{y_h}{m_h}\vtp,
   \]
   
   where $k+2,k+6,\dots,y_{h-1}$ are all four apart.
   We also have that $y_{h-1}=k+2+4(h-2)=c-4$ and $y_h=c-2$, thus $\chi_{h-1}$ and $\chi_h$ do not commute.
   Then as before, we have that
    \begin{align*}
\psi_k\vtt&=\psi_k\chi_1\Psichaind{k+2}{m_1}(\chi_2\Psichaind{k+6}{m_2}\dots\chi_{h-2}\Psichaind{k+2+4(h-3)}{m_{h-2}}\chi_{h-1}\Psichaind{c-4}{m_{h-1}})\chi_{h}\Psichaind{c-2}{m_h}\vtp\\
&\;\,\vdots\\
&=\psi_k\Psichaind{k-2}{m_1}\chi_1(\Psichaind{k+2}{m_2}\chi_2\dots\chi_{h-3}\Psichaind{k+2+4(h-4)}{m_{h-2}}\chi_{h-2}\Psichaind{k+2+4(h-3)}{m_{h-1}})\chi_{h-1}\chi_h\Psichaind{c-2}{m_h}\vtp.
\end{align*}

   Due to the conditions on $\ttt$ we must have that $\chi_h=\psi_{c-1}(\psi_x\dots\psi_{c+1})\psi_c$ where $x\ge c$, otherwise $\tp_c$ cannot be in the arm.
   Hence
   \begin{align*}
   \chi_{h-1}(\last)\chi_h\Psichaind{c-2}{m_h}&=\psi_{c-2}\psi_{c-1}(\psi_x\dots\psi_{c+1})\psi_c\Psichaind{c-2}{m_h}\vtp\\
   &=(\psi_x\dots\psi_{c+1})\psi_c\psi_{c-1}\psi_{c-2}\Psichaind{c-4}{m_h}\vtp
   \end{align*}
  by \cref{halfdom-leg}, so
   \[
   \psi_k\vtt=\psi_k\Psichaind{k-2}{m_1}\chi_1(\Psichaind{k+2}{m_2}\chi_2\dots\chi_{h-2}\Psichaind{k+2+4(h-3)}{m_{h-1}})\widehat{\chi}_{h-1}(\psi_x\dots\psi_{c+1})\psi_c\psi_{c-1}\psi_{c-2}\Psichaind{c-4}{m_h}\vtp.
   \]
\end{itemize}
 
    \item[c)] $\bi_{\tp_k}=(1,0),\bi_{\tp_c}=(1,0)$:
   
   Here we have that
   \[
   \vtt=\psi_k\chi_1\Psichaind{k+2}{m_1}(\chi_2\Psichaind{k+6}{m_2}\dots\chi_{h-1}\Psichaind{k+2+4(h-2)}{m_{h-1}})\chi_{h}\Psichaind{c-2}{m_h}\vtp.
   \]
   
   Once again, the indices $k+2, k+6, \dots, c-6=k+2+4(h-2),c-2$ are all four apart and we also see that $\chi_h=1$, hence
   
 \begin{align*}
\psi_k\vtt&=\psi_k\chi_1\Psichaind{k+2}{m_1}(\chi_2\Psichaind{k+6}{m_2}\dots\chi_{h-1}\Psichaind{k+2+4(h-1)}{m_{h-1}})\chi_{h}\Psichaind{c-2}{m_h}\vtp\\
&\,\;\vdots\\
&=\psi_k\Psichaind{k-2}{m_1}\chi_1(\Psichaind{k+2}{m_2}\chi_2\dots\chi_{h-2}\Psichaind{k+2+4(h-2)}{m_{h-1}})\chi_{h-1}\Psichaind{c-6}{m_h}\vtp.
\end{align*}

The second part of the proof considers the three cases when $\bi_{\tp_k}=(0,1)$, in which case $\psi_k$ does not commute with $\chi_1$.
     
     \begin{itemize}
 \item[d)] $\bi_{\tp_k}=(0,1),\bi_{\tp_c}=(0,1)$ and $\tp_c$ appears at the end of the arm:
 
Here
 \[
 \vtt=\psi_{k+2}\psi_{k+1}\psi_k\psi_{k-1}\psi_{k-2}\Psichaind{k-4}{m_1}(\chi_2\Psichaind{k+4}{m_2}\dots\chi_{h-1}\Psichaind{k+2+4(h-2)}{m_{h-1}})\chi_{h}\Psichaind{c-4}{m_h}\chi_\ell z^\la,
 \]
where the chain on the left might contain the additional terms $\psi_{k-3}$ and $\psi_{k+3}$ depending on whether $k-2$ and $k+4$ are in the arm.
In the proof, we will omit writing these as they do not play an active part in the rotation.
It is clear that $y_1=k-4$ and $y_2=k+4$ as both $\tp_k$ and $\tp_{k+2}$ are in the leg.
After these two $\Psi$-chains, just as before, $k+8,\dots,k+2+4(i-1),\dots,y_{h-1},y_h$ are all $4$ apart.
Observe that $y_h\neq c-4$ can only happen if $c=k+4$, in which case $h=1$ and $\vtt=\psi^*\chi_{h}\Psichaind{y_h}{m_h}\chi_\ell z^\la$.
We will deal with this special case at the end of the proof, and for now we will assume that $y_h=c-4$.

  \begin{align*}
\psi_k\vtt&=\psi_k\psi_{k+2}\psi_{k+1}\psi_k\psi_{k-1}\psi_{k-2}\Psichaind{k-4}{m_1}(\chi_2\Psichaind{k+4}{m_2}\dots\chi_{h-1}\Psichaind{k+2+4(h-2)}{m_{h-1}})\chi_{h}\Psichaind{c-4}{m_h}\chi_\ell z^\la\\
&=\psi_{k+2}(\psi_k\psi_{k+1}\psi_k)\psi_{k-1}\psi_{k-2}\Psichaind{k-4}{m_1}(\chi_2\Psichaind{k+4}{m_2}\dots\chi_{h-1}\Psichaind{k+2+4(h-2)}{m_{h-1}})\chi_{h}\Psichaind{c-4}{m_h}\chi_\ell z^\la\\
&=\psi_{k+2}(\psi_{k+1}\psi_k\underbrace{\psi_{k+1})\psi_{k-1}\psi_{k-2}\Psichaind{k-4}{m_1}(\chi_2\Psichaind{k+4}{m_2}\dots\chi_{h-1}\Psichaind{k+2+4(h-2)}{m_{h-1}})\chi_{h}\Psichaind{c-4}{m_h}\chi_\ell z^\la}_{=0\text{ by \cref{even}}}\\
&+\psi_{k+2}(y_k)\psi_{k-1}\psi_{k-2}\Psichaind{k-4}{m_1}(\chi_2\Psichaind{k+4}{m_2}\dots\chi_{h-1}\Psichaind{k+2+4(h-2)}{m_{h-1}})\chi_{h}\Psichaind{c-4}{m_h}\chi_\ell z^\la\\
&\begin{rcases}
-2\psi_{k+2}(y_{k+1})\psi_{k-1}\psi_{k-2}\Psichaind{k-4}{m_1}\\
\hspace{3cm}\cdot(\chi_2\Psichaind{k+4}{m_2}\dots\chi_{h-1}\Psichaind{k+2+4(h-2)}{m_{h-1}})\chi_{h}\Psichaind{c-4}{m_h}\chi_\ell z^\la\\
+\psi_{k+2}(y_{k+2})\psi_{k-1}\psi_{k-2}\Psichaind{k-4}{m_1}\\
\hspace{3cm}\cdot(\chi_2\Psichaind{k+4}{m_2}\dots\chi_{h-1}\Psichaind{k+2+4(h-2)}{m_{h-1}})\chi_{h}\Psichaind{c-4}{m_h}\chi_\ell z^\la
\end{rcases}
=0\text{ by \cref{standard-y}}\\
&=\psi_{k+2}\psi_{k-2}\Psichaind{k-4}{m_1}(\chi_2\Psichaind{k+4}{m_2}\dots\chi_{h-1}\Psichaind{k+2+4(h-2)}{m_{h-1}})\chi_{h}\Psichaind{c-4}{m_h}\chi_\ell z^\la\\
&=\psi_{k-2}\Psichaind{k-4}{m_1}(\chi_2\psi_{k+2}\Psi_{k+4}\Psi_{k+2}\Psichaind{k}{m_2}\dots\chi_{h-1}\Psichaind{k+2+4(h-2)}{m_{h-1}})\chi_{h}\Psichaind{c-4}{m_h}\chi_\ell z^\la\\
&=\psi_{k-2}\Psichaind{k-4}{m_1}(\chi_2\psi_{k+2}\Psichaind{k}{m_2}\dots\chi_{h-1}\Psichaind{k+2+4(h-2)}{m_{h-1}})\chi_{h}\Psichaind{c-4}{m_h}\chi_\ell z^\la\\
&=\psi_{k-2}\Psichaind{k-4}{m_1}(\psi_{k+2}\Psichaind{k}{m_2}\chi_3\chi_2\Psichaind{y_3}{m_3}\dots\chi_{h-1}\Psichaind{k+2+4(h-2)}{m_{h-1}})\chi_{h}\Psichaind{c-4}{m_h}\chi_\ell z^\la\\
&=\psi_{k-2}\Psichaind{k-4}{m_1}(\psi_{k+2}\Psichaind{k}{m_2}\chi_3(\chi_2\Psi_{y_3}\Psi_{y_3-2})\Psichaind{y_2}{m_3}\dots\chi_{h-1}\Psichaind{k+2+4(h-2)}{m_{h-1}})\chi_{h}\Psichaind{c-4}{m_h}\chi_\ell z^\la\\
&=\psi_{k-2}\Psichaind{k-4}{m_1}(\psi_{k+2}\Psichaind{k}{m_2}\chi_3\chi_2\Psichaind{y_2}{m_3}\dots\chi_{h-1}\Psichaind{k+2+4(h-2)}{m_{h-1}})\chi_{h}\Psichaind{c-4}{m_h}\chi_\ell z^\la\\
&\,\;\vdots\\
&=\psi_{k-2}\Psichaind{k-4}{m_1}(\psi_{k+2}\Psichaind{k}{m_2}\chi_2\dots\chi_{h-2}\Psichaind{k+2+4(h-3)}{m_{h-1}})\chi_{h-1}\Psichaind{c-8}{m_h}\chi_h\chi_\ell z^\la\\
&=\psi_{k-2}\Psichaind{k-4}{m_1}(\psi_{k+2}\Psichaind{k}{m_2}\chi_2\dots\chi_{h-2}\Psichaind{k+2+4(h-3)}{m_{h-1}})\chi_{h-1}\Psichaind{c-8}{m_h}\widehat{\chi}_h\psi_{c-4}\dots\psi_az^\la
\end{align*}

where the product $\chi_h\chi_\ell $ gives the same result as in case a).

Finally, suppose that $y_h< c-4$ and $c=k+4$:
\begin{align*}
\psi_k\vtt&=\psi_k\chi_{h}\Psichaind{y_h}{m_h}\chi_\ell z^\la\\
&=\psi_k\psi_{k+2}\psi_{k+1}\dots\psi_{y_h+2}\Psichaind{y_h}{m_h}\chi_\ell z^\la\\
&=\psi_{k+2}(\psi_k\psi_{k+1}\psi_k)\psi_{k-1}\dots\psi_{y_h+2}\Psichaind{y_h}{m_h}\chi_\ell z^\la\\
&=\psi_{k+2}(\psi_{k+1}\psi_k\underbrace{\psi_{k+1})\psi_{k-1}\dots\psi_{y_h+2}\Psichaind{y_h}{m_h}\chi_\ell z^\la}_{=0\text{ by \cref{even}}}\\
&+\psi_{k+2}(y_k)\psi_{k-1}\dots\psi_{y_h+2}\Psichaind{y_h}{m_h}\chi_\ell z^\la\\
&\begin{rcases}
-2\psi_{k+2}(y_{k+1})\psi_{k-1}\dots\psi_{y_h+2}\Psichaind{y_h}{m_h}\chi_\ell z^\la\\
+\psi_{k+2}(y_{k+2})\psi_{k-1}\dots\psi_{y_h+2}\Psichaind{y_h}{m_h}\chi_\ell z^\la
\end{rcases}
=0\text{ by \cref{standard-y}}\\
&=\psi_{k+2}\psi_{k-2}\dots\psi_{y_h+2}\Psichaind{y_h}{m_h}\chi_\ell z^\la\\
&=\psi_{k-2}\dots\psi_{y_h+2}\Psichaind{y_h}{m_h}(\psi_{k+2}\psi_{k+3}\psi_{k+2})\psi_{k+1}\dots\psi_a z^\la\\
&\,\;\vdots\qquad\text{ see the computation for $\chi_h(\last)\chi_\ell$ above}\\
&=\psi_{k-2}\dots\psi_{y_h+2}\Psichaind{y_h}{m_h}\psi_{k}\dots\psi_a z^\la\\
&=\psi_k\psi_{k-2}\dots\psi_{y_h+2}\Psichaind{y_h}{m_h}\psi_{k-1}\dots\psi_a z^\la.
\end{align*}
 
 \item[e)] $\bi_{\tp_k}=(0,1)$, $\bi_{\tp_c}=(0,1)$ and $\tp_c$ is not at the end of the arm:
 
   Similar to case b).
   Here, the chain ends with $\chi_{h-1}\Psichaind{k+2+4(h-2)}{m_{h-1}}\chi_{h}\Psichaind{c-2}{m_h}\vtp$ where $k+2+4(h-2)\le c-4$.
   Further, because of the shifted residue of $\tp_k$, we have that $y_1\le k-4$.
   If $k+2+4(h-2)= c-4$, $\chi_{h-1}$ and $\chi_h$ do not commute:
    \begin{align*}
\psi_k\vtt&=\psi_k\chi_1\Psichaind{y_1}{m_1}(\chi_2\Psichaind{k+4}{m_2}\dots\chi_{h-1}\Psichaind{c-4}{m_{h-1}})\chi_{h}\Psichaind{c-2}{m_h}\vtp\\
&=\psi_k\psi_{k+2}\psi_{k+1}\psi_k\psi_{k-1}\psi_{k-2}\dots\psi_{y_1+2}\Psichaind{y_1}{m_1}(\chi_2\Psichaind{k+4}{m_2}\dots\chi_{h-1}\Psichaind{c-4}{m_{h-1}})\chi_{h}\Psichaind{c-2}{m_h}\vtp\\
&\;\,\vdots\text{ see the computation in case $d)$}\\
&=\psi_{k-2}\dots\psi_{y_1+2}\Psichaind{y_1}{m_1}(\psi_{k+2}\Psichaind{k}{m_2}\chi_2\dots\chi_{h-2}\Psichaind{k+2+4(h-3)}{m_{h-1}})\chi_{h-1}\chi_h\Psichaind{c-2}{m_h}\vtp\\
&\;\,\vdots\\
&=\psi_{k-2}\dots\psi_{y_1+2}\Psichaind{y_1}{m_1}(\psi_{k+2}\Psichaind{k}{m_2}\chi_2\dots\chi_{h-2}\Psichaind{k+2+4(h-3)}{m_{h-1}})\\
&\cdot\widehat{\chi}_{h-1}(\psi_x\dots\psi_{c+1})\psi_c\psi_{c-1}\psi_{c-2}\Psichaind{c-4}{m_h}\vtp.
\end{align*}

If $k+2+4(h-2)< c-4$, we have that $\vtt=\psi_{k+2}\psi_{k+1}\dots\psi_{x+1}\Psichaind{x}{m_1}\chi_{h}\Psichaind{c-2}{m_h}\vtp$ where $x<k-4$, but the proof is essentially the same.
 
 \item[f)] $\bi_{\tp_k}=(0,1), \bi_{\tp_c}=(1,0)$:
 
 As in case c), $\chi_h=\id$ and $\psi_k\chi_1$ behaves the same way as in the previous case.
  \begin{align*}
\psi_k\vtt&=\psi_k\chi_1\Psichaind{y_1}{m_1}(\chi_2\Psichaind{k+4}{m_2}\dots\chi_{h-1}\Psichaind{k+2+4(h-2)}{m_{h-1}})\chi_{h}\Psichaind{c-2}{m_h}\vtp\\
&=\psi_k\psi_{k+2}\psi_{k+1}\psi_k\dots\psi_{y_1+2}\Psichaind{y_1}{m_1}(\chi_2\Psichaind{k+4}{m_2}\dots\chi_{h-1}\Psichaind{k+2+4(h-2)}{m_{h-1}})\chi_{h}\Psichaind{c-2}{m_h}\vtp\\
&\,\;\vdots\\
&=\psi_{k-2}\dots\psi_{y_1+2}\Psichaind{y_1}{m_1}(\psi_{k+2}\Psichaind{k}{m_2}\chi_2\dots\chi_{h-2}\Psichaind{k+2+4(h-3)}{m_{h-1}})\chi_{h-1}\Psichaind{c-6}{m_h}\vtp
\end{align*}
\end{itemize}

\underline{Case 2}

We have the same six subcases as in Case $1$ and with some slight modifications these work the same way as before, however $\psi_k$ will start acting on the end of the chain (that is the highest indexed $\Psi$-term) and we will be using \cref{handy} part $3$ and \cref{halfdom-arm} (instead of \cref{handy} part $2$ and \cref{halfdom-leg}).
Details can be included by switching the toggle in the {\tt arXiv} version of this paper.
\begin{answer}
Furthermore, the indices in the brackets in the middle all have the upper indices four apart to ensure that there are other paired entries between $\tp_c$ and $\tp_k$ in $\ttt$.

\begin{itemize}

    \item[a)] $\bi_{\tp_c}=(1,0)$, $\bi_{\tp_k}=(1,0)$:
    
As $\bi_{\tp_k}=(1,0)=\bi_{\tp_{k+2}}$, we have that $\chi_{h-1}=\id=\chi_h$. Hence:
    
\begin{align*}
\psi_k\vtt&=\psi_k\chi_1\Psichaind{c}{m_1}(\chi_2\Psichaind{c+4}{m_2}\dots\chi_{h-2}\Psichaind{k-6}{m_{h-2}}\chi_{h-1}\Psichaind{k-2}{m_{h-1}})\chi_h\Psichaind{k}{m_h}\vtp\\
&=\chi_1\Psichaind{c}{m_1}(\chi_2\Psichaind{c+4}{m_2}\dots\chi_{h-2}\Psichaind{k-6}{m_{h-2}}\chi_{h-1}\psi_k\Psi_{k-2}\Psi_k\Psichaind{k-4}{m_{h-1}})\chi_h\Psichaind{k-2}{m_h}\vtp\\
&=\chi_1\Psichaind{c}{m_1}(\chi_2\Psichaind{c+4}{m_2}\dots\chi_{h-2}\Psichaind{k-6}{m_{h-2}}\chi_{h-1}\psi_k\Psichaind{k-4}{m_{h-1}})\chi_h\Psichaind{k-2}{m_h}\vtp\\
&=\chi_1\Psichaind{c}{m_1}(\chi_2\Psichaind{c+4}{m_2}\dots\chi_{h-2}\Psichaind{k-6}{m_{h-2}}\Psichaind{k-4}{m_{h-1}})\psi_k\underbrace{\chi_{h-1}\chi_h}_{=1}\Psichaind{k-2}{m_h}\vtp\\
&=\chi_1\Psichaind{c}{m_1}(\chi_2\Psichaind{c+4}{m_2}\dots\chi_{h-2}\Psi_{k-6}\Psi_{k-4}\Psichaind{k-8}{m_{h-2}}\Psichaind{k-6}{m_{h-1}})\psi_k\Psichaind{k-2}{m_h}\vtp\\
&=\Psichaind{c-4}{m_0}\chi_1\Psichaind{c}{m_1}(\chi_2\Psichaind{c+4}{m_2}\dots\Psichaind{k-8}{m_{h-2}}\chi_{h-2}\Psichaind{k-6}{m_{h-1}})\psi_k\Psichaind{k-2}{m_h}\vtp\\
&=\Psichaind{c-4}{m_0}\chi_1\Psichaind{c}{m_1}(\chi_2\Psichaind{c+4}{m_2}\dots\chi_{h-3}\Psichaind{k-10}{m_{h-3}}\Psichaind{k-8}{m_{h-2}}\chi_{h-2}\Psichaind{k-6}{m_{h-1}})\psi_k\Psichaind{k-2}{m_h}\vtp\\
&=\Psichaind{c-4}{m_0}\chi_1\Psichaind{c}{m_1}(\chi_2\Psichaind{c+4}{m_2}\dots\chi_{h-3}\Psi_{k-10}\Psi_{k-8}\Psichaind{k-12}{m_{h-3}}\Psichaind{k-10}{m_{h-2}}\chi_{h-2}\Psichaind{k-6}{m_{h-1}})\psi_k\Psichaind{k-2}{m_h}\vtp\\
&=\Psichaind{c-4}{m_0}\chi_1\Psichaind{c}{m_1}(\chi_2\Psichaind{c+4}{m_2}\dots\chi_{h-3}\Psichaind{k-12}{m_{h-3}}\Psichaind{k-10}{m_{h-2}}\chi_{h-2}\Psichaind{k-6}{m_{h-1}})\psi_k\Psichaind{k-2}{m_h}\vtp\\
&=\Psichaind{c-4}{m_0}\chi_1\Psichaind{c}{m_1}(\chi_2\Psichaind{c+4}{m_2}\dots\Psichaind{k-12}{m_{h-3}}\chi_{h-3}\Psichaind{k-10}{m_{h-2}}\chi_{h-2}\Psichaind{k-6}{m_{h-1}})\psi_k\Psichaind{k-2}{m_h}\vtp\\
&\,\;\vdots\\
&=\psi_k\Psichaind{c-2}{m_1}\chi_1(\Psichaind{c}{m_2}\chi_2\dots\chi_{h-2}\Psichaind{k-6}{m_{h-1}})\Psichaind{k-2}{m_h}\vtp.
\end{align*}

   \item[b)] $\bi_{\tp_c}=(1,0),\bi_{\tp_k}=(0,1)$ and $\tp_{k+2}$ is at the end of the arm;
   
   Here we must have that $\chi_d=\psi_{k+1}\psi_k\dots\psi_a$ and that $\chi_h=\psi_{k-1}\psi_k$, otherwise we cannot have $\tp_k,\tp_{k+2}\in\Arm(\ttt)$. Then
   
\begin{align*}
\psi_k\vtt&=\psi_k\chi_1\Psichaind{c}{m_1}(\chi_2\Psichaind{c+4}{m_2}\dots\chi_{h-2}\Psichaind{k-8}{m_{h-2}}\chi_{h-1}\Psichaind{k-4}{m_{h-1}})\chi_h\Psichaind{k-2}{m_h}\chi_dz^\la\\
&=\chi_1\Psichaind{c}{m_1}(\chi_2\Psichaind{c+4}{m_2}\dots\chi_{h-2}\Psichaind{k-8}{m_{h-2}}\chi_{h-1}\Psichaind{k-4}{m_{h-1}})(\psi_k\psi_{k-1}\psi_k)\Psichaind{k-2}{m_h}\chi_dz^\la\\
&=\chi_1\Psichaind{c}{m_1}(\chi_2\Psichaind{c+4}{m_2}\dots\chi_{h-2}\Psichaind{k-8}{m_{h-2}}\chi_{h-1}\Psichaind{k-4}{m_{h-1}})(\psi_{k-1}\psi_k\underbrace{\psi_{k-1})\Psichaind{k-2}{m_h}\chi_dz^\la}_{=0\text{ by \cref{even}}}\\
&\begin{rcases}
-\chi_1\Psichaind{c}{m_1}(\chi_2\Psichaind{c+4}{m_2}\dots\chi_{h-2}\Psichaind{k-8}{m_{h-2}}\chi_{h-1}\Psichaind{k-4}{m_{h-1}})(y_{k-1})\Psichaind{k-2}{m_h}\chi_dz^\la\\
+2\chi_1\Psichaind{c}{m_1}(\chi_2\Psichaind{c+4}{m_2}\dots\chi_{h-2}\Psichaind{k-8}{m_{h-2}}\chi_{h-1}\Psichaind{k-4}{m_{h-1}})(y_k)\Psichaind{k-2}{m_h}\chi_dz^\la
\end{rcases}
=0\text{ by \cref{standard-y}}\\
&-\chi_1\Psichaind{c}{m_1}(\chi_2\Psichaind{c+4}{m_2}\dots\chi_{h-2}\Psichaind{k-8}{m_{h-2}}\chi_{h-1}\Psichaind{k-4}{m_{h-1}})(y_{k+1})\Psichaind{k-2}{m_h}\chi_dz^\la\\
&=\chi_1\Psichaind{c}{m_1}(\chi_2\Psichaind{c+4}{m_2}\dots\chi_{h-2}\Psichaind{k-8}{m_{h-2}}\chi_{h-1}\Psichaind{k-4}{m_{h-1}})\Psichaind{k-2}{m_h}\psi_k\dots\psi_az^\la\\
&=\chi_1\Psichaind{c}{m_1}(\chi_2\Psichaind{c+4}{m_2}\dots\chi_{h-2}\Psichaind{k-8}{m_{h-2}}\chi_{h-1}\Psi_{k-4}\Psi_{k-2}\Psichaind{k-6}{m_{h-1}})\Psichaind{k-4}{m_h}\psi_k\dots\psi_az^\la\\
&=\chi_1\Psichaind{c}{m_1}(\chi_2\Psichaind{c+4}{m_2}\dots\chi_{h-2}\Psichaind{k-8}{m_{h-2}}\Psichaind{k-6}{m_{h-1}})\chi_{h-1}\Psichaind{k-4}{m_h}\psi_k\dots\psi_az^\la\\
&\,\;\vdots\\
&=\Psichaind{c-2}{m_1}\chi_1(\Psichaind{c}{m_2}\chi_2\dots\chi_{h-2}\Psichaind{k-8}{m_{h-1}})\chi_{h-1}\Psichaind{k-4}{m_h}\psi_k\psi_{k-1}\psi_{k-2}\dots\psi_az^\la.
\end{align*}

   \item[c)] $\bi_{\tp_c}=(1,0),\bi_{\tp_k}=(0,1)$ and $\tp_{k+2}$ is not at the end of the arm:
   
   Here we have that $\chi_h=\psi_{k+1}\psi_{k+2}$ and that $\chi_{h-1}=\psi_{k-1}\psi_k$, otherwise we cannot have $\tp_k,\tp_{k+2}\in\Arm(\ttt)$. Then
   
  \begin{align*}
\psi_k\vtt&=\psi_k\chi_1\Psichaind{c}{m_1}(\chi_2\Psichaind{c+4}{m_2}\dots\chi_{h-2}\Psichaind{k-4}{m_{h-2}})\chi_{h-1}\Psichaind{k-2}{m_{h-1}}\chi_h\Psichaind{k}{m_h}\vtp\\
&=\chi_1\Psichaind{c}{m_1}(\chi_2\Psichaind{c+4}{m_2}\dots\chi_{h-2}\Psichaind{k-4}{m_{h-2}})\psi_k\chi_{h-1}\Psichaind{k-2}{m_{h-1}}\chi_h\Psichaind{k}{m_h}\vtp,
\end{align*}

where
\begin{align*}
    \psi_k\chi_{h-1}\Psichaind{k-2}{m_{h-1}}\chi_h\Psichaind{k}{m_h}\vtp&=   (\psi_k\psi_{k-1}\psi_k)\Psichaind{k-2}{m_{h-1}}\psi_{k+1}\psi_{k+2}\Psichaind{k}{m_h}\vtp\\
    &=(\psi_{k-1}\psi_k\underbrace{\psi_{k-1})\Psichaind{k-2}{m_{h-1}}\psi_{k+1}\psi_{k+2}\Psichaind{k}{m_h}\vtp}_{=0\text{ by \cref{even}}}\\
    &\begin{rcases}
 -(y_{k-1})\Psichaind{k-2}{m_{h-1}}\psi_{k+1}\psi_{k+2}\Psichaind{k}{m_h}\vtp\\
    +2(y_k)\Psichaind{k-2}{m_{h-1}}\psi_{k+1}\psi_{k+2}\Psichaind{k}{m_h}\vtp
    \end{rcases}
    =0\text{ by \cref{standard-y}}\\
    &-(y_{k+1})\Psichaind{k-2}{m_{h-1}}\psi_{k+1}\psi_{k+2}\Psichaind{k}{m_h}\vtp\\
   &=\Psichaind{k-2}{m_{h-1}}\psi_{k+2}\Psichaind{k}{m_h}\vtp.
\end{align*}

Thus

\begin{align*}
\psi_k\vtt&=\chi_1\Psichaind{c}{m_1}(\chi_2\Psichaind{c+4}{m_2}\dots\chi_{h-2}\Psichaind{k-4}{m_{h-2}})\Psichaind{k-2}{m_{h-1}}(\psi_x\dots)\psi_{k+2}\Psichaind{k}{m_h}\vtp\\
&=\chi_1\Psichaind{c}{m_1}(\chi_2\Psichaind{c+4}{m_2}\dots\chi_{h-2}\Psi_{k-4}\Psi_{k-2}\Psichaind{k-6}{m_{h-2}})\Psichaind{k-4}{m_{h-1}}(\psi_x\dots)\psi_{k+2}\Psichaind{k}{m_h}\vtp\\
&=\chi_1\Psichaind{c}{m_1}(\chi_2\Psichaind{c+4}{m_2}\dots\Psichaind{k-6}{m_{h-2}})\chi_{h-2}\Psichaind{k-4}{m_{h-1}}(\psi_x\dots)\psi_{k+2}\Psichaind{k}{m_h}\vtp\\
&\vdots\\
&=\Psichaind{c-2}{m_1}\chi_1(\Psichaind{c}{m_2}\chi_2\Psichaind{c+4}{m_3}\dots\chi_{h-2}\Psichaind{k-4}{m_{h-1}})(\psi_x\dots)\psi_{k+2}\Psichaind{k}{m_h}\vtp.
\end{align*}
\end{itemize}

The second part of the proof considers the three cases when $\bi_{\tp_c}=(0,1)$, in which case we need to consider an additional $\chi$ and an extra $\Psi$-chain on the left. This $\chi$-term, denoted $\chi_1$, shifts the residues of $\tp_c$, hence $\chi_1=\psi_c\psi_{c-1}\psi_{c-2}$.

\begin{itemize}
     
 \item[d)] $\bi_{\tp_c}=(0,1),\bi_{\tp_k}=(1,0)$:
  
  Similarly to a) we have that $\chi_{h-1}=\id=\chi_h$.
  We also have that $y_1\le c-4$.
 
   \begin{align*}
\psi_k\vtt&=\psi_k\chi_1\Psichaind{y_1}{m_1}(\chi_2\Psichaind{c+2}{m_2}\chi_3\dots\chi_{h-2}\Psichaind{k-6}{m_{h-2}})\chi_{h-1}\Psichaind{k-2}{m_{h-1}}\chi_h\Psichaind{k}{m_h}\vtp\\
&\,\;\vdots\\
&=\chi_1\Psichaind{y_1}{m_1}(\Psichaind{c}{m_2}\chi_2\dots\chi_{h-3}\Psichaind{k-8}{m_{h-2}})\chi_{h-2}\Psichaind{k-4}{m_{h-1}}\psi_k\Psichaind{k-2}{m_h}\vtp\\
&=\psi_c\psi_{c-1}\psi_{c-2}\dots\psi_{y_1+2}\Psichaind{y_1}{m_1}(\Psichaind{c}{m_2}\chi_2\dots\chi_{h-3}\Psichaind{k-8}{m_{h-2}})\chi_{h-2}\Psichaind{k-4}{m_{h-1}}\psi_k\Psichaind{k-2}{m_h}\vtp
\end{align*}

But $\psi_c\psi_{c-1}\psi_{c-2}\dots\psi_{y_1+2}\Psichaind{y_1}{m_1}\Psichaind{c}{m_2}$ satisfy the conditions stated in \cref{halfdom-arm} Case $1$.
So

\[
\psi_k\vtt=\psi_{c-2}\Psichaind{c-4}{m_1}\psi_{c-1}\psi_c\Psichaind{c-2}{m_2}(\chi_2\dots\chi_{h-2}\Psichaind{k-6}{m_{h-1}})\psi_k\Psichaind{k-2}{m_h}\vtp.
\]
 
  \item[e)] $\bi_{\tp_c}=(0,1),\bi_{\tp_k}=(0,1)$ and $\tp_{k+2}$ is at the end of the arm:
 
 Works almost the same way as case b) and we have that $y_1 \le c-4$ as in case d).

 \begin{align*}
\psi_k\vtt&=\psi_k\chi_1\Psichaind{y_1}{m_1}(\chi_2\Psichaind{c+2}{m_2}\chi_3\dots\chi_{h-2}\Psichaind{k-8}{m_{h-2}})\chi_{h-1}\Psichaind{k-4}{m_{h-1}}\chi_h\Psichaind{k-2}{m_h}\chi_dz^\la\\
&\,\;\vdots\\
&=\chi_1\Psichaind{y_1}{m_1}(\Psichaind{c}{m_2}\chi_2\dots\chi_{h-2}\Psichaind{k-8}{m_{h-1}})\chi_{h-1}\Psichaind{k-4}{m_h}\psi_{k-1}\psi_{k-2}\dots\psi_az^\la\\
&=\psi_c\psi_{c-1}\psi_{c-2}\dots\psi_{y_1+2}\Psichaind{y_1}{m_1}(\Psichaind{c}{m_2}\chi_2\dots\chi_{h-2}\Psichaind{k-8}{m_{h-1}})\chi_{h-1}\Psichaind{k-4}{m_h}\psi_{k-1}\psi_{k-2}\dots\psi_az^\la\\
&=\psi_k\psi_{c-2}\dots\psi_{y_1+2}\Psichaind{y_1}{m_1}\psi_{c-1}\psi_c\Psichaind{c-2}{m_2}(\chi_2\dots\chi_{h-2}\Psichaind{k-8}{m_{h-1}})\chi_{h-1}\Psichaind{k-4}{m_h}\psi_{k-1}\psi_{k-2}\dots\psi_az^\la
\end{align*}

 \item[f)] $\bi_{\tp_c}=(0,1),\bi_{\tp_k}=(0,1)$ and $\tp_{k+2}$ is not at the end of the arm:
 
 Once again, $y_1\le c-4$.
 Very similar to case c), the same conditions hold on $\chi_{h-1}$ and $\chi_h$.
   \begin{align*}
\psi_k\vtt&=\psi_k\chi_1\Psichaind{y_1}{m_1}(\chi_2\Psichaind{c+2}{m_2}\chi_3\dots\chi_{h-2}\Psichaind{k-4}{m_{h-2}})\chi_{h-1}\Psichaind{k-2}{m_{h-1}}\chi_h\Psichaind{k}{m_h}\vtp\\
&\,\;\vdots\\
&=\chi_1\Psichaind{y_1}{m_1}(\Psichaind{c}{m_2}\chi_2\dots\chi_{h-2}\Psichaind{k-4}{m_{h-1}})\psi_{k+2}\Psichaind{k}{m_h}\vtp\\
&=\psi_c\psi_{c-1}\psi_{c-2}\dots\psi_{y_1+2}\Psichaind{y_1}{m_1}(\Psichaind{c}{m_2}\chi_2\dots\chi_{h-3}\Psichaind{k-8}{m_{h-2}})\chi_{h-2}\Psichaind{k-4}{m_{h-1}}\psi_{k+2}\Psichaind{k}{m_h}\vtp\\
&=\psi_{c-2}\dots\psi_{y_1+2}\Psichaind{y_1}{m_1}\psi_{c-1}\psi_c\Psichaind{c-2}{m_2}(\chi_2\dots\chi_{h-3}\Psichaind{k-8}{m_{h-2}})\chi_{h-2}\Psichaind{k-4}{m_{h-1}}\psi_{k+2}\Psichaind{k}{m_h}\vtp,
\end{align*}
by \cref{halfdom-arm}.
\end{itemize}
\end{answer}
\end{proof}

\subsection{At least three consecutive pairs in $\Leg(\ttt)$ or $\Arm(\ttt)$}

Now that we have shown what happens when we act with $\psi_k$ when $\tp_k,\tp_{k+2}$ are in $\Leg(\ttt)$ and $\tp_c$ is in $\Arm(\ttt)$ (or vice versa, $\tp_k,\tp_{k+2}\in\Arm(\ttt)$ and $\tp_c\in\Leg(\ttt)$), we will extend with these results for at least three pairs in the leg (or in the arm) of $\ttt$.

\begin{lem}[cf.\cite{thesis}, Lemma~3.19]\label{4killing-1}
Let $\ttt\in\std(\la)$ with $\tp_k,\tp_{k+2}\in\Leg(\ttt)$.

\begin{enumerate}
    \item If $a$ is odd, and
 \begin{itemize}
 \item[a)] $\bi_{\tp_k}=(1,0)$, let $\vtt$ have normal form
\[
\vtt=\psi^*\chi_1\Psichaind{y_1}{m_1}\chi_2\dots\chi_{h}\Psichaind{y_h}{m_h}z^\la,
\]
where $m_h=a$. Suppose $3\le k \le n-1$ is odd and that $y_1$ is minimal such that $m_1 \le k \le y_1$.
    Then $\psi_k\vtt=0$ if and only if for all $i=1,2,\dots,h$ either $y_{i}\ge 4i+k$ or $y_i = 4i+k-2$ and $\chi_i\neq \id$;
\item[b)] $\bi_{\tp_k}=(0,1)$, let $\vtt$ have normal form
\[
\vtt=\psi^*\chi_0\Psichaind{y_0}{m_0}\chi_1\Psichaind{y_1}{m_1}\chi_2\dots\chi_{h}\Psichaind{y_h}{m_h}z^\la,
\]
where $m_h=a$. Suppose $3\le k \le n-1$ is odd and that $y_1$ is minimal such that $m_1 \le k+2 \le y_1$.
    Then $\psi_k\vtt=0$ if and only if for all $i=1,2,\dots,h$ either $y_{i}\ge 4i+k+2$ or $y_i = 4i+k$ and $\chi_i\neq \id$. 
     \end{itemize}
     
    \item If $a$ is even, and
    \begin{itemize}
 \item[c)] $\bi_{\tp_k}=(1,0)$, let $\vtt$ have normal form
\[
\vtt=\psi^*\chi_1\Psichaind{y_1}{m_1}\chi_2\dots\chi_{h}\Psichaind{y_h}{m_h}\chi_dz^\la,
\]
where $m_h=a-1$. Suppose $3\le k \le n-1$ is odd and that $y_1$ is minimal such that $m_1 \le k \le y_1$.
Then $\psi_k\vtt=0$ if and only if for all $i=1,2,\dots,h-1$ either $y_{i}\ge 4i+k$ or $y_i= 4i+k-2$ and $\chi_i\neq\id$ and $y_h\ge 4h+k$ or $y_h=4h+k-2,\chi_h\neq\id$ and $\ell>4h+k+2$;
\item[d)] $\bi_{\tp_k}=(0,1)$, let $\vtt$ have normal form
\[
\vtt=\psi^*\chi_0\Psichaind{y_0}{m_0}\chi_1\Psichaind{y_1}{m_1}\chi_2\dots\chi_{h}\Psichaind{y_h}{m_h}\chi_dz^\la,
\]
where $m_h=a-1$. Suppose $3\le k \le n-1$ is odd and that $y_1$ is minimal such that $m_1 \le k+2 \le y_1$.
Then $\psi_k\vtt=0$ if and only if for all $i=1,2,\dots,h-1$ either $y_{i}\ge 4i+k+2$ or $y_i= 4i+k$ and $\chi_i\neq\id$ and $y_h\ge 4h+k+2$ or $y_h=4h+k,\chi_h\neq\id$ and $\ell>4h+k+4$.
 \end{itemize}
\end{enumerate}
\end{lem}

\begin{proof}
Just like in the proof of \cref{3pairs}, we omit writing $\psi^*$ on the left.
We will also write $\chi_i=\widehat{\chi}_i\chi_i(\last)$ (recall that $\chi_i(\last)=\psi_{y_i+2}$).
We prove the statement by induction on the length of $\vtt$.
The proof itself is very similar to the proof of \cref{3pairs}.

\underline{``If'' part:}
\begin{enumerate}
\item \textbf{Odd $a$}

As $\tp_k$ and $\tp_{k+2}$ are both in the leg, we immediately see that $y_1\ge k+2$.
In addition, we only care about $\chi_1$ if $y_1=k+2$ because in all other cases the chain is long enough to propagate the leftmost $\psi_{k+4}$ in $\Psi_{k+4}$ from $\Psichaind{y_1}{m_1}$, instead of $\psi_{k+4}$ from $\chi_1$.
Now we can break down the proof into several subcases:
\begin{itemize}
    \item[a)] Assume that $\bi_{\tp_k}=(1,0)$ and $y_1\ge k+4$.

If $h=1$, we have that

\begin{align*}
   \psi_k \vtt&=\psi_k\chi_1\Psichaind{y_1}{k+6}\Psi_{k+4}\Psi_{k+2}\Psi_k\Psichaind{k-2}{m_1}z^\la\\
   &=\psi_k\chi_1\Psichaind{y_1}{k+6}\Psi_{k+4}\Psichaind{k-2}{m_1}z^\la\\
&=\psi_k\chi_1\Psichaind{y_1}{k+6}\Psichaind{k-2}{m_1}\Psi_{k+4}z^\la\\
&=0.
\end{align*}

Now suppose that $\vtt=\chi_1\Psichaind{y_1}{m_1}\chi_2\dots\chi_{h}\Psichaind{y_h}{m_h}z^\la$ where $h>1$ as stated in the lemma.
Then from the base case we have that 
\[
\psi_k\vtt=\psi_k\chi_k\Psichaind{y_1}{k+6}\Psichaind{k-2}{m_1}\Psi_{k+4}\underbrace{\chi_{2}\Psichaind{y_{2}}{m_{2}}\chi_{3}\dots\chi_h\Psichaind{y_h}{m_h}z^\la}_{=\vtp}.
\]
As $m_1\le k$, it implies that $m_{2}\le k+2$ and by assumption, we also have that $k+6\le y_2$.
\begin{itemize}
    \item[$\circ$] If $y_2\ge k+8$ we have that
\[
\Psichaind{y_{2}}{m_{2}}=\Psichaind{y_{2}}{k+10}\Psichaind{k+8}{k+2}\Psichaind{k}{m_2},
\]
 where the first and last chains might be empty, but our conditions guarantee that $\Psichaind{k+8}{k+2}$ appears.
This forces $\tp_{k+4},\tp_{k+6}$ to be in the leg of $\ttt'$.
Thus
\[
\psi_k\vtt=\psi_k\Psichaind{y_1}{k+6}\Psichaind{k-2}{m_1}\Psi_{k+4}\chi_{2}\Psichaind{y_{2}}{k+10}\Psichaind{k+8}{k+2}\Psichaind{k}{m_2}\chi_{3}\dots\chi_h\Psichaind{y_h}{m_h}z^\la=0
\]
by induction, since for each $i=1,2,\dots,h-1$, we have that either $y_{i+1}\ge 4i+(k+4)$ or $y_{i+1}\ge 4i+(k+2)$ (these follow immediately from the conditions on $\vtt$).
    \item[$\circ$] If $y_2= k+6$ we have that
    \[
\Psichaind{y_{2}}{m_{2}}=\Psichaind{k+6}{k+2}\Psichaind{k}{m_2}
\]
and we also see that $\chi_2\neq\id$.
Once again we see that $\tp_{k+4}$ and $\tp_{k+6}$ are in the leg of $\ttt'$, hence
\[
    \psi_k\vtt=\psi_k\Psichaind{y_1}{k+6}\Psichaind{k-2}{m_1}\Psi_{k+4}\chi_{2}\Psichaind{k+6}{k+2}\Psichaind{k}{m_2}\chi_{3}\dots\chi_h\Psichaind{y_h}{m_h}z^\la=0
\]
by induction for the same reasoning as in the previous case.
\end{itemize}

Similarly, we also need to consider the case when $y_1=k+2$.
By assumption, we must have that $\chi_1\neq \id$.

If $h=1$, we have that

\begin{align*}
   \psi_k \vtt&=\psi_k\chi_1\Psichaind{k+2}{m_1}z^\la\\
   &=\chi_1\psi_k\Psi_{k+2}\Psi_k\Psichaind{k-2}{m_1}z^\la\\
   &=\chi_1\psi_k\Psichaind{k-2}{m_1}z^\la\\
&=\psi_k\widehat{\chi}_1\Psichaind{k-2}{m_1}\psi_{k+4}z^\la=0,
\end{align*}
where $\chi_1=\widehat{\chi}_1\psi_{k+4}$.

Now suppose that $h>1$ and $\vtt=\chi_1\Psichaind{y_1}{m_1}\chi_2\dots\chi_{h}\Psichaind{y_h}{m_h}z^\la$ is as in the statement.
From the base case we have that 
\[
\psi_k\vtt=\psi_k\widehat{\chi}_1\Psichaind{k-2}{m_1}\psi_{k+4}\underbrace{\chi_{2}\Psichaind{y_{2}}{m_{2}}\chi_{3}\dots\chi_h\Psichaind{y_h}{m_h}z^\la}_{=\vtp}.
\]

By assumption $y_2\ge k+6$, and this forces $\tp_{k+4}$ and $\tp_{k+6}$ into $\Leg(\ttt')$.
We also know that $\psi_{k+4}$ and $\chi_2$ commute.
Then we have two options:
\begin{itemize}
    \item[$\circ$] if $y_2\ge k+8$, we have that
    \[
\Psichaind{y_{2}}{m_{2}}=\Psichaind{y_{2}}{k+10}\Psichaind{k+8}{k+2}\Psichaind{k}{m_2}
\]
and then
    \[
\psi_k\vtt=\psi_k\widehat{\chi}_1\Psichaind{k-2}{m_1}\psi_{k+4}\chi_{2}\Psichaind{y_{2}}{m_{2}}\vtp=0
    \]
    by induction (see the proof of $y_1\ge k+4$); or
     \item[$\circ$] if $y_2= k+6$, we have that
        \[
\Psichaind{y_{2}}{m_{2}}=\Psichaind{k+6}{k+2}\Psichaind{k}{m_2}
\]
and then
    \[
\psi_k\vtt=\psi_k\widehat{\chi}_1\Psichaind{k-2}{m_1}\psi_{k+4}\chi_{2}\Psichaind{y_{2}}{m_{2}}\vtp=0
\]
by induction (works entirely the same as the previous cases).
\end{itemize}

\item[b)] Assume that $\bi_{\tp_k}=(0,1)$.

In this case, we will have an additional chain on the left of $\Psi_{y_1}$ such that $\psi_k$ appears in $\chi_0$: as $\tp_{k},\tp_{k+2}$ are in the leg in $(0,1)$-positions, we have some $r\le k-2$ in $\Arm(\ttt)$ in the position $\ttt(1,m_0-1)$ and some $s\ge k+2$ in the position $\ttt(1,m_0)$.
Hence the chain $\chi_0$ which moves $s$ into the right position must involve the terms $\psi_{k+2}\psi_{k+1}\psi_k\psi_{k-1}\psi_{k-2}$ in this order.
Notice that $\psi_k$ acts on these additional terms exactly the same way as described in cases d)-f) in \cref{3pairs}:
\begin{align*}
    \psi_k\vtt&=\chi_0\Psichaind{y_0}{m_0}\chi_1\Psichaind{y_1}{m_1}\chi_2\dots\chi_{h}\Psichaind{y_h}{m_h}z^\la\\
    &=\psi_{s-1}\dots\psi_{k+2}\psi_{k+1}\psi_k\psi_{k-1}\psi_{k-2}\dots\psi_{y_0+2}\Psichaind{y_0}{m_0}\chi_1\Psichaind{y_1}{m_1}\chi_2\dots\chi_{h}\Psichaind{y_h}{m_h}z^\la\\
    &\,\;\vdots\\
    &=\psi_{s-1}\dots\psi_{k+2}\psi_{k-2}\dots\psi_{y_0+2}\Psichaind{y_0}{m_0}\chi_1\Psichaind{y_1}{m_1}\chi_2\dots\chi_{h}\Psichaind{y_h}{m_h}z^\la\\
    &=\psi_{s-1}\dots\psi_{k+3}\psi_{k-2}\dots\psi_{y_0+2}\Psichaind{y_0}{m_0}\psi_{k+2}\chi_1\underbrace{\Psichaind{y_1}{m_1}\chi_2\dots\chi_{h}\Psichaind{y_h}{m_h}z^\la}_{=\vtp}\\
    &=0 \text{ by case a).}
\end{align*}
(The chain $\Psichaind{y_0}{m_0}$ might be empty though.)

Now we see why the statement requires different lower bounds for $y_i$ if $\bi_{\tp_k}=(0,1)$.
Even though we would like to determine the action of $\psi_k$ on $\vtt$, in fact we are looking at $\psi_{k+2}\vtp$ (as $\tp_k,\tp_{k+2}\in\Leg(\ttt)$ with shifted residues it follows that $y_1\ge k+2$).
Applying case a) to this expression and shifting the indices appropriately completes the proof.
\end{itemize}

\item \textbf{Even $a$}
\begin{itemize}

   \item[c)] Suppose that $\bi_{\tp_k}=(1,0)$ and $y_1\ge k+2$.
   If $h=1$ and $y_1\ge k+4$, we have that

\begin{align*}
   \psi_k \vtt&=\psi_k\chi_1\Psichaind{y_1}{k+6}\Psi_{k+4}\Psi_{k+2}\Psi_k\Psichaind{k-2}{m_1}\chi_dz^\la\\
   &=\psi_k\chi_1\Psichaind{y_1}{k+6}\Psi_{k+4}\Psichaind{k-2}{m_1}\chi_dz^\la\\
&=\psi_k\chi_1\Psichaind{y_1}{k+6}\Psichaind{k-2}{m_1}\Psi_{k+4}\underbrace{(\psi_{\ell-1}\dots\psi_{k+8})\psi_{k+7}\psi_{k+6}\psi_{k+5}\psi_{k+4}\dots\psi_a}_{=\chi_d} z^\la\\
&\,\;\vdots\,\;\text{for the computation see \cref{3pairs} Case $1a)$}\\
&=\psi_k\chi_1\Psichaind{y_1}{k+6}\Psichaind{k-2}{m_1}\psi_{k+4}\psi_{k+3}\psi_{k+5}(\psi_{\ell-1}\dots\psi_{k+8})\psi_{k+7}\psi_{k+2}\dots\psi_a\psi_{k+6} z^\la\\
&=0.
\end{align*}
Note that $\psi_{\ell-1}\dots\psi_{k+8}$ might be empty, but our conditions guarantee that $\psi_{k+6}$ is there to eliminate $z^\la$.

A similar reasoning as in the previous cases shows that if $h>1$, then for all $i=1,2,\dots,h-2$ we either have that $y_{i+1}\ge 4i+(k+4)$ or $y_{i+1}\ge 4i+(k+2)$ if $\chi_{i+1}\neq 1$.
Then, unless $y_h=4h+k-2$, it immediately follows that $\ell>4h+k$ and if $y_h=4h+k-2$, we know that $\ell>y_h+4$ by assumption.
We see that on the right end of the chain we will have an additional braid relation (for shorthand notation, let $s:=4h+k$), where $\psi_s$ comes from either $\chi_h$ (if $y_h\neq s$) or from $\Psi_{s}$ (if $y_h=s$).
Hence, in order to complete the proof, it suffices to consider the following:
\begin{align*}
\psi_{s}\chi_dz^\la&=\psi_{s}\psi_{\ell-1}\dots\psi_{s+2}\psi_{s+1}\psi_{s}\psi_{s-1}\psi_{s-2}\dots\psi_az^\la\\
&=\psi_x\dots\psi_{s+2}(\psi_{s}\psi_{s+1}\psi_{s})\psi_{s-1}\psi_{s-2}\dots\psi_az^\la\\
    &\,\;\vdots\\
    &=\psi_{\ell-1}\dots\psi_{s+3}\psi_{s+2}\psi_{s-2}\dots\psi_az^\la\\
    &=\psi_{\ell-1}\dots\psi_{s+3}\psi_{s-2}\dots\psi_a\psi_{s+2}z^\la\\
    &=0.
\end{align*}

 \item[d)] Take $\bi_{\tp_k}=(1,0)$.
 
The proof is a combination of the proofs of cases b) and c).
\end{itemize}
\end{enumerate}

\underline{Converse:}
If $\bi_{\tp_k}=(1,0)$, let $t$ be minimal such that $y_{t}\le 4t+k-4$ or $y_t=4t+k-2$ and $\chi_{t}=\id$ (or if $a$ is even and $t=h$, then let $y_t=4t+k-2,\chi_{t}=\id$ or $\ell=4t+k+2$).
We see that if such a $t$ exists, it means that there is at least one more pair $\tp_c\in\Arm(\ttt)$ such that $k+2<c$.
Indeed, if $y_t=4t+k-2$ and $\chi_{t}=\id$, it is easy to see that $\tp_{4t+k}\in\Arm(\ttt)$.
If $y_{t}\le 4t+k-4$, by minimality of $t$ we know that either $y_{t-1}\ge 4t+k-4$ or $y_{t-1}=4t+k-6$ with $\chi_{t-1}\neq\id$.
If $y_{t-1}\ge 4t+k-4$ then $y_{t-1}\nless y_t$.
Hence, $y_{t-1}=4t+k-6$ and $\chi_{t-1}\neq\id$.
But this immediately implies that $\chi_t\neq\id$ too, thus $\tp_{4t+k-2}\in\Arm(\ttt)$.
For $\bi_{\tp_k}=(0,1)$, we can repeat the previous argument with increasing all indices by $2$.

If $\tp_c\in\Arm(\ttt)$ and there are no other pairs between $\tp_{k+2}$ and $\tp_c$, we have already shown in \cref{3pairs} that $\psi_k\vtt\neq0$.
Hence, from now on we will assume that there exists some minimal $c$ such that $k+2<c<d$ where $\tp_c\in\Leg(\ttt)$ and $\tp_d\in\Arm(\ttt)$.

By assumption, $y_1\ge k$, but as $\tp_{k+2}\in\Leg(\ttt)$, we have that $y_1\ge k+2$.
If $y_1=k+2$ and $\chi_1=\id$, then we would have that $\tp_{k+4}\in\Arm(\ttt)$.
Thus we can assume that $t\ge 2$.

If $t=2$ and $\tp_k=(1,0)$, by minimality of $t$ we have the following:
\begin{itemize}
    \item $y_1=k+2$ with $\chi_1\neq\id$ or $y_1\ge k+4$, and
    \item $y_2\le k+4$ or $y_2=k+6$ and $\chi_2=\id$ or $y_2=k+6,\chi_2\neq\id$ and $\ell=k+10$.
\end{itemize}
As $y_1<y_2$, we see that $k+2\le y_1\le k+4$.

If $t=2$ and $\tp_k=(0,1)$, we must have that:
\begin{itemize}
    \item $y_1=k+4$ with $\chi_1\neq \id$ or $y_1\ge k+6$,
    \item $y_2\le k+6$ or $y_2=k+8$ and $\chi_2=\id$ or $y_2=k+8,\chi_2\neq\id$ and $\ell=k+12$.
\end{itemize}
As $y_1<y_2$, we see that $k+4\le y_1\le k+6$.

Hence, for $t=2$ we need to consider only the following cases:
\begin{enumerate}
    \item $\bf{\bi_{\tp_k}=(1,0)}$

\begin{itemize}
    \item[a)] $y_1=k+4,y_2=k+6$ and $\chi_2=\id$, which forces $\chi_1=\id$ ($a$ can be even or odd):
    \begin{align*}
        \psi_k\vtt&=\psi_k\Psichaind{k+4}{m_1}\Psichaind{k+6}{m_2}\chi_3\Psichaind{y_3}{m_3}\underbrace{\dots\chi_{h}\Psichaind{y_h}{m_h}(\chi_d)z^\la}_{=\vtp}\\
        &=\psi_k\Psi_{k+4}\Psichaind{k-2}{m_1}\Psichaind{k+6}{m_2}\chi_3\Psichaind{y_3}{m_3}\vtp\\
        &=\psi_k\Psichaind{k-2}{m_1}\Psi_{k+4}\Psichaind{k+6}{m_2}\chi_3\Psichaind{y_3}{m_3}\vtp\\
        &=\psi_k\Psichaind{k-2}{m_1}\Psichaind{k+4}{m_2}\chi_3\Psichaind{y_3}{m_3}\vtp
    \end{align*}
\end{itemize}
The next two sets of conditions deal with the special case when $a$ is even, $t=2=h$ and $\tp_{k+10}$ appears at the end of $\Arm(\ttt)$.
These immediately force $y_2=k+6,\chi_2\neq\id$ and $\ell=k+10$.
\begin{itemize}
  \item[b)] $y_1=k+2$ and $\chi_1\neq\id$:
   \begin{align*}
   \psi_k\vtt&=\chi_1\psi_k\Psichaind{k+2}{m_1}\chi_2\Psichaind{k+6}{m_2}\psi_{k+9}\psi_{k+8}\dots\psi_az^\la\\
        &=\chi_1\psi_k\Psichaind{k-2}{m_1}\chi_2\Psichaind{k+6}{m_2}\psi_{k+9}\psi_{k+8}\dots\psi_az^\la\\
        &=\psi_k\Psichaind{k-2}{m_1}\chi_2\chi_1\Psi_{k+6}\Psi_{k+4}\Psichaind{k+2}{m_2}\psi_{k+9}\psi_{k+8}\dots\psi_az^\la\\
        &=\psi_k\Psichaind{k-2}{m_1}\chi_1\chi_2\Psichaind{k+2}{m_2}\psi_{k+9}\psi_{k+8}\dots\psi_az^\la\\
        &=\psi_k\Psichaind{k-2}{m_1}\chi_1\Psichaind{k+2}{m_2}\underbrace{\psi_{k+7}\psi_{k+8}}_{\chi_2}\psi_{k+9}\psi_{k+8}\dots\psi_az^\la\\
        &\,\;\vdots\\
        &=\psi_k\Psichaind{k-2}{m_1}\chi_1\Psichaind{k+2}{m_2}\psi_{k+7}\psi_{k+6}\psi_{k+5}\dots\psi_az^\la
    \end{align*}
    
   \item[c)] $y_1=k+4$. Here $\chi_1=\id$ or $\chi_1=(\psi_{k+5})\psi_{k+6}$:
   \begin{align*}
   \psi_k\vtt&=\chi_1\psi_k\Psichaind{k+4}{m_1}\chi_2\Psichaind{k+6}{m_2}\psi_{k+9}\psi_{k+8}\dots\psi_az^\la\\
        &=\chi_1\Psi_{k+4}\psi_k\Psichaind{k-2}{m_1}\chi_2\Psichaind{k+6}{m_2}\psi_{k+9}\psi_{k+8}\dots\psi_az^\la\\
        &=\psi_k\chi_1\Psichaind{k-2}{m_1}\chi_2\Psi_{k+4}\Psichaind{k+6}{m_2}\psi_{k+9}\psi_{k+8}\dots\psi_az^\la\\
        &=\psi_k\chi_1\Psichaind{k-2}{m_1}\chi_2\Psi_{k+4}\Psichaind{k+2}{m_2}\psi_{k+9}\psi_{k+8}\dots\psi_az^\la\\
        &=\psi_k\Psichaind{k-2}{m_1}\chi_1\Psichaind{k+4}{m_2}\underbrace{(\psi_{k+7})\psi_{k+8}}_{\chi_2}\psi_{k+9}\psi_{k+8}\dots\psi_az^\la\\
        &\,\;\vdots\\
        &=\psi_k\Psichaind{k-2}{m_1}\chi_1\Psichaind{k+4}{m_2}(\psi_{k+7})\psi_{k+6}\psi_{k+5}\dots\psi_az^\la.
    \end{align*}
    
\end{itemize}

 \item $\bf{\bi_{\tp_k}=(0,1)}$
 
 Recall from the previous part that $\chi_0$ contains the subchain $\psi_{k+2}\psi_{k+1}\dots\psi_{k-2}$.
 Without loss of generality, we may assume that the leftmost entry in $\chi_0$ is precisely $\psi_{k+2}$ as $\psi$ terms with larger indices can be absorbed into $\psi^*$ on the left.
 For shorthand notation we also write $\chi'_0=\psi_{k-2}\psi_{k-3}\dots\psi_{y_0+2}$.
 For the computation of $\psi_k\chi_0$, details can be found in part 1b) of this proof.

 \begin{itemize}
 
    \item[d)] $y_1=k+6,y_2=k+8$ and $\chi_2=\id$, which forces $\chi_1=\id$:
 \begin{align*}
        \psi_k\vtt&=\psi_k\chi_0\Psichaind{y_0}{m_0}\Psichaind{k+6}{m_1}\Psichaind{k+8}{m_2}\underbrace{\chi_3\Psichaind{y_3}{m_3}\dots\chi_{h}\Psichaind{y_h}{m_h}\chi_dz^\la}_{=\vtp}\\
         &\,\;\vdots\\
         &=\chi'_0\Psichaind{y_0}{m_0}\psi_{k+2}\Psichaind{k+6}{m_1}\Psichaind{k+8}{m_2}\vtp\\
         &=\chi'_0\Psichaind{y_0}{m_0}\Psi_{k+6}\psi_{k+2}\Psichaind{k}{m_1}\Psichaind{k+8}{m_2}\vtp\\
    &=\chi'_0\Psichaind{y_0}{m_0}\psi_{k+2}\Psichaind{k}{m_1}\Psi_{k+6}\Psichaind{k+8}{m_2}\vtp\\
    &=\chi'_0\Psichaind{y_0}{m_0}\psi_{k+2}\Psichaind{k}{m_1}\Psichaind{k+6}{m_2}\vtp.
    \end{align*}

     \item[e)] $y_1=k+4,\chi_1\neq\id,y_2=k+8$ and $\chi_2=\id$:
\begin{align*}
        \psi_k\vtt&=\psi_k\chi_0\Psichaind{y_0}{m_0}\chi_1\Psichaind{k+4}{m_1}\Psichaind{k+8}{m_2}\vtp\\
         &\,\;\vdots\\
         &=\chi_0'\Psichaind{y_0}{m_0}\underbrace{\psi_{k+2}\chi_1\Psichaind{k+4}{m_1}\Psichaind{k+8}{m_2}\vtp}_{\text{apply \cref{3pairs}}}\\
    &=\chi_0'\Psichaind{y_0}{m_0}\psi_{k+2}\Psichaind{k}{m_1}\chi_1\Psichaind{k+4}{m_2}\vtp.
    \end{align*}

     \item[f)] $y_1=k+4,y_2=k+6$ and $\chi_1\neq\id\neq\chi_2$ (note that in this case $\chi_1=\psi_{k+5}\psi_{k+6},\chi_2=(\psi_{k+9})\psi_{k+7}\psi_{k+8}$, so they do not commute with each other):
\begin{align*}
        \psi_k\vtt&=\psi_k\chi_0\Psichaind{y_0}{m_0}\chi_1\Psichaind{k+4}{m_1}\chi_2\Psichaind{k+6}{m_2}\vtp\\
         &\,\;\vdots\\
         &=\chi'_0\Psichaind{y_0}{m_0}\chi_1\psi_{k+2}\Psichaind{k+4}{m_1}\chi_2\Psichaind{k+6}{m_2}\vtp\\
        &=\chi'_0\Psichaind{y_0}{m_0}\chi_1\psi_{k+2}\Psichaind{k}{m_1}\chi_2\Psichaind{k+6}{m_2}\vtp\\
        &=\chi'_0\Psichaind{y_0}{m_0}\psi_{k+2}\chi_1\Psichaind{k}{m_1}\chi_2\Psichaind{k+6}{m_2}\vtp\\
         &=\chi_0'\Psichaind{y_0}{m_0}\psi_{k+2}\Psichaind{k}{m_1}\psi_{k+5}\underbrace{\psi_{k+6}\psi_{k+7}\psi_{k+8}\Psichaind{k+6}{m_2}}_{\text{apply \cref{halfdom-leg}}}\vtp\\
         &\,\;\vdots\\
    &=\chi'_0\Psichaind{y_0}{m_0}\chi_0'\Psichaind{k}{m_1}\psi_{k+5}\psi_{k+8}\psi_{k+7}\psi_{k+6}\Psichaind{k+4}{m_2}\vtp.
    \end{align*}
\end{itemize}
Finally, as for $\bi_{\tp_k}=(1,0)$, we consider the case when $a$ is even $t=2=h$ and $\tp_{k+12}$ is at the end of $\Arm(\ttt)$ separately.
Here we immediately see that $y_2=k+8,\chi_2\neq\id$ and $\ell=k+12$.
\begin{itemize}
     \item[g)] $y_1=k+4$ and $\chi_1\neq\id$; and
     \item[h)] $y_1=k+6$ and $\chi_1=\id$ or $\chi_1=\psi_{k+8}$.
     
     We omit these proofs as they are essentially the same as in the $(1,0)$-case.
 \end{itemize}
\end{enumerate}

Now suppose that $t>2$.

First, we look at the different combinations of the first two $\Psi$-chains on the left.
(As $t>2$, it ensures that the propagation cannot stop here.)
We have several cases to consider:
\begin{enumerate}
     \item $\bf{\bi_{\tp_k}=(0,1)}$

\begin{itemize}
    \item [a)] If $y_1=k+2,y_2=k+6$ and $\chi_1\neq\id\neq\chi_2$, we have that:
    \begin{align*}
    \psi_k\vtt&=\psi_k\chi_1\Psichaind{k+2}{m_1}\chi_2\Psichaind{k+6}{m_2}\chi_3\Psichaind{y_3}{m_3}\dots\chi_{h}\Psichaind{y_h}{m_h}\chi_dz^\la\\
       &=\chi_1\psi_k\Psichaind{k-2}{m_1}\chi_2\Psichaind{k+6}{m_2}\chi_3\Psichaind{y_3}{m_3}\dots\chi_{h}\Psichaind{y_h}{m_h}\chi_dz^\la\\
   &=\psi_k\Psichaind{k-2}{m_1}\chi_2\chi_1\Psi_{k+6}\Psi_{k+4}\Psichaind{k+2}{m_2}\chi_3\Psichaind{y_3}{m_3}\dots\chi_{h}\Psichaind{y_h}{m_h}\chi_dz^\la\\
    &=\psi_k\Psichaind{k-2}{m_1}\chi_1\Psichaind{k+2}{m_2}\chi_2\chi_3\Psichaind{y_3}{m_3}\dots\chi_{h}\Psichaind{y_h}{m_h}\chi_d z^\la.
    \end{align*}
    
\item [b)] If $y_1=k+2,y_2\ge k+8$ and $\chi_1\neq\id$ (note that $\chi_2$ might be empty), then:
    \begin{align*}
    \psi_k\vtt&=\chi_1\psi_k\Psichaind{k+2}{m_1}\chi_2\Psichaind{y_2}{m_2}\chi_3\Psichaind{y_3}{m_3}\dots\chi_{h}\Psichaind{y_h}{m_h}\chi_dz^\la\\
       &=\chi_1\psi_k\Psichaind{k-2}{m_1}\chi_2\Psichaind{y_2}{m_2}\chi_3\Psichaind{y_3}{m_3}\dots\chi_{h}\Psichaind{y_h}{m_h}\chi_dz^\la\\
   &=\psi_k\widehat{\chi}_1\Psichaind{k-2}{m_1}\chi_2\Psichaind{y_2}{k+8}\psi_{k+4}\Psi_{k+6}\Psi_{k+4}\Psichaind{k+2}{m_2}\chi_3\Psichaind{y_3}{m_3}\dots\chi_{h}\Psichaind{y_h}{m_h}\chi_dz^\la\\
    &=\psi_k\Psichaind{k-2}{m_1}\chi_1\Psichaind{k+2}{m_2}\chi_2\Psichaind{y_2}{k+8}\chi_3\Psichaind{y_3}{m_3}\dots\chi_{h}\Psichaind{y_h}{m_h}\chi_d z^\la.
    \end{align*}
    
    \item [c)] If $y_1= k+4,y_2= k+6$ and $\chi_2\neq\id$ (note that $\chi_1$ might be empty), we have that:
    \begin{align*}
    \psi_k\vtt&=\psi_k\chi_1\Psichaind{k+4}{m_1}\chi_2\Psichaind{k+6}{m_2}\chi_3\Psichaind{y_3}{m_3}\dots\chi_{h}\Psichaind{y_h}{m_h}\chi_dz^\la\\
&=\chi_1\Psi_{k+4}\psi_k\Psi_{k+2}\Psi_k\Psichaind{k-2}{m_1}\chi_2\Psichaind{k+6}{m_2}\chi_3\Psichaind{y_3}{m_3}\dots\chi_{h}\Psichaind{y_h}{m_h}\chi_dz^\la\\
       &=\psi_k\Psichaind{k-2}{m_1}\chi_1\chi_2\Psi_{k+4}\Psichaind{k+6}{m_2}\chi_3\Psichaind{y_3}{m_3}\dots\chi_{h}\Psichaind{y_h}{m_h}\chi_dz^\la\\
   &=\psi_k\Psichaind{k-2}{m_1}\chi_1\chi_2\Psi_{k+4}\Psi_{k+6}\Psi_{k+4}\Psichaind{k+2}{m_2}\chi_3\Psichaind{y_3}{m_3}\dots\chi_{h}\Psichaind{y_h}{m_h}\chi_dz^\la\\
    &=\psi_k\Psichaind{k-2}{m_1}\chi_1\Psichaind{k+4}{m_2}\chi_2\chi_3\Psichaind{y_3}{m_3}\dots\chi_{h}\Psichaind{y_h}{m_h}\chi_d z^\la.
    \end{align*}
    
     \item [d)] If $y_1\ge k+4,y_2\ge k+8$ (note that both $\chi_1$ and $\chi_2$ might be empty), then we have that:
    \begin{align*}
    \psi_k\vtt&=\psi_k\chi_1\Psichaind{y_1}{m_1}\chi_2\Psichaind{y_2}{m_2}\chi_3\Psichaind{y_3}{m_3}\dots\chi_{h}\Psichaind{y_h}{m_h}\chi_dz^\la\\
    &=\chi_1\Psichaind{y_1}{k+4}\psi_k\Psi_{k+2}\Psi_k\Psichaind{k-2}{m_1}\chi_2\Psichaind{y_2}{m_2}\chi_3\Psichaind{y_3}{m_3}\dots\chi_{h}\Psichaind{y_h}{m_h}\chi_dz^\la\\
       &=\chi_1\Psichaind{y_1}{k+6}\psi_k\Psichaind{k-2}{m_1}\chi_2\Psi_{k+4}\Psichaind{y_2}{m_2}\chi_3\Psichaind{y_3}{m_3}\dots\chi_{h}\Psichaind{y_h}{m_h}\chi_dz^\la\\
   &=\psi_k\Psichaind{k-2}{m_1}\chi_1\Psichaind{y_1}{k+6}\chi_2\Psichaind{y_2}{k+8}\Psi_{k+4}\Psi_{k+6}\Psi_{k+4}\Psichaind{k+2}{m_2}\chi_3\Psichaind{y_3}{m_3}\dots\chi_{h}\Psichaind{y_h}{m_h}\chi_dz^\la\\
   &=\psi_k\Psichaind{k-2}{m_1}\chi_1\Psichaind{y_1}{k+6}\chi_2\Psichaind{y_2}{k+8}\Psichaind{k+4}{m_2}\chi_3\Psichaind{y_3}{m_3}\dots\chi_{h}\Psichaind{y_h}{m_h}\chi_dz^\la\\
    &=\psi_k\Psichaind{k-2}{m_1}\chi_1\Psichaind{y_1}{m_2}\chi_2\Psichaind{y_2}{k+8}\chi_3\Psichaind{y_3}{m_3}\dots\chi_{h}\Psichaind{y_h}{m_h}\chi_d z^\la.
    \end{align*}

\end{itemize}

 \item $\bf{\bi_{\tp_k}=(0,1)}$
\begin{itemize}
    \item[a)] $y_1=k+4,y_2=k+8$ and $\chi_1\neq\id\neq\chi_2$;
    \item[b)] $y_1=k+4,y_2\ge k+10$ and $\chi_1\neq\id$ ($\chi_2$ might be empty);
     \item[c)] $y_1\ge k+6,y_2= k+8$ and $\chi_2\neq\id$ ($\chi_1$ might be empty);
   \item[b)] $y_1\ge k+6$ and $y_2\ge k+10$ (both $\chi_1$ and $\chi_2$ might be empty).
\end{itemize}
   We observe that these are the same conditions as for $\bi_{\tp_k}=(1,0)$, except that the indices are increased by $2$.
   The computations are identical.
\end{enumerate}

For $2\le i \le t-1$, we claim that whenever $4i+k\le y_{i}$, then
\begin{equation}\label{claim1}
\psi_k\vtt=\psi_k(\Psichaind{k-2}{m_1}\chi_1\Psichaind{y_1}{m_{2}}\chi_{2}\Psichaind{y_{2}}{m_{3}}\dots\chi_{i-1}\Psichaind{y_{i-1}}{m_{i}})\chi_{i}\Psichaind{y_{i}}{4i+k}\chi_{i+1}\Psichaind{y_{i+1}}{m_{i+1}}\dots\chi_h\Psichaind{y_h}{m_h}\chi_dz^\la
\end{equation}

and whenever $4i+k-2=y_i$ and $\chi_i\neq\id$, then
\begin{equation}\label{claim2}
\psi_k\vtt=\psi_k(\Psichaind{k-2}{m_1}\chi_1\Psichaind{y_1}{m_{2}}\chi_{2}\Psichaind{y_{2}}{m_{3}}\dots\chi_{i-1}\Psichaind{y_{i-1}}{m_{i}})\chi_{i}\chi_{i+1}\Psichaind{y_{i+1}}{m_{i+1}}\dots\chi_h\Psichaind{y_h}{m_h}\chi_dz^\la.
\end{equation}

We will prove \cref{claim1,claim2} by induction on $i$.
If $i=2$ and $k+4\le y_1$, then
\begin{equation*}
\psi_k\vtt=\psi_k\Psichaind{k-2}{m_1}\chi_1\Psichaind{y_1}{k+4}\chi_{2}\Psichaind{y_{2}}{m_{2}}\dots\chi_h\Psichaind{y_h}{m_h}\chi_dz^\la,
\end{equation*}
and if $y_1=k+2$ and $\chi_1\neq\id$, then
\begin{equation*}
\psi_k\vtt=\psi_k\Psichaind{k-2}{m_1}\chi_1\chi_{2}\Psichaind{y_{2}}{m_{2}}\dots\chi_h\Psichaind{y_h}{m_h}\chi_dz^\la.
\end{equation*}

For the induction step, we break down the proof into several subcases.
We have four cases to consider:
\begin{enumerate}
    \item We show that if for all $2\le i-1 \le t-2$ either \cref{claim1} or \cref{claim2} holds  and $y_{i-1}\ge4i+k-4$ and $y_{i}\ge4i+k$, then \cref{claim1} follows for $i$.
    
As $y_{i-1}\ge4i+k-4$ we can assume by induction that
\begin{equation}\label{ass1}
 \psi_k\vtt=\psi_k(\Psichaind{k-2}{m_1}\chi_1\Psichaind{y_1}{m_{2}}\chi_{2}\Psichaind{y_{2}}{m_{3}}\dots\chi_{i-2}\Psichaind{y_{i-2}}{m_{i-1}})\chi_{i-1}\Psichaind{y_{i-1}}{4i+k-4}\chi_{i}\Psichaind{y_{i}}{m_{i}}\dots\chi_h\Psichaind{y_h}{m_h}\chi_dz^\la.
\end{equation}

As $i-1\le t-2$ and $y_{i}\ge 4i+k$ we have that
\begin{align*}
&\,\,\,\chi_{i-1}\Psichaind{y_{i-1}}{4i+k-4}\chi_{i}\Psichaind{y_{i}}{m_{i}}\dots\chi_h\Psichaind{y_h}{m_h}\chi_dz^\la\\
&=\chi_{i-1}\Psichaind{y_{i-1}}{4i+k-2}\chi_{i}\Psichaind{y_{i}}{4i+k}\Psi_{4i+k-4}\Psichaind{4i+k-2}{m_{i}}\dots\chi_h\Psichaind{y_h}{m_h}\chi_dz^\la\\
&=\chi_{i-1}\Psichaind{y_{i-1}}{4i+k-2}\chi_{i}\Psichaind{y_{i}}{4i+k}\Psichaind{4i+k-4}{m_{i}}\dots\chi_h\Psichaind{y_h}{m_h}\chi_dz^\la\\
&=\chi_{i-1}\Psichaind{y_{i-1}}{m_{i}}\chi_{i}\Psichaind{y_{i}}{4i+k}\chi_{i+1}\Psichaind{y_{i+1}}{m_{i+1}}\dots\chi_h\Psichaind{y_h}{m_h}\chi_dz^\la.
\end{align*}
    
    \item We show that if for all $2\le i-1 \le t-2$ either \cref{claim1} or \cref{claim2} holds and $y_{i-1}\ge4i+k-4,y_{i}=4i+k-2,\chi_i\neq\id$, then \cref{claim2} follows for $i$.

As $y_{i-1}\ge4i+k-4$ we can assume that \cref{ass1} holds.
Then:
\begin{align*}
&\,\,\,\chi_{i-1}\Psichaind{y_{i-1}}{4i+k-4}\chi_{i}\Psichaind{4i+k-2}{m_{i}}\dots\chi_h\Psichaind{y_h}{m_h}\chi_dz^\la\\
&=\chi_{i-1}\Psichaind{y_{i-1}}{4i+k-2}\chi_{i}\Psi_{4i+k-4}\Psichaind{4i+k-2}{m_{i}}\dots\chi_h\Psichaind{y_h}{m_h}\chi_dz^\la\\
&=\chi_{i-1}\Psichaind{y_{i-1}}{4i+k-2}\chi_{i}\Psichaind{4i+k-4}{m_{i}}\dots\chi_h\Psichaind{y_h}{m_h}\chi_dz^\la\\
&=\chi_{i-1}\Psichaind{y_{i-1}}{m_{i}}\chi_{i}\chi_{i+1}\Psichaind{y_{i+1}}{m_{i+1}}\dots\chi_h\Psichaind{y_h}{m_h}\chi_dz^\la.
\end{align*}

    \item We show that if for all $2\le i-1 \le t-2$ either \cref{claim1} or \cref{claim2} holds and $y_{i-1}=4i+k-6,\chi_{i-1}\neq\id$ and $y_{i}\ge4i+k$ then \cref{claim1} follows for $i$.

As $y_{i-1}=4i+k-6$ and $\chi_{i-1}\neq\id$ we can assume by induction that
\begin{equation}\label{ass2}
 \psi_k\vtt=\psi_k(\Psichaind{k-2}{m_1}\chi_1\Psichaind{y_1}{m_{2}}\chi_{2}\Psichaind{y_{2}}{m_{3}}\dots\chi_{i-2}\Psichaind{y_{i-2}}{m_{i-1}})\chi_{i-1}\chi_{i}\Psichaind{y_{i}}{m_{i}}\dots\chi_h\Psichaind{y_h}{m_h}\chi_dz^\la.
\end{equation}

As $y_{i}\ge4i+k$, $\chi_{i-1}$ and $\chi_i$ commute, so:
\begin{align*}
&\,\,\,\chi_{i-1}\chi_i\Psichaind{y_{i}}{m_i}\chi_{i+1}\Psichaind{y_{i+1}}{m_{i+1}}\dots\chi_h\Psichaind{y_h}{m_h}\chi_dz^\la\\
&=\chi_{i}\widehat{\chi}_{i-1}\Psichaind{y_i}{4i+k}\psi_{4i+k-4}\Psichaind{4i+k-2}{m_{i}}\chi_{i+1}\Psichaind{y_{i+1}}{m_{i+1}}\dots\chi_h\Psichaind{y_h}{m_h}\chi_dz^\la\\
&=\chi_{i}\widehat{\chi}_{i-1}\Psichaind{y_i}{4i+k}\psi_{4i+k-4}\Psichaind{4i+k-6}{m_{i}}\chi_{i+1}\Psichaind{y_{i+1}}{m_{i+1}}\dots\chi_h\Psichaind{y_h}{m_h}\chi_dz^\la\\
&=\chi_{i-1}\Psichaind{y_{i-1}}{m_{i}}\chi_{i}\Psichaind{y_i}{4i+k}\chi_{i+1}\Psichaind{y_{i+1}}{m_{i+1}}\dots\chi_h\Psichaind{y_h}{m_h}\chi_dz^\la
\end{align*}
as $y_{i-1}=4i+k-6$.

    \item We show that if for all $2\le i-1 \le t-2$ either \cref{claim1} or \cref{claim2} holds and $y_{i-1}=4i+k-6,\chi_{i-1}\neq\id,y_{i}=4i+k-2$ and $\chi_i\neq\id$, then \cref{claim2} follows for $i$.

As $y_{i-1}=4i+k-6$ and $\chi_{i-1}\neq\id$ we can assume that \cref{ass2} holds.
As $y_{i}=4i+k-2$, $\chi_{i-1}$ and $\chi_i$ commute, so:
\begin{align*}
&\,\,\,\chi_{i-1}\chi_i\Psichaind{4i+k-2}{m_i}\chi_{i+1}\Psichaind{y_{i+1}}{m_{i+1}}\dots\chi_h\Psichaind{y_h}{m_h}\chi_dz^\la\\
&=\chi_{i}\chi_{i-1}\Psichaind{4i+k-2}{4i+k-4}\Psichaind{4i+k-6}{m_{i}}\chi_{i+1}\Psichaind{y_{i+1}}{m_{i+1}}\dots\chi_h\Psichaind{y_h}{m_h}\chi_dz^\la\\
&=\chi_{i}\chi_{i-1}\Psichaind{4i+k-6}{m_{i}}\chi_{i+1}\Psichaind{y_{i+1}}{m_{i+1}}\dots\chi_h\Psichaind{y_h}{m_h}\chi_dz^\la\\
&=\chi_{i-1}\Psichaind{y_{i-1}}{m_{i}}\chi_{i}\chi_{i+1}\Psichaind{y_{i+1}}{m_{i+1}}\dots\chi_h\Psichaind{y_h}{m_h}\chi_dz^\la,
\end{align*}
as $y_{i-1}=4i+k-6$.
\end{enumerate}

Finally, we will use \cref{claim1,claim2} to see what happens at the $t$-th chain.
We write $s=4t+k$.
\begin{itemize}
\item[a)] Suppose that $\chi_{t-1}=\id=\chi_{t}$.
By minimality of $t$, we have that $y_{t-1}\ge s-4$.
On the other hand, $s-2\ge y_t>y_{t-1}\ge s-4$, so in fact $y_{t-1}=s-4$ and $y_t=s-2$.

\begin{align*}
    \psi_k\vtt&=\psi_k(\Psichaind{k-2}{m_1}\chi_1\Psichaind{y_1}{m_{2}}\chi_{2}\Psichaind{y_{2}}{m_{3}}\dots\chi_{t-2}\Psichaind{y_{t-2}}{m_{t-1}})\Psichaind{y_{t-1}}{s-4}\Psichaind{y_{t}}{m_{t}}\dots\chi_h\Psichaind{y_h}{m_h}\chi_dz^\la\\
    &\qquad\text{by \cref{claim1}}\\
    &=\psi_k(\Psichaind{k-2}{m_1}\chi_1\Psichaind{y_1}{m_{2}}\chi_{2}\Psichaind{y_{2}}{m_{3}}\dots\chi_{t-2}\Psichaind{y_{t-2}}{m_{t-1}})\Psichaind{y_{t-1}}{m_{t}}\chi_{t+1}\Psichaind{y_{t+1}}{m_{t+1}}\dots\chi_h\Psichaind{y_h}{m_h}\chi_dz^\la
\end{align*}
We see that the last line is in our standard form and is non-zero.

\item[b)] Suppose $\chi_{t-1}\neq\id\neq\chi_t$.
As $\chi_{t-1}\neq\id$, we have that $y_{t-1}\ge s-6$.
By minimality of $t$ and using the fact that $\chi_t\neq\id$, we also see that $y_t\le s-4$ or if $t=h$ then $y_t=s-2$.
On the other hand, $s-4\ge y_t>y_{t-1}\ge s-6$.
Thus, $y_t=s-4$ and $y_{t-1}=s-6$.
We also see that $\chi_{t-1}=(\psi_{s-5})\psi_{s-4}$ and $\chi_{t}=(\psi_x\dots\psi_{s-1})\psi_{s-3}\psi_{s-2}$ and they do not commute (compare with case 2f) above).
Here $x+1=\ttt(1,m_t)$.
    
\begin{align*}
    \psi_k\vtt&=\psi_k(\Psichaind{k-2}{m_1}\chi_1\Psichaind{y_1}{m_{2}}\chi_{2}\Psichaind{y_{2}}{m_{3}}\dots\chi_{t-2}\Psichaind{y_{t-2}}{m_{t-1}})\chi_{t-1}\chi_{t}\Psichaind{s-2}{m_{t}}\dots\chi_h\Psichaind{y_h}{m_h}\chi_dz^\la\\
    &\qquad\text{by \cref{claim2}}\\
    &=\psi_k(\Psichaind{k-2}{m_1}\chi_1\Psichaind{y_1}{m_{2}}\chi_{2}\Psichaind{y_{2}}{m_{3}}\dots\chi_{t-2}\Psichaind{y_{t-2}}{m_{t-1}})\chi_{t-1}'\Psichaind{y_{t-1}}{m_{t}}\chi_{t+1}\Psichaind{y_{t+1}}{m_{t+1}}\dots\chi_h\Psichaind{y_h}{m_h}\chi_dz^\la
\end{align*}
using \cref{halfdom-leg}.
Here $\chi_{t-1}'=(\psi_{s-5}\psi_x\dots\psi_{s-1})\psi_{s-2}\psi_{s-3}\psi_{s-4}$.
If $t\neq h$ or $\chi_d=\id$, then the second line is in our standard form and is non-zero.
Otherwise, we need to consider the special case where $t=h$, $a$ is even and $\ell=s+2$.
In this case, $\chi_t=\psi_{s-3}\psi_s\psi_{s-1}\psi_{s-2}$.
Then:
\begin{align*}
    \psi_k\vtt&=\psi_k(\Psichaind{k-2}{m_1}\chi_1\Psichaind{y_1}{m_{2}}\chi_{2}\Psichaind{y_{2}}{m_{3}}\dots\chi_{t-2}\Psichaind{y_{t-2}}{m_{t-1}})\chi_{t-1}\chi_{t}\Psichaind{s-4}{m_{t}}\chi_dz^\la\\
    &\qquad\text{by \cref{claim2}}\\
&=\psi_k(\Psichaind{k-2}{m_1}\chi_1\Psichaind{y_1}{m_{2}}\chi_{2}\Psichaind{y_{2}}{m_{3}}\dots\chi_{t-2}\Psichaind{y_{t-2}}{m_{t-1}})\chi_{t-1}\Psichaind{s-6}{m_{t}}\chi_{t}\chi_dz^\la\\
&\,\;\vdots\\
&=\psi_k(\Psichaind{k-2}{m_1}\chi_1\Psichaind{y_1}{m_{2}}\chi_{2}\Psichaind{y_{2}}{m_{3}}\dots\chi_{t-2}\Psichaind{y_{t-2}}{m_{t-1}})\chi_{t-1}\Psichaind{y_{t-1}}{m_{t}}\psi_{s-1}\psi_{s-2}\psi_{s-3}\dots\psi_az^\la.
\end{align*}

\item[c)] Suppose $\chi_{t-1}\neq\id$ and $\chi_t=\id$.
As $\chi_{t-1}\neq\id$, we have that $s-4\ge y_{t-1}\ge s-6$.
By minimality of $t$, we also see that $y_t\le s-2$.
On the other hand, $s-6\le y_{t-1}< y_t\le s-2$.
If $y_t=y_{t-1}+2$, this would force $\chi_t\neq\id$ (see case b)).
Hence, $y_t=s-2$ and $y_{t-1}=s-6$.

\begin{align*}
   \psi_k\vtt&=\psi_k(\Psichaind{k-2}{m_1}\chi_1\Psichaind{y_1}{m_{2}}\chi_{2}\Psichaind{y_{2}}{m_{3}}\dots\chi_{t-2}\Psichaind{y_{t-2}}{m_{t-1}})\chi_{t-1}\Psichaind{y_{t}}{m_{t}}\dots\chi_h\Psichaind{y_h}{m_h}\chi_dz^\la\\
    &\qquad\text{by \cref{claim2}}\\
    &=\psi_k(\Psichaind{k-2}{m_1}\chi_1\Psichaind{y_1}{m_{2}}\chi_{2}\Psichaind{y_{2}}{m_{3}}\dots\chi_{t-2}\Psichaind{y_{t-2}}{m_{t-1}})\chi_{t-1}\Psichaind{y_{t-1}}{m_{t}}\chi_{t+1}\Psichaind{y_{t+1}}{m_{t+1}}\dots\chi_h\Psichaind{y_h}{m_h}\chi_dz^\la
\end{align*}
We see that the final line is in our standard form and is non-zero.

\item[d)] Suppose that $\chi_{t-1}=\id$ and $\chi_t\neq\id$.
Observe that this case can only happen if $a$ is even with $\ell=s+2,t=h$ and $y_t=s-2$.
By minimality of $t$, we have that $y_{t-1}\ge s-4$.
On the other hand, $s-2=y_t>y_{t-1}\ge s-4$, thus $y_{t-1}=s-4$.
We also see that $\chi_t=(\psi_{s-1})\psi_s$.
Using \cref{claim1}, we have that:

\begin{align*}
    \psi_k\vtt&=\psi_k(\Psichaind{k-2}{m_1}\chi_1\Psichaind{y_1}{m_{2}}\chi_{2}\Psichaind{y_{2}}{m_{3}}\dots\chi_{t-2}\Psichaind{y_{t-2}}{m_{t-1}})\Psi_{s-4}\chi_{t}\Psichaind{s-2}{m_{t}}\chi_dz^\la\\
&=\psi_k(\Psichaind{k-2}{m_1}\chi_1\Psichaind{y_1}{m_{2}}\chi_{2}\Psichaind{y_{2}}{m_{3}}\dots\chi_{t-2}\Psichaind{y_{t-2}}{m_{t-1}})\Psichaind{s-4}{m_{t}}\chi_t\chi_dz^\la\\
&\,\;\vdots\\
&=\psi_k(\Psichaind{k-2}{m_1}\chi_1\Psichaind{y_1}{m_{2}}\chi_{2}\Psichaind{y_{2}}{m_{3}}\dots\chi_{t-2}\Psichaind{y_{t-2}}{m_{t-1}})\Psichaind{y_{t-1}}{m_{t}}(\psi_{s-1})\psi_{s-2}\psi_{s-3}\dots\psi_az^\la.
\end{align*}
We see that the last line is in our standard form and is non-zero.
\end{itemize}

For $\bi_{\tp_k}=(0,1)$, we need to consider cases a)--\,d) as above, but with the appropriate indices increased by $2$.
\end{proof}

\begin{rem}
    Observe that \cref{3pairs} part 1 corresponds to the case where $t=c$ and for all $1\le i < t$ we have that $y_i=4i+k-2$ and $\chi_i\neq\id$ (or $y_i=4i+k$ and $\chi_i\neq\id$ for $\bi_{\tp_k}=(0,1)$).
    Thus, the first part of the proof of \cref{4killing-1} also shows the converse of \cref{3pairs} part 1, that is if there exists no $\tp_c$ as given in the statement then $\psi_k\vtt=0$.
\end{rem}

\begin{defn}\label{pair-leg}
Suppose that $k$ is odd such that
\begin{center}
    \begin{tikzpicture}
    \scalefont{0.8}
    \Yboxdimx{30pt}
    \Yboxdimy{20pt}
    \tgyoung(0cm,3cm,1^1\hdts v^1\hdts <p_1{-}1><p_1>^1\hdts <p_m{-}1><p_m>^1\hdts\ell,'11\vdts,<k{-}1>,k,<k{+}1>,<k{+}2>,'11\vdts,<x{-}1>,x,'11\vdts)
    \node at (-0.8,3.4) (a) {{\normalsize{$\ttt=$}}};
    \end{tikzpicture}
\end{center}
where $v$ is minimal such that $k<v$, the section of $\Arm(\ttt)$ $v,\dots,p_1-1,p_1,\dots,p_m=2r+k$ contains $r$ entries and the section of $\Leg(\ttt)$ $k-1,k,\dots,x$ contains $r+2$ entries such that there are $m$ (resp., $m+1$) pairs amongst these $r$ (resp., $r+2$) entries.
Note that if $v+1\in\Arm(\ttt)$ then $v=p_1-1$.

In this case, we define $(\ttt)^{2r+k,k}$ to be the standard $\la$-tableau where $\tp_k$ is moved into $\Arm(\ttt), \tp_{p_m}=\tp_{2r+k}$ is moved into $\Leg(\ttt)$ and all other elements of $\Arm((\ttt)^{2r+k,k})$ and $\Leg((\ttt)^{2r+k,k})$ agree with those of $\ttt$.
This can be easily visualised as a `clockwise' rotation of the entries $k-1,k,\dots,2r+k$.
Thus,\nopagebreak
\begin{center}
    \begin{tikzpicture}
    \scalefont{0.8}
    \Yboxdimx{30pt}
    \Yboxdimy{20pt}
    \tgyoung(0cm,3cm,1^1\hdts <k{-}1>k^1\hdts v^1\hdts<p_1{-}1><p_1>^1\hdts\ell,'11\vdts,<k{+}1>,<k{+}2>,'11\vdts,<x{-}1>,x,'11\vdts,<p_m{-}1>,<p_m>,'11\vdts)
    \node at (-1,3.4) (a) {{\normalsize{$(\ttt)^{p_m,k}=$}}};
    \end{tikzpicture}
\end{center}
where the omitted parts are identical to those in $\ttt$.
 \end{defn}

 \begin{eg}
Consider $\la=(8, 1^9)$ and take $k=3$. If $r=7$ and $\ttt$ is as below, we have that:
\begin{center}
    \begin{tikzpicture}
    \Yboxdim{15pt}
\tgyoung(0cm,0cm,17\eleven\twelve\thirteen\fourteen\sixteen\seventeen,2,3,4,5,6,8,9,\ten,\fifteen)
{\Ylinethick{2pt}
\tgyoung(0cm,0cm,:::;\twelve;\thirteen:;\sixteen;\seventeen,;2,;3,;4,;5,:,;8,;9)}
\tgyoung(8cm,0cm,1237\eleven\twelve\thirteen\fourteen,4,5,6,8,9,\ten,\fifteen,\sixteen,\seventeen).
{\Ylinethick{2pt}
\tgyoung(8cm,0cm,:;2;3::;\twelve;\thirteen:,;4,;5,:,;8,;9:,:,:,;\sixteen,;\seventeen).}
    \node at (-0.6,0.3) (a) {{\normalsize{$\ttt=$}}};
    \node at (7,0.3) (a) {{\normalsize{$(\ttt)^{17,3}=$}}};
    \end{tikzpicture}
\end{center}
Observe that in this case
\[\vtt=\psi_6\psi_{10}\psi_9\psi_8\psi_7\Psichaind{5}{3}\Psichaind{11}{5}\psi_{15}\Psichaind{13}{7}\psi_{16}\psi_{15}\dots\psi_8z^\la
\]
and
\[v_{(\ttt)^{17,3}}=\psi_6\psi_{10}\psi_9\psi_8\psi_7\Psi_5\Psichaind{11}{7}\psi_{13}\psi_{12}\dots\psi_8z^\la.
\]
\end{eg}

\begin{cor}\label{cor:psi-leg}
Assume that $\ttt\in\std(\la)$ and $\tp_k,\tp_{k+2}\in\Leg(\ttt)$.
\begin{enumerate}
    \item If $\bi_{\tp_k}=(1,0)$, let
\[
\vtt=\psi^*\chi_1\Psichaind{y_1}{m_1}\chi_2\dots\chi_{h}\Psichaind{y_h}{m_h}\chi_dz^\la.
\]
Suppose $\psi_k\vtt\neq0$ and that $t$ is minimal such that $y_t\le4t+k-4$ or $y_t=4t+k-2$ and $\chi_t=\id$ or if $a$ is even and $t=h$, assume that $y_t=4t+k-2$ and either $\chi_{t}=\id$ or $\ell=4t+k+2$.
\item If $\bi_{\tp_k}=(0,1)$, let
\[
\vtt=\psi^*\chi_0\Psichaind{y_0}{m_0}\chi_1\Psichaind{y_1}{m_1}\chi_2\dots\chi_{h}\Psichaind{y_h}{m_h}\chi_dz^\la.
\]
Suppose $\psi_k\vtt\neq0$ and that $t$ is minimal such that $y_t\le4t+k-2$ or $y_t=4t+k$ and $\chi_t=\id$ or if $a$ is even and $t=h$, assume that $y_t=4t+k$ and either $\chi_{t}=\id$ or $\ell=4t+k+4$.
 \end{enumerate}
 Then there exists some minimal $r$ such that $\ttt$ is as in \cref{pair-leg} and $\psi_k\vtt=\psi_kv_{(\ttt)^{2r+k,k}}$ which is in reduced form.
\end{cor}

\begin{proof}
We will continue to use the notation $s=4t+k$ from the proof of \cref{4killing-1}.
Observe that the final part of the proof of \cref{4killing-1} has all the necessary results to prove the statement.
First, we assume that $\bi_{\tp_k}=(1,0)$.
The case $\bi_{\tp_k}=(0,1)$ can be proved analogously.
\begin{itemize}
    \item[a)] If $\chi_t=\id$, then this corresponds to cases a) and c) in the proof of \cref{4killing-1}.
    In this case,
\[
    \psi_k\vtt=\psi_k(\Psichaind{k-2}{m_1}\chi_1\Psichaind{y_1}{m_{2}}\dots\chi_{t-2}\Psichaind{y_{t-2}}{m_{t-1}})(\chi_{t-1})\Psichaind{y_{t-1}}{m_{t}}\chi_{t+1}\Psichaind{y_{t+1}}{m_{t+1}}\dots\chi_h\Psichaind{y_h}{m_h}\chi_dz^\la.
\]
We are considering the entries from $k-1$ up to $y_t+2=s$, out of which $2t$ are in the arm and $2t+2$ are in the leg.

Setting $r=2t$, it is easy to see that $\ttt$ is as in \cref{pair-leg}.
Assume that the set $k+3,k+4,\dots,s-3,s-2$ contains $2d$ entries in total.
We have just shown that these entries are evenly distributed between the arm and the leg, that is there are $d$ many of them in $\Leg(\ttt)$.
If we start counting the entries from $k+3$ up to $s-2$ in the natural order, whenever an even (resp.\ odd) entry $g$ is unpaired in $\Leg(\ttt)$, then $g+1$ (resp.\ $g-1$) is in $\Arm(\ttt)$.
We can assume that there are $f-1$ pairs between $k+3$ and $s-2$ in $\Arm(\ttt)$.
Hence, there must be $f-1+2=f+1$ pairs in $\Leg(\ttt)$ between $k-1$ and $s-2$ (the plus two accounts for the pairs $\tp_k$ and $\tp_{k+2}$).
Thus, we see that $\ttt$ is indeed as in \cref{pair-leg} where $\tp_s=\tp_{p_m},f=m$ and $r=2t$.
Looking at the reduced expression, this is the rotation that moves $\tp_k$ into the arm and $\tp_{s}$ into the leg.

\item[b)] If $\chi_t\neq\id$ and either $t\neq h$ or $\chi_d=\id$, then this is case b) in the proof of \cref{4killing-1} and
\[
   \psi_k\vtt=\psi_k(\Psichaind{k-2}{m_1}\chi_1\Psichaind{y_1}{m_{2}}\dots\chi_{t-2}\Psichaind{y_{t-2}}{m_{t-1}})\chi'_{t-1}\Psichaind{y_{t-1}}{m_{t}}\chi_{t+1}\Psichaind{y_{t+1}}{m_{t+1}}\dots\chi_h\Psichaind{y_h}{m_h}\chi_dz^\la.
\]
We are considering the entries from $k-1$ up to $y_t+2=s-2$.
But $\bi_{s-2}=1$, so in total we move $2t-1$ entries in the arm and $2t+1$ in the leg.
Setting $r=2t-1$ and applying the arguments from the previous paragraph, we see that $\ttt$ is as in \cref{pair-leg}.

\item[c)] Finally, if $\chi_t\neq\id,t=h$ and $\chi_d\neq\id$, then this is case d) and the special case of case b) in the proof of \cref{4killing-1}.
We have that
\[
    \psi_k\vtt=\psi_k(\Psichaind{k-2}{m_1}\chi_1\Psichaind{y_1}{m_{2}}\dots\chi_{t-2}\Psichaind{y_{t-2}}{m_{t-1}})\chi_{t-1}\Psichaind{y_{t-1}}{m_{t}}(\psi_{s-1})\psi_{s-2}\psi_{s-3}\dots\psi_az^\la.
\]
We are considering the entries $k-1,k,\dots,s+2$.
In total, we have $4t+4$ of them, $2t+1$ are in the arm and $2t+3$ are in the leg.
Setting $r=2t+1$ and applying the arguments from case a), we see that $\ttt$ is as in \cref{pair-leg}.
\qedhere
\end{itemize}
\end{proof}

Recall from \cref{doms} that $\Psichainu{x}{y}:=\Psi_x\Psi_{x+2}\dots\Psi_{y}$.

\begin{defn}[cf.~\cite{thesis}, Def.~3.22]
    Let $\ttt\in\std(\la)$.
    The corresponding $\vtt$ can be written the following way, which we call reverse normal form RNF (compare with \cref{short2}).
    \[
    \vtt=\rho_0\rho_1\Psichainu{y_1}{m_1}\rho_2\Psichainu{y_2}{m_2}\dots\rho_h\Psichainu{y_h}{m_h}\rho_dz^\la
    \]
    Here $y_{i+1}\le y_i-2$ and $m_{i+1}=m_i-2$ for $1\le i \le h$.
    If $a$ is odd, then $m_h=a$ and $\rho_d=\id$.
    Otherwise (when $a$ is even), $m_h=a-1$ and $\rho_d=\psi_{\ell-1}\dots\psi_a=\chi_d$.
 \end{defn}

     Similarly to $\chi_i$, unless $\rho_i=\id$, we denote the rightmost entry of $\rho_i$ by $\rho_i(\last)=y_{i+1}-2$.
       Observe that the leftmost entry in $\rho_{i+1}$ is strictly less than $y_i-2$.
       It is easy to see that $\rho_i=\id$ whenever $y_i+4\ge y_{i-1}$.
 Further, $\rho_0$ corresponds to a chain of single $\psi$'s with increasing indices from left to right where $\rho_0(\last)=\psi_{m_1+2}$. 
 This chain appears if the pair at the end of $\Arm(\ttt)$ has shifted residues.

\begin{rem}
The action of RNF can be visualised as placing entries into $\Leg(\ttt)$, as opposed to the normal form where we were placing entries into $\Arm(\ttt)$.
To be more specific, RNF puts the pairs $\tp_{y_h},\tp_{y_{h-1}},\dots,\tp_{y_1}$ at the beginning of $\Leg(\ttt)$ (with some possible shifts coming from the $\rho$'s).
\end{rem}

\begin{eg}
Let $n=18,\la=(8,1^{\ten})$ and assume that $\ttt$ is as below.
\begin{center}
    \begin{tikzpicture}
    \Yboxdim{14pt}
\tgyoung(0cm,0cm,168\ten\eleven\fourteen\sixteen<18>,2,3,4,5,7,9,\twelve,\thirteen,\fifteen,\seventeen)
    \node at (-0.6,0.25) (a) {{\normalsize{$\ttt=$}}};
    \end{tikzpicture}
\end{center}
Then
\begin{align*}
    \vtt&=\psi_7\Psichaind{5}{3}\Psichaind{9}{5}\psi_{15}\Psichaind{13}{7}\psi_{17}\psi_{16}\dots\psi_{8}\\
    &=\psi_{15}\Psi_{13}\psi_{7}\Psichainu{9}{11}\Psichainu{5}{9}\Psichainu{3}{7}\psi_{17}\psi_{16}\dots\psi_{8}.
\end{align*}
\end{eg}

\begin{eg}
    We will illustrate using RNF by going through the different cases outlined in \cref{3pairs} part 2.
    \begin{itemize}
        \item[a)] $\bi_{\tp_c}=(1,0)$, $\bi_{\tp_k}=(1,0)$:
    
As $\bi_{\tp_c}=(1,0)=\bi_{\tp_{k}}$, we have that $\rho_{0}=\id=\rho_h$.
Also, note that that for each $1\le i \le h$ the chain $\Psichaind{y_{i+1}}{m_{i+1}}$ contains one more $\Psi$ than $\Psichaind{y_{i}}{m_{i}}$ except for $i=h-1$.
Hence:

\begin{align*}
\psi_k\vtt&=\psi_k\chi_1\Psichaind{c}{m_1}(\chi_2\Psichaind{c+4}{m_2}\dots\chi_{h-2}\Psichaind{k-6}{m_{h-2}}\chi_{h-1}\Psichaind{k-2}{m_{h-1}})\chi_h\Psichaind{k}{m_h}\vtp\\
    &=\psi_k\rho_1\Psichainu{k-2}{y_1}\rho_2\Psichainu{k-6}{y_2}\dots\rho_{h-1}\Psichainu{c+4}{m_{h-1}}\Psichainu{c}{m_h}\vtp\\
    &=\rho_1\psi_k\Psichainu{k-2}{y_1}\rho_2\Psichainu{k-6}{y_2}\dots\rho_{h-1}\Psichainu{c+4}{m_{h-1}}\Psichainu{c}{m_h}\vtp\\
    &=\rho_1\psi_k\Psichainu{k+2}{y_1}\rho_2\Psichainu{k-6}{y_2}\dots\rho_{h-1}\Psichainu{c+4}{m_{h-1}}\Psichainu{c}{m_h}\vtp\\
    &=\psi_k\Psichainu{k+2}{y_1}\rho_2\rho_1\Psichainu{k-6}{y_2}\dots\rho_{h-1}\Psichainu{c+4}{m_{h-1}}\Psichainu{c}{m_h}\vtp\\
    &=\psi_k\Psichainu{k+2}{y_1}\rho_2\rho_1\Psichainu{k-2}{y_2}\dots\rho_{h-1}\Psichainu{c+4}{m_{h-1}}\Psichainu{c}{m_h}\vtp\\
&\,\;\vdots\\
 &=\psi_k\Psichainu{k+2}{y_1}\rho_1\Psichainu{k-2}{y_2}\dots\rho_{h-2}\Psichainu{c+8}{m_{h-1}}\rho_{h-1}\Psichainu{c+4}{m_h}\vtp\\
 &=\psi_k\Psichaind{c-2}{m_1}\chi_1(\Psichaind{c}{m_2}\chi_2\dots\chi_{h-2}\Psichaind{k-6}{m_{h-1}})\Psichaind{k-2}{m_h}\vtp.
\end{align*}

 \item[b)] $\bi_{\tp_c}=(1,0)$, $\bi_{\tp_k}=(0,1)$ and $\tp_{k+2}$ is at the end of the arm:
 
 As $\bi_{\tp_c}=(1,0)$ and $\bi_{\tp_{k}}=(0,1)$, we have that $\rho_{0}=(\psi_{k-3})\psi_{k-2}\psi_{k-1}\psi_k,\rho_h=\id,\rho_d=\psi_{k+1}\psi_k\dots\psi_a$ and $m_1=k-2$.
 Observe that $\rho_0$ is the concatenation of $\chi_{h-1}$ and $\chi_h$.

\begin{align*}
\psi_k\vtt&=\psi_k\chi_1\Psichaind{c}{m_1}(\chi_2\Psichaind{c+4}{m_2}\dots\chi_{h-2}\Psichaind{k-8}{m_{h-2}}\chi_{h-1}\Psichaind{k-4}{m_{h-1}})\chi_h\Psichaind{k-2}{m_h}\chi_dz^\la\\
&=\psi_k\rho_0\rho_1\Psichainu{k-4}{k-2}\rho_2\Psichainu{k-8}{k-4}\dots\rho_{h-1}\Psichainu{c+4}{m_{h-1}}\rho_{h}\Psichainu{c}{m_h}\dots\rho_dz^\la\\
&=(\psi_{k-3})\psi_{k-2}\psi_k\psi_{k-1}\psi_k\rho_1\Psichainu{k-4}{k-2}\rho_2\Psichainu{k-8}{k-4}\dots\rho_{h-1}\Psichainu{c+4}{m_{h-1}}\rho_{h}\Psichainu{c}{m_h}\dots\rho_dz^\la\\
&\,\;\vdots\\
&=(\psi_{k-3})\psi_{k-2}\rho_1\Psichainu{k-4}{k-2}\rho_2\Psichainu{k-8}{k-4}\dots\rho_{h-1}\Psichainu{c+4}{m_{h-1}}\rho_{h}\Psichainu{c}{m_h}\underbrace{\dots\widehat{\rho}_dz^\la}_{=\vtp}\\
&\,\;\vdots\\
&=(\psi_{k-3})\psi_{k-2}\rho_1\Psi_{k-4}
\rho_2\Psichainu{k-8}{k-6}\dots\rho_{h-2}\Psichainu{c+8}{m_{h-1}}\rho_{h-1}\Psichainu{c+4}{m_h}\vtp\\
&=\Psichaind{c-2}{m_1}\chi_1(\Psichaind{c}{m_2}\chi_2\dots\chi_{h-2}\Psichaind{k-8}{m_{h-1}})\chi_{h-1}\Psichaind{k-4}{m_h}\widehat{\rho}_dz^\la.
\end{align*}
where $\widehat{\rho}_d=\psi_k\psi_{k-1}\dots\psi_a$.

Details for the rest of the cases can be included by switching the toggle in the {\tt arXiv} version of this paper.
\begin{answer}
 \item[c)] $\bi_{\tp_c}=(1,0)$, $\bi_{\tp_k}=(0,1)$ and $\tp_{k+2}$ is not at the end of the arm:
 
 Same as the the previous case, except that here $\rho_{0}=(\psi_{k-3})\psi_{k-2}\psi_{k-1}\psi_k\psi_{k+1}\psi_{k+2}$ and $m_1\ge k$.

\begin{align*}
\psi_k\vtt&=\psi_k\chi_1\Psichaind{c}{m_1}(\chi_2\Psichaind{c+4}{m_2}\dots\chi_{h-2}\Psichaind{k-4}{m_{h-2}})\chi_{h-1}\Psichaind{k-2}{m_{h-1}}\chi_h\Psichaind{k}{m_h}\vtp\\
&=\psi_k\rho_0\rho_1\Psichainu{k-4}{m_1}\rho_2\Psichainu{k-8}{m_2}\dots\rho_{h}\Psichainu{c}{m_h}\vtp\\
&\,\;\vdots\\
&=(\psi_{k-3})\psi_{k-2}\psi_{k+2}\Psichainu{k}{m_1}\rho_1\Psichainu{k-4}{m_2}\dots\rho_{h-1}\Psichainu{c-2}{m_h}\vtp\\
&=\Psichaind{c+4}{m_1}\chi_1(\Psichaind{c}{m_2}\chi_2\Psichaind{c-2}{m_3}\dots\chi_{h-2}\Psichaind{k-4}{m_{h-1}})(\psi_x\dots)\psi_{k+2}\Psichaind{k}{m_h}\vtp
\end{align*}
 
  \item[d)] $\bi_{\tp_c}=(0,1)$, $\bi_{\tp_k}=(1,0)$:
  
We have $\rho_{h-1}\neq\id\neq\rho_h$.

   \begin{align*}
\psi_k\vtt&=\psi_k\chi_1\Psichaind{y_1}{m_1}(\chi_2\Psichaind{c+2}{m_2}\dots\chi_{h-2}\Psichaind{k-6}{m_{h-2}})\chi_{h-1}\Psichaind{k-2}{m_{h-1}}\chi_h\Psichaind{k}{m_h}\vtp\\
&=\psi_k\rho_1\Psichainu{k-2}{y_1}\rho_2\Psichainu{k-6}{y_2}\dots\rho_{h-1}\Psichainu{c+2}{m_{h-1}}\rho_h\Psichainu{c}{m_h}\vtp\\
&\,\;\vdots\\
&=\psi_k\rho_1\Psichainu{k-2}{y_1}\rho_2\Psichainu{k-6}{y_2}\dots\rho_{h-2}\Psichainu{c+6}{m_{h-1}}\underbrace{\rho_{h-1}\rho_h\Psi_c}_{\text{\cref{halfdom-arm}}}\Psichainu{c+2}{m_h}\vtp\\
&=\psi_k\rho_1\Psichainu{k-2}{y_1}\rho_2\Psichainu{k-6}{y_2}\dots\rho_{h-2}\Psichainu{c+6}{m_{h-1}}\psi_{c-2}\psi_{c-1}\psi_c\Psichainu{c+2}{m_h}\vtp\\
&=\psi_{c-2}\Psichaind{c-4}{m_1}\psi_{c-1}\psi_c\Psichaind{c-2}{m_2}(\chi_2\Psichaind{c+2}{m_3}\dots\chi_{h-2}\Psichaind{k-6}{m_{h-1}})\psi_k\Psichaind{k-2}{m_h}\vtp
\end{align*}

  \item[e)] $\bi_{\tp_c}=(0,1)$, $\bi_{\tp_k}=(0,1)$:

  As shown in cases b) and c), whether $\tp_{k+2}$ is at the end of the arm does not modify the computations significantly, so we will assume that $\tp_{k+2}$ appears somewhere along the arm.
  Recall that if $k+2=\ttt(1,a)$ with $a$ even, then
\[
\vtt=\chi_1\Psichaind{y_1}{m_1}(\chi_2\Psichaind{c+2}{m_2}\chi_3\dots\chi_{h-2}\Psichaind{k-8}{m_{h-2}})\chi_{h-1}\Psichaind{k-4}{m_{h-1}}\chi_h\Psichaind{k-2}{m_h}\chi_dz^\la.
\]
Otherwise,
\[
\vtt=\chi_1\Psichaind{y_1}{m_1}(\chi_2\Psichaind{c+2}{m_2}\chi_3\dots\chi_{h-2}\Psichaind{k-4}{m_{h-2}})\chi_{h-1}\Psichaind{k-2}{m_{h-1}}\chi_h\Psichaind{k}{m_h}\vtp.
\]

Then it is easy to see that
  
\begin{align*}
\psi_k\vtt&=\psi_k\rho_0\rho_1\Psichainu{k-2}{y_1}\rho_2\Psichainu{k-6}{y_2}\dots\rho_{h-2}\Psichainu{c+6}{m_{h-2}}\rho_{h-1}\Psichainu{c+2}{m_{h-1}}\rho_h\Psichainu{c}{m_h}\vtp\\
&=\psi_k\psi^*\rho_1\Psichainu{k-2}{y_1}\rho_2\Psichainu{k-6}{y_2}\dots\rho_{h-2}\Psichainu{c+6}{m_{h-1}}\psi_{c-2}\psi_{c-1}\psi_c\Psichainu{c+2}{m_h}\vtp\\
&=\psi_{c-2}\Psichaind{c-4}{m_1}\psi_{c-1}\psi_c\Psichaind{c-2}{m_2}(\chi_2\Psichaind{c+2}{m_3}\dots\chi_{h-2}\Psichaind{k-6}{m_{h-1}})\psi_k\Psichaind{k-2}{m_h}\vtp
\end{align*}
where depending on the position of $k+2$ in $\Arm(\ttt)$ either $\rho_0$ or $\rho_d$ got shorter by a $\psi_{k+1}$.
\end{answer}
    \end{itemize}
\end{eg}

\begin{lem}[cf.\cite{thesis}, Lemma~3.23]\label{4killing-2}
Let $\ttt\in\std(\la)$ with $\tp_k,\tp_{k+2}\in\Arm(\ttt)$.
\begin{itemize}
 \item[a)] If $\bi_{\tp_k}=(1,0)$, let $\vtt$ have reverse normal form
\[
\vtt=\psi^*\rho_1\Psichainu{y_1}{m_1}\rho_2\Psichainu{y_2}{m_2}\dots\rho{h}\Psichainu{y_h}{m_h}(\rho_d)z^\la,
\]
where $m_h=a(-1)$. Suppose $3\le k \le n-1$ is odd and that $y_1$ is minimal such that $y_1 \le k-2$ and $m_1\ge k$.
    Then $\psi_k\vtt=0$ if and only if for all $i=1,2,\dots,h$ either $y_{i}\le k-4i$ or $y_i = k-4i+2$ and $\rho_i\neq \id$.
    
\item[b)] If $\bi_{\tp_k}=(0,1)$ and $\tp_{k+2}$ is not at the end of $\Arm(\ttt)$, let $\vtt$ have reverse normal form
\[
\vtt=\psi^*\rho_0\Psichainu{y_0}{m_0}\rho_1\Psichainu{y_1}{m_1}\rho_2\Psichainu{y_2}{m_2}\dots\rho_{h}\Psichainu{y_h}{m_h}(\rho_d)z^\la,
\]
where $m_h=a(-1)$. Suppose $3\le k \le n-1$ is odd and that $y_1$ is minimal such that $y_1 \le k-4$ and $m_1\ge k$.
 Then $\psi_k\vtt=0$ if and only if for all $i=1,2,\dots,h$ either $y_{i}\le k-4i-2$ or $y_i = k-4i$ and $\rho_i\neq \id$.

 \item[c)] If $\bi_{\tp_k}=(0,1)$ and $\tp_{k+2}$ is at the end of $\Arm(\ttt)$, let $\vtt$ have reverse normal form
\[
\vtt=\psi^*\rho_0\Psichainu{y_0}{m_0}\rho_1\Psichainu{y_1}{m_1}\rho_2\Psichainu{y_2}{m_2}\dots\rho_{h}\Psichainu{y_h}{m_h}\rho_dz^\la,
\]
where $m_h=a-1$.
Suppose $3\le k \le n-1$ is odd and that $y_1$ is minimal such that $y_1 \le k-4$ and $m_1=k-2$.
 Then $\psi_k\vtt=0$ if and only if for all $i=1,2,\dots,h$ either $y_{i}\le k-4i-2$ or $y_i = k-4i$ and $\rho_i\neq \id$.
     \end{itemize}
\end{lem}

\begin{rem}
    In \cref{4killing-1}, when $a$ is even, we had to consider some additional restrictions on $\ell$.
    However, the above lemma already states that $\tp_k,\tp_{k+2}\in\Arm(\ttt)$, hence if $\tp_{k+2}$ is at the end of the arm, this already forces $\ell=k+2$.
\end{rem}

\begin{proof}
    After modifying the appropriate indices and notations, the proof is almost identical to the proof of \cref{4killing-1}.
\end{proof}

\begin{rem}
    Note that \cref{3pairs} part 2 corresponds to the case where $t=c$ and for all $1\le i < t$ we have that $y_i=k-4i+2$ and $\chi_i\neq\id$ (or $y_i=k-4i$ and $\chi_i\neq\id$ if $\bi_{\tp_k}=(0,1)$).
    Thus, \cref{4killing-2} also shows the converse of \cref{3pairs} part 2, that is if there exists no $\tp_c$ as given in the statement then $\psi_k\vtt=0$.
\end{rem}

\begin{defn}\label{pair-arm}
Suppose that
\begin{center}
    \begin{tikzpicture}
    \scalefont{0.8}
    \Yboxdimx{28pt}
    \Yboxdimy{20pt}
    \tgyoung(0cm,3cm,1^1\hdts <x{-}1>x^1\hdts <k{-}1>k<k{+}1><k{+}2>^1\hdts\ell,'11\vdts,<p_m{-}1>,<p_m>,'11\vdts,<p_1{-}1>,<p_1>,'11\vdts,v,'11\vdts)
    \node at (-0.8,3.4) (a) {{\normalsize{$\ttt=$}}};
    \end{tikzpicture}
\end{center}
where $v$ is maximal such that $v<k-1$, the section of $\Leg(\ttt)$ $p_m-1,p_m,\dots,p_1-1,\dots,v$ contains $m$ pairs and the section of the arm $x-1,x, \dots,k+2$ contains $m+1$ pairs (observe that $p_m=k+2-2r$).
Note that if $v-1\in\Leg(\ttt)$ then we have that $v=p_1$.

In this case, we define $(\ttt)_{k+2-2r,k+2}$ to be the standard $\la$-tableau where $\tp_{k+2}$ is moved into $\Leg(\ttt), \tp_{p_m}$ is moved into $\Arm(\ttt)$ and all other entries of $\Leg((\ttt)_{k+2-2r,k+2})$ and $\Arm((\ttt)_{k+2-2r,k+2})$ agree with those of $\ttt$.
This can be easily visualised as a `clockwise' rotation of the entries $p_m-1,p_m,\dots,k+2$.
Thus,
\begin{center}
    \begin{tikzpicture}
    \scalefont{0.8}
    \Yboxdimx{28pt}
    \Yboxdimy{20pt}
    \tgyoung(0cm,3cm,1^1\hdts <p_m{-}1><p_m>^1\hdts<x{-}1>x^1\hdts <k{-}1>k^1\hdts\ell,'11\vdts,<p_1{-}1>,<p_1>,'11\vdts,v,'11\vdts,<k{+}1>,<k{+}2>,'11\vdts)
    \node at (-1,3.4) (a) {{\normalsize{$(\ttt)_{m,k}=$}}};
    \end{tikzpicture}
\end{center}
where the omitted parts are identical to those in $\ttt$.
 \end{defn}
 
  \begin{eg}
  Consider $\la=(9, 1^{9})$ and take $k=15$. If $r=6$ and $\ttt$ is as below, we have that:
\begin{center}
    \begin{tikzpicture}
    \Yboxdim{15pt}
\tgyoung(0cm,0cm,18\ten\eleven\twelve\fourteen\sixteen\seventeen,2,3,4,5,6,7,9,\thirteen,<18>)
{\Ylinethick{2pt}
\tgyoung(0cm,0cm,::;\ten;\eleven:;\fourteen;\fifteen;\sixteen;\seventeen,:,:,;4,;5,;6,;7)}
\tgyoung(8.7cm,0cm,1458\ten\eleven\twelve\fourteen\fifteen,2,3,6,7,9,\thirteen,\sixteen,\seventeen,<18>).
{\Ylinethick{2pt}
\tgyoung(8.7cm,0cm,:;4;5:;\ten;\eleven:;\fourteen;\fifteen,:,:,;6,;7,:,:,;\sixteen,;\seventeen).}
    \node at (-0.6,0.3) (a) {{\normalsize{$\ttt=$}}};
    \node at (7.7,0.3) (a) {{\normalsize{$(\ttt)_{5,17}=$}}};
    \end{tikzpicture}
\end{center}
Observe that in this case
\[
\vtt=\psi_9\Psichaind{7}{3}\psi_{10}\psi_{11}\Psichaind{9}{5}\Psichaind{13}{7}\Psichaind{15}{9}
=\psi_{10}\psi_{11}\Psichainu{13}{15}\psi_9\Psichainu{7}{13}\Psichainu{5}{11}\Psichainu{3}{9}
\]
and
\[
v_{(\ttt)_{5,17}}=\Psi_3\psi_9\Psichaind{7}{5}\psi_{10}\psi_{11}\Psichaind{9}{7}\Psichaind{13}{9}=
\psi_{10}\psi_{11}\Psi_{13}\psi_9\Psichainu{7}{11}\Psichainu{3}{9}.
\]
\end{eg}

\begin{cor}\label{cor:psi-arm}
Assume $\ttt\in\std(\la)$ with $\tp_k,\tp_{k+2}\in\Arm(\ttt)$.
\begin{itemize}
 \item[a)] If $\bi_{\tp_k}=(1,0)$, let $\vtt$ have reverse normal form
\[
\vtt=\psi^*\rho_1\Psichainu{y_1}{m_1}\rho_2\Psichainu{y_2}{m_2}\dots\rho{h}\Psichainu{y_h}{m_h}(\rho_d)z^\la.
\]
Suppose that $\psi_k\vtt\neq0$ and that $t$ is minimal such that $y_t > k-4t$ or $y_t=k-4t+2$ and $\rho_t=\id$.
    
\item[b)] If $\bi_{\tp_k}=(0,1)$, let $\vtt$ have reverse normal form
\[
\vtt=\psi^*\rho_0\Psichainu{y_0}{m_0}\rho_1\Psichainu{y_1}{m_1}\rho_2\Psichainu{y_2}{m_2}\dots\rho_{h}\Psichainu{y_h}{m_h}(\rho_d)z^\la.
\]
Suppose that $\psi_k\vtt\neq0$ and that $t$ is minimal such that $y_t > k-4t-2$ or $y_t=k-4t$ and $\rho_t=\id$.
     \end{itemize}
Then there exists some minimal $r$ such that $\ttt$ is as in \cref{pair-arm} and $\psi_k\vtt=\psi_kv_{(\ttt)_{k+2-2r,k+2}}$ which is in reduced form.
\end{cor}

\begin{proof}
   Just as \cref{cor:psi-leg}, the result follows immediately from \cref{4killing-2}.
\end{proof}

If we have a tableau $\ttt$ with $\tp_k,\tp_{k+2}$ in the arm or in the leg, we can easily calculate whether $\psi_k$ rotates the dominoes as given in \cref{pair-arm,pair-leg} or if $\psi_k\vtt=0$.
Assume that $\tp_k,\tp_{k+2}\in\Leg(\ttt)$ and that we wish to determine the action of $\psi_k$ on $\vtt$.
We start counting the pairs in the natural order (starting with $\tp_k$).
If we have counted $m+1$ pairs in the leg and $m$ in the arm, we stop and perform the rotation as given in \cref{cor:psi-leg}.
If we reach the highest indexed pair without ever satisfying this condition, $\psi_k\vtt=0$.

Similarly, assume that $\tp_k,\tp_{k+2}\in\Arm(\ttt)$ and consider the expression $\psi_k\vtt$.
We start counting the pairs backwards from $\tp_{k+2}$.
If at any point we have counted $m+1$ pairs in the arm and $m$ in the leg, we stop and perform the rotation as given in \cref{cor:psi-arm}.
If we arrive at the lowest indexed pair without ever satisfying this condition, once again we see that $\psi_k\vtt=0$.
Note that the highest (resp.\ lowest) indexed pair is not necessarily $\tp_n$ (resp.\ $\tp_3$).

\begin{prop}
    The algorithms above give the correct reduced form for $\psi_k\vtt$.
\end{prop}

\begin{proof}
    We start with the first algorithm.
    The proof itself is just a translation between the algorithm and \cref{cor:psi-leg}.
    Assume that $y_1$ is minimal such that $m_1\le k\le y_1$ , that is the first pair the algorithm counts in $\Arm(\ttt)$ is $\tp_{y_1}$.

    When we get to the $r$th pair in $\Arm(\ttt)$, we must have counted at least $r+2$ in $\Leg(\ttt)$, otherwise the algorithm would terminate.
    Thus we see that we counted at least $2r+2$ pairs in total, starting with $\tp_k$ and so $y_r\ge 4r+k$ or $y_r=4r+k-2$ and $\chi_r\neq \id$.

    If the algorithm gets to the final pair without ever satisfying the condition that we have counted one more pair in $\Leg(\ttt)$ than in $\Arm(\ttt)$, it is easy to see that the above inequality (resp.\ equality) holds for all $r$, so we know by \cref{4killing-1} that $\psi_k\vtt=0$.
    Otherwise, if we counted $f$ many pairs in $\Arm(\ttt)$ and $f+1$ in $\Leg(\ttt)$, then by \cref{cor:psi-leg}, we have the desired result.
    
    The case when $y_1$ is minimal such that $m_1\le k+2\le y_1$ and the second algorithm can be proved analogously.
\end{proof}

At last, we have completely determined the action of the different generators on the basis elements $\vtt$ of $S^\la$.
We summarise our results in the next statement.

\begin{prop}\label{action}
Let $\la=(a,1^b)$, $\ttt\in\std(\la)$ with corresponding $\vtt$.
\begin{enumerate}
    \item $e(\bj) \vtt = \delta_{\bi^\ttt, \bj} \vtt$
    \item \[ y_k \vtt = 
    \begin{cases}
            -v_{ s_k \ttt }& \text{ if $k$ is odd, $i_k = i_{k+1}$, $k \in \Leg(\ttt)$, and $k+1 \in \Arm(\ttt)$,}\\[.5em]
              v_{ s_{k-1} \ttt }& \text{ if $k$ is odd, $i_{k-1} = i_k$, $k-1 \in \Leg(\ttt)$, and $k \in \Arm(\ttt)$,}\\[.5em]
             0& \text{ otherwise.}
              \end{cases}
              \]
    \item For $k$ even:
    \[\psi_k\vtt=
    \begin{cases} v_{s_k\ttt}& \text{ if $k \in \Arm(\ttt)$, and $k+1 \in \Leg(\ttt)$},\\[.5em]
    0& \text{ otherwise.}\\
    \end{cases}\]
     \item For $k$ odd:
     \[\psi_k\vtt=
    \begin{cases}
    v_{s_k\ttt}& \text{ if $1<k \in \Arm(\ttt)$, and $k+1 \in \Leg(\ttt)$},\\[.5em]
    -2v_{s_{k-1}s_{k+1}s_k\ttt}& \text{ if $\tp_k\in \Leg(\ttt),\tp_{k+2}\in\Arm(\ttt)$ and $\bi_{\tp_k}=\bi_{\tp_{k+2}}$},\\[.5em]
    \psi_k\vtt=\vs& \text{ if $k-1\in\Leg(\ttt)$, $k,k+1,k+2\in\Arm(\ttt)$,}\\ 
    & \text{ $i_{k-1}=i_k$ and $\tts$ is as in \cref{halfdom-arm}},\\[.5em]
    \psi_k\vtt=v_\ttp& \text{ if $k-1,k,k+1\in\Leg(\ttt), k+2\in\Arm(\ttt)$,}\\
    & \text{ $i_{k+1}=i_{k+2}$ and $\ttp$ is as in \cref{halfdom-leg}},\\[.5em]
    \psi_r\vtt=\psi_rv_{(\ttt)^{2r,k}}& \text{ if $\tp_k,\tp_{k+2}\in\Leg(\ttt)$ and $\ttt$ satisfies}\\
    & \text{ the conditions of \cref{pair-leg},}\\[.5em]
    \psi_r\vtt=\psi_rv_{(\ttt)_{2r,k}}& \text{ if $\tp_k,\tp_{k+2}\in\Arm(\ttt)$ and $\ttt$ satisfies}\\
    & \text{ the conditions of \cref{pair-arm},}\\[.5em]
    0& \text{ otherwise.}
    \end{cases}
    \]
\end{enumerate}
\end{prop}

%% file: parts/homs-1.tex
As before, let $\la = (a,1^b)$ and fix $\mu = (\mu_1, \dots, \mu_Z)$ a partition of $n$.
If there exists a nonzero homomorphism $\varphi:S^\mu \rightarrow S^\la$ and $\varphi(z^\mu)= v=\sum c_\ttt\vtt$, then $c_\ttt\neq0$ if and only if $\bi^\mu=\bi^\ttt$.
We call a node $(a,b) \in \mu$ an {\em arm node} (resp.\ {\em leg node}) with respect to $\ttt$ if $c_\ttt\neq0$ and $\ttt^\mu(i,j)$ is in $\Arm(\ttt)$ (resp.\ in $\Leg(\ttt)$).

\subsection{A novel way to label hook partitions}

In this section, we prove some results that will be used in the proof of our main theorem. 
In particular, we consider the action of the $y$ and $\psi$ generators on the basis elements $\vtt\in S^\la$.
In order to do this, we introduce an usual, albeit extremely important set of partitions and a rule to assign a coefficient to each.
Then, we show that the sum of the coefficients of the partitions in this set always equals to a certain constant (\cref{prop:2c}).
Next, we introduce a convenient way to label tableaux of hook shapes whose arms contain only pairs and unpaired even numbers.
The rest of the section is dedicated to showing that if we label tableaux in this way, then the resulting labels and the set of partitions from earlier coincide and all $\vtt$ involved in $v$ is labelled by a tableaux of this form.

\begin{lem}\label{y-action}
If $v=\sum c_\ttt\vtt$, then $(J_1^\mu + J_2^\mu)v = 0$ if and only if for any $\ttt \in \std(\la)$ with $\bi^\mu=\bi^\ttt$, $c_\ttt \neq 0$ and any $1 \le a < k$ for which $\mu_a$  is even, $(a, \mu_a)$ is an arm node and $(a+1, 1)$ is a leg node.
\end{lem}

\begin{proof}
Let $v = \sum c_{\ttt} \vtt$ and assume that $(J_1^\mu + J_2^\mu)v = 0$, then we have $y_r v =0$ for all $1 \le r \le n$. 
By~\cref{standard-y}, $y_r \vtt = 0$ unless $r$ and $r+1$ or $r$ and $r-1$ have the same residue. 
We only show the case where $i_r=i_{r+1}$, the other has a very similar proof.
We can only have $i_r=i_{r+1}$ if $r$ and $r+1$ are in different rows in $\ttt^\mu$, with $r$ at the end of a row and $r+1$ in the first node of the next row.
Now, this is only possible if the row containing $r$ has even length, otherwise $r$ and $r+1$ have different residues and $y_r \vtt = 0$ for all such $r$.
If $i_r = i_{r+1}$ with $r\in \Leg(\ttt)$ and $r+1 \in \Arm(\ttt)$ for some $\ttt$, then $y_r\vtt=-v_{s_r \ttt}$.
Then $y_r v \neq 0$, unless there is some $\ttt'$ involved in $v$ that gives $y_r\vtp=v_{s_r \ttt}$ with $i_r=i_{r+1}$.
But this is impossible, hence if $J_2^\mu v =0$, we must have that whenever $i_r=i_{r+1}$ then $r$ is in $\Arm(\ttt)$ or $r+1$ is in $\Leg(\ttt)$. Furthermore, as $i_r=i_{r+1}$ whenever $r$ is in $\Arm(\ttt)$, we must have that $r+1$ is in $\Leg(\ttt)$, because they cannot be next to each other with the same residue.

On the other hand, suppose that $r=\ttt^\mu(a,\mu_a)$ is an arm node and $r+1=\ttt^\mu(a+1, 1)$ is a leg node for every even length row $\mu_a$ in $\ttt^\mu$ with $\bi^\ttt=\bi^\mu=(i_1, \dots, i_n)$ for all $\ttt\in\std(\la)$ with $c_\ttt\neq0$.
Fix $\ttt$ such that $c_\ttt\neq0$ for $1\le r \le n-1$.
If $r-1,r,r+1$ are in different rows in $\ttt^\mu$, then $i_r\neq i_{r-1},i_{r+1}$, so by \cref{standard-y}, $y_r\vtt=0$ 
We immediately see that $y_1\vtt=0$ as $i_1\neq i_2$.
If $r$ is at the end of an even length row in $\ttt^\mu$, then by our assumption $r$ is in $\Arm(\ttt)$, hence $y_r\vtt=0=y_{r+1}\vtt$.
Finally, if $n-1$ and $n$ are in the same row in $\ttt^\mu$, then $i_{n-1}\neq i_n$, so $y_n\vtt=0$. If $n-1=\ttt^\mu(a-1,\mu_{a-1})$ and $\mu_{a-1}$ is odd, then $i_{n-1}\neq i_n$ and if $\mu_{a-1}$, by our assumption $n-1\in\Arm(\ttt)$ and $n\in\Leg(\ttt)$.
Hence $y_r \vtt =  0=y_{r+1} \vtt$ for all $\ttt$, thus it follows that $y_r v =0$ for all $1 \le r \le n$.
\end{proof}

\begin{lem}\label{arm-odd}
If $v=\sum c_\ttt\vtt$ such that $(J_1^\mu + J_2^\mu)v = 0$, $c_\ttt\neq0$ and $r,r+1,r+2\in\Arm(\ttt)$ for $r$ odd, then we must also have that $r-1\in\Arm(\ttt)$.
\end{lem}

\begin{proof}
Assume otherwise, i.e.\ $r-1\in\Leg(\ttt)$ and $r,r+1,r+2\in\Arm(\ttt)$. Due to the parity conditions, we must have that $i_{r-1}=i_r$. Then by the previous lemma, we see that in this case we must have $r-1\in\Arm(\ttt)$ and $r\in\Leg(\ttt)$, a contradiction. Hence, if $c_\ttt\neq0$ and $r,r+1,r+2\in\Arm(\ttt)$, we have that $r-1$ is also in $\Arm(\ttt)$.
\end{proof}

\begin{lem}\label{leg-odd}
If $v=\sum c_\ttt\vtt$ such that $(J_1^\mu + J_2^\mu)v = 0$, $\ct\neq0$ and $r-1,r,r+1\in\Leg(\ttt)$ for $r$ odd, then $r+2\in\Leg(\ttt)$.
\end{lem}

\begin{proof}
Analogous to the proof of \cref{arm-odd}.
\end{proof}

Recall the notion of a pair from \cref{pairs}.

\begin{defn}
For $\la=(a,1^b)$, we define the set of \emph{pair-tableaux} of shape $\la$, denoted by $\P(\la)$, to be the set of tableaux in which if $k\in\Leg(\ttt)$ for some even $2\le k \le n-1$, then we must have that $k+1\in\Leg(\ttt)$ and if $k\in\Arm(\ttt)$ for some odd $3\le k \le n$, then we must have that $k-1\in\Arm(\ttt)$.
\end{defn}

The following is an immediate consequence of the previous three lemmas.

\begin{cor}\label{cor:pair}
    Suppose $\varphi\in\hom(S^\mu,S^\la)$ and $\varphi(z^\mu)=\sum\ct\vtt$.
    If $\ct\neq 0$, then $\ttt\in\P(\la)$.
\end{cor}

\begin{lem}\label{y-pair}
    If $\ttt\in\P(\la)$, then
    \begin{itemize}
        \item[i)] $y_k\vtt=0$ for all $1\le k \le n$,
        \item[ii)] $\psi_k\vtt=0$ for all even $2\le k\le n-1$ such that $k$ and $k+1$ are adjacent in $\ttt$, and
        \item[iii)] $\psi_k\vtt=v_{s_k\ttt}$ for all even $2\le k\le n-1$ such that $k\in\Arm(\ttt)$ and $k+1\in\Leg(\ttt)$.
    \end{itemize} 
\end{lem}

\begin{proof}
    Follows immediately from the definition of $\P(\ttt)$ and \cref{action}.
\end{proof}

\begin{defn}\label{def:coeff}
Let $(\de_1^{h_1},\de_2^{h_2},\dots,\de_d^{h_d})$ be a weakly decreasing sequence of non-negative integers such that each $\de_i$ appears with multiplicity $h_i$.
We define the \emph{delta (coefficient) function} 
\[
  \De(\de_1^{h_1},\dots,\de_d^{h_d}):=\prod_1^d \fkd_x(\de_1^{h_1},\dots,\de_d^{h_d}).
\]
where
\[
\fkd_x(\de_1^{h_1},\dots,\de_d^{h_d})=\binom{\de_x-\caly+\calx}{\calx},
\]
and $\caly$ and $\calx$ are determined by the following two rules:
\begin{enumerate}
    \item If $\de_{x-1}-\de_x\ge h_x$, we set $\calx:=h_x$ (if $x=1$, we define $\calx:=h_1$).
    Otherwise, we find the maximal $1\le y<x-1$ such that $\de_y-\de_x \ge\sum_{y+1}^x h_\rho$ and set $\calx:=\sum_{y+1}^x h_\rho$.
    If no such $y$ exists, we set $\calx:=\sum_{1}^x h_\rho$.
    \item If $\de_{y}-\de_z\ge \sum_{y+1}^zh_\rho$ for all $1\le y\le x$ and $x+1\le z\le d$ or $x=d$, we set $\caly:=0$.
    Otherwise, we have to two cases to consider.
    \begin{itemize}
        \item[a)] If there exists some minimal $x+1\le z\le d$ such that $\de_x-\de_z<\sum_{x+1}^zh_\rho$, we set $\caly:=\de_z$.
        \item[b)] Otherwise, we find the maximal $1\le y\le x-1$ and the minimal $x+1\le z\le d$ such that $\de_{y}-\de_z< \sum_{y+1}^zh_\rho$ and set $\caly:=\de_z$.
    \end{itemize}
\end{enumerate}
\end{defn}

\begin{rem}
    It is easy to show that if 2a) fails, then the minimal $z$ and maximal $y$ satisfying 2b) are unique.
\end{rem}

\begin{eg}
    For example, if we have the sequence $(13,10^3,8^2,6^3,3^2,2)$, then
    \begin{align*}
        \De(13,10^3,8^2,6^3,3^2,2)&=\prod_{x=1}^6\fkd_x(13,10^3,8^2,6^3,3^2,2)\\[.5em]
        &=\underbrace{\binom{13-6+1}{1}}_{\fkd_1}\underbrace{\binom{10-6+3}{3}}_{\fkd_2}\underbrace{\binom{8-6+2}{2}}_{\fkd_3}\\[.5em]
        &\cdot\underbrace{\binom{6+9}{9}}_{\fkd_4}\underbrace{\binom{3+2}{2}}_{\fkd_6}\underbrace{\binom{2+1}{1}}_{\fkd_6}.
    \end{align*}
\end{eg}

In order to state the next definition, we need some additional notation.
For some $i\ge0$ and $v,s\in\bbz_{>0}$ and $g_v,G_s\in\bbz_{\ge0},h_v,H_s\in\bbz_{>0}$, we set
\begin{align*}
    \iBoxed{s}&=\underbrace{(i+G_s,i+G_s,\dots,i+G_s}_{H_s}),\\[.5em]
    \Boxedi{v}&=\underbrace{(i-g_v,i-g_v,\dots,i-g_v}_{h_v}).
\end{align*}

As mentioned at the beginning of the section, we now introduce a set of partitions which will play a crucial rule when giving an explicit description of $\Hom(S^\mu,S^\la)$.

\begin{defn}\label{def:tau}
    Fix $g_v,G_s\in\bbz{\ge0}$ such that $g_v<g_{v+1}$ and $G_s<G_{s+1}$ for all $1\le v\le \omega$ and $1\le s\le \al$.
    Let
    \begin{align*}
    \tau_1&=\big(\iBoxed{\al},\dots,\iBoxed{2},\iBoxed{1} ,\underbrace{i}_{j\text{th position}}, \\[.5em]
    &\hspace{2cm}\Boxedi{1},\Boxedi{2}, \dots   ,\Boxedi{\omega} \big)
    \end{align*}
    or 
    \begin{align*}
    \tau_2&=\big(\iBoxed{\al},\dots,\iBoxed{2},\iBoxed{1} , i,\underbrace{i}_{j\text{th position}}, \\[.5em]
    &\hspace{2cm}\Boxedi{1},\Boxedi{2}, \dots   ,\Boxedi{\omega} \big)
    \end{align*}
    be a partition with $R=1+\sum_{\rho=1}^\al H_\rho+\sum_{\rho=1}^\omega h_\rho$ or $R=2+\sum_{\rho=1}^\al H_\rho+\sum_{\rho=1}^\omega h_\rho$ parts and fix an integer $1\le j\le R$.
    For $\tj{j}{0}:=\tau_1$ or $\tau_2$, we define the set of partitions $\tij$ as follows.
    If $\tau_j=\tau_{j+1}$, we set $\tij:=\varnothing$.
    Otherwise, if $\tj{j}{0}=\tau_1$, define $\tij':=\{\tj{j}{0},\tj{j}{-1},\tj{j}{1},\tj{j}{2},\dots,\tj{j}{\omega+1}\}$, where
    \begin{align*}
    \tj{j}{-1}&=\big(\iBoxed{\al},\dots,\iBoxed{1}, \underbrace{i-1}_{j\text{th position}},  \Boxedi{1}, \dots   ,\Boxedi{\omega} \big),\\[1em] 
     \tj{j}{1}&=\big(\iBoxed{\al},\dots,\iBoxed{1}, \underbrace{i-2}_{j\text{th position}},  \Boxedi{1}, \dots   ,\Boxedi{\omega} \big),\\[1em] 
     \tj{j}{v+1}&=\big(\iBoxed{\al},\dots,\iBoxed{1},  \Boxedi[+1]{1} ,\Boxedi[+1]{2},\dots,\\
      &\hspace{2cm} \Boxedi[+1]{v-1},\boxed{\quad\cali_{g_{v}+1,h_v+1}\quad},\Boxedi{v+1},\dots,\Boxedi{\omega} \big).
\end{align*}
    (in $\tj{j}{v+1}$ the leftmost $i-(g_1+1)$ is in position $j$) and if $\tj{j}{0}=\tau_1$, define $\tij':=\{\tj{j}{0},\newline \tj{j}{-1},\tj{j}{1},\tj{j}{2}, \dots,\tj{j}{\omega+1}, \tu{j}{1},\tu{j}{2},\dots,\tu{j}{\al+1}\}$, where
    \begin{align*}
    \tj{j}{-1}&=\big(\iBoxed{\al},\dots,\iBoxed{1}, \underbrace{i}_{j-1\text{th position}},i-1,  \Boxedi{1}, \dots   ,\Boxedi{\omega} \big),\\[1em] 
     \tj{j}{1}&=\big(\iBoxed{\al},\dots,\iBoxed{1}, \underbrace{i}_{j-1\text{th position}},i-2,  \Boxedi{1}, \dots   ,\Boxedi{\omega} \big),\\[1em] 
     \tj{j}{v+1}&=\big(\iBoxed{\al},\dots,\iBoxed{1} , \underbrace{i}_{j-1\text{th position}},  \Boxedi[+1]{1} ,\Boxedi[+1]{2},\dots,\\
      &\hspace{2cm} \Boxedi[+1]{v-1},\boxed{\quad\cali_{g_{v}+1,h_v+1}\quad},\Boxedi{v+1},\dots,\Boxedi{\omega} \big),\\[1em]
    \tu{j}{1}&=\big(\iBoxed{\al},\dots,\iBoxed{1} ,\underbrace{i-1}_{j-1\text{th position}},i-1 , \Boxedi{1} , \dots   ,\Boxedi{\omega} \big),\\[1em]
     \tu{j}{s+1}&=\big(\iBoxed{\al},\dots,\iBoxed{s},\boxed{\quad\cali^{G_{s-1},H_{s-1}-1}\quad},\iBoxed[-1]{s-2},\\[1em]
     &\hspace{2cm}\dots,\iBoxed[-1]{1},i-1 ,\underbrace{i-1}_{j-1\text{th position}},  i-1  ,\Boxedi{1} ,\dots   ,\Boxedi{\omega} \big).
\end{align*}
    Then we define $\tij$ to be the set of all $\tau\in\tij'$ satisfying one of the following:
    \begin{itemize}
        \item if $\tau=\tj{j}{v+1}$, then $g_{v}=\sum_{\rho=1}^vh_\rho+1$ and for each $1\le r\le \sum_{\rho=1}^vh_\rho$, we have that $\tau_{j+r}\le \tau_j-1-r$, or
        \item if $\tau=\tu{j}{s+1}$ then $G_{s+1}=\sum_{\rho=1}^s H_\rho+1$ and for each $1\le p\le \sum_{\rho=1}^sH_\rho$, we have that $\tau_{j-1-p}\le \tau_{j-1}+p$.
    \end{itemize}
\end{defn}

\begin{rmk}
    We can visualise the set $\tij$ as follows.
    The Young diagrams of $\tj{j}{1}$ and $\tu{j}{1}$ are obtained from $\tj{j}{0}$ by removing a horizontal (resp.\ vertical) $2$-strip, starting in the last node of row $\tau_j$.
    (This implies that for $1\le k$, $\tj{j}{k}\in\tij$, i.e.\ $\tj{j}{k}$ is a valid partition, if and only if $\tau_j\ge\tau_{j+1}+2$.
    Similarly, for $1\le s$, $\tu{j}{s}\in\tij$ if and only if $\tau_{j-1}=\tau_j$.)
    
    The Young diagram of each $\tj{j}{v+1}$ is obtained from $\tj{j}{1}$ by removing a rim consisting of $2\cdot\sum_{\rho=1}^{v}h_\rho$ nodes, starting in row $j$ and ending in the last node of row $j+\sum_{\rho=1}^vh_\rho$.
    Finally, the Young diagram of each $\tu{j}{s+1}$ is obtained from $\tj{j}{1}$ by removing a rim consisting of $2\cdot\sum_{\rho=1}^{s}h_\rho$ nodes, starting at the end of row $j-2$ and ending in the last node of row $j-1-\sum_{\rho=1}^{s-1}+1$).

    In particular, every $\tj{j}{v}$ and $\tu{j}{s}$ can be obtained from $\tj{j}{0}$ by removing a rim starting in the node $(j,\tau_j)$ and ending in a node which is in the set $\calr:=\calr_1\cup\calr_2$ where
    \begin{align*}
        \calr_1:=&\{(j-m-1,\tau_j+m)\mid 0\le m< m_1\le j-2,(j-m_1-1,\tau_j+m_1+1)\not\in \tj{j}{-1}\},\\[.5em]
        \calr_2:=&\{(j+m,\tau_j-m-1)\mid 0\le m< m_2\le R-j,(j+m_2,\tau_j-m_2)\not\in \tj{j}{-1}\}.
    \end{align*}
\end{rmk}

\begin{eg}\label{eg:taus}
    Let $\tj{4}{0}=(10,8,7^2,5,3^2,2^3)$, so we have that
    \begin{center}
        \begin{tikzpicture}
        \Yboxdim{0.5cm}
            \Yfillcolour{yellow!20}
                \tyng(0cm,0cm,10,8,7,7,5,3,3,2,2,2)
                \node at (-1, 0.15) (a) {$[\tj{4}{0}]=$};
        \end{tikzpicture}.
    \end{center}
    Next, $\tj{4}{-1}=(10,8,7,6,5,3^2,2^3)$, hence
    \begin{center}
        \begin{tikzpicture}
        \Yboxdim{0.5cm}
        \Yfillcolour{yellow!20}
                \tyng(0cm,0cm,10,8,7,7,5,3,3,2,2,2)
                {
                \Ylinethick{1.5pt}
                \Ylinecolour{red}
                \Yfillcolour{red!30}
                \tgyoung(0cm,0cm),:::::::,:::::::,:::::::,::::::;)
                }
                \node at (-1, 0.15) (a) {$[\tj{4}{-1}]=$};
        \end{tikzpicture}
    \end{center}
    where the node highlighted in red is node removed from $[\tj{j}{0}]$.

    We can visualise the set $\calr$ as below.
    In fact, we see that every blue node sits at the bottom of its column in $[\tj{j}{-1}]$ and the set $\tij$ contains seven partitions in total.
    (Observe that if $\tj{j}{0}=\tau_1$, then row $j-1$ already crosses the blue line, which is why we did not define any $\tu{j}{s}$ in this case.)
    \begin{center}
        \begin{tikzpicture}
        \Yboxdim{0.5cm}
        \tyng(0cm,0cm,10,8,7,7,5,3,3,2,2,2)
                {
                \Ylinethick{1.5pt}
                \Ylinecolour{red}
                \Yfillcolour{red!30}
                \tgyoung(0cm,0cm),:::::::,:::::::,::::::;,:::::;;)
                }
                {
                \Ylinethick{1.5pt}
                \Ylinecolour{blue}
                \Yfillcolour{blue!30}
                \tgyoung(0cm,0cm),:::::::,:::::::;,::::::;,:::::;,::::;,:::,::;)
                }
                \draw[blue,line width=1.5pt] (4.5,.5)--(0,-4);
        \end{tikzpicture}
    \end{center}
    In the following, we highlight removed nodes in blue.
    First, we have $\tj{4}{1}=(10,8,7,5,5,3^2,2^3)$ and $\tu{4}{1}=(10,8,6,6,5,3^2,2^3)$.
    \begin{center}
        \begin{tikzpicture}
        \Yfillcolour{yellow!20}
                \Yboxdim{0.5cm}
                \tyng(0cm,0cm,10,8,7,7,5,3,3,2,2,2)
                {
                \Ylinethick{1.5pt}
                \Ylinecolour{blue}
                \Yfillcolour{blue!30}
                \tgyoung(0cm,0cm),:::::::,:::::::,:::::::,:::::;;)
                }
                \tyng(7.5cm,0cm,10,8,7,7,5,3,3,2,2,2)
                {
                \Ylinethick{1.5pt}
                \Ylinecolour{blue}
                \Yfillcolour{blue!30}
                \tgyoung(7.5cm,0cm),:::::::,:::::::,::::::;,::::::;)
                }
                \node at (-1, 0.15) (a) {$[\tj{4}{1}]=$};
                \node at (6.5, 0.3) (a) {$[\tu{4}{1}]=$};
        \end{tikzpicture}.
    \end{center}
    Then we have $\tj{4}{2}=(10,8,7,4^2,3^2,2^3)$ and $\tj{4}{2}=(10,8,7,4,2^6)$.
    \begin{center}
        \begin{tikzpicture}
        \Yfillcolour{yellow!20}
                \Yboxdim{0.5cm}
                \tyng(0cm,0cm,10,8,7,7,5,3,3,2,2,2)
                {
                \Ylinethick{1.5pt}
                \Ylinecolour{blue}
                \Yfillcolour{blue!30}
                \tgyoung(0cm,0cm),:::::::,:::::::,:::::::,::::;;;,::::;)
                }
                \tyng(7.5cm,0cm,10,8,7,7,5,3,3,2,2,2)
                {
                \Ylinethick{1.5pt}
                \Ylinecolour{blue}
                \Yfillcolour{blue!30}
                \tgyoung(7.5cm,0cm),:::::::,:::::::,:,::::;;;,::;;;,::;,::;)
                }
                \node at (-1, 0.15) (a) {$[\tj{4}{2}]=$};
                \node at (6.5, 0.15) (a) {$[\tj{4}{3}]=$};
        \end{tikzpicture}.
    \end{center}
    Finally, $\tu{4}{2}=(10,6^3,5,3^2,2^3)$.
    \begin{center}
        \begin{tikzpicture}
        \Yfillcolour{yellow!20}
                \Yboxdim{0.5cm}
                \tyng(0cm,0cm,10,8,7,7,5,3,3,2,2,2)
                {
                \Ylinethick{1.5pt}
                \Ylinecolour{blue}
                \Yfillcolour{blue!30}
                \tgyoung(0cm,0cm),:::::::,::::::;;,::::::;,::::::;)
                }
                \node at (-1, 0.3) (a) {$[\tu{4}{2}]=$};
        \end{tikzpicture}.
    \end{center}
\end{eg}

\begin{rem}
If $g_{v}=g_{v+1}-h_{v+1}$ for all $1\le v\le z$ and $G_{s+1}=G_{s}+H_{s}$ for all $1\le s\le x$, then $\tij=\{\tj{j}{-1},\tj{j}{0},\tj{j}{1},\tj{j}{2},\dots,\tj{j}{z+1},\tu{j}{1},\tu{j}{2},\dots,\tu{j}{x+1}\}$ and
\[
\vert \tij \vert=2+z+1+x+1=x+z+4.
\]
In particular, if $x=\al$ and $z=\omega$ and $H_s=1=h_v$ for all $1\le v\le z$ and $1\le s\le x$, then $x+z+4$ is maximal.
\end{rem}

The following statement shows that there is a special relation between the $\De$-coefficients of the partitions appearing in $\tij$ where $\tij$ is as in \cref{def:tau}.

\begin{prop}\label{prop:2c}
    For any $\tij$ with $1\le j\le R$ and $i>1$, we have that
    \[
    \sum_{v=0}^{z+1}\De(\tj{j}{v})+\sum_{s=1}^{x+1}\De(\tu{j}{s})=2\De(\tj{j}{-1})
    \]
    where $\tj{j}{v}$ and $\tu{j}{s}$ ($-1\le v\le z+1,1\le s\le x+1$) are as in \cref{def:tau} and $\De(\tau)$ is computed according to \cref{def:coeff}.
\end{prop}

\begin{proof}
We will prove the statement by considering the different sizes of $\tij$.
Abusing notation, if $\tau$ is a partition of the form
\[
\big(\iBoxed{\al},\dots,\iBoxed{1} ,\underbrace{i}_{j\text{th position}},\Boxedi{1},\dots,\Boxedi{\omega} \big)
\]
or 
\[
\big(\iBoxed{\al},\dots,\iBoxed{1} ,i,\underbrace{i}_{j\text{th position}},\Boxedi{1},\dots,\Boxedi{\omega} \big)
\]
we write $\De(\tau)=\fkd_{G_\al}\dots\fkd_{G_1}\fkd(i)\fkd_{g_1}\dots\fkd_{g_\omega}$.
We also set
\[
\fkd(G_1^{\al}):=\fkd_{G_\al}\dots\fkd_{G_1}\quad\text{and}\quad \fkd(g_1^{\omega}):=\fkd_{g_1}\dots\fkd_{g_\omega}.
\]
To simplify notation, we also write $i'=i-\caly$.\\

\textbf{The case when }$\vert\tij\vert=2$

In this case $\tij=\{\tau_{(j,0)},\tau_{(j,-1)}\}$.
Recall that if $\tj{j}{0}=\tau_2$, then $\tu{j}{1}\in\tij$, so here we must have $\tj{j}{0}=\tau_1$.
Then
\begin{align*}
    \tj{j}{0}&=\big(\iBoxed{\al},\dots,\iBoxed{1} ,\underbrace{i}_{j\text{th position}},\Boxedi{1},\dots   ,\Boxedi{\omega} \big)\\[1em]
    \tau_{(j,-1)}&=\big(\iBoxed{\al},\dots,\iBoxed{1} , \underbrace{i-1}_{j\text{th position}},\Boxedi{1},\dots   ,\Boxedi{\omega} \big) 
\end{align*}
where $G_1\ge 1$ and $g_1=1$ (since if $g_1\ge 2$, then $\tj{j}{1}\in\tij$).
If $g_1=1=h_1$ we have
\begin{align*}
    \De(\tau_{(j,0)})&=\fkd(G_1^{\al})(i'+1)((i-1)+1-\caly)\fkd(g_2^{\omega})\\[.5em]
    &=(i'+1)\cdot i'\cdot \fkd(G_1^{\al})\fkd(g_2^{\omega})
\end{align*}
and
\begin{align*}
\De(\tau_{(j,-1)})&=\fkd(G_1^{\al})\binom{(i-1)-\caly+2}{2}\fkd(g_2^{\omega})\\[.5em]
&=\binom{i'+1}{2}\cdot \fkd(G_1^{\al})\fkd(g_2^{\omega})
\end{align*}
so $\De(\tau_{(j,0)})=2\De(\tau_{(j,-1)})$ as required.

Otherwise, if $h_1>1$
\begin{align*}
    \De(\tau_{(j,0)})&=\fkd(G_1^{\al})(i-(i-1)+1)\binom{(i-1)-\caly+\calx}{\calx}\fkd(g_2^{\omega})\\[.5em]
    &=2\cdot\binom{i'-1+\calx}{\calx}\cdot\fkd(G_1^{\al})\fkd(g_2^{\omega})\\[1em]
     \De(\tj{j}{-1})&=\fkd(G_1^{\al})\binom{(i-1)-\caly+\calx}{\calx}\fkd(g_2^{\omega})=\binom{i'-1+\calx}{\calx}\cdot\fkd(G_1^{\al})\fkd(g_2^{\omega}).
\end{align*}
\\

\textbf{The case when }$\vert\tij\vert=3$

We have two cases to consider:
\begin{itemize}
    \item[i)] $\tij=\{\tau_{(j,0)},\tau_{(j,-1)},\tau_{(j,1)}\}$, and
    \item[ii)] $\tij=\{\tau_{(j,0)},\tau_{(j,-1)},\tau^{(j,1)}\}$.
\end{itemize}

\textbf{Case i)}

Here $\tij=\{\tau_{(j,0)},\tau_{(j,-1)},\tau_{(j,1)}\}$.
Once again, we must have that $\tj{j}{0}=\tau_1$, so
\begin{align*}
    \tj{j}{0}&=\big(\iBoxed{\al},\dots,\iBoxed{1} ,\underbrace{i}_{j\text{th position}},\Boxedi{1},\dots   ,\Boxedi{\omega} \big)\\[1em]
    \tau_{(j,-1)}&=\big(\iBoxed{\al},\dots,\iBoxed{1} , \underbrace{i-1}_{j\text{th position}},\Boxedi{1},\dots   ,\Boxedi{\omega} \big)\\[1em]
    \tau_{(j,1)}&=\big(\iBoxed{\al},\dots,\iBoxed{1} , \underbrace{i-2}_{j\text{th position}},\Boxedi{1},\dots   ,\Boxedi{\omega} \big) 
\end{align*}
where $G_1\ge 1$ and $g_1\ge 2$.
Since $\tj{j}{v+1}\not\in\tij$ for all $1\le v\le \omega$, there are three possibilities.
\begin{enumerate}
    \item[1)] If $g_v\ge \sum_{\rho=1}^vh_\rho+1$ for all $1\le v\le \omega$, then $\calr_2=\{j,\tau_j-1\}$, that is there are more blue nodes below row $j$ (see \cref{eg:taus}) and $\caly=0$ for $\fkd(i)$ in each $\tau\in\tij$.
    This case also includes $j=R$.
    \item[2)] If $h_1\ge g_1$, then row $j+h_1$ already crosses the blue line and $\fkd(g_2^\omega)$ is the same in all three partitions in $\tij$.
    \item[3)] If for some $1\le v'<v\le \omega$ we have that $g_{v'}>\sum_{\rho=1}^{v'}h_\rho+1$ and $g_v\le \sum_{\rho=1}^vh_\rho$, we see that $\fkd(g_1^{v'})$ is the same in all partitions in $\tij$ and we compute $\fkd(g_v)$ for each.
    Here, if $g_v<\sum_{\rho=1}^vh_\rho$, then $\caly=i-g_v$ in $\fkd(i)$ for all three partitions above, however if $g_v=\sum_{\rho=1}^vh_\rho$, then $\caly$ is different in $\fkd(i)\in\De(\tj{j}{0})$ from $\caly$ in $\fkd(i)$ in the other two.
\end{enumerate}
Thus, we have three computations: case 2) breaks into $g_1=h_1$ and $h_1>g_1$ (if $g_1\ge h_1+1$, then the computation is identical to the $h_1>g_1$ case by setting $\caly=0$ in $\fkd(i)$) and the special case from case 3).

\textbf{Case 2a):} If $h_{1}=g_1$, we have that
\begin{align*}
    \De(\tau_{(j,0)})&=\fkd(G_1^{\al})(i'+1)\binom{(i-g_1)+h_1-\caly}{h_1}\fkd(g_2^{\omega})\\[.5em]
    &=(i'+1)\binom{i'}{h_1}\cdot\fkd(G_1^{\al})\fkd(g_2^{\omega})
\end{align*}
and
\begin{align*}
    \De(\tau_{(j,-1)})&=\fkd(G_1^{\al})((i-1)-(i-g_1)+1)\binom{(i-g_1)+h_1+1-\caly}{h_1+1}\fkd(g_2^{\omega})\\[.5em]
    &=g_1\cdot\binom{i'+1}{h_1+1}\cdot\fkd(G_1^{\al})\fkd(g_2^{\omega})
\end{align*}
and 
\begin{align*}
    \De(\tau_{(j,1)})&=\fkd(G_1^{\al})((i-2)-(i-g_1)+1)\binom{(i-g_1)+h_1+1-\caly}{h_1+1}\fkd(g_2^{\omega})\\[.5em]
    &=(g_1-1)\cdot\binom{i'+1}{h_1+1}\cdot\fkd(G_1^{\al})\fkd(g_2^{\omega})
\end{align*}
Now
\[
    (i'+1)\binom{i'}{h_1}=\frac{(i'+1)!}{h_1!(i'-h_1)!},
\]
so
\begin{align*}
    \De(\tau_{(j,0)})+\De(\tau_{(j,1)})=&\Big(\frac{(i'+1)!}{h_1!(i'-h_1)!}+(h_1-1)\frac{(i'+1)!}{(h_1+1)!(i'-h_1)!}\Big)\cdot\fkd(G_1^{\al})\fkd(g_2^{\omega})\\
    &=2h_1\cdot\frac{(i'+1)!}{(h_1+1)!(i'-h_1)!}\cdot\fkd(G_1^{\al})\fkd(g_2^{\omega})=2\De(\tau_{(j,-1)})
\end{align*}
as required.

\textbf{Case 2b):} Otherwise, if $h_{1}>g_1$,
\begin{align*}
        \De(\tau_{(j,0)})&=\fkd(G_1^{\al})(i-(i-g_1)+1)\fkd(g_1^{\omega})=(g_1+1)\fkd(G_1^{\al})\fkd(g_1^{\omega})\\[.5em]
        \De(\tau_{(j,-1)})&=\fkd(G_1^{\al})(i-1-(i-g_1)+1)\fkd(g_1^{\omega})=g_1\cdot\fkd(G_1^{\al})\fkd(g_1^{\omega})\\[.5em]
        \De(\tau_{(j,1)})&=\fkd(G_1^{\al})(i-2-(i-g_1)+1)\fkd(g_1^{\omega})=(g_1-1)\fkd(G_1^{\al})\fkd(g_1^{\omega}).
\end{align*}
Hence $\De(\tau_{(j,0)})+\De(\tau_{(j,1)})=2\De(\tau_{(j,-1)})$ as required.
We see that the final computation also applies to the case when $j=R$, as then $i-g_1=0$ and $\fkd(g_1^\omega)$ is empty.

\textbf{Case 3):} If $g_{v'}>\sum_{\rho=1}^{v'}h_\rho+1$ for all $1\le v'<v\le \omega$ and $g_v=\sum_{\rho=1}^vh_\rho$.
\begin{align*}
    \De(\tj{j}{0})&=\fkd(G_1^\al)(i-\caly+1)\fkd(g_1^{v-1})\binom{i-g_v-\caly+\sum_{\rho=1}^vh_\rho}{\sum_{\rho=1}^vh_\rho}\fkd(g_{v+1}^\omega)\\[.5em]
    &=(i'+1)\cdot\binom{i'}{g_v}\fkd(G_1^\al)\fkd(g_1^{v-1})\fkd(g_{v+1}^\omega)\\[1em]
    \De(\tau_{(j,-1)})&=\fkd(G_1^{\al})((i-1)-(i-g_v)+1)\fkd(g_1^{v-1})\binom{i-g_v-\caly+\sum_{\rho=1}^vh_\rho+1}{\sum_{\rho=1}^vh_\rho+1}\fkd(g_{v+1}^\omega)\\[.5em]
    &=g_v\cdot\binom{i'+1}{g_v+1}\cdot\fkd(G_1^\al)\fkd(g_1^{v-1})\fkd(g_{v+1}^\omega)\\[1em]
    \De(\tau_{(j,1)})&=\fkd(G_1^{\al})((i-2)-(i-g_v)+1)\fkd(g_1^{v-1})\binom{i-g_v-\caly+\sum_{\rho=1}^vh_\rho+1}{\sum_{\rho=1}^vh_\rho+1}\fkd(g_{v+1}^\omega)\\[.5em]
    &=(g_v-1)\cdot\binom{i'+1}{g_v+1}\cdot\fkd(G_1^\al)\fkd(g_1^{v-1})\fkd(g_{v+1}^\omega)
\end{align*}
Comparing coefficients with case 2a), we see that $\De(\tj{j}{0})+\De(\tj{j}{1})=2\De(\tj{j}{-1})$.
\\

\textbf{Case ii)}

Here $\tij=\{\tau_{(j,0)},\tau_{(j,-1)},\tau^{(j,1)}\}$.
Similar to the previous case, however we have that $g_1=1$ and $G_2\neq 1$.
As $G_{s+1}\neq \sum_{\rho=1}^sH_s+1$ for all $1\le s\le \al-1$, we only have two cases to consider: $h_1=1$ or $h_1>1$.
(We do not have case 3 as $g_1\ge1$.)

We consider the case when $G_1=0$ below.
However it is easy to see that if $G_1\ge 2$, then setting $\calx_i=2$ and replacing $\fkd(G_2^\al)$ with $\fkd(G_1^\al)$ yields the same result.
In general, $\calx_i=\sum_{\rho=1}^{y}H_\rho$ where $y$ is minimal such that $G_{y+1}>\sum_{\rho=1}^{y}H_\rho+1$.

\textbf{Case 1:} If $h_1=1$, then
\begin{align*}
    \tj{j}{0}&=\fkd(G_2^\al)\binom{i'+\calx_i}{\calx_i}\cdot (i-1-\caly+1)\cdot \fkd(g_2^\omega)\\[.5em]
    &=\binom{i'+\calx_i}{\calx_i}\cdot i' \cdot \fkd(G_2^\al)\fkd(g_2^\omega)\\[1em]
    \tj{j}{-1}&=\fkd(G_2^\al)\binom{i-(i-1)+\calx_i-1}{\calx_i-1}\binom{i-1-\caly+\calx_i+1}{\calx_i+1}\fkd(g_2^\omega)\\[.5em]
    &=\calx_i\cdot\binom{i'+\calx_i}{\calx_i+1}\cdot \fkd(G_2^\al)\fkd(g_2^\omega)\\[1em]
    \tu{j}{1}&=\fkd(G_2^\al)\binom{i-(i-1)+\calx_i-2}{\calx_i-2}\binom{i-1-\caly+\calx_i+1}{\calx_i+1}\fkd(g_2^\omega)\\[.5em]
    &=(\calx_i-1)\cdot\binom{i'+\calx_i}{\calx_i+1}\cdot \fkd(G_2^\al)\fkd(g_2^\omega).
\end{align*}
We also see that
\[
\binom{i'+\calx_i}{\calx_i}=\frac{(i'+\calx_i)!}{\calx_i!\cdot i'!}=\frac{\calx_i+1}{i'}\cdot\binom{i'+\calx_i}{\calx_i+1}.
\]
Thus
\[
    \De(\tj{j}{0})+\De(\tu{j}{1})=\frac{1}{\calx_i}\cdot\Big((\calx_i+1)+(\calx_i-1)\Big)\De(\tj{j}{-1})=2\cdot\De(\tj{j}{-1}).
\]

\textbf{Case 2:} If $h_1>1$, then
\begin{align*}
    \tj{j}{0}&=\fkd(G_2^\al)\binom{i-(i-1)+\calx_i}{\calx_i}\binom{i-1-\caly+\calx_{i-1}}{\calx_{i-1}} \fkd(g_2^\omega)\\[.5em]
    &=(\calx_i+1)\cdot\binom{i'-1+\calx_{i-1}}{\calx_{i-1}} \cdot \fkd(G_2^\al)\fkd(g_2^\omega)\\[1em]
    \tj{j}{-1}&=\fkd(G_2^\al)\binom{i-(i-1)+\calx_i-1}{\calx_i-1}\binom{i-1-\caly+\calx_{i-1}}{\calx_{i-1}}\fkd(g_2^\omega)\\[.5em]
    &=\calx_i\cdot\binom{i'-1+\calx_{i-1}}{\calx_{i-1}} \cdot \fkd(G_2^\al)\fkd(g_2^\omega)\\[1em]
    \tu{j}{1}&=\fkd(G_2^\al)\binom{i-(i-1)+\calx_i-2}{\calx_i-2}\binom{i-1-\caly+\calx_{i-1}}{\calx_{i-1}}\fkd(g_2^\omega)\\[.5em]
    &=(\calx_i-1)\cdot\binom{i'-1+\calx_{i-1}}{\calx_{i-1}} \cdot \fkd(G_2^\al)\fkd(g_2^\omega).
\end{align*}
\\

\textbf{The case when }$\vert\tij\vert\ge4$

We have four cases to consider:
\begin{itemize}
    \item[i)] $\tij=\{\tau_{(j,0)},\tau_{(j,-1)},\tau_{(j,1)},\tau^{(j,1)}\}$,
    \item[ii)] $\tij=\{\tau_{(j,0)},\tau_{(j,-1)},\tau_{(j,1)},\dots,\tj{j}{z+1}\}$,
     \item[iii)] $\tij=\{\tau_{(j,0)},\tau_{(j,-1)},\tau^{(j,1)},\dots,\tu{j}{x+1}\}$, and
    \item[iv)] $\tij=\{\tau_{(j,0)},\tau_{(j,-1)},\tau_{(j,1)},\dots,\tj{j}{z+1},\tau^{(j,1)},\dots,\tu{j}{x+1}\}$.
\end{itemize}

This case can be considered similarly.
Details can be included by switching the toggle in the {\tt arXiv} version of this paper.
\begin{answer}

\textbf{Case i)}

We have that $\tij=\{\tau_{(j,0)},\tau_{(j,-1)},\tau_{(j,1)},\tau^{(j,1)}\}$.
As $\tu{j}{1}\in\tij$, we must have that $\tj{j}{0}=\tau_2$.
Then
\begin{align*}
    \tj{j}{0}&=\big(\iBoxed{\al},\dots,\iBoxed{1} ,i,\underbrace{i}_{j\text{th position}},\Boxedi{1},\dots   ,\Boxedi{\omega} \big)\\[1em]
    \tau_{(j,-1)}&=\big(\iBoxed{\al},\dots,\iBoxed{1} , i,\underbrace{i-1}_{j\text{th position}},\Boxedi{1},\dots   ,\Boxedi{\omega} \big)\\[1em]
    \tau_{(j,1)}&=\big(\iBoxed{\al},\dots,\iBoxed{1} , i,\underbrace{i-2}_{j\text{th position}},\Boxedi{1},\dots   ,\Boxedi{\omega} \big)\\[1em]
    \tu{j}{1}&=\big(\iBoxed{\al},\dots,\iBoxed{1} , i-1,\underbrace{i-1}_{j\text{th position}},\Boxedi{1},\dots   ,\Boxedi{\omega} \big) 
\end{align*}
where $G_1\neq 1$ and $g_1\ge 2$ (if $G_1=1$, then $\tu{j}{2}\in\tij$).
We have the same three cases to consider as in the case $\tij=\{\tj{j}{0},\tj{j}{-1},\tj{j}{1}\}$, but now we can also have that $G_1=0$.
Similar to the previous case, we consider $G_1=0$ below.
If $G_1\ge 2$, then setting $\calx_i=2$ and replacing $\fkd(G_2^\al)$ with $\fkd(G_1^\al)$ gives the same result.
In general, $\calx_i=\sum_{\rho=1}^{y}H_\rho$ where $y$ is minimal such that $G_{y+1}>\sum_{\rho=1}^{y}H_\rho+1$.

\textbf{Case 2a):} If $G_1\neq 1$ and $h_1=g_1$, we have that
\begin{align*}
    \De(\tj{j}{0})&=\fkd(G_2^{\al})\binom{i'+\calx_i}{\calx_i}\binom{(i-g_1)+h_1-\caly}{h_1}\fkd(g_2^{\omega})\\[.5em]
    &=\binom{i'+\calx_i}{\calx_i}\binom{i'}{h_1}\cdot\fkd(G_2^{\al})\fkd(g_2^{\omega})\\[1em]
    \De(\tj{j}{-1})&=\fkd(G_2^{\al})\binom{i-(i-g_1)+\calx_i-1}{\calx_i-1}(i-1-(i-g_1)+1)\\[.5em]
    &\cdot \binom{(i-g_1)+h_1+\calx_i-\caly}{h_1+\calx_i}\fkd(g_2^{\omega})\\[.5em]
    &=\binom{g_1+\calx_i-1}{\calx_i-1}\cdot g_1\cdot\binom{i'+\calx_i}{g_1+\calx_i}\cdot\fkd(G_2^{\al})\fkd(g_2^{\omega})\\[1em]
    \De(\tj{j}{1})&=\fkd(G_2^{\al})\binom{i-(i-g_1)+\calx_i-1}{\calx_i-1}(i-2-(i-g_1)+1)\\[.5em]
    &\cdot\binom{(i-g_1)+h_1+\calx_i-\caly}{h_1+\calx_i}\fkd(g_2^{\omega})\\[.5em]
    &=\binom{g_1+\calx_i-1}{\calx_i-1}\cdot (g_1-1)\cdot\binom{i'+\calx_i}{g_1+\calx_i}\cdot\fkd(G_2^{\al})\fkd(g_2^{\omega})\\[1em]
    \De(\tu{j}{1})&=\fkd(G_2^{\al})\binom{i-(i-1)+\calx_i-2}{\calx_i-2}\binom{i-1-(i-g_1)+\calx_i}{\calx_i}\\[.5em]
    &\cdot\binom{(i-g_1)+h_1+\calx_i-\caly}{h_1+\calx_i}\fkd(g_2^{\omega})\\[.5em]
    &=(\calx_i-1)\cdot\binom{g_1+\calx_i-1}{\calx_i}\binom{i'+\calx_i}{g_1+\calx_i}\cdot\fkd(G_2^{\al})\fkd(g_2^{\omega}).
\end{align*}
Expanding the terms in $\De(\tj{j}{0})$, we have
\[
\binom{i'+\calx_i}{\calx_i}\binom{i'}{g_1}=\frac{(i'+\calx_i)!}{i'!\calx_i!}\cdot\frac{i'!}{(i'-g_1)!g_1!}=\binom{\calx_i+g_1}{\calx_i}\binom{i'+\calx_i}{\calx_i+g_1}
\]
and
\[
\binom{\calx_i+g_1}{\calx_i}=\frac{\calx_i+g_1}{\calx_i}\cdot\binom{g_1+\calx_i-1}{\calx_i-1}.
\]
In $\De(\tu{j}{1})$ we have
\[
\binom{g_1+\calx_i-1}{\calx_i}=\frac{g_1}{\calx_i}\cdot\binom{g_1+\calx_i-1}{\calx_i-1}.
\]
Thus
\begin{align*}
    \De(\tj{j}{0})+\De(\tj{j}{1})+\De(\tu{j}{1})&=\frac{1}{g_1}\Bigg(\frac{\calx_i+g_1}{\calx_i}+(g_1-1)+\frac{g_1\cdot(\calx_i-1)}{\calx_i}\Bigg)\De(\tj{j}{-1})\\[.5em]
    &=\frac{1}{g_1}\cdot\frac{\calx_i+g_1+g_1\calx_i-\calx_i+g_1\calx_i-g_1}{\calx_i}\De(\tj{j}{-1})\\[.5em]
    &=\frac{1}{g_1}\cdot2\cdot g_1\cdot\De(\tj{j}{-1})
\end{align*}
as required.

\textbf{Case 2b):} Otherwise, if $h_1>g_1$, we compute the $\De$ function as below.
\begin{align*}
    \De(\tj{j}{0})&=\fkd(G_2^{\al})\binom{i-(i-g_1)+\calx_i}{\calx_i}\fkd(g_1^{\omega})=\binom{g_1+\calx_i}{\calx_i}\cdot\fkd(G_2^{\al})\fkd(g_1^{\omega})\\[1em]
    \De(\tj{j}{-1})&=\fkd(G_2^{\al})\binom{i-(i-g_1)+\calx_i-1}{\calx_i-1}(i-1-(i-g_1)+1)\fkd(g_1^{\omega})\\[.5em]
    &=\binom{g_1+\calx_i-1}{\calx_i-1}\cdot g_1\cdot \fkd(G_2^{\al})\fkd(g_1^{\omega})\\[1em]
    \De(\tj{j}{1})&=\fkd(G_2^{\al})\binom{i-(i-g_1)+\calx_i-1}{\calx_i-1}(i-2-(i-g_1)+1)\fkd(g_1^{\omega})\\[.5em]
    &=\binom{g_1+\calx_i-1}{\calx_i-1}\cdot (g_1-1)\cdot \fkd(G_2^{\al})\fkd(g_1^{\omega})\\[1em]
    \De(\tu{j}{1})&=\fkd(G_2^{\al})\binom{i-(i-1)+\calx_i-2}{\calx_i-2}\binom{i-1-(i-g_1)+\calx_i}{\calx_i}\fkd(g_1^{\omega})\\[.5em]
    &=(\calx_i-1)\cdot\binom{g_1+\calx_i-1}{\calx_i-1}\cdot\fkd(G_2^{\al})\fkd(g_1^{\omega}).
\end{align*}
The computation for summing the coefficients is very similar to the previous, hence we omit it.

\textbf{Case 3):}
If $g_{v'}>\sum_{\rho=1}^{v'}h_\rho+1$ for all $1\le v'<v\le \omega$ and $g_v=\sum_{\rho=1}^vh_\rho$.
\begin{align*}
    \De(\tj{j}{0})&=\fkd(G_2^{\al})\binom{i+\calx_i-\caly}{\calx_i}\fkd(g_1^{v-1})\binom{i-g_v+\sum_{\rho=1}^vh_\rho-\caly}{\sum_{\rho=1}^vh_\rho}\fkd({g_{v+1}^\omega})\\[.5em]
    &=\binom{i'+\calx_i}{\calx_i}\binom{i'}{g_v}\cdot\fkd(G_2^{\al})\fkd(g_1^{v-1})\fkd({g_{v+1}^\omega})\\[1em]
    \De(\tj{j}{-1})&=\fkd(G_2^{\al})\binom{i-(i-g_v)+\calx_i-1}{\calx_i-1}(i-1-(i-g_v)+1)\\[.5em]
    &\cdot\fkd(g_1^{v-1})\binom{i-g_v+\sum_{\rho=1}^vh_\rho+\calx_i-\caly}{\sum_{\rho=1}^vh_\rho+\calx_i}\fkd({g_{v+1}^\omega})\\[.5em]
    &=\binom{g_v+\calx_i-1}{\calx_i-1}\cdot g_v\cdot\binom{i'+\calx_i}{g_v+\calx_i}\cdot\fkd(G_2^{\al})\fkd(g_1^{v-1})\fkd{g_{v+1}^\omega}\\[1em]
    \De(\tj{j}{1})&=\fkd(G_2^{\al})\binom{i-(i-g_v)+\calx_i-1}{\calx_i-1}(i-2-(i-g_v)+1)\\[.5em]
    &\cdot\fkd(g_1^{v-1})\binom{i-g_v+\sum_{\rho=1}^vh_\rho+\calx_i-\caly}{\sum_{\rho=1}^vh_\rho+\calx_i}\fkd({g_{v+1}^\omega})\\[.5em]
    &=\binom{g_v+\calx_i-1}{\calx_i-1}\cdot (g_v-1)\cdot\binom{i'+\calx_i}{g_v+\calx_i}\cdot\fkd(G_2^{\al})\fkd(g_1^{v-1})\fkd{g_{v+1}^\omega}\\[1em]
    \De(\tu{j}{1})&=\fkd(G_2^{\al})\binom{i-(i-g_v)+\calx_i-2}{\calx_i-2}\binom{i-1-(i-g_v)+\calx_i}{\calx_i}\\[.5em]
    &\cdot\fkd(g_1^{v-1})\binom{i-g_v+\sum_{\rho=1}^vh_\rho+\calx_i-\caly}{\sum_{\rho=1}^vh_\rho+\calx_i}\fkd({g_{v+1}^\omega})\\[.5em]
    &=(\calx_i-1)\cdot\binom{g_v+\calx_i-1}{\calx_i-1}\cdot\binom{i'+\calx_i}{g_v+\calx_i}\cdot\fkd(G_2^{\al})\fkd(g_1^{v-1})\fkd{g_{v+1}^\omega}.
\end{align*}
Comparing the coefficients with the previous two cases, we see that $\De(\tj{j}{0})+\De(\tj{j}{1})+\De(\tu{j}{1})=2\De(\tj{j}{-1})$.
\\

\textbf{Cases ii) and iii)} are special cases of case iv) and we deal with them at the end of the proof.\\

\textbf{Case iv)}

Here $\tij=\{\tau_{(j,0)},\tau_{(j,-1)},\tau_{(j,1)},\dots,\tj{j}{z+1},\tau^{(j,1)},\dots,\tu{j}{x+1}\}$ and $\tj{j}{0}=\tau_2$ with $G_1\ge0$ and $g_1\ge 1$.
Recall from the previous cases that we have a special case when $\fkd(i)\in\De(\tj{j}{0})$ has a different $\caly$ than in other $\tau\in\tij$.
The next set of conditions describes this case.
\begin{itemize}
    \item[1)] If
    \begin{itemize}
        \item[$\bullet$] $g_{z+1}=\sum_{\rho=1}^{z+1}h_\rho$, and
        \item[$\bullet$] $g_v\ge \sum_{\rho=1}^vh_\rho$ for $z+2\le v\le \omega$,
    \end{itemize}
    then the $\caly$ term in $\fkd(i)$ is different in $\De(\tj{j}{0})$ from other $\De(\tj{j}{v}),\De(\tu{j}{s})$.
    \item[2)] Otherwise, either $\caly=i-g_{z+1}$ in $\fkd(i)$ for all $\tau\in\tij$ or $0$ in $\fkd(i)$ for all $\tau\in\tij$.
\end{itemize}
In particular, the only difference in the computations for the two cases is when calculating $\De(\tj{j}{0})$, but the rest of the computations agree.

Before starting the calculations, we need to set up some notation. 
We define
\[
j':=\sum_1^xH_\rho+2 \qquad\text{and} \qquad i':=i-\caly=g_{z+1}\qquad\text{and} \qquad d':=\sum_1^xH_\rho+2+\sum_{v=1}^zh_v.
\]

It is easy to see that if $G_1\ge1$, then
\begin{equation*}
    \De(\tj{j}{-1})=\fkd(G_{1}^\al)\cdot(i'+1)\cdot i' \cdot\fkd(g_1^\omega)
\end{equation*}
and if $G_1=0$, then
\begin{equation*}
    \De(\tj{j}{-1})=\fkd(G_2^\al)\cdot\fkd(G_1\cdot i)\cdot i' \cdot\fkd(g_1^\omega).\tag{$\ast$}\label{eq:t-minusone}
\end{equation*}
where the term $\fkd(G_1 \cdot i)$ takes into account the extra $i$ in position $j-1$.
For shorthand notation, we write $\fkd(G_2^\al)\cdot\fkd(G_1\cdot i)=\fkd(G_1^x\cdot i)$ which ends with $\fkd(G_1)$ if $G_1=0$ and $(i'+1)$ if $G_1\ge1$.

Observe that for $1\le s\le x-1$, if $\tu{j}{s+1},\tu{j}{s+2}\in\tij$ we have
\[
    \fkd(G_s)=\binom{i+G_s-\caly+H_s}{H_s}
\]
and if $\tu{j}{s+1}\in\tij,\tu{j}{s+2}\not\in\tij$, then
\[
    \fkd(G_s^y)=\fkd(G_{s+1}^y)\cdot\binom{i+G_s-\caly+\sum_{\rho=s}^y H_\rho}{\sum_{\rho=s}^y H_\rho}
\]
where $s< y\le x-1$ is minimal such that $\tu{j}{y+1}\not\in\tij$ and $\tu{j}{y+2}\in\tij$, except if $G_1=0$, in which case $\tu{j}{2}\not\in\tij$ and we have that
\[
\fkd(G_1)=\fkd(G_{2}^y)\cdot\binom{i-\caly+\sum_{\rho=s}^y H_\rho+1}{\sum_{\rho=s}^y H_\rho+1}
\]
In fact, $\fkd(G_{s+1}^y)$ is the same in most $\tu{j}{q}\in\tij$, except for one special case, which is going to be considered separately later.
In general, we will not write $\fkd(G_{s+1}^y)$ explicitly.
\begin{align*}
    \fkd(G_{1}^x\cdot i)&=\prod_{\substack{s=1\\ \tu{j}{s+1}\in\tij}}^x\binom{i+G_s-\caly+\calx_s}{\calx_s}\prod_{\substack{s=1\\ \tu{j}{s+1}
    \not\in\tij}}^{x-1}\fkd(G_s)\\[.5em]
    &=\prod_{\substack{s=1\\ \tu{j}{s+1}\in\tij}}^{x}\frac{(i'+G_s+\calx_s)!}{(i'+G_s)!\calx_s!}\prod_{\substack{s=1\\ \tu{j}{s+1}
    \not\in\tij}}^{x-1}\fkd(G_s)\\[.5em]
     &=\prod_{\substack{s=1\\ \tu{j}{s+1}\in\tij}}^{x}\frac{(i'+G_s+\calx_s)(i'+G_s+\calx_s-1)\dots(i'+G_s+1)}{\calx_s!}\prod_{\substack{s=1\\ \tu{j}{s+1}
    \not\in\tij}}^{x-1}\fkd(G_s)\\[.5em]
    &=(i'+j'-1)(i'+j'-2)\dots(i'+1)\cdot\prod_{\substack{s=1\\ \tu{j}{s+1}\in\tij}}^x(\calx_s!)^{-1}\prod_{\substack{s=1\\ \tu{j}{s+1}
    \not\in\tij}}^{x-1}\fkd(G_s)\tag{$\ast\ast$}\label{eq:G-upper}
\end{align*}
where $\calx_s=H_s$ or $\sum_s^yH_\rho$ with $y$ as above and for all $1\le s\le x$ such that $\tu{j}{s+1}\in\tij$, we have that $G_{s+1}=\sum_{\rho=1}^sH_\rho+1$.
(Recall that $i'+1$ either comes from $\fkd(G_{1+i})$ or from the single $i$ in the $j-1$th position.)

Similarly, for $1\le v\le z-1$, if $\tj{j}{v+1},\tj{j}{v+2}\in\tij$, we have
\[
    \fkd(g_v)=\binom{i-g_v-\caly+h_v}{h_v}
\]
and if $\tj{j}{v+1}\in\tij,\tj{j}{v+2}\not\in\tij$, then 
\[
    \fkd(g_v^y)=\fkd(g_{v+1}^y)\cdot\binom{i+v_s-\caly+\sum_v^y h_\rho}{\sum_v^y h_\rho}
\]
where $v< y\le z-1$ is minimal such that $\tj{j}{y+1}\not\in\tij$ and $\tj{j}{y+2}\in\tij$.
The term $\fkd(g_{v+1}^y)$ gives the same coefficient in all $\tau\in\tij$, so we can avoid computing it explicitly. 
Then in $\tj{j}{0},\tj{j}{-1}$ and $\tu{j}{s+1}$ for all $0\le s\le x$, we have
\begin{align*}
    \fkd(g_1^z)&=\prod^z_{\substack{v=1\\ \tj{j}{v+1}\in\tij}}\binom{i-g_v-\caly+\calx_v}{\calx_v}\prod_{\substack{v=1\\ \tj{j}{v+1}
    \not\in\tij}}^{z-1}\fkd(g_v)\\[.5em]
    &=\prod^z_{\substack{v=1\\ \tj{j}{v+1}\in\tij}}\frac{(i'-g_v+\calx_v)!}{(i'-g_v)!\calx_v!}\prod_{\substack{v=1\\ \tj{j}{v+1}
    \not\in\tij}}^{z-1}\fkd(g_v)\\[.5em]
    &=\prod^z_{\substack{v=1\\ \tj{j}{v+1}\in\tij}}(i'-g_v+\calx_v)(i'-g_v+\calx_v-1)\dots(i'-g_v+1)\\[.5em]
    &\cdot\prod^z_{\substack{v=1\\ \tj{j}{v+1}\in\tij}}(\calx_v!)^{-1}\prod_{\substack{v=1\\ \tj{j}{v+1}
    \not\in\tij}}^{z-1}\fkd(g_v)\\[.5em]
    &=\frac{(i'-1)!}{(i'-g_{z})!}\prod^z_{\substack{v=1\\ \tj{j}{v+1}\in\tij}}(\calx_v!)^{-1}\prod_{\substack{v=1\\ \tj{j}{v+1}
    \not\in\tij}}^{z-1}\fkd(g_v)\tag{$\ast\ast\ast$}\label{eq:g-lower}
\end{align*}
where $\calx_v=h_v$ or $\sum_{\rho=y}^vh_\rho$.

Hence substituting \cref{eq:G-upper,eq:g-lower} into \cref{eq:t-minusone}, we obtain
\begin{align*}
    \De(\tj{j}{-1})&=\fkd(G_{x+1}^\al)\fkd(G_{1}^x\cdot i)\cdot i'\cdot \fkd(g_1^z)\fkd(g_{z+1}^\omega)\\[.5em]
    &=\frac{i'!}{(i'-g_z)!}\cdot\prod_{\substack{s=1\\ \tu{j}{s+1}\in\tij}}^x(\calx_s!)^{-1}\prod_{\substack{s=1\\ \tu{j}{s+1}
    \not\in\tij}}^{x-1}\fkd(G_s)\\[.5em]
    &\cdot\prod^z_{\substack{v=1\\ \tj{j}{v+1}\in\tij}}(\calx_v!)^{-1}\prod_{\substack{v=1\\ \tj{j}{v+1}
    \not\in\tij}}^{z-1}\fkd(g_v).\tag{$\star$}\label{eq:t-}
\end{align*}

As $g_z=\sum_{v=1}^zh_v+1$, we see that
\begin{align*}
    i'-g_{z}&=i'-\Big(\sum_{v=1}^zh_v+1\Big)\\[.5em]
    &=i'+\Big(\sum_{s=1}^xH_s+2\Big)-\Bigg(\Big(\sum_{s=1}^xH_s+2\Big)+\Big(\sum_{v=1}^zh_v\Big)\Bigg)-1\\[.5em]
    &=i'+j'-d'-1.
\end{align*}

As mentioned above, for $\tj{j}{0}$, we have two cases to consider.
For shorthand notation, we write
\[
\cala:=\prod^{x-1}_{\substack{s=1\\ \tu{j}{s+1}\not\in\tij}}\fkd(G_s)
\]
\begin{itemize}
    \item[a)] If we are in case 1), then $g_{z+1}=\sum_{\rho=1}^{z+1}h_\rho$.
    \begin{align*}
        \De(\tj{j}{0})&=\fkd(G_{x+1}^\al)\prod^x_{\substack{s=1\\ \tu{j}{s+1}\in\tij}}\binom{i+G_s-i+\calx_s}{\calx_s}\cdot\binom{i-\caly+\sum_1^xH_s+2}{\sum^x_1H_s+2 }\\[.5em]
        &\cdot\fkd(g_1^z)\cdot\binom{i-g_{z+1}-\caly+\sum^{z+1}_{\rho=1}h_\rho}{\sum^{z+1}_{\rho=1}h_\rho}\cdot\fkd(g_{z+2}^\omega)\cdot\cala\\[.5em]
        &=\prod^x_{\substack{s=1\\ \tu{j}{s+1}\in\tij}}\binom{G_s+\calx_s}{\calx_s}\cdot\binom{i'+j'}{j'}\binom{i'}{g_{z+1}}\cdot\fkd(G_{x+1}^\al)\fkd(g_1^z)\fkd(g_{z+2}^\omega)\cdot\cala\\[.5em]
        &=\prod^x_{\substack{s=1\\ \tu{j}{s+1}\in\tij}}(G_s+H_s)(G_s+H_s-1)\dots(G_s+1)\cdot\prod^x_{\substack{s=1\\ \tu{j}{s+1}\in\tij}}(\calx_s!)^{-1}\\[.5em]
        &\cdot\frac{(i'+j')!}{i'!j'!}\cdot\frac{i'!}{g_{z+1}!\cdot(i-g_{z+1})!}\cdot\fkd(G_{x+1}^\al)\fkd(g_1^z)\fkd(g_{z+2}^\omega)\cdot\cala\\[.5em]
        &=(\sum_1^xH_s+1)(\sum_1^xH_s)\dots(G_1+1)\cdot\frac{(i'+j')!}{i'!j'!}\\[.5em]
        &\cdot\frac{i'!}{g_{z+1}!\cdot(i-g_{z+1})!}\cdot\prod^x_{\substack{s=1\\ \tu{j}{s+1}\in\tij}}(\calx_s!)^{-1}\cdot\fkd(G_{x+1}^\al)\fkd(g_1^z)\fkd(g_{z+2}^\omega)\cdot\cala\\[.5em]
        &=(j'-1)!\cdot\frac{(i'+j')!}{i'!j'!}\cdot\frac{i'!}{g_{z+1}!\cdot(i-g_{z+1})!}\cdot\prod^x_{\substack{s=1\\ \tu{j}{s+1}\in\tij}}(\calx_s!)^{-1}\\[.5em]
        &\cdot\fkd(G_{x+1}^\al)\fkd(g_1^z)\fkd(g_{z+2}^\omega)\cdot\cala\\[.5em]
        &=\frac{(i'+j')!}{j'\cdot g_{z+1}!(i'-g_{z+1})!}\cdot\fkd(G_{x+1}^\al)\fkd(g_1^z)\fkd(g_{z+2}^\omega)\cdot\cala.
    \end{align*}

    In this case we also see that
    \begin{align*}
        \De(\tj{j}{-1})&=\fkd(G_2^\al)\binom{i-(i-g_{z+1})+j'-1}{j'-1}\cdot(i-1-(i-g_{z+1})+1)\\[.5em]
        &\cdot\fkd(g_1^z)\cdot\binom{i-g_{z+1}-\caly+\sum^{z+1}_{\rho=1}h_\rho+j'}{\sum^{z+1}_{\rho=1}h_\rho+j'}\cdot\fkd(g_{z+2}^\omega)\\[.5em]
        &=\binom{j'-1+g_{z+1}}{j'-1}\cdot g_{z+1}\cdot\binom{i'+j'}{g_{z+1}+j'}\cdot\fkd(G_1^\al)\fkd(g_1^z)\fkd(g_{z+2}^\omega)\\[.5em]
        &=\frac{(j'-1+g_{z+1})!}{(j'-1)!\cdot g_{z+1}!}\cdot g_{z+1}\cdot\frac{(i'+j')!}{(i'-g_{z+1})!(g_{z+1}+j')!}\cdot\fkd(G_1^\al)\fkd(g_1^z)\fkd(g_{z+2}^\omega)\\[.5em]
        &=\frac{g_{z+1}}{(j'-1)!\cdot g_{z+1}!}\cdot\frac{(i'+j')!}{(i'-g_{z+1})!(g_{z+1}+j')}\cdot\fkd(G_1^\al)\fkd(g_1^z)\fkd(g_{z+2}^\omega).
    \end{align*}
    \item[b)] If we are in case 2), we have the following.
    (Note that if $G_1=0$, then $\tu{j}{2}\not\in\tij$.)
    \begin{align*}
        \De(\tj{j}{0})&=\fkd(G_{x+1}^\al)\prod^x_{\substack{s=1\\ \tu{j}{s+1}\in\tij}}\binom{i+G_s-i+\calx_s}{\calx_s}\cdot\binom{i-\caly+\sum_1^xH_s+2}{\sum^x_1H_s+2 }\cdot\fkd(g_1^\omega)\cdot\cala\\[.5em]
        &=\prod^x_{\substack{s=1\\ \tu{j}{s+1}\in\tij}}\binom{G_s+\calx_s}{\calx_s}\cdot\binom{i'+j'}{j'}\cdot\fkd(G_{x+1}^\al)\fkd(g_1^\omega)\cdot\cala\\[.5em]
        &=\prod^x_{\substack{s=1\\ \tu{j}{s+1}\in\tij}}(G_s+H_s)(G_s+H_s-1)\dots(G_s+1)\\[.5em]
        &\cdot\prod^x_{\substack{s=1\\ \tu{j}{s+1}\in\tij}}(\calx_s!)^{-1}\cdot\frac{(i'+j')!}{i'!j'!}\cdot\fkd(G_{x+1}^\al)\fkd(g_1^\omega)\cdot\cala\\[.5em]
        &=(\sum_1^xH_s+1)(\sum_1^xH_s)\dots(G_1+1)\frac{(i'+j')!}{i'!j'!}\\[.5em]
        &\cdot\prod^x_{\substack{s=1\\ \tu{j}{s+1}\in\tij}}(\calx_s!)^{-1}\cdot\fkd(G_{x+1}^\al)\fkd(g_1^\omega)\cdot\cala\\[.5em]
        &=(j'-1)!\cdot\frac{(i'+j')!}{i'!j'!}\cdot\prod^x_{\substack{s=1\\ \tu{j}{s+1}\in\tij}}(\calx_s!)^{-1}\cdot\fkd(G_{x+1}^\al)\fkd(g_1^\omega)\cdot\cala
    \end{align*}
As $i'=g_{z+1}$, we see that 
\begin{equation*}
    \De(\tj{j}{0})=\frac{i'+j'}{i'\cdot j'}\cdot\De(\tj{j}{-1}), \tag{1}\label{eq:t-zero}
\end{equation*}
in both cases.
\end{itemize}

Next, we look at $\De(\tj{j}{v+1})$ with $0\le v\le z$.
For $\tj{j}{1}$, we need to consider four cases.
\begin{itemize}
    \item[a)] If $h_1=1$ and $\tj{j}{2}\in\tij$, then $g_1=2$, so we have that
    \begin{align*}
        \De(\tj{j}{1})&=\fkd(G_{1}^\al\cdot i)\binom{(i-2)-\caly+2}{2}\fkd(g_2^z)\\[.5em]
        &=\binom{i'}{2}\fkd(G_{1}^\al\cdot i)\fkd(g_2^z)\\[.5em]
        &=\frac{1}{2}\cdot\De(\tj{j}{-1})=\frac{h_1}{g_1}\cdot\De(\tj{j}{-1}).
    \end{align*}
    \item[b)] If $h_1>1$ and $\tj{j}{2}\in\tij$, then $g_1=h_1+1$, so we have that
    \begin{align*}
        \De(\tj{j}{1})&=\fkd(G_1^\al\cdot i)(i-2-(i-g_1)+1)\binom{(i-g_1)-\caly+h_1+1}{h_1+1}\fkd(g_2^z)\\[.5em]
        &=(g_1-1)\binom{i'-g_1+h_1+1}{h_1+1}\cdot\fkd(G_1^\al\cdot i)\fkd(g_2^\omega)\\[.5em]
        &=(g_1-1)\cdot\binom{i'}{g_1}\cdot\fkd(G_1^\al\cdot i)\fkd(g_2^\omega)\\[.5em]
        &=(g_1-1)\cdot\frac{i'!}{(g_1)!(i'-g_1)!}\cdot\fkd(G_1^\al\cdot i)\fkd(g_2^\omega)\\[.5em]
        &=(g_1-1)\cdot\frac{i'\cdot(i'-1)!}{(g_1)\cdot(g_1-1)!(i'-g_1)!}\cdot\fkd(G_1^\al\cdot i)\fkd(g_2^\omega)\\[.5em]
        &=\frac{h_1}{g_1}\cdot\De(\tj{j}{-1}),
    \end{align*}
    as $\binom{i'-1}{g_1-1}=\fkd(g_1)$.
    
    \item[c)] If $z=0$ then $\tj{j}{1}$ is the highest index $\tj{j}{v}\in\tij$.
    \begin{align*}
        \De(\tj{j}{p+1})&=\fkd(G_1^\al\cdot i)((i-2)-\caly+1)\fkd(g_{z+1}^\omega)\\[.5em]
        &=(i'-1)\cdot\fkd(G_1^\al\cdot i)\fkd(g_{z+1}^\omega)\\[.5em]
        &=\frac{i'-1}{i'}\cdot\De(\tj{j}{-1}).
    \end{align*}
    
    \item[d)] If $\tj{j}{2}\not\in\tij$, let $2 \le y<z$ be minimal such that $\tj{j}{y}\not\in\tij$ and $\tj{j}{y+1}\in\tij$.
     \begin{align*}
        \De(\tj{j}{1})&=\fkd(G_1^\al\cdot i)(i-2-(i-g_y)+1)\fkd(g_1^{y-1})\binom{(i-g_y)-\caly+\sum_{\rho=1}^yh_\rho+1}{\sum_{\rho=1}^yh_\rho+1}\fkd(g_{y+1}^\omega)\\[.5em]
        &=(g_y-1)\binom{i'-g_y+\sum_{\rho=1}^yh_\rho+1}{\sum_{\rho=1}^yh_\rho+1}\cdot\fkd(G_1^\al\cdot i)\fkd(g_1^{y-1})\fkd(g_{y+1}^\omega)\\[.5em]
        &=(g_y-1)\cdot\binom{i'}{g_y}\cdot\fkd(G_1^\al\cdot i)\fkd(g_1^{y-1})\fkd(g_{y+1}^\omega)\\[.5em]
        &=(g_y-1)\cdot\frac{i'!}{(g_y)!(i'-g_y)!}\cdot\fkd(G_1^\al\cdot i)\fkd(g_1^{y-1})\fkd(g_{y+1}^\omega)\\[.5em]
        &=(g_y-1)\cdot\frac{i'!}{(g_y)\cdot(g_y-1)!(i'-g_y)!}\cdot\fkd(G_1^\al\cdot i)\fkd(g_1^{y-1})\fkd(g_{y+1}^\omega)\\[.5em]
        &=\frac{\sum_{\rho=1}^{y}h_\rho}{g_y}\cdot\De(\tj{j}{-1}).
    \end{align*}
\end{itemize}

Next, we consider $\De(\tj{j}{v+1})$ for $1\le v<z$.
\begin{align*}
    \De(\tj{j}{v+1})&=\fkd(G_1^\al \cdot i)\cdot \prod_{\substack{w=1\\ \tj{j}{w+1}\in\tij}}^{v-1}\binom{(i-(g_w+1)-(i-(g_{v}+1))+\calx_w}{\calx_w}\\[.5em]
    &\cdot\binom{(i-(g_v+1)-(i-g_{v+1})+\sum^{v}_{\rho=1}h_\rho+1}{\sum^{v}_{\rho=1}h_\rho+1}\\[.5em]
    &\cdot\binom{(i-g_{v+1})-\caly+\sum^{v+1}_{\rho=1}h_\rho+1}{\sum^{v+1}_{\rho=1}h_\rho+1}\cdot \fkd(g_{v+2}^\omega)\cdot \prod_{\substack{w=1\\ \tj{j}{w+1}
    \not\in\tij}}^{z-1}\fkd(g_w)\\[.5em]
    &=\bigg(\prod_{\substack{w=1\\ \tj{j}{w+1}\in\tij}}^{v-1}\binom{g_v-g_w+\calx_w}{\calx_w}\bigg)\cdot\frac{(g_{v+1}-1)!}{g_v!(h_{v+1}-1)!}\\[.5em]
    &\cdot\frac{i'!}{(g_{v+1})!(i'-g_{v+1})!}\cdot\fkd(G_1^\al\cdot i)\fkd(g_{v+2}^\omega)\cdot \prod_{\substack{w=1\\ \tj{j}{w+1}
    \not\in\tij}}^{z-1}\fkd(g_w)\\[.5em]
    &=\bigg(\prod_{\substack{w=1\\ \tj{j}{w+1}\in\tij}}^{v-1}\frac{(g_v-g_w+\calx_w)!}{\calx_w!(g_v-g_w)!}\bigg)\cdot\frac{h_{v+1}}{g_v!}\\[.5em]
    &\cdot\frac{i'\cdot(i'-1)\dots(i'-g_{v+1}+1)}{g_{v+1}}\cdot\fkd(G_1^\al\cdot i)\fkd(g_{v+2}^\omega)\cdot \prod_{\substack{w=1\\ \tj{j}{w+1}
    \not\in\tij}}^{z-1}\fkd(g_w)\\[.5em]
    &=\bigg(\prod_1^{v-1}(g_v-g_w+\calx_w)(g_v-g_w-\calx_w-1)\dots(g_v-g_w+1)\bigg)\\[.5em]
    &\cdot\frac{h_{k+1}}{g_k!}\cdot\frac{1}{g_{k+1}}\cdot\De(\tj{j}{-1})\\[.5em]
    &=(g_v-1)!\cdot\frac{h_{v+1}}{g_v!}\cdot\frac{1}{g_{v+1}}\cdot\De(\tj{j}{-1})=\frac{h_{v+1}}{g_v\cdot g_{v+1}}\cdot\De(\tj{j}{-1})
\end{align*}

Note that if $g_v+1=g_{v+1}$, then we must have that $h_{v+1}=1$ and the third term in the first equality equals one.
If $\tj{j}{v+2}\not\in\tij$, we have to replace the third and fourth terms in the first equality with
\[
\binom{(i-(g_v+1)-(i-g_{y})+\sum^{y}_{\rho=1}h_\rho+1}{\sum^{y}_{\rho=1}h_\rho+1}\binom{(i-g_{y})-\caly+\sum^{y}_{\rho=1}h_\rho+1}{\sum^{y}_{\rho=1}h_\rho+1}\cdot\fkd(g_{v+2}^\omega)
\]
where $v< y\le z$ is minimal such that $\tj{j}{y+1}\not\in\tij$ and $\tj{j}{y+2}\in\tij$, which gives 
\[
\De(\tj{j}{v+1})=\frac{\sum_{\rho=v+1}^yh_{\rho}}{g_v\cdot g_{y}}\cdot\De(\tj{j}{-1})=\frac{\sum_{\rho=v+1}^yh_{\rho}}{g_v\cdot (g_v+\sum_{\rho=v+1}^yh_{\rho})}\cdot\De(\tj{j}{-1})
\]
but the computation is essentially the same (see also case d) of $\De(\tj{j}{1})$).

\textbf{\underline{Claim}}

    We will show that
    \[
    \frac{m}{r(r+m)}=\sum_{y=r}^{r+m-1}\frac{1}{y(y+1)}.
    \]

    We will use induction on $m$.
    If $m=1$, then the claim is clearly true.
    So assume that the statement holds for $k<m$.
Then
\begin{align*}
    \sum_{y=r}^{r+m-1}\frac{1}{y(y+1)}&=\sum_{k=r}^{r+m-2}\frac{1}{k(k+1)}+\frac{1}{(r+m-1)(r+m)}\\[.5em]
    &=\underbrace{\frac{m-1}{r(r+m-1)}}_{\text{by induction}}+\frac{1}{(r+m-1)(r+m)}\\[.5em]
    &=\frac{(m-1)(r+m)+r}{r(r+m-1)(r+m)}\\[.5em]
    &=\frac{mr+m^2-m}{r(r+m-1)(r+m)}\\[.5em]
    &=\frac{m(r+m-1)}{r(r+m-1)(r+m)}\\[.5em]
    &=\frac{m}{r(r+m)}\,.
\end{align*}

Setting $g_0=1$ and using the above result, we see that
\begin{align*}
    \sum_{\substack{v=0\\ \tj{j}{v+1}\in\tij}}^{z-1}\De(\tj{j}{v+1})&=\sum_{\substack{v=0\\ \tj{j}{v+1}\in\tij}}^{z-1}\frac{\calx_v}{g_{v}\cdot(g_{v}+\calx_v)}\cdot\De(\tj{j}{-1})\\[.5em]
    &=\sum_{\substack{v=0\\ \tj{j}{v+1}\in\tij}}^{z-1}\sum_{y=g_v}^{g_v+\calx_v-1}\frac{1}{y(y+1)}\cdot\De(\tj{j}{-1})\\[.5em]
    &=\sum_{y=1}^{d'-j'}\frac{1}{y(y+1)}\cdot\De(\tj{j}{-1})\\[.5em]
    &=\bigg(1-\frac{1}{d'-j'+1}\bigg)\cdot\De(\tj{j}{-1}).\tag{2}\label{eq:t-lower}
\end{align*}

Next, we consider $\tj{j}{z+1}$ (if $z=0$, this has already been dealt with above).
Recall that if $\tj{j}{v+1}\in\tij$, then $g_v=\sum_{\rho=1}^vh_\rho+1$.
\begin{align*}
    \De(\tj{j}{z+1})&=\fkd(G_1^\al\cdot i)\cdot\prod_{\substack{{v=1}\\ \tj{j}{v+1}\in\tij}}^{z-1}\binom{(i-(g_v+1)-(i-(g_z+1))+\calx_v}{\calx_v}\\[.5em]
    &\cdot\binom{(i-(g_{z}+1))-\caly+\sum^{z}_{\rho=1}h_\rho+1}{\sum^{z}_{\rho=1}h_\rho+1}\fkd(g_{z+1}^\omega)\cdot\prod_{\substack{v=1\\ \tj{j}{v+1}\not\in\tij}}^{z-1}\fkd(g_v)\\[.5em]
    &=\prod_{\substack{{v=1}\\ \tj{j}{v+1}\in\tij}}^{z-1}\binom{g_{z}-g_v+\calx_v}{\calx_v}\cdot\binom{i'-1}{g_z}\cdot\fkd(G_1^\al\cdot i)\fkd(g_{z+1}^\omega)\cdot\prod_{\substack{v=1\\ \tj{j}{v+1}\not\in\tij}}^{z-1}\fkd(g_v)\\[.5em]
    &=(g_z-g_1+\calx_1)(g_z-g_1+\calx_1-1)\dots(g_z-g_{z-1}+1)\cdot\frac{(i'-1)!}{g_z!(i'-g_z-1)!}\\[.5em]
    &\cdot\prod_{\substack{{v=1}\\ \tj{j}{v+1}\in\tij}}^{z-1}(\calx_v!)^{-1}\cdot\fkd(G_1^\al\cdot i)\fkd(g_{z+1}^\omega)\cdot\prod_{\substack{v=1\\ \tj{j}{v+1}\not\in\tij}}^{z-1}\fkd(g_v)\\[.5em]
    &=(g_z-1)!\cdot\frac{(i'-1)!}{g_z!(i'-g_z-1)!}\cdot\prod_{\substack{{v=1}\\ \tj{j}{v+1}\in\tij}}^{z-1}(\calx_v!)^{-1}\\[.5em]
    &\cdot\fkd(G_1^\al\cdot i)\fkd(g_{z+1}^\omega)\cdot\prod_{\substack{v=1\\ \tj{j}{v+1}\not\in\tij}}^{z-1}\fkd(g_v)\\[.5em]
    &=(g_z-1)!\cdot\frac{(i'-1)(i'-2)\dots(i'-g_z)}{g_z!}\cdot\prod_{\substack{{v=1}\\ \tj{j}{v+1}\in\tij}}^{z-1}(\calx_v!)^{-1}\\[.5em]
    &\cdot\fkd(G_1^\al\cdot i)\fkd(g_{z+1}^\omega)\prod_{\substack{v=1\\ \tj{j}{v+1}\not\in\tij}}^{z-1}\fkd(g_v)\cdot\prod_{\substack{{v=1}\\ \tj{j}{v+1}\in\tij}}^{z-1}(\calx_v!)^{-1}\\[.5em]
    &=\frac{(i'-1)(i'-2)\dots(i'-(d'-j'+1))}{g_z}\cdot\prod_{\substack{{v=1}\\ \tj{j}{v+1}\in\tij}}^{z-1}(\calx_v!)^{-1}\\[.5em]
    &\cdot\fkd(G_1^\al\cdot i)\fkd(g_{z+1}^\omega)\prod_{\substack{v=1\\ \tj{j}{v+1}\not\in\tij}}^{z-1}\fkd(g_v)
\end{align*}
Comparing this with \cref{eq:t-}, we see that
\begin{equation*}
   \De(\tj{j}{z+1})=\frac{i'+j'-d'-1}{i'(d'-j'+1)}\cdot\De(\tj{j}{-1}).\tag{3}\label{eq:t-last} 
\end{equation*}

Finally, we look at $\De(\tu{j}{s+1})$ with $0\le s\le x$.
In the following, we must have that $\tj{j}{0}=\tau_2$.
We first consider $\tu{j}{1}$.
\begin{enumerate}
    \item[a)] If $G_1=1$, then $\tu{j}{2}\in\tij$ and we have that
\begin{align*}
    \De(\tu{j}{1}) &=\fkd(G_1^\al)\cdot\binom{(i-1)-\caly+2}{2}\cdot\fkd(g_1^\omega)\\[.5em]
    &=\frac{(i'+1)\cdot i'}{2}\cdot\fkd(G_1^\al)\fkd(g_1^\omega)=\frac{1}{2}\cdot\De(\tj{j}{-1}).
\end{align*}
    \item[b)] Assume that $x\neq0$.
    If $G_1=0$, then $\tu{j}{2}\not\in\tij$ and let $2 \le y<x$ be minimal such that $\tu{j}{y}\not\in\tij$ and $\tu{j}{y+1}\in\tij$.
     We have three cases to consider:
     \begin{itemize}
         \item[$\bullet$] If $G_2<H_1$, we have that
            \begin{align*}
        \De(\tu{j}{1})&=\fkd(G_2^\al)\cdot\binom{i-(i-1)+\sum_{\rho=1}^{y-1}H_\rho}{\sum_{\rho=1}^{y-1}H_\rho}\binom{i-1-\caly+\sum_{\rho=1}^{y-1}H_\rho+2}{\sum_{\rho=1}^{y-1}H_\rho+2}\cdot\fkd(g_1^\omega)\\[.5em]
        &=(\sum_{\rho=1}^{y-1}H_\rho+1)\cdot\binom{i'+1+\sum_{\rho=1}^{y-1}H_\rho}{\sum_{\rho=1}^{y-1}H_\rho+2}\cdot\fkd(G_2^\al)\fkd(g_1^\omega)\\[.5em]
        &=\frac{\sum_{\rho=1}^{y-1}H_\rho+1}{(\sum_{\rho=1}^{y-1}H_\rho+2)}\cdot\De(\tj{j}{-1})=\frac{G_y}{(G_1+1)\cdot(G_y+1)}\cdot\De(\tj{j}{-1}).
            \end{align*}
        \item[$\bullet$] If $G_2-G_1=H_1$ (i.e.\ $G_2=H_1$).
        In this case, we only need to consider $G_{y'}=G_{y'-1}+H_{y'-1}$ for all $2\le y'\le y-1$ (that is the second term in the first row above is different), otherwise there exists some $2\le y'\le y-1$ such that $G_{y'}-G_1<\sum_{\rho=1}^{y'-1}H_\rho-1$ (i.e.\ $G_{y'}<\sum_{\rho=1}^{y'-1}H_\rho-1$).
        We have that
        \begin{align*}
            \De(\tj{j}{-1})&=\fkd(G_y^\al)\prod_{y'=2}^{y-1}\binom{i+G_{y'}-i+H_{y'}}{H_{y'}}\\[.5em]
            &\cdot\binom{i'+\sum_{\rho=1}^{y-1}H_{\rho}+1}{\sum_{\rho=1}^{y-1}H_{\rho}+1}\cdot ((i-1)-\caly+1)\cdot\fkd(g_1^\omega)\\[.5em]
            &=\prod_{y'=2}^{y-1}\binom{G_{y'}+H_{y'}}{H_{y'}}\binom{i'+G_y}{G_y}\cdot i'\cdot\fkd(G_y^\al)\fkd(g_1^\omega)\\[.5em]
            &=(G_{y-1}+H_{y-1})(G_{y-1}+H_{y-1}-1)\dots(H_1+1)\\[.5em]
            &\cdot\prod_{y'=2}^{y-1}(H_{y'}!)^{-1}\cdot i'\cdot\fkd(G_y^\al)\fkd(g_1^\omega)\\[.5em]
            &=(G_y-1)!\prod_{y'=2}^{y-1}(H_{y'}!)^{-1}\cdot i'\cdot\fkd(G_y^\al)\fkd(g_1^\omega)
        \end{align*}
        and
        \begin{align*}
            \De(\tu{j}{1})&=\fkd(G_y^\al)\prod_{y'=1}^{y-1}\binom{i+G_{y'}-(i-1)+H_{y'}}{H_{y'}}\cdot\binom{i'-1+\sum_{\rho=1}^{y-1}H_{\rho}+2}{\sum_{\rho=1}^{y-1}H_{\rho}+2}\cdot\fkd(g_1^\omega)\\[.5em]
            &=\prod_{y'=1}^{y-1}\binom{G_{y'}+H_{y'}+1}{H_{y'}}\cdot \binom{i'+G_y}{G_y+1}\cdot\fkd(G_y^\al)\fkd(g_1^\omega)\\[.5em]
            &=(G_{y-1}+H_{y-1}+1)(G_{y-1}+H_{y-1})\dots(H_1+2)\\[.5em]
            &\cdot\prod_{y'=1}^{y-1}(H_{y'}!)^{-1}\cdot (H_1+1)\binom{i'+G_y}{G_y+1}\cdot\fkd(G_y^\al)\fkd(g_1^\omega)\\[.5em]
            &=(G_y)!\prod_{y'=1}^{y-1}(H_{y'}!)^{-1}\cdot\frac{i'}{G_y+1}\cdot\fkd(G_y^\al)\fkd(g_1^\omega).
        \end{align*}
     So
     \[
        \De(\tj{j}{-1})=\frac{G_y}{(G_1+1)\cdot (G_y+1)}\cdot\De(\tj{j}{-1}).
     \]
     \item[$\bullet$] If $G_2-G_1-1>H_1-1$, then $x=0$.
     We consider this in case c).
      \end{itemize}
    \item[c)] If $x=0$, then $\tu{j}{1}$ is the highest index $\tu{j}{s}\in\tij$ and there are no more blue nodes above row $j-1$ (see \cref{eg:taus}).
    We have three cases to consider.
    \begin{itemize}
        \item[$\bullet$] If $G_1\ge2$, then row $j-2$ already crosses the blue line.
        It is easy to see that the computation is the same as in case a).
        \item[$\bullet$] If $G_{s+1}\le\sum_{\rho=1}^sH_\rho$ for all $1\le s\le \al-1$, then the rows above row $j-1$ never touch the blue line.
        We have that
        \begin{align*}
            \De(\tu{j}{1})&=\fkd(G_2^\al)\binom{i-(i-1)+\sum_{s=1}^\al H_s}{\sum_{s=1}^\al H_s}\binom{i-1-\caly+\sum_{s=1}^\al H_s+2}{\sum_{s=1}^\al H_s+2}\fkd(g_1^\omega)\\[.5em]
            &=(\sum_{s=1}^\al H_s+1)\cdot\binom{i'+1+\sum_{s=1}^\al H_s}{\sum_{s=1}^\al H_s+2}\cdot\fkd(G_2^\al)\fkd(g_1^\omega)\\[1em]
            \De(\tj{j}{-1})&=\fkd(G_2^\al)\binom{i-(i-1)+\sum_{s=1}^\al H_s+1}{\sum_{s=1}^\al H_s+1}(i-1-\caly+1)\cdot\fkd(g_1^\omega)\\[.5em]
            &=\binom{i'+1+\sum_{s=1}^\al H_s}{\sum_{s=1}^\al H_s+1}\cdot i' \cdot\fkd(G_2^\al)\fkd(g_1^\omega).
        \end{align*}
        In this case we also see that instead of $j'$ we have $\sum_{\rho=1}^\al H_\rho$ appearing in all the computations above, so abusing notation, we will write $j'$ for this sum.
        Hence
        \[
        \De(\tj{j}{0})=\frac{j'-1}{j'}\De(\tj{j}{-1})=\Big(1-\frac{1}{j'}\Big)\De(\tj{j}{-1}).
        \]
         \item[$\bullet$] If $G_{s+1}\le\sum_{\rho=1}^sH_\rho$ for some minimal $1\le s< y$ such that $G_{y+1}>\sum_{\rho=1}^{y}H_\rho$, then replacing $\sum_{\rho=1}^{\al-1}H_\rho$ with $\sum_{\rho=1}^{y}H_\rho$ in the previous case gives the same result.
    \end{itemize}
\end{enumerate}

Finally, we look at $\tu{j}{s+1}$ for $1\le s\le x$.
There are two cases to consider.
\begin{itemize}
    \item[a)] If $H_s\ge1$ and $\tu{j}{s+2}\in\tij$, we have that
\begin{align*}
    \De(\tu{j}{s+1})&=\fkd(G_{s+1}^\al)\cdot\binom{i+G_{s}-(i-1)+H_{s}-1}{H_{s}-1}\\[.5em]
    &\cdot\prod_1^{s-1}\binom{i+G_{q}-1-(i-1)+H_q}{H_q}\cdot\binom{i-1-\caly+\sum_1^sH_q+2}{\sum_1^sH_q+2}\cdot\fkd(g_1^\omega)\\[.5em]
    &=\binom{G_s+H_s}{H_s-1}\cdot\prod_1^{s-1}\binom{G_{q}+H_q}{H_q}\cdot\binom{i'+G_{s+1}}{G_{s+1}+1}\cdot\fkd(G_{s+1}^\al)\fkd(g_1^\omega)\\[.5em]
    &=\frac{G_{s+1}!}{(H_s-1)!(G_s+1)!}\cdot\prod_1^{s-1}\frac{(G_{q}+H_q)!}{H_q!G_q!}\cdot\frac{(i'+G_{s+1})!}{(G_{s+1}+1)!(i'-1)!}\cdot\fkd(G_{s+1}^\al)\fkd(g_1^\omega)\\[.5em]
    &=\frac{G_{s+1}!\cdot H_s}{(H_s)!(G_s+1)!}\cdot\prod_1^{s-1}\frac{(G_{q}+H_q)(G_{q}+H_q-1)\dots(G_q+1)}{H_q!}\\[.5em]
    &\cdot\frac{(i'+G_{s+1})(i'+G_{s+1}-1)\dots i'}{(G_{s+1}+1)!}\cdot\fkd(G_{s+1}^\al)\fkd(g_1^\omega)\\[.5em]
    &=\frac{G_{s+1}!\cdot H_s}{(H_s)!(G_s+1)!}\frac{(G_s)!}{\prod_1^{s-1}H_q!}\cdot\frac{(i'+G_{s+1})(i'+G_{s+1}-1)\dots i'}{(G_{s+1}+1)!}\cdot\fkd(G_{s+1}^\al)\fkd(g_1^\omega)\\[.5em]
    &=\frac{H_s\cdot G_{s+1}!G_s!}{(G_s+1)!(G_{s+1}+1)!}\De(\tj{j}{-1})\\[.5em]
    &=\frac{H_s}{(G_s+1)(G_{s+1}+1)}\De(\tj{j}{-1})=\frac{H_s}{(G_s+1)(G_s+H_s+1)}\De(\tj{j}{-1}).
\end{align*}

    \item[b)] If $H_s>1$ and $\tu{j}{s+2}\not\in\tij$.
     Let $s+2 \le y<x$ be minimal such that $\tj{j}{y}\not\in\tij$ and $\tj{j}{y+1}\in\tij$.
    The proof is analogous to the proof of case b) in $\De(\tu{j}{1})$.
    In this case,
    \[
    \De(\tj{j}{s+1})=\frac{\sum_{\rho=s}^{y-1}H_\rho}{(G_s+1)\cdot (G_y+1)}
    \]
\end{itemize}

Hence
\begin{align*}
    \sum_0^{x}\De(\tu{j}{s})&=\sum_{1}^{j'-1}\frac{1}{r(r+1)}\cdot\De(\tj{j}{-1})\\[.5em]
    &=\big(1-\frac{1}{j'}\big)\cdot\De(\tj{j}{-1}).\tag{4}\label{eq:t-higher}
\end{align*}

Combining our results from \cref{eq:t-zero,eq:t-lower,eq:t-last,eq:t-higher}, we have that
\begin{align*}
    &\De(\tj{j}{0})+\sum_0^{z}\De(\tj{j}{v+1})+\De(\tj{j}{z+1})+\sum_0^{x}\De(\tu{j}{s+1})\\[.5em]
    &=\frac{i'+j'}{i'j'}\cdot\De(\tj{j}{-1})+\big(1-\frac{1}{d'-j'+1}\big)\cdot\De(\tj{j}{-1})\\[.5em]
    &+\frac{i'+j'-d'-1}{i'(d'-j'+1)}\cdot\De(\tj{j}{-1})+\big(1-\frac{1}{j'}\big)\cdot\De(\tj{j}{-1})\\[.5em]
    &=2\De(\tj{j}{-1}).
\end{align*}
    
\textbf{Case ii)}

Here $\tij=\{\tau_{(j,0)},\tau_{(j,-1)},\tau_{(j,1)},\dots,\tj{j}{z+1}\}$.
We have that $\tj{j}{0}=\tau_1,G_1=1$ and $g_1\ge2$.
We only need to consider $\De(\tj{j}{0})$, all other $\De(\tj{j}{v+1})$ for $0\le v\le z$ can be calculated as before.
\begin{itemize}
    \item[a)] If $g_{z+1}=\sum_{\rho=1}^{z+1}h_\rho$ and $g_v\ge \sum_{\rho=1}^{v}h_\rho$ for $z+2\le v\le \omega$ (case a) above):
\begin{align*}
    \De(\tj{j}{0})&=\fkd(G_1^\al)\cdot(i'+1)\fkd(g_1^z)\binom{i'-g_{z+1}+\sum_{\rho=1}^{z+1}h_\rho}{\sum_{\rho=1}^{z+1}h_\rho}\fkd(g_{z+2}^\omega)\\[.5em]
    &=(i'+1)\binom{i'}{g_{z+1}}\cdot\fkd(G_1^\al)\fkd(g_1^z)\fkd(g_{z+2}^\omega).
\end{align*}
Then
\begin{align*}
    \De(\tj{j}{-1})&=\fkd(G_1^\al)\cdot(i-1-(i-g_{z+1}+1)\fkd(g_1^z)\binom{i'-g_{z+1}+\sum_{\rho=1}^{z+1}h_\rho+1}{\sum_{\rho=1}^{z+1}h_\rho+1}\fkd(g_{z+2}^\omega)\\[.5em]
    &=g_{z+1}\cdot\binom{i'+1}{g_{z+1}+1}\cdot\fkd(G_1^\al)\fkd(g_1^z)\fkd(g_{z+2}^\omega).
\end{align*}

\item[b)] If $g_{z+1}<\sum_{\rho=1}^{z+1}h_\rho$, then
\begin{align*}
    \De(\tj{j}{0})&=\fkd(G_1^\al)\cdot(i-(i-g_{z+1})+1)\cdot\fkd(g_{1}^\omega)=(i'+1)\cdot\fkd(G_1^\al)\fkd(g_1^\omega)\\[1em]
    \De(\tj{j}{-1})&=\fkd(G_1^\al)\cdot(i-1-(i-g_{z+1})+1)\cdot\fkd(g_{1}^\omega)=i' \cdot\fkd(G_1^\al)\fkd(g_1^\omega).
\end{align*}
\end{itemize}

In both cases
\[
\De(\tj{j}{0})=\frac{i'+1}{i'}\cdot\De(\tj{j}{-1}).
\]
As $j'=1$, we have that
\begin{align*}
    &\De(\tj{j}{0})+\sum_0^{z-1}\De(\tj{j}{v+1})+\De(\tj{j}{z+1})\\[.5em]
    &=\frac{i'+1}{i'}\cdot\De(\tj{j}{-1})+\big(1-\frac{1}{d'}\big)\cdot\De(\tj{j}{-1})+\frac{i'-d'}{i'\cdot d'}\cdot\De(\tj{j}{-1})\\[.5em]
    &=2\De(\tj{j}{-1}).
\end{align*}

\textbf{Case iii)}

Here $\tij=\{\tau_{(j,0)},\tau_{(j,-1)},\tau^{(j,1)},\dots,\tu{j}{x+1}\}, G_1\neq 1$ and $g_1=1$.
Once again, we only need to consider $\De(\tj{j}{0})$.
\begin{itemize}
    \item[a)] If $h_1=1$, we have that
\begin{align*}
    \De(\tj{j}{0})&=\fkd(G_2^\al)\binom{i-\caly+\sum_{\rho=1}^xH_\rho+2}{\sum_{\rho=1}^xH_\rho+2}\cdot(i-1-\caly+1)\cdot\fkd(g_2^\omega)\\[.5em]
    &=\binom{i'+j'}{j'}\cdot i'\cdot\fkd(G_2^\al)\fkd(g_2^\omega)\\[1em]
    \De(\tj{j}{-1})&=\fkd(G_2^\al)\binom{i-(i-1)+\sum_{\rho=1}^xH_\rho+1}{\sum_{\rho=1}^xH_\rho+1}\binom{i-1-\caly+\sum_{\rho=1}^xH_\rho+3}{\sum_{\rho=1}^xH_\rho+3}\cdot\fkd(g_2^\omega)\\[.5em]
    &=j'\cdot\binom{i'+j'}{j'+1}\cdot \fkd(G_2^\al)\fkd(g_2^\omega).
\end{align*}
    \item[b)] If $h_1>1$, we have that
\begin{align*}
    \De(\tj{j}{0})&=\fkd(G_2^\al)\binom{i-(i-1)+\sum_{\rho=1}^xH_\rho+2}{\sum_{\rho=1}^xH_\rho+2}\cdot\fkd(g_1^\omega)=(j'+1)\cdot\fkd(G_2^\al)\fkd(g_1^\omega)\\[1em]
    \De(\tj{j}{-1})&=\fkd(G_2^\al)\binom{i-(i-1)+\sum_{\rho=1}^xH_\rho+1}{\sum_{\rho=1}^xH_\rho+1}\cdot\fkd(g_1^\omega)=j'\cdot \fkd(G_2^\al)\fkd(g_1^\omega).
\end{align*}
\end{itemize}
In both cases
\[
\De(\tj{j}{0})=\frac{j'+1}{j'}\De(\tj{j}{-1})
\]
and so
\[
    \De(\tj{j}{0})+\sum_{s=0}^{x}\De(\tu{j}{s+1})=\Bigg(\frac{j'+1}{j'}+\Big(1-\frac{1}{j'}\Big)\Bigg)\cdot\De(\tj{j}{-1})=2\De(\tj{j}{-1}).\qedhere
\]
\end{answer}
\end{proof}

\begin{eg}
    Continuing with \cref{eg:taus}, we have that
    \[
        \De(\tj{4}{-1})=\De(10,8,7,6,5,3^2,2^3)=\frac{12\cdot11\cdot9\cdot7\cdot6\cdot5\cdot4\cdot3}{2}=\De(\tj{j}{-1}).
    \]
    Then
    \begin{align*}
        \De(\tj{4}{0})&=\De(10,8,7^2,5,3^2,2^3)=\frac{8}{3\cdot5}\cdot\De(\tj{j}{-1}),\\[.5em]
        \De(\tj{4}{1})&=\De(10,8,7,5^2,3^2,2^3)=\frac{1}{2}\cdot\De(\tj{j}{-1}),\\[.5em]
        \De(\tj{4}{2})&=\De(10,8,7,4^2,3^2,2^3)=\frac{2}{2\cdot4}\cdot\De(\tj{j}{-1}),\\[.5em]
        \De(\tj{4}{3})&=\De(10,8,7,4,2^6)=\frac{1}{4\cdot5}\cdot\De(\tj{j}{-1}),\\[.5em]
        \De(\tu{4}{1})&=\De(10,8,6^2,5,3^2,2^3)=\frac{1}{2}\cdot\De(\tj{j}{-1}),\\[.5em]
        \De(\tu{4}{2})&=\De(10,6^3,5,3^2,2^3)=\frac{1}{2\cdot3}\cdot\De(\tj{j}{-1}).
    \end{align*}
    So we have
    \[
    \sum_{\rho=0}^{3}\De(\tj{4}{\rho})+\sum_{\rho=1}^{2}\De(\tu{4}{\rho})=2\De(\tj{j}{-1})
    \]
    as required.
\end{eg}

At this point, we turn our attention back to tableaux of hook shapes and give a very useful way to label (some of) them.

\begin{defn}\label{def:tableau}
    Let $\la=(a,1^b)$ and $\ttt\in\P(\la)$.
    Fix a set of positive even integers $\cale=\{e_1,\dots,e_\rho\}$ such that $e_i+2=e_{i+1}$ for all $1\le i\le \rho-1$ and $1\le\rho\le a-1$.
    We define $\sfz(\la)_\cale\subset\P(\la)$ to be the set of all $\ttt\in\P(\la)$, such that for some $v_i<v_{i+1}$, $\Arm(\ttt)$ contains entries $\tp_{r_1},\tp_{r_2},\dots,\tp_{r_s},\cale,\tp_{r_{s+1}}\tp_{r_{s+2}},\dots,\tp_{r_d},(n)$ in this order where $2d+\rho+1(+1)=a$ (we see that $n\in\Arm(\ttt)$ can only happen if $n$ is even).
    If $\ttt_{\lzero}$ is the \emph{`least dominant'} pair-tableau in $\sfz(\la)_\cale$, that is $\Arm(\ttt_{\lzero})$ contains entries $\tp_{e_1-2x-1},\tp_{e_1-2x+1},\dots,\tp_{e_1-3},\tp_{e_1-1}, \cale,\newline\tp_{e_\rho+3},\tp_{e_\rho+5},\dots,\tp_{n-2},\tp_n$ (or shift each entry by one if $n$ is even) for some $1<x$ such that $e_1-2x-1\ge3$, we define $\ttt_{(1)}:=\ttt_{\lzero}$ with corresponding $Z_{1}=(0,0,\dots,0)$ of length $d$.
    
    Otherwise, we label each $\ttt_{(i)}\in\sfz(\la)_\cale$ by $Z_{i}=(\de_1,\de_2,\dots,\de_d)$ where
    \[
    \de_j=\begin{cases}
    \frac{1(j)-i(j)}{2}-\rho\quad\qquad\text{if }1(j)>e_\rho\text{ and } i(j)<e_1\\[1em]
    \frac{1(j)-i(j)}{2}\quad\qquad\qquad\text{otherwise}
    \end{cases}
    \]
    and $\tp_{1(j)}$ (resp.\ $\tp_{i(j)}$) is the $j$-th pair in $\Arm(\ttt_{(1)})$ (resp.\ $\Arm(\ttt_{(i)})$) and $1\le j\le d$.
    We see that $0\le\de_j\le (b-\rho)/2$.

    In particular, for $\ttt_{(i)},\ttt_{(j)}\in\sfz(\la)_\cale$, we have that $i<j$ if $Z_i<Z_j$ in the lexicographic order on partitions.
\end{defn}

\begin{eg}
    Let $\la=(5,1^4)$ and $\cale=\varnothing$.
    Then there are six tableaux in $\sfz(\la)_\cale$.
    \begin{center}
\begin{tikzpicture}
\tgyoung(0cm,3.5cm,16789,2,3,4,5)
\tyoung(4cm,3.5cm,14589,2,3,6,7)
\tyoung(8cm,3.5cm,14567,2,3,8,9)
\tyoung(0cm,0cm,12389,4,5,6,7)
\tyoung(4cm,0cm,12367,4,5,8,9)
\tyoung(8cm,0cm,12345,6,7,8,9)
\node at (-.75, 3)   (f) {$\ttt_{(1)}=$};
\node at (3.25, 3)   (a) {$\ttt_{(2)}=$};
\node at (7.25, 3)   (b) {$\ttt_{(3)}=$};
\node at (-.75, -.5)   (c) {$\ttt_{(4)}=$};
\node at (3.25, -.5)   (d) {$\ttt_{(5)}=$};
\node at (7.25, -.5)   (e) {$\ttt_{(6)}=$};
\end{tikzpicture}
\end{center}
Indeed, we have $Z_1=(0,0)<Z_2=(1,0)<Z_3=(1,1)<Z_4=(2,0)<Z_5=(2,1)<Z_6=(2,2)$.
\end{eg}

\begin{rem}
    Hence, now we have a way to label certain hook tableaux by partitions.
    In fact, we will see later that $\mu=(\mu_1,\dots,\mu_Z)$ uniquely determines the set
\[
\cale=\{\ttt^\mu(m,\mu_m) \text{ is even} \mid 1\le m\le Z\},
\]
and by \cref{action,y-action} if $\varphi(z^\mu)=v=\sum\ct\vtt \in S^\la$, then $\ct\neq0$ only if $\ttt\in\sfz(\la)_\cale$. 
\end{rem}

If $Z_p=\tj{j}{y}$ or $\tu{j}{y}$ labels $\ttt_{(p)}$ for some fixed $i$, we write $Z_p\in\tij$ and label the corresponding basis vector by $\vt{j,y}$.
We will always make sure that it is clear from the context whether we are referring to $\tj{j}{y}$ or $\tu{j}{y}$.
To simplify notation, we also write $v_{s_k\tau(j,y)}$ for the basis vector corresponding to the tableau $s_k\ttt_{(p)}$.

For $k$ odd such that $\psi_k\vtt\neq0$, we will rename the different actions appearing in \cref{action} depending on the number of pairs involved in the resulting rotation (see \cref{cor:psi-leg,cor:psi-arm}).
\begin{itemize}
    \item $0-0$ move: $\tp_k\in\Arm(\ttt)$ and $\tp_{k+2}\in\Leg(\ttt)$.
    \item $1-1$ move: $\tp_{k+2}\in\Arm(\ttt)$ and $\tp_{k}\in\Leg(\ttt)$.
    \item $(L)-(L{-}1)$ move: $\tp_{k},\tp_{k+2}\in\Leg(\ttt)$.
    Here $\ttt$ is as in \cref{pair-leg} and $L=m+1$.
    In general, we will refer to these as \emph{leg moves}.
    \item $(A{-}1)-(A)$ move: $\tp_{k},\tp_{k+2}\in\Arm(\ttt)$.
    Here $\ttt$ is as in \cref{pair-arm} and $A=m+1$.
    In general, we will refer to these as \emph{arm moves}.
\end{itemize}

Abusing notation, if $\tj{j}{y}$ or $\tu{j}{y}$ labels $\ttt_{(p)}\in\sfz(\la)_\cale$, we write $\Arm(\tau(j,y))$ and $\Leg(\tau(j,y))$ for $\Arm(\ttt_{(p)})$ and $\Leg(\ttt_{(p)})$, respectively.
For the rest of this section, we also assume that $\tp_k$ is the $j$th pair in $\Arm(\tau(j,0))$ and simply write $\tau(y)$ instead of $\tj{j}{y}$ and $\tu{j}{y}$.

By \cref{action}, we see that if $\psi_k\neq0$, then the pair $\tp_k$ is rotated into $\Arm(\ttt)$.
Assume that $\varphi\in\hom(S^\mu,S^\la)\neq0$ and $\varphi(z^\mu)=v\sum\ct\vtt$.
When considering the action of $\psi_k$ on $v$, we would like to break $v$ into smaller sums such that
\begin{equation}\label{eq:double-sum}
    \psi_kv=\psi_k\sum_{x}c_{\ttt_{(x)}}v_{\ttt_{(x)}}=\sum_x\sum_yc_{\ttt_{(y)}}\underbrace{\psi_kv_{\ttt_{(y)}}}_{=a_y\cdot v_{s_k\ttt_{(x)}}}.
\end{equation}
(By \cref{action}, we already know that $a=1$ or $-2$.)
The rest of this section is dedicated to finding the right set of tableaux and coefficients and proving that
\[
\psi_k\sum_yc_{\ttt_{(y)}}v_{\ttt_{(y)}}=0.
\]

\begin{lem}\label{lem:delta-tableaux}
    If $\ttt_{(p)}\in\sfz(\la)_\cale$ with $Z_p\in\tij$ and $\tp_k$ is the $j$th pair in $\Arm(\tau(0))$ and $e_1\neq k+1$, then
    \begin{itemize}
        \item[i)] $\psi_kv_{\ttt_{(p)}}=a_\tau\cdot v_{s_k \tau_{(0)}}$ for all odd $3\le k\le n-2$ where $a_\tau=-2$ if $Z_p=\tau_{(-1)}$ and $a_\tau=1$ otherwise,
        \item[ii)] if $k=n-1$ is odd, then
        \[
        \psi_k\vtd{p}=
        \begin{cases}
            v_{s_k\ttt_{(p)}}&\text{if }n\in\Leg(\tau(0))\\
            0&\text{if }n\in\Arm(\tau(0)),
        \end{cases}
        \]
        \item[iii)] $\psi_r\vtd{p}=v_{s_r\ttt_{(p)}}$ if $r\in\cale$.
    \end{itemize}
\end{lem}

\begin{proof}
    Case iii) follows from \cref{y-pair}.
    If $k=n-1$ and $n\in\Arm(\tau(0))$, then $\psi_{n-1}\vtd{p}=0$ by \cref{action} as $n\in\Arm(\tau(v))$ for all $\ttt_{(p)}\in\sfz(\la)_\cale$.
    If $k=n-1$ and $n\in\Leg(\tau(0))$, the statement holds by \cref{action}.
    
    For case i), we first consider leg moves, that is for $\vt{y}$ the corresponding $\ttt_{(p)}$ is labelled by $Z_p=\tau_{(y)}$.
    Since $e_1\neq k+1$, it is easy to see that for all $3\le k\le n-2$,
    \begin{align*}    
        \psi_k\vt{0}&=v_{s_k\tau(0)}\,\qquad\qquad\text{as }\tp_k\in\Arm(\tau(0)),\tp_{k+2}\in\Leg(\tau(0))\\[.5em]
        \psi_k\vt{-1}&=-2v_{s_k\tau(0)}\qquad\;\;\;\text{as }\tp_k\in\Leg(\tau(-1)),\tp_{k+2}\in\Arm(\tau(-1))
    \end{align*}

    Also, by \cref{action}
    \[
        \psi_k\vt{1}=v_{s_k\tau(0)}
    \]
    as $\tp_k,\tp_{k+2}\in\Leg(\tau(1))$ and $\tp_{x}\in\Arm(\tau(1))$ where $x\ge k+4$ ($\tp_x$ must exist by definition).
    The condition $e_1\ge k+3$ or $e_1\le k-3$ ensures that for all $1\le v\le z$ we have that $\tp_k,\tp_{k+2},\tp_x\in\Leg(\tau(v+1))$ (as $g_1\ge2$).
    Further, the position of $\cale$ does not actually matter, since by \cref{action} if $\tp_k,\tp_{k+2}\in\Leg(\tau(v+1))$, to compute $\psi_k\vt{v+1}$ we only need to count the pairs in $\Leg(\tau(v+1))$ and $\Arm(\tau(v+1))$, starting with $\tp_k$ and even entries are ignored.
    After $\tp_k$, we see that there are $g_1$ pairs in $\Leg(\tau(v+1))$ which are smaller than the one in position $j$ in $\Arm(\tau(v+1))$.
    The next $h_1$ pairs are in $\Arm(\tau(v+1))$, followed by $g_2-g_1$ pairs in the leg.
    In general, each $g_m$ ($1\le m\le v$) many pairs in $\Leg(\tau(v+1))$ are followed by $g_{m+1}-g_m$ pairs in $\Arm(\tau(v+1))$.
    In total there are $\sum_1^vh_m+1=g_v$ pairs in $\Arm(\tau(v+1))$ between positions $j$ and $j+g_{v+1}-1$.
    Then, we count the number of pairs in $\Leg(\tau(v+1))$ as 
    \[
    g_1+g_2-g_1+g_3-g_2+\dots+g_{v}-g_{v-1}+1=g_{v}+1
    \]
    where the plus one accounts for the pair $\tp_k$.

    Assume that $\tau\in\tij$.
    Using the facts that for all $1\le r\le \sum_{\rho=1}^vh_\rho$, $\tau_{j+r}\le \tau_j-1-r$, and that for all $\sum_{\rho=1}^{w-1}h_\rho< r\le \sum_{\rho=1}^{w}h_\rho$ with $1\le w\le v$, we have that $\tau_{j+r}=i-g_w$, we see that
    \begin{align*}
        i-g_w=\tau_{j+r}&\le i-1-\sum^w_{\rho=1}h_\rho\\
        &\Rightarrow g_w\ge\sum_{\rho=1}^{w}h_\rho+1
    \end{align*}
    In particular, if $g_w= \sum_{\rho=1}^{w}h_\rho+1$, then $\tj{j}{w+1}\in\tij$.
    Hence $g_v$ is the minimal number of pairs in $\Arm(\tau(v+1))$ in the sense of \cref{pair-leg}.
    Thus all $\ttt_{(p)}$ labelled by $Z_p\in\tij$ satisfy the conditions of case 3 in \cref{action}, so $\psi_kv_{\ttt_{(p)}}=v_{s_k\tau(0)}$ as claimed.

    Arm moves can be proved in a similar way.
\end{proof}

Using the previous lemma, we will show that setting $\ct=\De(\ttt)$ and $\ttt_{(y)}\in\tij$ in \cref{eq:double-sum}, we obtain a homomorphism given by
\[
\psi_kv=\psi_k\sum_{\ttt_{(x)}\in\sfz(\la)_\cale}c_{\ttt_{(x)}}v_{\ttt_{(x)}}=\psi_k\sum_{\substack{1\le j\le R\\ 0\le i\le \frac{b-\rho}{2}}}\underbrace{\sum_{\tau\in\tij} \De(\tau)v_{\tau}}_{=:\cals_{i,j}}=\sum_{\substack{1\le j\le R\\ 0\le i\le \frac{b-\rho}{2}}}\sum_{\tau\in\tij} \De(\tau)\underbrace{\psi_kv_{\tau}}_{=a_\tau v_{s_k\tau(0)}}
\]
where $\De(\tau)$ is computed according to \cref{def:coeff}.

\begin{rem}
    If $\cals_{i,j}$ is as above and $\vert \tij\vert=m$, then it is easy to see that $2\le m \le R+2$.
    The lower bound is clear, as by definition we have that $\{\tj{j}{0},\tj{j}{-1}\}\subseteq \tij$.
    For the upper bound, we just need to count the number of possible leg and arm moves, which gives us $R-j+1+j-1+2=R+2$ where the plus two accounts for the moves $0-0$ and $1-1$.
\end{rem}

\begin{prop}\label{prop:main}
    Suppose that $\ttt_{(p)}\in\sfz(\la)_\cale$ with corresponding $Z_p$ and that $3\le k\le n-2$ is odd such that $\tp_k$ is the $j$th (where $1\le j\le R$) pair in $\Arm(\ttt_{(p)})$ and $\de_j=i$ (where $0\le i\le \frac{b-\rho}{2}$) in $Z_p$ .
    Then for each pair of integers $(i,j)$, we have that $\psi_kv_{\ttt_{(p)}}=v_{s_k\tau(0)}$ if and only if $e_1\neq k+1$ and $Z_p\in\tij$.
\end{prop}

\begin{proof}
    Fix $3\le k\le n-2$.
    We have already shown in \cref{lem:delta-tableaux} that if $Z_p\in\tij$ and $e_1\neq k+1$, then $\psi_kv_{\ttt_{(p)}}=c\cdot\vt{0}$ where $c=-2$  if $Z_p=\tau_{(-1)}$ and $c=1$ otherwise.
    
    To prove the converse, suppose $e_1=k+1$.
    In this case,
    \[
    \psi_k\vtt=0
    \]
    for all $\ttt\in\sfz(\la)_\cale$ by \cref{action} since we either have $k-1,k,k+1\in\Arm(\ttt),k+2\in\Leg(\ttt)$ or $k+1\in\Arm(\ttt),k-1,k,k+2\in\Leg(\ttt)$.
    For the rest of the proof, we assume that $e_1\neq k+1$.
    
    By \cref{action}, we see that
    \[
    \psi_k v_{\ttt_{(p)}}=-2v_{s_{k-1}s_{k+1}s_k\ttt_{(p)}}
    \]
    if and only if $\tp_{k+2}\in\Arm(\ttt)$ and $\tp_{k}\in\Leg(\ttt)$ which is precisely the case when $Z_p=\tau(-1)$ (then we also see that $s_{k-1}s_{k+1}s_k\ttt_{(p)}=s_k\tau(0)$).
    Since $e_1\neq k+1$, $\tp_{k+2}$ exists for all $3\le k\le n-2$.
    We also see that if $k=n-1$, then $j=R$ and $i=0$, so $\tau(-1)$ is not a valid partition and $\tij=\varnothing$.
    
    Assume that $\psi_kv_{\ttt_{(p)}}\neq0$ such that the resulting rotation changes at most one pair in $\Arm(\ttt_{(p)})$.
    It is clear that there are at most three possible moves ($0-0,1-1,2-1$) and we must have that $\de_{j-1}\ge i+1$ (as otherwise the $1-2$ move can occur) and $\de_{j+1}=i-1$ or $\de_{j+1}\le i-2$ such that if $\de_{j+1}=i-x$, then it is repeated either less than or more than $x$ times.
    We see that any such $\ttt_{(p)}$ is labelled by some $Z_p\in\tij$.
 
    Next, assume that $\psi_kv_{\ttt_{(p)}}\neq0$ such that the resulting rotation changes at most two pairs in $\Arm(\ttt_{(p)})$ (these moves are three listed in the previous paragraph plus $3-2$ and $1-2$).
    \begin{itemize}
    \item If $\de_{j-1}=\de_j=\de_{j+1}= i$, then $\psi_kv_{\ttt_{(p)}}$ is either zero or the resulting rotation moves $\tp_k$ from the $j$th position.
    \item If $\de_{j+1}=i-1$, then the $0-0$ move is allowed by.
    But by \cref{action}, we see that leg moves are only possible if $\tp_k$ and $\tp _{k+2}$ are both in the leg and pair in position $j$ is larger than $k+2$, which requires $\de_{j+1}\le i-2$.
    If $\de_{j-1}\ge i+1$, then we already saw that the rotation can involve at most one pair in the arm.
    If $\de_{j-1}= i$, we must have that $\de_{j-2}\ge i+1$, otherwise the $2-3$ move would be allowed.
    \item If $\de_{j-1}=i=\de_j$, and $\de_{j+1}
    \le i-x$ ($x\ge2$), then $x$ must be as above and $\de_{j-2}\ge i+2$.
    If $\de_{j-1}=i+1$, we already know that the rotation can change at most one pair in the arm.
    \end{itemize}
    We see that each $\ttt_{(p)}$ described above is labelled by $Z_p\in\tij$.
    
    Next, assume that $\psi_kv_{\ttt_{(p)}}\neq0$ such that the resulting rotation changes at least three pairs in $\Arm(\ttt_{(p)})$ and that $\tj{j}{0}$ is of the form as in \cref{def:tau}.
    
    Suppose that for some $1\le v'\le \omega$ such that $\tj{j}{v'}\not\in\tij$.
    It is clear that if $g_{v'}\neq \sum_{\rho=1}^{v'}h_\rho+1$, then by \cref{action} and a similar argument as in the proof of \cref{lem:delta-tableaux}, we have that $\psi_k\vt{v'}=0$.

    Finally, suppose that $g_{v'}= \sum_{\rho=1}^{v'}h_\rho+1$ and for some $1\le r\le \sum^{v'}_{\rho=1}h_\rho$, we have that $\tau_{j+r}>\tau_j-2-r$.
    This implies that $g_{w}< \sum_{\rho=1}^{w}h_\rho+1$ for some $w\le v'$.
    In particular, looking at the proof of \cref{lem:delta-tableaux}, the minimal number of pairs is not $g_{v'}$ in this case, hence $\psi_k\vt{v'}\neq v_{s_k\tau(0)}$.

    Arm moves can be considered using a similar reasoning.
\end{proof}

\begin{cor}\label{cor:coeff}
    We have that $\psi_k\cals_{i,j}=0$ for all odd $3\le k\le n-2,1\le j\le R$ and $0\le i\le \frac{b-\rho}{2}$.
\end{cor}

\begin{proof}
    By \cref{prop:2c,lem:delta-tableaux}, we have that
    \begin{align*}
    \psi_k\cals_{i,j}&=\psi_k\sum_{\tau\in\tij}\De(\tau)v_\tau\\[.5em]
    &=\psi_k\big(\sum_{v=0}^{z+1}\De(\tj{j}{v})\vt{v}+\sum_{s=1}^{x+1}\De(\tu{j}{s})\vt{s}\big)+\psi_k\De(\tj{j}{-1})\vt{-1}\\[.5em]
    &=2\De(\tj{j}{-1}) v_{s_k\tau(0)}+(-2)\De(\tj{j}{-1}) v_{s_k\tau(0)}=0.
    \qedhere
    \end{align*}
\end{proof}

\begin{prop}\label{prop:number}
   If $\la=(a,1^b)$ with $a\ge3$ and $b\ge2$, then every $\ttt\in\sfz(\la)_\cale$ appears in at least one $\cals_{i,j}$ for some odd $3\le k\le n-2$.
\end{prop}

\begin{proof}
    We see that $\ttt\in\sfz(\la)_\cale$ does not appear in $\cals_{i,j}$ if and only if $\psi_k\vtt=0$ for all odd $3\le k\le n-2$, but this is impossible.
\end{proof}

%% file: parts/homs-2.tex
\subsection{Homomorphisms from hook partitions}\label{subs:hooks}

In this section, we study homomorphisms from $S^\mu$ to $S^\la$ where $\mu$ is also a hook shape.
First, we look at $\mu=(1^n)$ or $(n)$ and completely determine $\hom(S^\mu,S^\la)$ in these cases.
Next, we generalise these results for all $\mu=(c,1^d)$ by finding at least one nonzero homomorphism if it exists.
The results are summarised in \cref{hook-result}.

\begin{rem}
   As we will see, all the homomorphisms constructed in \cref{subs:hooks,subs:nonhooks} have integer coefficients and leading coefficient one, so our results hold for any field $\bbf$ with $\Ch(\bbf)=p$, we only need to reduce the coefficients modulo $p$. 
\end{rem}

\subsubsection{\texorpdfstring{$\mu=(1^n)$ or $\mu=(n)$}{mu row or column}} \label{column}

It is clear that if $\varphi\in\hom(S^\mu,S^\la)$ then $\varphi(z^\mu)=v=\sum\ct\vtt$ such that $\bi^\mu=\bi^\ttt$.
This together with \cref{arm-odd,leg-odd} imply that 
\[
\varphi(z^\mu)=v=\sum_{\substack{\ttt\in\P(\la)\\\bi^\mu=\bi^\ttt}}\ct\vtt.
\]
By \cref{y-pair}, we see that $(J^\mu_1+J^\mu_2)v=0$ for all such $v$.
If $\mu=(n)$, there are no Garnir nodes and if $\mu=(1^n)$, the Garnir relations are given by $\psi_kz^\mu=0$ for all $1\le k\le n-1$.
Hence in order to prove the existence of a homomorphism, we only need to show that $\psi_kv=0$ for all $2\le k \le n-1$ and this will ensure that $(J^\mu_3+J^\mu_4)v=0$.
The residue conditions also imply that all entries in $\ttt$ must appear in pairs (except if $n$ is even, in which case $n$ is not in a pair) and so $\psi_r\vtt=0$ for all even $2\le r\le n-2$.

\begin{rem}
    Considering \cref{arm-odd,leg-odd}, it might seem that we could still have $r-1,r+2\in\Arm(\ttt)$ and $r,r+1\in\Leg(\ttt)$ for some odd $r$ such that $\bi_{r-1}=\bi_{r+1}$, however by a similar argument as in the remark after \cref{standard-y}, we see that this is impossible.
\end{rem}

First, we consider $n$ odd.

\begin{prop}\label{oddcolumn}
Let $n$ be odd, $\mu=(1^n)$ or $(n)$ and $\la=(a,1^b)$.
We have that $\hom(S^\mu,S^\la)$ is one dimensional, spanned by
\[
\varphi(z^\mu)= v=\sum_{\substack{\ttt\in\P(\la)\\ \bi^\mu=\bi^\ttt}} \De(\ttt)\vtt,
\]
if $a$ is odd and zero otherwise.
\end{prop}

\begin{proof}
If $n$ is odd and $a$ is even, then the modules $S^\mu$ and $S^\la$ lie in different blocks and so $\dim\hom(S^\mu,S^\la)=0$.
So we can assume that $a$ is odd.
As $\hom(S^\nu,S^{(1^n)})$ is at most one dimensional, with the only possible homomorphism given by $z^\nu\mapsto v_{\ttt^{(1^n)}}$, we can use the isomorphism given in the introduction to show that $\dim\hom(S^\mu,S^\la)$ is either zero or one.
Thus, once we find a nonzero $v$, then $v$ is unique up to scalar multiplication.
By \cref{prop:number}, we see that all elements of the set $\{\ttt\in\sfz(\la)_\varnothing\mid \bi^\mu=\bi^\ttt\}$ appear in $v$ and by \cref{cor:coeff}, $\psi_kv=0$ for odd $2\le k\le n-1$, hence $(J^\mu_3+J^\mu_4)v=0$ and we proved the statement.
\end{proof}

\begin{eg}
Let $\mu=(1^9)$ and $\la=(5,1^4)$. Then $\dim\hom(S^{(1^9)},S^{(5,1^4)})=1$ and
\begin{center}
\begin{tikzpicture}
\tgyoung(0cm,3.5cm,16789,2,3,4,5)
\tyoung(3.5cm,3.5cm,14589,2,3,6,7)
\tyoung(7cm,3.5cm,14567,2,3,8,9)
\tyoung(0cm,0cm,12389,4,5,6,7)
\tyoung(3.5cm,0cm,12367,4,5,8,9)
\tyoung(7cm,0cm,12345,6,7,8,9)
\node at (-.75, 3)   (f) {$v=$};
\node at (2.75, 3)   (a) {$+2$};
\node at (6.25, 3)   (b) {$+3$};
\node at (-.75, -.5)   (c) {$+3$};
\node at (2.75, -.5)   (d) {$+6$};
\node at (6.25, -.5)   (e) {$+6$};
\node at (9.5, -.5)   (g) {$.$};
\end{tikzpicture}
\end{center}
\end{eg}

Next we consider $n$ even. It is easy to see that by the same reasoning as for $n$ odd, we only have to check $(J^\mu_3+J^\mu_4)v=0$.

\begin{prop}\label{evencolumn}
Let $n$ be even, $\mu=(1^n)$  or $(n)$ and $\la=(a,1^b)$.
We have that $\hom(S^\mu,S^\la)=1$ is one dimensional, spanned by
\[
\varphi(z^\mu) =v=\sum_{\substack{\ttt\in\P(\la)\\ \bi^\mu=\bi^\ttt}} \De(\ttt)\vtt,
\]
if $a$ is even and zero otherwise.
\end{prop}

\begin{proof}
If $a$ is even, the proof is essentially the same as in \cref{oddcolumn}.
Assume that there exists some $\varphi'\in\hom(S^{\mu},S^\la)$ with $n$ even and $a$ odd such that $\varphi'(z^\mu)=v'=\sum_{\ttt\in\P(\la)} \ct\vtt$ and $\tp_{n-1}\in\Arm(\ttt)$.
By \cref{action}, $\psi_{n-1}\vtt= v_{s_{n-1}\ttt}$.
But it is easy to see that there exists no other $\tts\in\P(\la)$ such that $\psi_{n-1}\vs=v_{s_{n-1}\ttt}$, hence we must have that $\tp_{n-1}\in\Leg(\ttt)$.

Next, we look at the possible position of $\tp_{n-3}$ in $\ttt$:
    \begin{enumerate}
        \item[1)] $\tp_{n-3}\in\Arm(\ttt),\tp_{n-1}\in\Leg(\ttt)$: in this case $\psi_{n-3}\vtt=v_{s_{n-3}\ttt}$;
        \item[2)] $\tp_{n-3},\tp_{n-1}\in\Leg(\ttt)$: in this case $\psi_{n-3}\vtt=0$ by \cref{action}.
    \end{enumerate}
We see that if $n-3\in\Arm(\ttt)$, then $\psi_{n-3}\vtt\neq0$ by case 1) above.
As $\tp_{n-1}\in\Leg(\ttt)$ for all $\ttt$ involved in $v$, it is clear that $\psi_{n-3}v\neq0$.
Thus only case 2) is possible.
We can use the same reasoning for all $3\le k<n-3$ and conclude that if $n$ is even and $a$ is odd then $\la=(1^n)$ or $(n)$, that is $S^\mu\cong S^\la$ and we only have the trivial homomorphism.
\end{proof}

\begin{eg}
Let $\mu=(1^8)$ and $\la=(4,1^4)$. Then $\hom(S^{(1^8)},S^{(4,1^4)})=1$ where $\varphi(z^\mu)= v$ and
\begin{center}
\begin{tikzpicture}
\tgyoung(0cm,3.5cm,1678,2,3,4,5)
\tyoung(3.5cm,3.5cm,1458,2,3,6,7)
\tyoung(7cm,3.5cm,1238,4,5,6,7)
\node at (-.75, 3)   (f) {$v=$};
\node at (2.75, 3)   (a) {$+2$};
\node at (6.25, 3)   (b) {$+3$};
\node at (9.5, 3)   (g) {$.$};
\end{tikzpicture}
\end{center}
\end{eg}

\subsubsection{\texorpdfstring{$\mu=(c,1^d)$ and $c>a$}{mu hook and c>a}} \label{sec:large-c}

The strategy is very similar to that of \cref{column}.
Using \cref{arm-odd,leg-odd} and the residue conditions, we see that if $(J^\mu_1+J^\mu_2)v=0$ then
\[
v=\sum_{\substack{\ttt\in\P(\la)\\\bi^\mu=\bi^\ttt}}\ct\vtt.
\]
Observe that in this case there is one extra Garnir relation $\psi_Gz^\mu=\psi_1\psi_2\dots\psi_az^\mu=0$ coming from the Garnir node $\ttt(1,1)$, but since $\psi_1$ appears on the left, $\psi_G\vtt=0$ for all $\ttt\in\sfz(\la)_\cale$ by \cref{action}.
Hence to prove that $\varphi(z^\mu)=v$ for $\varphi\in\hom(S^\mu,S^\la)\neq0$, we only need to show that $\psi_kv=0$ for all $2\le k \le n-1$ as this will ensure that $(J^\mu_3+J^\mu_4)v=0$.
Using \cref{action}, we see that the only possible way that some even $r$ is unpaired in $\Arm(\ttt)$ is if $\psi_rz^\mu\neq0$.

We need to consider the parities of $n,c$ and $a$ separately.

\begin{prop}\label{prop:large-c}
Let $\mu=(c,1^d)$ and $\la=(a,1^b)$ be two partitions of $n$ with $c>a$.
Then $\hom(S^\mu,S^\la)$ is at least one dimensional if
\begin{itemize}
    \item[i)] $n$ is odd and $c\equiv_2 a$, or
    \item[ii)]  $n$ is even,
\end{itemize}
and zero otherwise.
\end{prop}

\begin{proof}
Define $\bar\la:=(a,1^{c-a})$.
First, we are going to construct a homomorphism $\bar\varphi(z^{(c)})=\bar v$ with $\bar v$ summing over elements of the set $\{\bar\ttt\in\sfz(\bar\la)_\cale\mid \bi^{(c)}=\bi^{\bar\ttt}\}$ and then extend this to a homomorphism from $S^\mu$ to $S^\la$ by placing the entries $c+1,c+2,\dots,n$ in $\Leg(\bar\ttt)$.

 \begin{enumerate}
    \item Odd $n$
    
    Just as in \cref{oddcolumn}, we see that if $n$ is odd and $c\not\equiv_2 a$, then $S^\mu$ and $S^\la$ are in different blocks.
    Thus we assume that $c\equiv_2 a$.
    \begin{itemize}
        \item[a)] $a$ odd
        
    In this case $c$ is also odd and we take
    \[
        \bar v=\sum_{\substack{\bar\ttt\in\sfz(\bar\la)_\varnothing\\ \bi^\mu=\bi^{\bar\ttt}}}\De(\bar\ttt)\vtb.
    \]
    \item[b)] $a$ even
    
    Here $c$ is even and we take
    \[
        \bar v=\sum_{\substack{\bar\ttt\in\sfz(\bar\la)_{\{c\}}\\ \bi^\mu=\bi^{\bar\ttt}}}\De(\bar\ttt)\vtb.
    \]
    \end{itemize}
    \item Even $n$
    
      In this case, the residue conditions allow $c$ to be either odd or even, depending on where $n$ is placed. 
    \begin{itemize}
        \item[a)] $a$ odd
        
        If $c$ is odd, we construct $\bar v$ as described in case 1.
        If $c$ is even, $\bar v$ is the same as in 1b) (note that $n\in\Arm(\bar\ttt)$ for all $\bar \ttt$ appearing in the sum).
        \item[b)] $a$ even
        
        If $c$ is odd, then $\bar v$ agrees with the one in 1a) (observe that $n\in\Arm(\bar\ttt)$ for all $\bar \ttt$ appearing in the sum).
        Finally, if $c$ is even, $\bar v$ is the same as in 1b).
        \end{itemize}
\end{enumerate}

We extend each $\bar\ttt$ above to $\ttt\in\std(\la)$ by adjoining $c+1,c+2,\dots,n$ to $\Leg(\bar\ttt)$, giving terms $v$.
By \cref{action}, it is clear then that $\psi_r v=0=y_s v$ for all $c<r< n$ and $c<s \le n$ and by \cref{oddcolumn}, we have that $J^\mu v=0$.
\end{proof}

\begin{rem}
The homomorphisms $\varphi$ described in \cref{prop:large-c} might not be unique, indeed if $n$ is odd and $c\equiv_2 a$ or $n$ is even and $c\not\equiv_2 a$, then by \cref{oddcolumn,evencolumn} we see that if $\varphi'\in\hom(S^{(n)},S^\la)$ and $\varphi'(z^{(n)})=v'$, then there exists some $\varphi\in\hom(S^\mu,S^\la)$ such that $\varphi(z^\mu)=v'$.
As $c>a$, it is clear that $v\neq v'$ and we have that $\dim\hom(S^\mu,S^\la)\ge2$.

However, if $n$ is even and $c\equiv_2 a$, then $n\not\in\Arm(\ttt)$ and by a similar reasoning as in \cref{evencolumn}, we see that the entries $c+1,\dots,n$ must all be placed in $\Leg(\ttt)$.
As $\dim\hom(S^{(c)},S^\la)=1$, this implies that $\hom(S^{\mu},S^\la)$ is also one dimensional in this case.
\end{rem}

\begin{eg}
    Let $n>7$.
    If $\mu=(7,1^{n-7})$ and $\la=(3,1^{n-3})$, then $\dim\hom(S^\mu,S^\la)\ge2$ and the homomorphism described in \cref{prop:large-c} case 1a) is drawn below where we omitted all entries greater than $7$.
    \begin{center}
\begin{tikzpicture}
\tgyoung(0cm,3.5cm,167,2,3,4,5,'12\vdts)
\tgyoung(3.5cm,3.5cm,145,2,3,6,7,'12\vdts)
\tgyoung(7cm,3.5cm,123,4,5,6,7,'12\vdts)
\node at (-.75, 3)   (f) {$v=$};
\node at (2.75, 3)   (a) {$+2$};
\node at (6.25, 3)   (b) {$+3$};
\end{tikzpicture}
\end{center}
\end{eg}

\subsubsection{\texorpdfstring{$\mu=(c,1^d)$ and $c=a$ or $a-1$}{mu=la}}

In \cite{ls14}, the author obtained a generalised eigenspace decomposition for $S^\la$ for hook partitions of shape $\la$ in the case when $n$ is odd.
In doing so, he obtained an upper bound on $\dim\End_{R^\La_n}(S^\la)$ and constructed an endomorphism for odd $n$.
Here we summarise these results and extend for the case when $n$ is even and $\mu=(a-1,1^{b+1})$.

\begin{prop}[\cite{thesis} Lemma~3.28 and \cite{ls14}, Thm.~4.3]
\begin{itemize}
    \item[i)] Let $n$ and $a$ both be odd and take $\la=(a,1^b)$.
    Then
    \begin{itemize}
        \item[$\bullet$] $\dim\End_{R^\La_n}(S^\la)\le 1+b/2$ if $b<n/2$, and
        \item[$\bullet$] $\dim\End_{R^\La_n}(S^\la)\le 1+(a-1)/2$ if $b\ge n/2$.
    \end{itemize}
    \item[ii)] If $n$ is even, then $\dim\End_{R^\La_n}(S^\la)=1$.
\end{itemize}
\end{prop}

\begin{defn}[cf.~Def.~6.1,~\cite{ls14}]
For $n$ odd, let $k$ and $r$ be odd integers  with $3\le k\le a<r\le n$.
We will denote by $\ttt_{k,r}$ the tableau with  $\{\tp_3,\tp_5,\dots,\tp_a,\tp_r\}\setminus\{\tp_k\}$ in the arm.
If $n$ is even, we choose odd integers $k$ and $r$ such that $3\le k\le a<r\le n-1$ and $\ttt_{k,r}$ denotes the tableau with even arm length where the pairs $\{\tp_3,\tp_5,\dots,\tp_a,\tp_r\}\setminus\{\tp_k\}$ and $n$ are in the arm.
\end{defn}

\begin{eg}
If $\lambda=(7,1^4)$ then
\[\ttt_{5,9}=\young(1236789,4,5,<10>,<11>)
\qquad
\text{and}
\qquad
\ttt_{7,11}=\young(12345<10><11>,6,7,8,9).
\]
\end{eg}

\begin{propc}{ls14}{cf.~Prop.~6.2}\label{prop:endo}
If $n$ and $a$ are both odd and $\mu=\la$ or if $n$ is even and $c=a-1$ is odd, then there exists a homomorphism $\varphi\in\hom(S^\mu,S^\la)$ given by $\varphi(z^\mu)=v$ where
\[
v=\sum_{\substack{3\le k\le a\\ a+2\le r\le n\\k,r \text{ odd}}}\mfrac{k-1}{2}\cdot\mfrac{n+2-r}{2}v_{\ttt_{k,r}}.
\]
\end{propc}

\begin{proof}
If $n$ is odd, the result follows by \cite[Prop.~6.2]{ls14}.
If $n$ is even, by the same reasoning as in the proof of \cref{evencolumn}, we must have that $n\in\Arm(\ttt_{k,r})$.
\end{proof}

\begin{eg}
If $\mu=(7,1^5)$ and $\la=(8,1^4)$, then we have a homomorphism from $S^\mu$ to $S^\la$ given by
\[
v=v_{\ttt_{5,9}}+2\;v_{\ttt_{3,9}}+2\;v_{\ttt_{5,11}}+4\;v_{\ttt_{5,9}}+3\;v_{\ttt_{7,11}}+6\;v_{\ttt_{7,9}}.
\]
\end{eg}

\subsubsection{\texorpdfstring{$\mu=(c,1^d)$ and $c\le a$}{mu hook and c<a}}

The next statement follows from \cref{oddcolumn}.

\begin{prop}\label{prop:n-less}
If $n$ is odd and $c\le a\le n$, then $\hom(S^\mu,S^\la)$ is at least one dimensional if $c\equiv_2 a$ and zero otherwise.
\end{prop}

\begin{proof}
    The parity restrictions follow by the same reasoning as in \cref{prop:large-c}.
    
    Assume that $c$ and $a$ are both odd.
    In \cref{oddcolumn} we have already constructed $\varphi(z^{(1^n)})=v$.
    Taking the same $v$ here, we already know that $\psi_kv=0$ for all $2\le k \le n-1$, so the condition $\psi_rv=0$ for $2\le r\neq c \le n$ is immediately satisfied.
    
    If $c$ and $a$ are both even, by \cref{y-action} we have that
    \[
    \varphi(z^\mu)=v=\sum_{\substack{\ttt\in\sfz(\la)_{\{c\}}\\ \bi^\mu=\bi^\ttt}}\ct\vtt.
    \]
    Using \cref{oddcolumn} again, we can construct $v$ in the same way by setting $\ct=\De(\ttt)$.
    By \cref{prop:main}, $\psi_{c-1}\vtt=0$ for all $\ttt\in\sfz(\la)_{\{c\}}$ and so by \cref{action,cor:coeff}, $\psi_kv=0$ for all $2\le k\neq c \le n-1$ and we proved the statement.
\end{proof}

\begin{prop}\label{even-parity}
If $n$ is even and $c<a$, then $\hom(S^\mu,S^\la)$ is at least one dimensional if $c\not\equiv_2 a$ and zero otherwise.
\end{prop}

\begin{proof}
    If $n$ is even and $c\equiv_2 a$, then we would have
    \[
    \varphi(z^\mu)=v=\sum_{\substack{\ttt\in\sfz(\la)_\cale\\ \bi^\mu=\bi^\ttt}}\ct\vtt.
    \]
    By a similar argument as in \cref{evencolumn}, this is impossible and we must have that $n\in\Arm(\ttt)$, but then $c\not\equiv_2 a$, a contradiction.
    
    If $c$ is even and $a$ is odd, we see that $\cale=\{c\}$ and if $c$ is odd and $a$ is even, then $\cale=\varnothing$.
    Using a similar reasoning as in the proof of \cref{prop:n-less}, we see that $J^\mu v=0$.
\end{proof}

We summarise the results from this section in the following statement.

\begin{thm}\label{hook-result}
Let $\mu=(c,1^d)$ and $\la=(a,1^b)$. Then $\hom(S^\mu,S^\la)$ is at least one dimensional if
\begin{itemize}
    \item[i)] $n$ is odd and $c\equiv_2 a$, or
    \item[ii)] $n$ is even, unless $c\equiv_2 a$ and $c< a$.
\end{itemize}
Moreover, if $c>a$ and $c+n\not\equiv_2 a$, then $\dim\hom(S^\mu,S^\la)\ge2$.
\end{thm}

\begin{rem}
    In particular, if $c=a$ is odd and $a\ge5,b\ge4$, then $\dim\hom(S^\mu,S^\la)\ge 3$, as the non-trivial homomorphisms in \cref{prop:endo,prop:n-less} are clearly distinct.
\end{rem}

\subsection{Homomorphisms from non-hooks}\label{subs:nonhooks}

In this section, we first use \cref{y-action} to describe homomorphisms from $S^\mu$ to $S^\la$ where $\mu$ is a two-column partition.
Combining this with our findings from the previous sections, we obtain a generalised version of \cref{hook-result} which describes $\hom(S^\mu,S^\la)$ in the case when $\mu$ is a \emph{two-column hook partition}, a partition of the form $\mu=(c,2^s,1^d)$ (see \cref{thm:main}).

\subsubsection{Homomorphisms from two-column partitions}

\begin{thm}\label{thm:two-column}
Let $\mu=(2^r,1^s)$ with $r\ge 1$ and $s\ge 0$ be a partition of $n$.
Then $\hom(S^\mu,S^\la)$ is one dimensional, spanned by
\[
    \varphi(z^\mu)=v=\sum_{\substack{\ttt\in\sfz(\la)_{\cale}\\ \bi^\mu=\bi^\ttt\\ \cale=\{2,4,\dots,2r\}}}\De(\ttt)\vtt,
\]
if $r+1\le a\le n-r$, and either
\begin{itemize}
    \item $n+a\equiv_2 r$, or
    \item $n$ is even and $a=r+1$,
\end{itemize}
and zero otherwise.
\end{thm}

\begin{proof}
    By \cref{y-action}, if $\vtt$ is involved in $v$ and $x=\ttt^\mu(y,2)$, then $x\in\Arm(\ttt)$ for all $2\le y\le r$.
    Thus $(J^\mu_1+J^\mu_2)v=0$ if and only if
    \[
    v=\sum_{\substack{\ttt\in\sfz(\la)_{\cale}\\ \bi^\mu=\bi^\ttt}}\ct\vtt
    \]
    where $\cale=\{2,4,\dots,2r\}$ and so $r+1\le a\le n-r$.
    
    Similar to the proof of \cref{prop:large-c}, setting $\bar\la=(a-r,1^{b-r})$, we can construct
    \[
    \bar\varphi(z^{(s)})=\bar v=\sum_{\substack{\bar\ttt\in\sfz(\bar\la)_{\varnothing}\\ \bi^\mu=\bi^{\bar \ttt}}}\De(\bar\ttt)\vtb.
    \]
    For each $\bar\ttt$ above, we define $\ttt\in\std(\la)$, by increasing every entry in $\bar \ttt$ by $2r$ and placing the entries in $\cale$ at the beginning of $\Arm(\ttt)$ and $\{3,5,\dots,2r+1\}$ at beginning of $\Leg(\ttt)$.
    Then we obtain $v$ such that $y_sv=0$ for all $1\le s\le n$ and $\psi_k v=0$ for all $k\in\{3,5,\dots,2r-1\}$ and $2r+1\le k \le n-1$, hence $(J_1^\mu+J_2^\mu+J_3^\mu)v=0$.
    Next, we look at the Garnir relations.
    There are three different non-trivial relations:
    \begin{itemize}
        \item[1)] $G_1z^\mu=\psi_{k}\psi_{k+1}z^\mu=0$ if $k=\ttt^\mu(y,1)$ for $k$ odd and $1\le y \le r$,
        \item[2)] $G_2z^\mu=\psi_{k+1}\psi_{k}z^\mu=0$ if $k=\ttt^\mu(y,2)$ for $k$ even and $1\le y \le r$, and
        \item[3)] $\psi_kz^\mu=0$ for all $2r+1\le k\le n-1$.
    \end{itemize}
    By \cref{action}, we see that $G_1\vtt=0=G_2\vtt$ and we already proved case 3) above, so $J^\mu v=0$.
    It is clear that if $n$ is odd, then so is $s$.
    If $n\equiv_2 s\not\equiv_2 (a-r)$, then $S^{(s)}$ and $S^{\bar\la}$ lie in different blocks and $\dim\hom(S^{(s)},S^{\bar\la})=0$.
    Hence if $\dim\hom(S^\mu,S^\la)\neq0$, then $n+(a-r)+r=n+a\equiv_2 r$.
    If $n$ is even, using a similar argument as in \cref{evencolumn}, we must have that $n\in\Arm(\ttt)$, which implies $\dim\hom(S^\mu,S^\la)\neq0$ then $n+a\equiv_2r$.
    Moreover, if $n$ is even, then it is easy to see that there exists $v=\vtt$ where $\Arm(\ttt)$ only contains $\cale$, that is $a=r+1$.
    
    Finally, we prove that $\dim\hom(S^\mu,S^\la)=1$.
    Indeed, if $s\le1$, then $\mid\ttt\in\sfz(\la)_\cale\mid=1$ and $v=\vtt$ where $\Arm(\ttt)$ contains only entries of $\cale$.
    If $s>1$, observe that $\bar\varphi\in\hom(S^{(s)},S^{(a-r,b-r)})$, so $\bar v$ is unique up to scalar multiplication.
    Using a similar argument as in the remark after \cref{prop:large-c}, we see that $\dim\hom(S^\mu,S^\la)=1$.
\end{proof}

\subsubsection{Homomorphisms from two-column hook partitions}

\begin{thm}\label{thm:2chook}
Let $\mu=(c,2^s,1^d)$ with $c$ even.
Then $\hom(S^\mu,S^\la)$ is at least one dimensional if $a\ge s+2$; and either
\begin{itemize}
    \item $n$ is odd, $a\equiv_2 s$ and $s+2\le a\le n-(s+1)$, or
    \item $n$ is even and $s+2\le a\le c+s$ or $a=c+s+1+2y$ for $0\le y\le d/2-2$,
\end{itemize}
and zero otherwise.
\end{thm}

\begin{proof}
    Using the same method as in \cref{thm:two-column}, it is easy to see that we extend $v$ from \cref{prop:n-less,even-parity}, which is a homomorphism from $S^{(c,1^d)}$ to $S^{(a-s,1^{b-s})}$, to a homomorphism $\varphi(z^\mu)=v'$, where $\varphi\in\hom(S^\mu,S^\la)$, by increasing every entry greater than $c+1$ by $2s$ and placing the entries $\cale=\{c+2,c+4,\dots,c+2s\}$ in the arm and the entries $\{c+3,c+5,\dots,c+2s+1\}$ in the leg.
\end{proof}

\begin{eg}
    Let $\mu=(4,2^2,1^5)$ and $\la=(8,1^5)$.
    By \cref{prop:n-less}, there exists $\varphi(z^{(4,1^5)})=v$ as below where $\varphi\in\hom(S^{(4,1^5)},S^{(6,1^3)})$.
    \begin{center}
\begin{tikzpicture}
\tgyoung(0cm,0cm,145689,2,3,7)
\tgyoung(5cm,0cm,123689,4,5,7)
\tgyoung(10cm,0cm,123456,7,8,9)
\node at (-.75, -.5)   (f) {$v=$};
\node at (4.5, -.5)   (a) {$+2$};
\node at (9.5, -.5)   (b) {$+3$};
\end{tikzpicture}
\end{center}
    Then we can extend this a homomorphism $v'$ from $S^\mu$ to $S^\la$.
        \begin{center}
\begin{tikzpicture}
\tgyoung(0cm,0cm,1468\ten\eleven<12><13>,2,3,5,7,9)
\tgyoung(5cm,0cm,123468<12><13>,5,7,9,\ten,\eleven)
\tgyoung(10cm,0cm,123468\ten\eleven,5,7,9,<12>,<13>)
\node at (-.75, -.5)   (f) {$v'=$};
\node at (4.5, -.5)   (a) {$+2$};
\node at (9.5, -.5)   (b) {$+3$};
\end{tikzpicture}
\end{center}
\end{eg}

Combining our results from \cref{hook-result,thm:2chook}, we have just proved our main theorem:

\begin{thm}[Main Theorem]\label{thm:main}
  Let $\la=(a,1^b)$ and $\mu=(c,2^s,1^d)$ where if $s>0$ then $c$ is even.
  Then $\hom(S^\mu,S^\la)$ is at least one dimensional if for $c$ even, we have that $a\ge s+2$ and either
    \begin{itemize}
        \item $n$ is odd, $a\equiv_2 s$ and $s+1\le a\le n-(s+1)$, or
        \item $n$ is even and $s+2\le a\le c+s$ or $a=c+s+1+2y$ for $0\le y\le d/2-2$,
\end{itemize}
and zero otherwise.
\end{thm}

\begin{rmk}
We have a more precise description of the dimensions of $\hom(S^\mu,S^\la)$ in the following cases:
\begin{itemize}
    \item[1)] $\hom(S^\mu,S^\la)$ is at least three dimensional if $c=a$ is odd and $a\ge5,b\ge4$,
    \item[2)] $\hom(S^\mu,S^\la)$ is at least two dimensional if $c>a$ (unless we are in case 3b) below),
    \item[3)] $\hom(S^\mu,S^\la)$ is one dimensional if
    \begin{itemize}
        \item[a)] $n$ is even, $c\equiv_2 a$ and $c\ge a$, or
        \item[b)] if $c=2$ (i.e.\ $\mu=(2^{s+1},1^d)$).
    \end{itemize}
\end{itemize}

   If $s=1$ and $a=c+1$, then we recover the known Carter-Payne homomorphism of degree $1$ and can explicitly describe it by $z^\mu\mapsto \vtt$, where $\Arm(\ttt)$ consists of $2,3,\dots,c,c+2$.
   For more details, we refer the reader to \cite[Thm.~3.12]{lm14}.
\end{rmk}

\begin{rem}
    In this final remark, we give the graded the dimension of the homomorphism from \cref{thm:main}.
    The following is easy to verify:
    \begin{itemize} 
        \item[1)] If $\mu=(c,1^d)$, then
        \begin{itemize}
            \item[$\bullet$] $\deg(\ttt^\mu)=(c-1)/2$ if $c$ is odd, and 
            \item[$\bullet$] $\deg(\ttt^\mu)=c/2$ if $c$ is even.
        \end{itemize}
        \item[2)] If $\mu=(c,2^s,1^d)$ with $c$ even, then $\deg(\ttt^\mu)=c/2+s$.
    \end{itemize}
    Now we only need to compute $\deg(\ttt)-\deg(\ttt^\mu)=:\cald$ for all cases  in \cref{thm:main} where $\ttt$ has shape $\la=(a,1^b)$.
    Assume that $\mu=(c,1^d)$.
    \begin{enumerate}
        \item[1)] If $n$ is odd, then $\cald=(a-c)/2$.
        \item[2)] If $n$ is even, then
        \begin{itemize}
            \item[$\bullet$] $\cald=(a-c)/2$ if $c\equiv_2 a$, 
            \item[$\bullet$] $\cald=(a-c+1)/2$ if $c$ is odd and $a$ is even, and
            \item[$\bullet$] $\cald=(a-c+3)/2$ if $c$ is even and $a$ is odd.
        \end{itemize}
    \end{enumerate}
    Finally, let $\mu=(c,2^s,1^d)$ with $c$ even and $s\ge 1$.
    \begin{enumerate}
        \item[1)] If $n$ is odd, then
        \begin{itemize}
            \item[$\bullet$] $\cald=(a-c)/2$ if $s$ is even, and
            \item[$\bullet$] $\cald=(a-c+1)/2$ if $s$ is odd.
        \end{itemize}
        \item[2)] If $n$ is even, then
        \begin{itemize}
            \item[$\bullet$] $\cald=(a-c+2)/2$ if $s$ is odd and $a$ is even, 
            \item[$\bullet$] $\cald=(a-c+1)/2$ if both $s$ and $a$ are odd,
            \item[$\bullet$] $\cald=(a-c)/2$ if $s$ and $a$ are both even, and
            \item[$\bullet$] $\cald=(a-c+3)/2$ if $s$ is even and $a$ is odd. 
        \end{itemize}
    \end{enumerate}
    Then we have that 
    \[
    \dim_q\hom(S^\mu,S^{\la})=q^\cald,
    \]
    where $\hom(S^\mu,S^{\la})$ is as in \cref{thm:main}.
\end{rem}

\subsection{The trivial module}

In the following section, we obtain a graded version of James's well-known result for the trivial module $S^{(n)}$ (see \cite[Thm.~24.4]{James}).
First, we need to recall some elementary results about binomial coefficients which will be used in the statement of \cref{trivial}.
Let $p$ be a fixed prime.

\begin{defn}
  For $n \in \bbz_{>0}$, we define $\ell_p(n) = \min \{ i \mid p^i > n \}$.
  We also set $\ell_p(0) = -\infty$.
\end{defn}

There are two one dimensional $S^\la$ modules, indexed by the partitions $\la=(1^n)$ and $(n)$.
The module $S^{(n)}$ is called the \emph{trivial module} and $S^{(1^n)}$ is called the \emph{sign module}.
As $e=2$, we have that $S^{(n)}\cong S^{(1^n)}$ up to grading shift.

\begin{thm}\label{trivial}
Let $\mu=(\mu_1,\mu_2,\dots,\mu_Z)$ be a partition of $n$.
Then $\hom(S^\mu,S^{(1^n)})$ is one dimensional if
\begin{itemize}
    \item[i)] $\Ch\bbf=0,\mu_1$ is odd and $\mu=(\mu_1,1^{n-\mu_1})$ or
    \item[ii)] $\Ch\bbf=p>0$ and
    \begin{itemize}
    \item[$\bullet$] for all $1\le i\le Z-1$, $\mu_i$ is odd, and
    \item[$\bullet$]  if $n$ is odd, then $\mu_Z$ is odd, and
    \item[$\bullet$]  $\mu_i\equiv -1 \;\md\;2p^{\ell_p(a{i+1})}$ for all $1\le i< Z$ where $a_i=\lfloor \mu_i/2\rfloor$,
\end{itemize}
\end{itemize}
and zero otherwise.
In this case
\[
\dim_q\hom(S^\mu,S^{\la})=q^{-(\sum_1^Z a_i)}.
\]
\end{thm}

\begin{proof}
The first part of the statement follows by \cite[Theorem~3.3]{Lyle07}.
To calculate the degree, notice that $\deg(\ttt^\mu)=\sum_1^Z a_i$.
Hence
\[
\deg(\ttt^{(1^n)})-\deg(\ttt^\mu)=-\sum_1^Z a_i.\qedhere
\]
\end{proof}

\begin{eg}
Let $\mu=(7,3,2)$. If $p=2$, then $\dim\hom(S^\mu,S^{(12)})=1$.
\begin{center}
\begin{tikzpicture}
\tyoung(0cm,1.5cm,1234567,89<10>,\eleven<12>)
\tyoung(6cm,3cm,1,2,3,4,5,6,7,8,9,<10>,\eleven,<12>)
\draw[->] (4,1) -- (5,1);
\end{tikzpicture}
\end{center}
We can easily compute the graded dimension by calculating the degrees of $\ttt^\mu$ and $\ttt^\la$:
\[
\deg\ttt^\la-\deg\ttt^{\mu}=0-5=-5.
\]
Hence $\dim_q\hom(S^{\mu},S^{(1^{12})})=q^{-5}$.
\end{eg}

\begin{rem}
    If $\la=(n)$ where depending on the parity $n=2d$ or $2d+1$, then $\deg(\ttt^\la)=d$.
    Thus
    \[
    \deg(\ttt^{(n)})-\deg(\ttt^\mu)=d-\sum_1^Z a_i.
    \]
\end{rem}

\subsection{Further questions remaining}

Our strategy relied upon reducing more complex shapes (without complicated Garnir relations) into single rows or columns and then applying \cref{oddcolumn}.
In this paper, we gave a complete description of $\hom(S^\mu,S^\la)$ where $\mu$ can be visualised as a hook partition with a rectangle of width two in the middle.
The first row having even parity and no other rows having length more than two eliminated complicated Garnir relations.
We would like to emphasise that shapes of $\mu$ with non-trivial Garnir relations seem to require an entirely different treatment as there are more relations to consider and the combinatorics governing the coefficients of the tableaux in the homomorphism space is different.
As we see in \cref{def:coeff}, the $\De$ coefficients are all products of positive integers.
One of the main reasons that prevented us to further continue our research into other shapes of $\mu$ was the fact that even traditionally ``well-behaving" shapes, such as two-part partitions, give rise to homomorphisms which have large prime numbers as coefficients.

Nevertheless, using results from \cite{dg15}, we can present a simple proof for the case when $n$ is odd and $\mu$ is a two-row partition, however we note that this method does not yield an explicit description of the homomorphism or the degree.

\begin{lem}\label{lem:tworows}
Assume that $n$ is odd and let $\mu=(\mu_1,\mu_2)$ and $\la=(a,1^b)$ be two partitions of $n$ with $\mu_1\ge\mu_2\ge2$.
If $\Ch(\bbf)=2$, then $\hom(S^\mu,S^\la)$ is one dimensional if and only if $\mu_2\le a-1\le \mu_1$ if and only if $\mu_2\le b+1\le \mu_1$.
\end{lem}

\begin{proof}
    We note that our Specht modules are dual to those defined by James, so that $(S^\la)^\circledast$ denotes the Specht modules indexed by $\la$ as in \cite{James}.

    Let $e=2$ and $\Ch\bbf=p=2$, so $R^\La_n\cong\bbf\fkS_n$.
    We write $\bbf\fkS_n$ throughout the proof, but all necessary Hecke algebra analogues can be found in \cite{dg20}, except for the Littlewood--Richardson rule.
    Although it holds in the Grothendieck group, we cannot find a Hecke analogue of \cite[Thm.~5.3]{jp79} in the literature, however it is almost certainly true. 

    We want to understand
    \[
        \Hom_{\bbf\fkS_n}(S^{(c+d,c)},S^{(a,1^b)})\cong\Hom_{\bbf\fkS_n}((S^{(b+1,1^{a-1})^\circledast},(S^{(2^c,1^d)})^\circledast).
    \]
    We may assume that $(S^{(a,1^b)})^\circledast$ and $(S^{(2^c,1^d)})^\circledast$ lie in the same block, as otherwise there exists no homomorphisms between them.
    As $n$ is odd, we have that
    \[
    M(a,b)=(S^{(a,1^b)})^\circledast\oplus (S^{(a+1,1^{b-1})})^\circledast
    \]
    since $(S^{(a,1^b)})^\circledast$ and $(S^{(a+1,1^{b-1})})^\circledast$ are in different blocks (see \cite[Prop.~7.1.1]{dg20}).
    
    Set $\tau=(2^c,1^d)$.
    It follows that $(S^{(a+1,1^{b-1})})^\circledast$ and $(S^{\tau})^\circledast$ also lie in different blocks and so
    \begin{align*}
        \Hom_{\bbf\fkS_n}(M(a,b),(S^\tau)^\circledast)&\cong \Hom_{\bbf\fkS_n}((S^{(a,1^b)})^\circledast,(S^\tau)^\circledast)\oplus\underbrace{\Hom_{\bbf\fkS_n} (S^{(a+1,1^{b-1})})^\circledast,(S^\tau)^\circledast)}_{=0}\\[.5em]
        &\cong\Hom_{\bbf\fkS_n}((S^{(a,1^b)})^\circledast,(S^\tau)^\circledast)
    \end{align*}
    We also have that
    \[
     \Hom_{\bbf\fkS_n}(M(a,b),(S^\tau)^\circledast)\cong\Hom_{\bbf(\fkS_a\times\fkS_b)}(1_{\fkS_a\times\fkS_b},(S^\tau)^\circledast\downarrow_{\fkS_a\times\fkS_b}).
    \]

    By \cite[Lemma~1.3.9.(i)]{dg15}, $(S^\tau)^\circledast\downarrow_{\fkS_a\times\fkS_b}$ has a filtration with factors isomorphic to $(S^\nu)^\circledast\otimes (S^{(\tau/\nu)})^\circledast$ such that $\nu\subseteq\tau$.
    Furthermore, each submodule $(S^{(\tau/\nu)})^\circledast$ has a Specht filtration and these factors $(S^{\xi})^\circledast$ are given by the Littlewood--Richardson rule \cite[Lemma~1.3.9.(ii)]{dg15}.
    It follows that
    \[
    \dim\Hom_{\bbf\fkS_n}((S^{(a,1^b)})^\circledast,(S^\tau)^\circledast)\le \sum_{\substack{\nu\vdash a\\\mu\vdash b}}\dim\Hom_{\bbf\fkS_a\times\bbf\fkS_b}(1_{\fkS_a\times\fkS_b},(S^\nu)^\circledast\otimes(S^\mu)^\circledast).
    \]

    By Kunneth's formula, we also have that
    \[
    \Hom_{\bbf\fkS_a\times\bbf\fkS_b}(1_{\fkS_a\times\fkS_b},(S^\nu)^\circledast\otimes(S^\mu)^\circledast)\cong \Hom_{\bbf\fkS_a}(1_{\fkS_a},(S^\nu)^\circledast)\otimes\Hom_{\bbf\fkS_b}(1_{\fkS_b},(S^\mu)^\circledast)
    \]

    Recall that $\nu_1,\mu_1\le2$ and so by James's Theorem (\cite[Theorem~24.4]{James}), we have that
    \[
    \Hom_{\bbf\fkS_a}(1_{\fkS_a},(S^\nu)^\circledast)\cong
    \begin{cases}
        \bbf\qquad\qquad\text{if }\nu=(2)\text{ or } \nu=(1^a),\\
        0\qquad\qquad\text{otherwise,}
        \end{cases}
    \]
    and similarly
    \[
    \Hom_{\bbf\fkS_b}(1_{\fkS_b},(S^\mu)^\circledast)\cong
    \begin{cases}
        \bbf\qquad\qquad\text{if }\mu=(2)\text{ or } \mu=(1^b),\\
        0\qquad\qquad\text{otherwise.}
        \end{cases}
    \]

    First, assume that $\tau=(2^c,1^d)\neq(2,1^b)$.
    It is easy to see that if $a< c$ or $a> c+d$, then either $\mu$ or $\nu$ cannot be as required by James's Theorem and so
    \[
    \Hom_{\bbf\fkS_a}(1_{\fkS_a},(S^\nu)^\circledast)\otimes\Hom_{\bbf\fkS_b}(1_{\fkS_b},(S^\mu)^\circledast)=0
    \]
    for all $\mu,\nu$, which gives
    \[
    \Hom_{\bbf\fkS_n}((S^{(a,1^b)})^\circledast,(S^\tau)^\circledast)=0.
    \]
    
    On the other hand, if $c\le a\le c+d$, then we see that $(S^{(2^c,1^d)})^\circledast\downarrow_{\fkS_a\times\fkS_b}$ has $(S^{(1^a)})^\circledast\otimes(S^{(1^b)})^\circledast$ as a submodule while the other factors are of the form $(S^{\nu})^\circledast\otimes(S^{\mu})^\circledast$ and they do not satisfy the conditions in James's Theorem and thus
    \[
    \Hom_{\bbf\fkS_n}((S^{(a,1^b)})^\circledast,(S^\tau)^\circledast)\cong\bbf.
    \]

    Finally, if $\tau=(2,1^b)$, then
    \[
    \End_{\bbf\fkS_n}((S^\tau)^\circledast)\cong\bbf,
    \]
    since it is a $2$-restricted partition.

    Thus 
    \[
    \Hom_{\bbf\fkS_n}((S^{(a,1^b)})^\circledast,(S^{(2^c,1^d)})^\circledast)
    \]
    is one dimensional if $c\le a\le c+d$ and zero otherwise.
    Hence by the first isomorphism in our proof, for $\la=(a,1^b)$ and $\mu=(c+d,c)$ such that $a+b=n$ is odd, we have that $\Hom(S^\la,S^\mu)$ is one dimensional if $c\le a-1\le c+d$ and $c\le b+1\le c+d$.
\end{proof}

\begin{rem}
    We note that \cref{lem:tworows} considers ungraded homomorphisms.
    The above method can also be applied to two-column partitions, however unlike our result (\cref{thm:two-column}), it does not give an explicit description of the homomorphism.
\end{rem}